\documentclass[11pt]{article}

\usepackage{etex}
\usepackage{graphics}
\usepackage{color}
\usepackage{cancel}
\usepackage{stmaryrd}
\usepackage{amsmath}
\usepackage{amsthm}
\usepackage{amssymb}
\usepackage{proof}
\usepackage{verbatim,cmll}
\usepackage{setspace}
\usepackage{lscape}
\usepackage{latexsym,relsize}
\usepackage{amscd}
\usepackage{multicol}
\usepackage{bussproofs}
\usepackage[position=top]{subfig}
\usepackage{leqno}
\usepackage{tikz}
\usepackage{authblk}
\usepackage{multicol}
\usetikzlibrary{arrows}
\usetikzlibrary{matrix}
\usetikzlibrary{patterns}
\usetikzlibrary{shapes}

\usepackage[left=3cm,top=3cm,right=3cm,bottom=2cm]{geometry}
\newcommand{\commment}[1]{}

\def\aol{\rule[0.5865ex]{1.38ex}{0.1ex}}

\def\pdra{\mbox{$\,>\mkern-8mu\raisebox{-0.065ex}{\aol}\,$}}
\def\pdla{\mbox{\rotatebox[origin=c]{180}{$\,>\mkern-8mu\raisebox{-0.065ex}{\aol}\,$}}}

\def\mANDORatom#1{\hbox{\hbox to 0pt{$#1\TriangleUp$\hss}$#1\TriangleDown$}}

\newcommand{\mcAND}{%
\mathrel{\ooalign{\raisebox{-0.39ex}{$\mbox{\TriangleUp}$}\cr\kern4.2pt{\raisebox{-0.13ex}{$\cdot$}}}}}
\newcommand{\mcand}{%
\mathrel{\ooalign{$\vartriangle$\cr\kern1.99pt{\raisebox{-0.17ex}{$\cdot$}}}}}

\newcommand{\nAND}{%
\mathrel{\ooalign{$\mbox{\TriangleUp}$\cr\kern0pt$\mbox{\rotatebox[origin=c]{180}{\TriangleUp}}$}}}

\newcommand{\nand}{%
\mathrel{\ooalign{$\vartriangle$\cr\kern0pt$\triangledown$}}}

\newcommand{\mcBAND}{%
\mathrel{\ooalign{\raisebox{-0.39ex}{$\mbox{\FilledTriangleUp}$}\cr\kern4.2pt{\raisebox{-0.13ex}{${\color{white}\cdot}$}}}}}

\newcommand{\mcband}{%
\mathrel{\ooalign{$\blacktriangle$\cr\kern1.99pt{\raisebox{-0.17ex}{${\color{white}\cdot}$}}}}}

\newcommand{\mcRA}{%
\mathrel{\ooalign{
                  \raisebox{-0.3ex}{$\rotatebox[origin=c]{-90}{$\mbox{{\TriangleUp}}$}$}
                                                                            \cr\kern2.7pt{\raisebox{0.2ex}{$\cdot\mkern1.3mu$}}}}}

\newcommand{\mcra}{%
\mathrel{\ooalign{$\,{\vartriangleright\,}$\cr\kern3pt{\raisebox{0ex}{$\cdot$}}}}}

\newcommand{\mcraline}{%
-{\mkern-6mu{\mathrel{\ooalign{$\,{\vartriangleright\,}$\cr\kern3pt{\raisebox{0ex}{$\cdot$}}}}}}}

\newcommand{\mdraline}{%
{\mathrel{\ooalign{$\,{\vartriangleright\,}$\cr\kern3pt{\raisebox{0ex}{$\cdot$}}}}}{\mkern-6mu}-}

\newcommand{\cra}{%
\mathrel{\ooalign{$\,-{\mkern-3mu\vartriangleright\,}$\cr\kern8pt{\raisebox{0ex}{$\cdot$}}}}}

\newcommand{\mcBRA}{%
\mathrel{\ooalign{
                  \raisebox{-0.3ex}{$\rotatebox[origin=c]{-90}{$\mbox{\FilledTriangleUp}$}$}
                                                                            \cr\kern2.7pt{\raisebox{0.2ex}{${\color{white}\cdot}$}}}}}

\newcommand{\mcbra}{%
\mathrel{\ooalign{$\,-{\mkern-3mu\blacktriangleright\,}$\cr\kern8pt{\raisebox{0ex}{$\cdot$}}}}}

\newcommand{\mcLA}{%
\mathrel{\ooalign{
                  \raisebox{-0.3ex}{$\rotatebox[origin=c]{90}{$\mbox{\TriangleUp}$}$}
                                                                                     \cr\kern5.5pt{\raisebox{0.2ex}{$\cdot$}}
                                                                                                                              }}}

\newcommand{\mcla}{%
\mathrel{\ooalign{$\,{\vartriangleleft\,}$\cr\kern5pt{\raisebox{0ex}{$\cdot$}}}}}

\newcommand{\mclaline}{%
-{\mkern-6mu{\mathrel{\ooalign{$\,{\vartriangleleft\,}$\cr\kern5pt{\raisebox{0ex}{$\cdot$}}}}}}}

\newcommand{\mcBLA}{%
\mathrel{\ooalign{
                  \raisebox{-0.3ex}{$\rotatebox[origin=c]{90}{$\mbox{\FilledTriangleUp}$}$}
                                                                                     \cr\kern5.5pt{\raisebox{0.2ex}{${\color{white}\cdot}$}}
                                                                                                                                            }}}

  \numberwithin{equation}{section}
\newcommand\val[1]{{\lbrack\!\lbrack} {#1}{\rbrack\!\rbrack}}

\newcommand{\pand}{\wedge}

\newcommand{\por}{\vee}

\newcommand{\pra}{\rightarrow}

\newcommand{\pla}{\leftarrow}
\newcommand{\gI}{%
\mathrel{\ooalign{$\mbox{T}$\cr\kern0pt$\mbox{\rotatebox[origin=c]{180}{T}}$}}}

\newcommand{\gbot}{\rotatebox[origin=c]{180}{$\tau$}}

\def\aol{\rule[0.5865ex]{1.38ex}{0.1ex}}

\makeatletter
\newcommand{\WKnowProxy}[2]{%
  {\mathbin{\ooalign{$#1\circ#2 $\cr\hidewidth
   \raise.155ex\hbox{$#1{\scriptstyle{\ast}}#2$}\hidewidth\cr  }}}}
\makeatother

\makeatletter
\newcommand{\BKnowProxy}[2]{%
  {\mathbin{\ooalign{$#1\bullet#2 $\cr\hidewidth
   \raise.155ex\hbox{$#1{\scriptstyle{\color{white}{\ast}}}#2$}\hidewidth\cr  }}}}
\makeatother









\newcommand{\fns}{\footnotesize}
\newcommand{\mc}{\multicolumn}

\usepackage[geometry]{ifsym}

\usepackage{pgf}%,pgfarrows,pgfnodes,pgfautomata,pgfheaps}
\usepackage{amsmath,amssymb}
\usepackage{mathrsfs}
\usepackage{bbm}
\usepackage{verbatim}
\usepackage{epsfig}
\usepackage{graphicx}
\usepackage{color}
\usepackage{import}
\usepackage{bussproofs}
\usepackage{qtree}
\usepackage{tree-dvips}
\usepackage{lingmacros}
\usepackage{amsfonts}
\usepackage{mathdots}
\usepackage{amssymb}
\usepackage{pifont}
\usepackage[safe]{tipa}
\usepackage{newlfont}
\usepackage{latexsym}
\usepackage{multirow}
\usepackage{array,cmll}
\usepackage{textcomp}
\usepackage{accents}
\usepackage{rotating}
\usepackage{logreq}
\usepackage{hyperref}
  \usepackage{amsmath}
  \usepackage{amsthm}
  \usepackage{latexsym}
  \usepackage{amscd}
  \usepackage{amssymb}
  \usepackage{bussproofs}
  \usepackage{txfonts}

  \usepackage{tikz}
  \usepackage[position=top]{subfig}

  \usepackage{leqno}
  \usepackage{txfonts}
  \usepackage{tikz}

\def\fCenter{{\mbox{$\ \vdash\ $}}}

\EnableBpAbbreviations

\renewcommand{\epsilon}{\varepsilon}

\newcommand{\ox}{\overline{x}}

\newcommand{\oz}{\overline{z}}

\newcommand{\diamdot}{\Diamond\!\!\!\cdot\ }
\newcommand{\lhddot}{{\lhd}\!\!\cdot\ }
\newcommand{\rhddot}{{\rhd}\!\!\!\cdot\ }

\newcommand{\nomh}{\mathbf{h}}
\newcommand{\nomi}{\mathbf{i}}
\newcommand{\nomj}{\mathbf{j}}
\newcommand{\nomk}{\mathbf{k}}
\newcommand{\noml}{\mathbf{l}}
\newcommand{\cnomm}{\mathbf{m}}
\newcommand{\cnomn}{\mathbf{n}}
\newcommand{\cnomo}{\mathbf{o}}

\newcommand{\blhd}{\blacktriangleleft}
\newcommand{\brhd}{\blacktriangleright}

\newcommand{\rhu}{\rightharpoonup}
\newcommand{\lhu}{\leftharpoonup}
\newcommand{\bba}{\mathbb{A}}
\newcommand{\bbA}{\mathbb{A}}

\newcommand{\jty}{J^{\infty}}
\newcommand{\mty}{M^{\infty}}

\newcommand{\bigamp}{\mathop{\mbox{\Large \&}}}

\newcommand{\amp}{\mathop{\&}}
\newcommand{\marginnote}[1]{\marginpar{\raggedright\tiny{#1}}}

\newcommand{\Scdot}{\bigcirc\!\!\!\cdot}
\newcommand{\Sstar}{\bigcirc\!\!\!\!\star}
\newcommand{\SDiamond}{\bigcirc\!\!\!\!\!\Diamond}
\newcommand{\SBox}{\bigcirc\!\!\!\!\!\Box}
\newcommand{\Slhd}{\bigcirc\!\!\!\!\!\lhd}
\newcommand{\Srhd}{\bigcirc\!\!\!\!\rhd}
\newcommand{\SDiamondblack}{\bigcirc\!\!\!\!\Diamondblack}
\newcommand{\Sblacksquare}{\bigcirc\!\!\!\!\blacksquare}
\newcommand{\Sblacktriangleleft}{\bigcirc\!\!\!\!\!\!\blacktriangleleft}
\newcommand{\Sblacktriangleright}{\bigcirc\!\!\!\!\!\!\blacktriangleright}

\newcommand*\circled[1]{\tikz[baseline=(char.base)]{
		\node[shape=circle ,draw, minimum size=3.5mm, inner sep=0pt] (char) {#1};}}

\newcommand*\dcircled[1]{\tikz[baseline=-3pt]{
		\node[shape=circle ,draw, minimum size=3.5mm, inner sep=0pt] (char) {#1};}}

\tikzset{
	treenode/.style = {align=center, inner sep=0pt, text centered},
	Ske/.style = {treenode, ellipse, double, draw=black,
		minimum width=6pt, thick},% arbre rouge noir, noeud noir
	PIA/.style = {treenode, ellipse, black, draw=black,
		minimum width=6pt},% arbre rouge noir, noeud rouge
	Crit/.style = {treenode, rectangle, draw=black,
		minimum width=0.5em, minimum height=0.5em}% arbre rouge noir, nil
}

\theoremstyle{plain}
\newtheorem{thm}{Theorem}

\newtheorem{lem}[thm]{Lemma}
\newtheorem{cor}[thm]{Corollary}
\newtheorem{prop}[thm]{Proposition}
\newtheorem{fact}[thm]{Fact}
\newtheorem{lemma}[thm]{Lemma}

\theoremstyle{definition}
\newtheorem{dfn}[thm]{Definition}
\newtheorem{definition}[thm]{Definition}
\newtheorem{exa}[thm]{Example}
\newtheorem{example}[thm]{Example}

\newtheorem{remark}[thm]{Remark}

\title{Unified Correspondence as a Proof-Theoretic Tool}
\author[1]{Giuseppe Greco}
\author[2]{Minghui Ma}
\author[1,3]{Alessandra Palmigiano}
\author[1]{Apostolos Tzimoulis}
\author[1]{Zhiguang Zhao\footnote{The research of the first, third, fourth and fifth author has been made possible by the NWO Vidi grant 016.138.314, by the NWO Aspasia grant 015.008.054, and by a Delft Technology Fellowship awarded in 2013. The second author is supported by the China national funding of social sciences (grant no.12CZ054).}\footnote{We would like to thank Agata Ciabattoni for inviting three of the authors to a research visit to her group in Vienna. The feedback and questions we received from her and Revantha Ramanayake during this visit significantly improved the paper.}}

\affil[1]{\small Faculty of Technology, Policy and Management, Delft University of Technology, the Netherlands}
\affil[2]{\small Institute for Logic and Intelligence, Southwest University,Chongqing, China}
\affil[3]{\small Department of Pure and Applied Mathematics, University of Johannesburg, South Africa}
\date{}
\begin{document}
	\maketitle
	\begin{abstract}
		The present paper aims at establishing formal connections between correspondence phenomena, well known from the area of modal logic, and the theory of display calculi, originated by Belnap.
		
		These connections have been seminally observed and exploited by Marcus Kracht, in the context of his characterization of the modal axioms (which he calls primitive formulas) which can be effectively transformed into `analytic' structural rules of display calculi. In this context, a rule is `analytic' if adding it to a display calculus preserves Belnap's cut-elimination theorem.
		
		In recent years, the state-of-the-art in correspondence theory has been uniformly extended from classical modal logic to diverse families of nonclassical logics, ranging from (bi-)intuitionistic (modal) logics, linear, relevant and other substructural logics, to hybrid logics and mu-calculi. This generalization has given rise to a theory called unified correspondence, the most important technical tools of which are the algorithm ALBA, and the syntactic characterization of Sahlqvist-type classes of formulas and inequalities which is uniform in the setting of normal DLE-logics (logics the algebraic semantics
of which is based on bounded distributive lattices).
		
		We apply unified correspondence theory, with its tools and insights, to extend Kracht's results and prove his claims in the setting of DLE-logics.
The results of the present paper characterize the space of properly displayable DLE-logics. 
	\end{abstract}
	\tableofcontents
	
	\section{Introduction}
	
	The present paper applies the results and insights of unified correspondence theory \cite{CoGhPa13} to establish formal connections between correspondence phenomena, well known from the area of modal logic, and the theory of display calculi, introduced by Belnap \cite{Belnap}.

	\paragraph{Sahlqvist correspondence theory.}
	Sahlqvist theory \cite{Sah75} is among the most celebrated and useful results of the classical theory of modal logic, and one of the hallmarks of its success. It provides an algorithmic, syntactic identification of a class of modal formulas whose associated normal modal logics are {\em strongly complete} with respect to {\em elementary} (i.e.\ first-order definable) classes of frames.
	
	\paragraph{Unified correspondence.}
%Sahlqvist theory has a long history in normal modal logic, going back to cite{Sa75}. The Sahlqvist theorem in cite{Sa75} gives a syntactic definition of a class of modal formulas, the {\em Sahlqvist class}, each member of which defines an elementary (i.e.\ first order definable) class of frames and is canonical.
	In recent years, building on duality-theoretic insights \cite{ConPalSou12}, an encompassing perspective has emerged which has made it possible to export the state-of-the-art in Sahlqvist theory from modal logic to a wide range of logics which includes, among others, intuitionistic and distributive lattice-based (normal modal) logics \cite{CoPa10}, non-normal (regular) modal logics \cite{PaSoZh15r}, substructural logics \cite{CoPa11}, %mono modal logics \cite{FrPaSa14},
	hybrid logics \cite{ConRob}, and mu-calculus \cite{CoCr14,CFPS}.
	
	The breadth of this work has stimulated many and varied applications. Some are closely related to the core concerns of the theory itself, such as the understanding of the relationship between different methodologies for obtaining canonicity results \cite{PaSoZh14,CPZ:constructive}, or of the phenomenon of pseudo-correspondence \cite{CGPSZ14}. Other, possibly surprising applications include the dual characterizations of classes of finite lattices \cite{FrPaSa14}.
%and the identification of the syntactic shape of axioms which can be translated into structural rules of a properly displayable calculus cite{GMPTZ}.
Finally, the insights of unified correspondence theory have made it possible to determine the extent to which the Sahlqvist theory of classes of normal DLEs can be reduced to the Sahlqvist theory of normal Boolean expansions, by means of G\"{o}del-type translations \cite{CPZ:Trans}.
	These and other results have given rise to a theory called \emph{unified correspondence} \cite{CoGhPa13}.
	
	\paragraph{Tools of unified correspondence theory.}
	The most important technical tools in unified correspondence are: (a) a very general syntactic definition of the class of Sahlqvist formulas, which applies uniformly to each logical signature and is given purely in terms of the order-theoretic properties of the algebraic interpretations of the logical connectives; (b) the algorithm ALBA, which effectively computes first-order correspondents of input term-inequalities, and is guaranteed to succeed on a wide class of inequalities (the so-called \emph{inductive} inequalities) which, like the Sahlqvist class, can be defined uniformly in each mentioned signature, and which properly and significantly extends the Sahlqvist class.
	
	\paragraph{Unified correspondence and display calculi.}
	The present paper aims at applying the tools of unified correspondence to address the identification of the syntactic shape of axioms which can be translated into analytic structural rules\footnote{\emph{Analytic} rules (cf.\ Definition \ref{def:analytic}) are those which can be added to a {\em proper display calculus} (cf.\ Section \ref{PS:para:CanonicalCutElimination}) obtaining another proper display calculus.} of a display calculus, and the definition of an effective procedure for transforming axioms into such rules. In recent years, these questions have been intensely investigated in the context of various proof-theoretic formalisms (cf.\ \cite{negri2005proof,ciabattoni2008axioms,ciabattoni2009expanding,gore2011correspondence,ciabattoni2012algebraic,lellmann2013correspondence,lahav2013frame,marin2014label,lellmann2014axioms}).  Perhaps the first paper in this line of research is \cite{Kracht}, which addresses these questions in the setting of display calculi
	for basic normal modal and tense logic. Interestingly, in \cite{Kracht}, the connections between Sahlqvist theory and display calculi started to be observed, but have not been systematically explored there nor (to the knowledge of the authors) in subsequent papers in the same research line. % have been seminally observed by Marcus Kracht in cite{Kracht}, in the context of his characterization of those formulas of the language of basic modal logic (which he calls \emph{primitive formulas}) which can be effectively transformed into analytic structural rules of display calculi. %In this context,

	%In recent years, correspondence theory has been uniformly extended from classical modal logic to diverse families of nonclassical logics, ranging from (bi-)intuitionistic (modal) logics, linear, relevant and other substructural logics (Non-Distributive Logics), to hybrid logics (Conradie-robinson) and mu-calculi (mu).

	\paragraph{Contributions.}
	The two tools of unified correspondence can be put to use to generalize Kracht's transformation procedure from axioms into analytic rules. This generalization concerns more than one aspect. Firstly, in the same way in which the definitions of Sahlqvist and inductive inequalities can be given uniformly in each logical signature, the definition of primitive formulas/inequalities is introduced for any logical framework the algebraic semantics of which is based on distributive lattices with operators (these will be referred to as {\em DLE-logics}, (cf.\ Definition \ref{def:DLE:logic:general} and Footnote \ref{footnote:DLE vs DLO} for terminology). Secondly, in the context of each such logical framework, we introduce a hierarchy of subclasses of inductive inequalities, progressively extending the primitive inequalities, the largest of which is the class of so-called \emph{analytic inductive inequalities}. This is a syntactic generalization of the class of primitive formulas/inequalities. We provide an effective procedure, based on ALBA, which transforms each analytic inductive inequality into an equivalent set of analytic rules. Moreover, we show that any analytic rule can be effectively and equivalently transformed into some analytic inductive inequality. Finally, we show that any analytic rule can be effectively and equivalently transformed into one of a particularly nice shape, collectively referred to as {\em special} rules.

	\paragraph{Structure of the paper.} In Section \ref{PS:ssec:DisplayLogic}, preliminaries on display calculi are collected.
In Section \ref{sec:DLE-logics and ALBA prelim}, the setting of basic DLE-logics is introduced, and the algorithm ALBA for them. In Section
\ref{sec:display calculi DL DL ast}, the display calculi $\mathbf{DL}$ and $\mathbf{DL}^\ast$ for DLE-logics are introduced, and their basic properties are proven.
	In Section \ref{sec:primitive special main strategy}, %defines primitive inequalities, and proves that DLE$^*$-logics axiomatized by primitive inequalities can be properly displayed. Section 6 provides a hierarchy of inequalities which can be equally transformed into primitive ones by ALBA. Sections 7 presents is the discussion and conclusion.
Kracht's notion of primitive formulas is generalized to primitive inequalities in each DLE-language, as well as their connection with special structural rules
for display calculi (cf.\ Definition \ref{def:special}). It is also shown that, for any language $\mathcal{L}_{\mathrm{DLE}}$, each primitive
$\mathcal{L}_{\mathrm{DLE}}$-inequality is equivalent on perfect $\mathcal{L}_{\mathrm{DLE}}$-algebras to a set of special structural rules in the language
of the associated display calculus $\mathbf{DL}$, and that the validity of each such special structural rule is equivalent to the validity of some primitive
$\mathcal{L}_{\mathrm{DLE}}$-inequality. In Section \ref{sec:extended classes} we extend the algorithm generating special structural rules in the language of $\mathbf{DL}$ from input
primitive $\mathcal{L}_{\mathrm{DLE}}$-inequalities to a hierarchy of classes of non-primitive $\mathcal{L}_{\mathrm{DLE}}$-inequalities, the most general of which
is referred to as
{\em restricted analytic inductive inequalities} (cf.\ Definition \ref{def:type4}). Our procedure for obtaining this extension makes use of ALBA to equivalently
transform any restricted analytic inductive $\mathcal{L}_{\mathrm{DLE}}$-inequality into one or more primitive $\mathcal{L}^\ast_{\mathrm{DLE}}$-inequalities.
In Section \ref{sec:analytic}, the class of restricted analytic inductive inequalities is further extended to the {\em analytic inductive inequalities} (cf.\ Definition
\ref{def:type5}). Each analytic inductive inequality can be equivalently transformed into some analytic
rule of a restricted shape, captured in the notion of
{\em quasi-special} structural rule (cf.\ Definition \ref{def:quasispecial}) in the language of $\mathbf{DL}$. Once again, the key step of the latter procedure
makes use of ALBA, this time to equivalently transform any analytic inductive inequality into one or more suitable quasi-inequalities in
$\mathcal{L}^*_{\mathrm{DLE}}$.
We also show that each analytic rule is equivalent to some analytic inductive inequality. This back-and-forth correspondence between analytic rules and analytic
inductive inequalities characterizes the space of properly displayable DLE-logics as the axiomatic extensions of the basic DLE-logic obtained by means of analytic
inductive inequalities.
In Section \ref{sec:special rules as expressive as analytic}, we show that
for any language $\mathcal{L}_\mathrm{DLE}$, any properly displayable DLE-logic is specially
displayable, which implies that any properly displayable $\mathcal{L}^\ast_\mathrm{DLE}$-logic can be axiomatized by means of primitive
$\mathcal{L}^\ast_\mathrm{DLE}$-inequalities. This last result generalizes an analogous statement made by Kracht in the setting of properly displayable tense
modal logics, which was proven in \cite{CiRa13, CiRa14} in the same setting. In Section \ref{sec:comparison}, we outline a comparison between the present treatment and that
of \cite{CiRa13, CiRa14}. In Section \ref{sec:conclusions} we present our conclusions. Various proofs are collected in Appendices \ref{appendix:cut elim}--\ref{appedix:analytic and I2}.

	%%%%%%%%%%%%%%%%%%%%%%%%%%%
	
	%%%%%%%%%%%%%%%%%%%%%%%%%%%

	\section{Preliminaries on display calculi}
	\label{PS:ssec:DisplayLogic}
	
	In the present section, we provide an informal introduction to the main features of display calculi without any attempt at being self-contained. We refer the reader to \cite{Wa98} for an expanded treatment. Our presentation follows \cite[Section 2.2]{FGKPS14a}.
	
	Display calculi are among the approaches in structural proof theory aimed at the uniform development of an inferential theory of meaning of logical constants (logical connectives) aligned with the principles of proof-theoretic semantics \cite{schroeder2006validity,sep-proof-theoretic-semantics}. Display calculi have been successful in giving adequate proof-theoretic semantic accounts of logics---such as certain modal and substructural logics \cite{Gore1}, and more recently also Dynamic Epistemic Logic \cite{Multitype} and PDL \cite{PDL}---which have notoriously been difficult to treat with other approaches. Here we mainly report and elaborate on the work of Belnap \cite{Belnap}, Wansing \cite{Wa98}, Gor\'{e} \cite{Gore1, Gore98}, and Restall \cite{Restall}.
	
	\subsection{Belnap's display logic} Nuel Belnap introduced the first display calculus, which he calls {\em Display Logic} \cite{Belnap}, as a sequent system augmenting and refining Gentzen's basic observations on structural rules. Belnap's refinement is based on the introduction of a special syntax for the constituents of each sequent. Indeed, his calculus treats sequents $ X \vdash Y$ where $X$ and $Y$ are so-called \emph{structures}, i.e.\ syntactic objects inductively defined from formulas using an array of special meta-logical connectives. Belnap's basic idea is that, in the standard Gentzen formulation, the comma symbol `$,$' separating formulas in the precedent and in the succedent of sequents can be recognized as a metalinguistic connective,  the behaviour of which is defined by the structural rules.
	
	Belnap took this idea further by admitting not only the comma, but also other meta-logical connectives to build up structures out of formulas, and called them {\em structural connectives}. Just like the comma in standard Gentzen sequents is interpreted contextually (that is, as conjunction when occurring on the left-hand side and as disjunction when occurring on the right-hand side), each structural connective typically corresponds to a pair of logical connectives, and is interpreted as one or the other of them contextually (more of this in Section \ref{ssec:soundness}). Structural connectives maintain relations with one another, the most fundamental of which take the form of adjunctions and residuations. These relations make it possible for the calculus to enjoy the powerful property which gives it its name, namely, the {\em display property}. Before introducing it formally, let us agree on some auxiliary definitions and nomenclature: \emph{structures} are defined much in the same way as formulas, taking formulas as atomic components and closing under the given structural connectives; therefore, each structure can be uniquely associated with a generation tree. Every node of such a generation tree defines a {\em substructure}. A \emph{sequent} $X\vdash Y$ is a pair of structures $X,Y$. The display property, stated similarly to the one below, appears in \cite[Theorem 3.2]{Belnap}: %where $X\vdash Y$ is referred to as a consecution, and $X$ the antecedent and $Y$ the consequent:
	
	\begin{definition} \label{PS:def: display prop} A proof system enjoys the {\em display property} iff for every sequent $X \vdash Y$ and every substructure $Z$ of either $X$ or $Y$, the sequent $X \vdash Y$ can be equivalently transformed, using the rules of the system, into a sequent which is either of the form $Z \vdash W$ or of the form $W \vdash Z$, for some structure $W$. In the first case, $Z$ is \emph{displayed in precedent position}, and in the second case, $Z$ is \emph{displayed in succedent position}.\footnote{In the following sections, we will find it useful to differentiate between the full and the relativized display property (cf.\ discussion before Proposition \ref{prop: DL has relativized display property}).}
		The rules enabling this equivalent rewriting are called \emph{display postulates}.
	\end{definition}
	
	Thanks to the fact that display postulates are semantically based on adjunction and residuation, exactly one of the two alternatives mentioned in the definition above can soundly occur. In other words, in a calculus enjoying the display property, any substructure of any sequent $X \vdash Y$ is always displayed either only in precedent position or only in succedent position. This is why we can talk about occurrences of substructures in {\em precedent} or in {\em succedent} position, even if they are nested deep within a given sequent, as illustrated in the following example which is based on the display postulates between the structural connectives $;$ and $>$:
	\begin{center}
		{\AX$Y \fCenter X > Z$
			\UI$X\,; Y \fCenter Z$
			\UI$Y\,; X \fCenter Z$
			\UI$X \fCenter Y > Z$
			\DisplayProof}
		\label{PS:example adj}
	\end{center}
	\noindent In the derivation above, the structure $X$ is on the right side of the turnstile, but it is displayable on the left, and therefore is in precedent position.
	The display property is a crucial technical ingredient
	for Belnap's cut elimination metatheorem: for instance, it provides the core mechanism for the satisfaction of the crucial condition C$_8$, discussed in the following subsection. %but it is also at the basis of Belnap's methodology for characterizing operational connectives: according to Belnap, any logical connective should be introduced {\em in isolation}, i.e., when it is introduced, the context on the side it has been introduced must be empty. The display property guarantees that this condition is not too restrictive.
	
	\subsection{Proper display calculi and canonical cut elimination}
	\label{PS:para:CanonicalCutElimination}
	In \cite{Belnap}, a metatheorem is proven, which gives sufficient conditions in order for a sequent calculus to enjoy cut elimination.\footnote{As Belnap observed on page 389 in \cite{Belnap}: `The eight conditions are supposed to be a reminiscent of those of Curry' in \cite{Curry}.} This metatheorem captures the essentials of the Gentzen-style cut elimination procedure, and is the main technical motivation for the design of Display Logic. Belnap's metatheorem gives a set of eight conditions on sequent calculi, which are relatively easy to check, since most of them are verified by inspection on the shape of the rules. Together, these conditions guarantee that the cut is eliminable in the given sequent calculus, and that the calculus enjoys the subformula property. When Belnap's metatheorem can be applied, it provides a much smoother and more modular route to cut elimination than the Gentzen-style proofs. Moreover, as we will see later, a Belnap style cut elimination theorem is robust with respect to adding a general class of structural rules, and with respect to adding new logical connectives, whereas a Gentzen-style cut elimination proof for the modified system cannot be deduced from the old one, but must be proved from scratch.
	
	In a slogan, we could say that Belnap-style cut elimination is to ordinary cut elimination what canonicity is to completeness: indeed, canonicity provides a {\em uniform strategy} to achieve completeness. In the same way, the conditions required by Belnap's metatheorem ensure that {\em one and the same} given set of transformation steps is enough to achieve Gentzen-style cut elimination for any system satisfying them. %\footnote{The relationship between canonicity and Belnap-style cut elimination is in fact more than a mere analogy, see cite[Theorem 20]{Kracht}.}
	
	In what follows, we review and discuss eight conditions which are stronger in certain respects than those in \cite{Belnap},\footnote{See also \cite{Be2, Restall} and the `second formulation' of condition C6/7 in Section 4.4 of \cite{Wa98}.} and which define the notion of {\em proper display calculus} in \cite{Wa98}.\footnote{See the `first formulation' of conditions C6, C7 in Section 4.1 of \cite{Wa98}.}
	
	\paragraph*{ C$_1$: Preservation of formulas.} This condition requires each formula occurring in a premise of a given inference to be the subformula of some formula in the conclusion of that inference. That is, structures may disappear, but not formulas. This condition is not included in the list of sufficient conditions of the cut elimination metatheorem, but, in the presence of cut elimination, it guarantees the subformula property of a system.
	Condition $C_1$ can be verified by inspection on the shape of the rules. In practice, condition C$_1$ bans rules in which structure variables occurring in some premise to not occur also in the conclusion, since in concrete derivations these are typically instantiated with (structures containing) formulas which would then disappear in the application of the rule.
	\paragraph*{ C$_2$: Shape-alikeness of parameters.}
	This condition is based on the relation of {\em congruence} between {\em parameters} (i.e., non-active parts) in inferences; %the rules;
	the congruence relation is an equivalence relation which is meant to identify the different occurrences of the same formula or substructure along the branches of a derivation \cite[Section 4]{Belnap}, \cite[Definition 6.5]{Restall}. %For example, in the following inference
	%\[
	%\AxiomC{$A \vdash A,A$}
	%\UnaryInfC{$A\vdash A$}
	%\DisplayProof
	%\]
	%the two occurrences of $A$ on the left are congruent and the three occurrences of $A$ on the right are congruent.
	Condition C$_2$ requires that congruent parameters be occurrences of the same structure. This can be understood as a condition on the {\em design} of the rules of the system if the congruence relation is understood as part of the specification of each given rule; that is, each schematic rule of the system comes with an explicit specification of which elements are congruent to which (and then the congruence relation is defined as the reflexive and transitive closure of the resulting relation). In this respect, C$_2$ is nothing but a sanity check, requiring that the congruence is defined in such a way that indeed identifies the occurrences which are intuitively ``the same''.\footnote{Our convention throughout the paper is that congruent parameters are denoted by the same letter. For instance, in the rule $$\frac{X;Y\vdash Z}{Y;X\vdash Z}$$ the structures $X,Y$ and $Z$ are parametric and the occurrences of $X$ (resp.\ $Y$, $Z$) in the premise and the conclusion are congruent.} %In what follows, we will denote that two occurrences
	
	\paragraph*{C$_3$: Non-proliferation of parameters.} Like the previous one, also this condition is actually about the definition of the congruence relation on parameters. Condition C$_3$ requires that, for every inference (i.e.\ rule application), each of its parameters is congruent to at most one parameter in the conclusion of that inference. Hence, the condition stipulates that for a rule such as the following,
	\begin{center}
		\AX$X \fCenter Y$
		\UI$X\,, X \fCenter Y$
		\DisplayProof
	\end{center}
	\noindent the structure $X$ from the premise is congruent to {\em only one} occurrence of $X$ in the conclusion sequent. Indeed, the introduced occurrence of $X$ should be considered congruent only to itself. Moreover, given that the congruence is an equivalence relation, condition C$_3$ implies that, within a given sequent, any substructure is congruent only to itself. In practice, in the general schematic formulation of rules, we will use the same structure variable for two different parametric occurrences if and only if they are congruent, so a rule such as the one above is de facto banned.
	
	\begin{remark}
		\label{PS:rem: history tree}
		Conditions C$_2$ and C$_3$ make it possible to follow the history of a formula along the branches of any given derivation. In particular, C$_3$ implies that the the history of any formula within a given derivation has the shape of a tree, which we refer to as the \emph{history-tree} of that formula in the given derivation. Notice, however, that the history-tree of a formula might have a different shape than the portion of the underlying derivation corresponding to it; for instance, the following application of the Contraction rule gives rise to a bifurcation of the history-tree of $A$ which is absesent in the underlying branch of the derivation tree, given that Contraction is a unary rule.
		\medskip
		
		\begin{tabular}{lcr}
			\ \ \ \ \ \ \ \ \ \ \ \ \ \ \ \ \ \ \ \ \ \ \ \ \ \ \ \ \ \ \ \ \ \ \ \ \ \ \ \ \ \ \ \ \ \ \ \ \ \ \
			\bottomAlignProof
			\AXC{$\vdots$}
			\noLine
			\UI$A\,, A \fCenter X$
			\UI$A \fCenter X$
			\DisplayProof
			
			& \ \ \ \ \ \ \ \ &
			
			$\unitlength=0.80mm
			\put(0.00,10.00){\circle*{2}}
			\put(10.00,0.00){\circle*{2}}
			\put(20.00, 10.00){\circle*{2}}
			\put(40.00, 10.00){\circle*{2}}
			\put(0.00,10.00){\line(1,-1){10}}
			\put(10.00,0.00){\line(1,1){10}}
			\put(40.00, 0.00){\circle*{2}}
			\put(40.00,0.00){\line(0,1){10}}
			$
			\\
		\end{tabular}
		
	\end{remark}
	
	\paragraph*{C$_4$: Position-alikeness of parameters.} This condition bans any rule in which a (sub)structure in precedent (resp.~succedent) position in a premise is congruent to a (sub)structure in succedent (resp.\ precedent) position in the conclusion.
	
	\bigskip
	
	\paragraph*{C$_5$: Display of principal constituents.} This condition requires that any principal occurrence (that is, a non-parametric formula occurring in the conclusion of a rule application, cf.\ \cite[Condition C5]{Belnap})  be always either the entire antecedent or the entire consequent part of the sequent in which it occurs. In the following section, a generalization of this condition will be discussed, in view of its application to the main focus of interest of the present chapter.
	
	\bigskip
	The following conditions C$_6$ and C$_7$ are not reported below as they are stated in the original paper \cite{Belnap}, but as they appear in \cite[Section 4.1]{Wa98}. %More about this difference is discussed in Section \ref{PS:ssec:further directions}. \marginnote{Here we say more on this in some section. Do we just erase this sentence?}
	%
	%Assume that the calculus enjoys separation;
	
	%\paragraph*{C$_6$.} This condition requires that, for each rule $R$, which is not closed under arbitrary substitution of parametric constituent formulas, for any offending occurrence in antecedent position, the then this occurrence be closed under su for is stated in terms of exists a parametric occurrence of a formula $R$ in antecedent position which is not closed under arbitrary substitution, then
	\paragraph*{C$_6$: Closure under substitution for succedent parameters.} This condition requires each rule to be closed under simultaneous substitution of arbitrary structures for congruent formulas which occur in succedent position.
	Condition C$_6$ ensures, for instance, that if the following inference is an application of the rule $R$:
	
	\begin{center}
		\AX$(X \fCenter Y) \big([A]^{suc}_{i} \,|\, i \in I\big)$
		\RightLabel{$R$}
		\UI$(X' \fCenter Y') [A]^{suc}$
		\DisplayProof
	\end{center}
	
	\noindent and $\big([A]^{suc}_{i} \,|\, i \in I\big)$ represents all and only the occurrences of $A$ in the premiss which are congruent to the occurrence of $A$ in the conclusion\footnote{Clearly, if $I = \varnothing$, then the occurrence of $A$ in the conclusion is congruent to itself.}, %is congruent to all the occurrences $[A]^{suc}_{i}$ if the set of indices $I$ is nonempty (otherwise $A$ is congruent to itself),
	then also the following inference is an application of the same rule $R$:
	
	\begin{center}
		\AX$(X \fCenter Y) \big([Z/A]^{suc}_{i} \,|\, i \in I\big)$
		\RightLabel{$R$}
		\UI$(X' \fCenter Y') [Z/A]^{suc}$
		\DisplayProof
	\end{center}
	
	\noindent where the structure $Z$ is substituted for $A$.
	
	\noindent This condition caters for the step in the cut elimination procedure in which the cut needs to be ``pushed up'' over rules in which the cut-formula in succedent position is parametric. Indeed, condition C$_6$ guarantees that, in the picture below, a well-formed subtree $\pi_1[Y/A]$ can be obtained from $\pi_1$ by replacing any occurrence of $A$ corresponding to a node in the history tree of the cut-formula $A$ by $Y$, and hence the following transformation step is guaranteed go through uniformly and ``canonically'':
	
	\begin{center}
		%\footnotesize{
		\bottomAlignProof
		\begin{tabular}{lcr}
			\AXC{\ \ \ $\vdots$ \raisebox{1mm}{$\pi'_1$}}
			\noLine
			\UI$X' \fCenter A$
			\noLine
			\UIC{\ \ \ $\vdots$ \raisebox{1mm}{$\pi_1$}}
			\noLine
			\UI$X \fCenter A$
			\AXC{\ \ \ $\vdots$ \raisebox{1mm}{$\pi_2$}}
			\noLine
			\UI$A \fCenter Y$
			\BI$X \fCenter Y$
			\DisplayProof
			& $\rightsquigarrow$ &
			\bottomAlignProof
			\AXC{\ \ \ $\vdots$ \raisebox{1mm}{$\pi'_1$}}
			\noLine
			\UI$X' \fCenter A$
			\AXC{\ \ \ $\vdots$ \raisebox{1mm}{$\pi_2$}}
			\noLine
			\UI$A \fCenter Y$
			\BI$X' \fCenter Y $
			\noLine
			\UIC{\ \ \ \ \ \ \ \ \ $\vdots$ \raisebox{1mm}{$\pi_1[Y/A]$}}
			\noLine
			\UI$X \fCenter Y$
			\DisplayProof
			\\
		\end{tabular}
		%}
	\end{center}
	\noindent if each rule in $\pi_1$ verifies condition C$_6$.
	
	\paragraph*{C$_7$: Closure under substitution for precedent parameters.} This condition requires each rule to be closed under simultaneous substitution of arbitrary structures for congruent formulas which occur in precedent position.
	Condition C$_7$ can be understood analogously to C$_6$, relative to formulas in precedent position. Therefore, for instance, if the following inference is an application of the rule $R$:
	
	\begin{center}
		\AxiomC{$(X \fCenter Y) \big([A]^{pre}_{i} \,|\, i \in I\big)$}
		\RightLabel{$R$}
		\UnaryInfC{$ (X' \vdash Y') [A]^{pre}$}
		\DisplayProof
	\end{center}
	
	\noindent then also the following inference is an instance of $R$:
	
	\begin{center}
		\AxiomC{$(X \fCenter Y) \big([Z/A]^{pre}_{i} \,|\, i \in I\big)$}
		\RightLabel{$R$}
		\UnaryInfC{$ (X' \vdash Y') [Z/A]^{pre}$}
		\DisplayProof
	\end{center}
	
	\noindent Similarly to what has been discussed for condition C$_6$, condition C$_7$ caters for the step in the cut elimination procedure in which the cut needs to be ``pushed up'' over rules in which the cut-formula in precedent position is parametric.
	%We will return on conditions $C_6$ and $C_7$ in the following subsection.
	%The occurrence of the formula $A$ in precedent position can be substituted with an arbitrary structure $Z$, both in the premise(s) and in the conclusion, in such a manner that the transition between the two remains possible by an application of the rule $R$. Every calculus enjoying cut elimination \'a la Belnap should satisfy either condition $C_6$ or $C_7$.@ In cite{Restall} it is proven that both of these conditions are not necessary.
	
	\paragraph*{C$_8$: Eliminability of matching principal constituents.}
	
	This condition requests a standard Gentzen-style checking, which is now limited to the case in which both cut formulas are {\em principal}, i.e.\ each of them has been introduced with the last rule application of each corresponding subdeduction. In this case, analogously to the proof Gentzen-style, condition C$_8$ requires being able to transform the given deduction into a deduction with the same conclusion in which either the cut is eliminated altogether, or is transformed in one or more applications of cut involving proper subformulas of the original cut-formulas.

	\begin{thm} (cf.\ \cite[Section 3.3, Appendix A]{Wan02})
		\label{thm:meta}
		Any calculus satisfying conditions C$_2$, C$_3$, C$_4$, C$_5$, C$_6$, C$_7$, C$_8$ enjoys cut elimination. If C$_1$ is also satisfied, then the calculus enjoys the subformula property.
	\end{thm}

	\paragraph*{Rules introducing logical connectives.} %\marginnote{Expand on this last paragraph. The story about the neat separation of roles between structural and operational rules is important since it helps to understand the basic problem Kracht's paper is about, namely characterize axioms in modal logic corresponding to rules which can be added in a modular fashion so as to preserve cut elimination. But this should be written explicitly and at the moment there is nothing.}
	In display calculi, these rules, sometimes referred to as {\em operational} or \emph{logical rules} as opposed to structural rules, typically occur in two flavors: operational rules which translate one structural connective in the premises in the corresponding connective in the conclusion, and operational rules in which both the operational connective and its structural counterpart are introduced in the conclusion. An example of this pattern is provided below for the case of the modal operator `diamond':
	\[
	\AX$\circ A \fCenter X$
	\RightLabel{$\Diamond_{L}$}
	\UI$\Diamond A \fCenter X$
	\DisplayProof
	\qquad
	\AX$X \fCenter A$
	\RightLabel{$\Diamond_{R}$}
	\UI$\circ X \fCenter \Diamond A$
	\DisplayProof
	\]
	
	\noindent In Section \ref{sec:display calculi DL DL ast}, this introduction pattern will be justified from a semantic viewpoint and generalized to logical connectives of arbitrary arity and polarity of their coordinates. From this example, it is clear that the introduction rules capture the rock bottom behavior of the logical connective in question; additional properties (for instance, normality, in the case in point), which might vary depending on the logical system, are to be captured at the level of additional (purely structural) rules. This enforces a clear-cut division of labour between operational rules, which only encode the basic proof-theoretic meaning of logical connectives, and structural rules, which account for all extra relations and properties, and which can be modularly added or removed, thus accounting for the space of axiomatic extensions of a given base logic. Besides being important from the viewpoint of a proof-theoretic semantic account of logical connectives, this neat division of labour is also key to the research program in proof theory aimed at developing more robust versions of Gentzen's cut-elimination theory. Indeed, as we have seen, Belnap's strategy in this respect precisely pivots on the identification of conditions (mainly on the structural rules of a display calculus) which guarantee that structural rules satisfying them can be safely added in a modular fashion to proper display calculi without disturbing the canonical cut elimination. In the following subsection, we will expand on the consequences of these conditions on the design of structural rules. Specifically, we report on three general shapes of structural rules. Identifying axioms or formulas which can be effectively translated into rules of one of these shapes is the main goal of the present paper.
	
	%Summing up, the two main benefits of display calculi are a ``canonical'' proof of cut elimination, and an explicit and modular account of logical connectives.
	
	%\marginnote{make the connection with what comes next: precisely because we want uniformity, we introduce a very general setting which can be handled by display calculi. The way we introduce it is algebraically oriented since we want to apply ALBA to it, but in section \ref{} we give a display calculus for it.}
	
	\subsection{Analytic, special and quasi-special structural rules}\label{ssec:rules}
	In the remainder of the paper, we will adopt the following convention regarding structural variables and terms: variables $X,Y,Z,W$ denote structures, and so do $S,T,U,V$. However, when describing rule schemas in abstract terms, we will often write e.g.\ $X\vdash S$, and in this context we understand that $X,Y,Z,W$ denote structure variables actually occurring in the given rule scheme, whereas $S,T,U,V$ are used as meta-variables for (possibly) compound structural terms such as $X\, ;\, Y$.
	
	\begin{definition}[Analytic structural rules]\label{def:analytic}
		(cf.\ \cite[Definition 3.13]{CiRa14}) A structural rule which satisfies conditions C$_1$-C$_7$ is an \emph{analytic structural rule}.
	\end{definition}
	
	%It immediately follows from Belnap's theorem \ref{thm:meta} and the definition above that
	Clearly, adding analytic structural rules to a proper display calculus (cf.\ Section \ref{PS:para:CanonicalCutElimination}) yields a proper display calculus.

	\begin{remark} In the setting of calculi with the relativized display property\footnote{cf.\ discussion before Proposition \ref{prop: DL has relativized display property}}, if a given analytic structural rule $\rho$ can be applied in concrete derivations of the calculus
	%\marginnote{do we need to explain with an example what we mean?}
then $\rho$ is interderivable, modulo applications of display postulates, with a rule of the following form:
		\begin{center}
			\begin{tabular}{ccc}
				\AxiomC{$(S^i_j\vdash Y^i\mid 1\leq i\leq n\textrm{ and } 1\leq j\leq n_i)\quad (X^k\vdash T_\ell^k\mid 1\leq k\leq m\textrm{ and } 1\leq\ell\leq m_k)$}
				\UnaryInfC{$(S\vdash T)[Y^i]^{suc}[X^k]^{pre}$}
				\DisplayProof \\
			\end{tabular}
		\end{center}
		where $X^k$ (resp.\ $Y^i$) might occur in $S^i_j$ or in $T^k_\ell$ in precedent (resp.\ succedent) position for some $i,j,k,\ell$ and moreover,  $X^k$ and $Y^j$ occur exactly once in $S\vdash T$ in  precedent and succedent position respectively for all $j,k$.
	\end{remark}
	The most common analytic rules occur in the following proper subclass:
	\begin{definition}[Special structural rules]\label{def:special}
		(cf.\ \cite[Section 5, discussion after Theorem 15]{Kracht} ) \emph{Special structural rules} are analytic structural rules of one of the following forms:
		\[
		\mbox{\AxiomC{$(X\vdash T_i\mid 1\leq i\leq n)$}
			\UnaryInfC{$X\vdash T$}
			\DisplayProof}
		\qquad
		\mbox{\AxiomC{$(S_i\vdash Y\mid 1\leq i\leq n)$}
			\UnaryInfC{$S\vdash Y$}
			\DisplayProof}
		\]
		where $X$ (resp.\ $Y$) does not occur in any $T_i$ (resp.\ $S_i$) for $1\leq i\leq n$ nor in $T$ (resp.\ $S$).
	\end{definition}
	%\marginnote{say something more on special rules, expand on what Kracht did and that they will play an important role below; mention the section on primitive formulas}
	%\marginnote{mention the difference between the hard-coded and the intensional definition of special rules: the first is recognized immediately but the second is such that when applying the general strategy to a type 4 we get immediately a special rule in the intensional sense }
	In \cite{Kracht}, Kracht establishes a correspondence between special rules and primitive formulas in the setting of tense modal logic, which will be generalized in Section \ref{ssec:left right primitive} below.
	\begin{remark}
\label{rmk:two understandings of special}
		An alternative  way to define special rules, which would also be perhaps more in line with the spirit of display calculi, would be as those rules \[
		\mbox{\AxiomC{$(S_i\vdash T_i\mid 1\leq i\leq n)$}
			\UnaryInfC{$S\vdash T$}
			\DisplayProof}
		\]
		such that some variable $X$ occurs exactly once in each premise and in the conclusion, and always in the same (antecedent or consequent) position. In this way, the class of special rules would be closed under under application of display postulates. Applying the general procedure described in Section \ref{ssec:type5} to primitive inequalities (cf.\ Definition \ref{def:primitive}) always yields special rules in the less restrictive sense here specified, but not in the sense of Definition \ref{def:special} above. This fact might be taken as a motivation for adopting the less restrictive definition. However, the more restrictive definition can be immediately verified of a concrete rule, which is the reason why we prefer it over the less restricted one.
	\end{remark}
	In \cite{Kracht}, Kracht states without proof that any analytic structural rules in the language of classical tense logic $Kt$ is equivalent to some special structural rule. Kracht's claim has been proved with model-theoretic techniques in \cite{CiRa14}, \cite{ramanayake2011cut}. In Section \ref{sec:special rules as expressive as analytic}, we generalize these results using ALBA from classical tense logic to arbitrary DLE-logics. %Moreover, as a consequence of the main results of sections \ref{} and \ref{}, we obtain that every analytic structural rule in our main setting, namely the language of $\mathbf{DL}$ (cf.\ Section \ref{}), is equivalent to some quasi-special structural rule, defined as follows:
	The following definition is instrumental in achieving this generalization:
	\begin{definition}[Quasi-special structural rules]\label{def:quasispecial}
		\emph{Quasi-special structural rules} are analytic structural rules of the following form:
		\begin{center}
			\begin{tabular}{ccc}
				\AxiomC{$(S^i_j\vdash Y^i\mid 1\leq i\leq n\textrm{ and } 1\leq j\leq n_i)\quad (X^k\vdash T_\ell^k\mid 1\leq k\leq m\textrm{ and } 1\leq\ell\leq m_k)$}
				\UnaryInfC{$(S\vdash T)[Y^i]^{suc}[X^k]^{pre}$}
				\DisplayProof \\
			\end{tabular}
		\end{center}
		where $X^k$ and $Y^i$ \emph{do not} occur in any $S^i_j$, $T^k_\ell$ (and occur in $S\vdash T$ exactly once).
	\end{definition}

	\section{Preliminaries on DLE-logics and ALBA}
\label{sec:DLE-logics and ALBA prelim}
	In the present section, we collect preliminaries on logics for distributive lattice expansions (or {\em DLE-logics}), reporting in particular on their language, axiomatization and algebraic semantics. Then we report on the definition of inductive DLE-inequalities, and outline, without any attempt at being self-contained, the algorithm ALBA\footnote{ALBA is the acronym of Ackermann Lemma Based Algorithm.} (cf.\ \cite{CoPa10, CoGhPa13}) for each DLE-language.
	
	\subsection{Syntax and semantics for DLE-logics}
	\label{subset:language:algsemantics}
	Our base language is an unspecified but fixed language $\mathcal{L}_\mathrm{DLE}$, to be interpreted over distributive lattice expansions of compatible similarity type. This setting uniformly accounts for many well known logical systems, such as distributive and positive modal logic, intuitionistic and bi-intuitionistic (modal) logic, tense logic, and (distributive) full Lambek calculus.
	
	In our treatment, we will make heavy use of the following auxiliary definition: an {\em order-type} over $n\in \mathbb{N}$\footnote{Throughout the paper, order-types will be typically associated with arrays of variables $\vec p: = (p_1,\ldots, p_n)$. When the order of the variables in $\vec p$ is not specified, we will sometimes abuse notation and write $\varepsilon(p) = 1$ or $\varepsilon(p) = \partial$.} is an $n$-tuple $\epsilon\in \{1, \partial\}^n$. For every order type $\epsilon$, we denote its {\em opposite} order type by $\epsilon^\partial$, that is, $\epsilon^\partial_i = 1$ iff $\epsilon_i=\partial$ for every $1 \leq i \leq n$. For any lattice $\bba$, we let $\bba^1: = \bba$ and $\bba^\partial$ be the dual lattice, that is, the lattice associated with the converse partial order of $\bba$. For any order type $\varepsilon$, we let $\bba^\varepsilon: = \Pi_{i = 1}^n \bba^{\varepsilon_i}$.
	
	The language $\mathcal{L}_\mathrm{DLE}(\mathcal{F}, \mathcal{G})$ (from now on abbreviated as $\mathcal{L}_\mathrm{DLE}$) takes as parameters: 1) a denumerable set of proposition letters $\mathsf{AtProp}$, elements of which are denoted $p,q,r$, possibly with indexes; 2) disjoint sets of connectives $\mathcal{F}$ and $\mathcal{G}$.\footnote{It will be clear from the treatment in the present and the following sections that the connectives in $\mathcal{F}$ (resp.\ $\mathcal{G}$) correspond to those referred to as {\em positive} (resp.\ {\em negative}) connectives in \cite{ciabattoni2008axioms}. The reason why this terminology is not adopted in the present paper is explained later on in Footnote \ref{footnote: why not adopt terminology of CiGaTe}. Our assumption that the sets $\mathcal{F}$ and $\mathcal{G}$ are disjoint is motivated by the desideratum of generality and modularity. Indeed, for instance, the order theoretic properties of Boolean negation $\neg$ guarantee that this connective belongs both to $\mathcal{F}$ and to $\mathcal{G}$. In such cases we prefer to define two copies $\neg_\mathcal{F}\in\mathcal{F}$ and $\neg_\mathcal{G}\in\mathcal{G}$, and introduce structural rules which encode the fact that these two copies coincide.} Each $f\in \mathcal{F}$ and $g\in \mathcal{G}$ has arity $n_f\in \mathbb{N}$ (resp.\ $n_g\in \mathbb{N}$) and is associated with some order-type $\varepsilon_f$ over $n_f$ (resp.\ $\varepsilon_g$ over $n_g$).\footnote{Unary $f$ (resp.\ $g$) will be sometimes denoted as $\Diamond$ (resp.\ $\Box$) if the order-type is 1, and $\lhd$ (resp.\ $\rhd$) if the order-type is $\partial$.} The terms (formulas) of $\mathcal{L}_\mathrm{DLE}$ are defined recursively as follows:
	\[
	\phi ::= p \mid \bot \mid \top \mid \phi \wedge \phi \mid \phi \vee \phi \mid f(\overline{\phi}) \mid g(\overline{\phi})
	\]
	where $p \in \mathsf{AtProp}$, $f \in \mathcal{F}$, $g \in \mathcal{G}$. Terms in $\mathcal{L}_\mathrm{DLE}$ will be denoted either by $s,t$, or by lowercase Greek letters such as $\varphi, \psi, \gamma$ etc. In the context of sequents and prooftrees, $\mathcal{L}_\mathrm{DLE}$-formulas will be denoted by uppercase
letters $A$, $B$, etc.
	
	\begin{definition}
		\label{def:DLE}
		For any tuple $(\mathcal{F}, \mathcal{G})$ of disjoint sets of function symbols as above, a {\em distributive lattice expansion} (abbreviated as DLE) is a tuple $\bba = (D, \mathcal{F}^\bbA, \mathcal{G}^\bbA)$ such that $D$ is a bounded distributive lattice, $\mathcal{F}^\bbA = \{f^\bbA\mid f\in \mathcal{F}\}$ and $\mathcal{G}^\bbA = \{g^\bbA\mid g\in \mathcal{G}\}$, such that every $f^\bbA\in\mathcal{F}^\bbA$ (resp.\ $g^\bbA\in\mathcal{G}^\bbA$) is an $n_f$-ary (resp.\ $n_g$-ary) operation on $\bbA$. A DLE is {\em normal} if every $f^\bbA\in\mathcal{F}^\bbA$ (resp.\ $g^\bbA\in\mathcal{G}^\bbA$) preserves finite joins (resp.\ meets) in each coordinate with $\epsilon_f(i)=1$ (resp.\ $\epsilon_g(i)=1$) and reverses finite meets (resp.\ joins) in each coordinate with $\epsilon_f(i)=\partial$ (resp.\ $\epsilon_g(i)=\partial$).\footnote{\label{footnote:DLE vs DLO} Normal DLEs are sometimes referred to as {\em distributive lattices with operators} (DLOs). This terminology directly derives from the setting of Boolean algebras with operators, in which operators are understood as operations which preserve finite meets in each coordinate. However, this terminology results somewhat ambiguous in the lattice setting, in which primitive operations are typically maps which are operators if seen as $\bbA^\epsilon\to \bbA^\eta$ for some order-type $\epsilon$ on $n$ and some order-type $\eta\in \{1, \partial\}$. Rather than speaking of distributive lattices with $(\varepsilon, \eta)$-operators, we then speak of normal DLEs.} Let $\mathbb{DLE}$ be the class of DLEs. Sometimes we will refer to certain DLEs as $\mathcal{L}_\mathrm{DLE}$-algebras when we wish to emphasize that these algebras have a compatible signature with the logical language we have fixed.
	\end{definition}
In the remainder of the paper, %\marginnote{Apostolos: I replaced the double appearance of ``in the remainder of the paper''; hope it is ok.},
we will abuse notation and write e.g.\ $f$ for $f^\bbA$.
Normal DLEs constitute the main semantic environment of the present paper. Henceforth, every DLE is assumed to be normal; hence the adjective `normal' will be typically dropped. The class of all DLEs is equational, and can be axiomatized by the usual distributive lattice identities and the following equations for any $f\in \mathcal{F}$ (resp.\ $g\in \mathcal{G}$) and $1\leq i\leq n_f$ (resp.\ for each $1\leq j\leq n_g$):
	\begin{itemize}
		\item if $\varepsilon_f(i) = 1$, then $f(p_1,\ldots, p\vee q,\ldots,p_{n_f}) = f(p_1,\ldots, p,\ldots,p_{n_f})\vee f(p_1,\ldots, q,\ldots,p_{n_f})$ and $f(p_1,\ldots, \bot,\ldots,p_{n_f}) = \bot$,
		\item if $\varepsilon_f(i) = \partial$, then $f(p_1,\ldots, p\wedge q,\ldots,p_{n_f}) = f(p_1,\ldots, p,\ldots,p_{n_f})\vee f(p_1,\ldots, q,\ldots,p_{n_f})$ and $f(p_1,\ldots, \top,\ldots,p_{n_f}) = \bot$,
		\item if $\varepsilon_g(j) = 1$, then $g(p_1,\ldots, p\wedge q,\ldots,p_{n_g}) = g(p_1,\ldots, p,\ldots,p_{n_g})\wedge g(p_1,\ldots, q,\ldots,p_{n_g})$ and $g(p_1,\ldots, \top,\ldots,p_{n_g}) = \top$,
		\item if $\varepsilon_g(j) = \partial$, then $g(p_1,\ldots, p\vee q,\ldots,p_{n_g}) = g(p_1,\ldots, p,\ldots,p_{n_g})\wedge g(p_1,\ldots, q,\ldots,p_{n_g})$ and $g(p_1,\ldots, \bot,\ldots,p_{n_g}) = \top$.
	\end{itemize}
	Each language $\mathcal{L}_\mathrm{DLE}$ is interpreted in the appropriate class of DLEs. In particular, for every DLE $\bba$, each operation $f^\bba\in \mathcal{F}^\bbA$ (resp.\ $g^\bba\in \mathcal{G}^\bbA$) is finitely join-preserving (resp.\ meet-preserving) in each coordinate when regarded as a map $f^\bba: \bba^{\varepsilon_f}\to \bba$ (resp.\ $g^\bba: \bba^{\varepsilon_g}\to \bba$).
	
	 The generic DLE-logic is not equivalent to a sentential logic. Hence the consequence relation of these logics cannot be uniformly captured in terms of theorems, but rather in terms of sequents, which motivates the following definition:
	\begin{definition}
		\label{def:DLE:logic:general}
		For any language $\mathcal{L}_\mathrm{DLE} = \mathcal{L}_\mathrm{DLE}(\mathcal{F}, \mathcal{G})$, the {\em basic}, or {\em minimal} $\mathcal{L}_\mathrm{DLE}$-{\em logic} is a set of sequents $\phi\vdash\psi$, with $\phi,\psi\in\mathcal{L}_\mathrm{DLE}$, which contains the following axioms:
		\begin{itemize}
			\item Sequents for lattice operations:\footnote{In what follows we will use the turnstile symbol $\vdash$ both as sequent separator and also as the consequence relation of the logic.}
			\begin{align*}
				&p\vdash p, && \bot\vdash p, && p\vdash \top, & & p\wedge (q\vee r)\vdash (p\wedge q)\vee (p\wedge r), &\\
				&p\vdash p\vee q, && q\vdash p\vee q, && p\wedge q\vdash p, && p\wedge q\vdash q, &
			\end{align*}
			\item Sequents for additional connectives:
			\begin{align*}
				& f(p_1,\ldots, \bot,\ldots,p_{n_f}) \vdash \bot,~\mathrm{for}~ \varepsilon_f(i) = 1,\\
				& f(p_1,\ldots, \top,\ldots,p_{n_f}) \vdash \bot,~\mathrm{for}~ \varepsilon_f(i) = \partial,\\
				&\top\vdash g(p_1,\ldots, \top,\ldots,p_{n_g}),~\mathrm{for}~ \varepsilon_g(i) = 1,\\
				&\top\vdash g(p_1,\ldots, \bot,\ldots,p_{n_g}),~\mathrm{for}~ \varepsilon_g(i) = \partial,\\
				&f(p_1,\ldots, p\vee q,\ldots,p_{n_f}) \vdash f(p_1,\ldots, p,\ldots,p_{n_f})\vee f(p_1,\ldots, q,\ldots,p_{n_f}),~\mathrm{for}~ \varepsilon_f(i) = 1,\\
				&f(p_1,\ldots, p\wedge q,\ldots,p_{n_f}) \vdash f(p_1,\ldots, p,\ldots,p_{n_f})\vee f(p_1,\ldots, q,\ldots,p_{n_f}),~\mathrm{for}~ \varepsilon_f(i) = \partial,\\
				& g(p_1,\ldots, p,\ldots,p_{n_g})\wedge g(p_1,\ldots, q,\ldots,p_{n_g})\vdash g(p_1,\ldots, p\wedge q,\ldots,p_{n_g}),~\mathrm{for}~ \varepsilon_g(i) = 1,\\
				& g(p_1,\ldots, p,\ldots,p_{n_g})\wedge g(p_1,\ldots, q,\ldots,p_{n_g})\vdash g(p_1,\ldots, p\vee q,\ldots,p_{n_g}),~\mathrm{for}~ \varepsilon_g(i) = \partial,
			\end{align*}
		\end{itemize}
		and is closed under the following inference rules:
		\begin{displaymath}
			\frac{\phi\vdash \chi\quad \chi\vdash \psi}{\phi\vdash \psi}
			\quad
			\frac{\phi\vdash \psi}{\phi(\chi/p)\vdash\psi(\chi/p)}
			\quad
			\frac{\chi\vdash\phi\quad \chi\vdash\psi}{\chi\vdash \phi\wedge\psi}
			\quad
			\frac{\phi\vdash\chi\quad \psi\vdash\chi}{\phi\vee\psi\vdash\chi}
		\end{displaymath}
		\begin{displaymath}
			\frac{\phi\vdash\psi}{f(p_1,\ldots,\phi,\ldots,p_n)\vdash f(p_1,\ldots,\psi,\ldots,p_n)}{~(\varepsilon_f(i) = 1)}
		\end{displaymath}
		\begin{displaymath}
			\frac{\phi\vdash\psi}{f(p_1,\ldots,\psi,\ldots,p_n)\vdash f(p_1,\ldots,\phi,\ldots,p_n)}{~(\varepsilon_f(i) = \partial)}
		\end{displaymath}
		\begin{displaymath}
			\frac{\phi\vdash\psi}{g(p_1,\ldots,\phi,\ldots,p_n)\vdash g(p_1,\ldots,\psi,\ldots,p_n)}{~(\varepsilon_g(i) = 1)}
		\end{displaymath}
		\begin{displaymath}
			\frac{\phi\vdash\psi}{g(p_1,\ldots,\psi,\ldots,p_n)\vdash g(p_1,\ldots,\phi,\ldots,p_n)}{~(\varepsilon_g(i) = \partial)}.
		\end{displaymath}
		The minimal DLE-logic is denoted by $\mathbf{L}_\mathrm{DLE}$. For any DLE-language $\mathcal{L}_{\mathrm{DLE}}$, by a {\em $\mathrm{DLE}$-logic} we understand any axiomatic extension of the basic $\mathcal{L}_{\mathrm{DLE}}$-logic in $\mathcal{L}_{\mathrm{DLE}}$.
	\end{definition}
	
	For every DLE $\bba$, the symbol $\vdash$ is interpreted as the lattice order $\leq$. A sequent $\phi\vdash\psi$ is valid in $\bba$ if $h(\phi)\leq h(\psi)$ for every homomorphism $h$ from the $\mathcal{L}_\mathrm{DLE}$-algebra of formulas over $\mathsf{AtProp}$ to $\bba$. The notation $\mathbb{DLE}\models\phi\vdash\psi$ indicates that $\phi\vdash\psi$ is valid in every DLE. Then, by means of a routine Lindenbaum-Tarski construction, it can be shown that the minimal DLE-logic $\mathbf{L}_\mathrm{DLE}$ is sound and complete with respect to its correspondent class of algebras $\mathbb{DLE}$, i.e.\ that any sequent $\phi\vdash\psi$ is provable in $\mathbf{L}_\mathrm{DLE}$ iff $\mathbb{DLE}\models\phi\vdash\psi$. %Moreover, it is not hard to see that every consistent DLE-logic is characterized by the class of algebras for it.
	
	\subsection{The expanded language $\mathcal{L}_\mathrm{DLE}^*$}
	\label{ssec:expanded language}
	Any given language $\mathcal{L}_\mathrm{DLE} = \mathcal{L}_\mathrm{DLE}(\mathcal{F}, \mathcal{G})$ can be associated with the language $\mathcal{L}_\mathrm{DLE}^* = \mathcal{L}_\mathrm{DLE}(\mathcal{F}^*, \mathcal{G}^*)$, where $\mathcal{F}^*\supseteq \mathcal{F}$ and $\mathcal{G}^*\supseteq \mathcal{G}$ are obtained by expanding $\mathcal{L}_\mathrm{DLE}$ with the following connectives:
	\begin{enumerate}
		\item the binary connectives $\leftarrow$ and $\rightarrow$, the intended interpretations of which are the right residuals of $\wedge$ in the first and second coordinate respectively, and
		$\pdla$ and $ \pdra$, the intended interpretations of which are the left residuals of $\vee$ in the first and second coordinate, respectively;
		\item the $n_f$-ary connective $f^\sharp_i$ for $0\leq i\leq n_f$, the intended interpretation of which is the right residual of $f\in\mathcal{F}$ in its $i$th coordinate if $\varepsilon_f(i) = 1$ (resp.\ its Galois-adjoint if $\varepsilon_f(i) = \partial$);
		\item the $n_g$-ary connective $g^\flat_i$ for $0\leq i\leq n_g$, the intended interpretation of which is the left residual of $g\in\mathcal{G}$ in its $i$th coordinate if $\varepsilon_g(i) = 1$ (resp.\ its Galois-adjoint if $\varepsilon_g(i) = \partial$).
		% $ g^\flat_j$ for each and $g\in \mathcal{G}$, where and $0\leq j\leq n_g$ ($f^\sharp_i$ is the right residual of $f$ in the $i$-th coordinate, and $g^\flat_j$ is the left residual of $g$ in the $j$-th coordinate).
		\footnote{The adjoints of the unary connectives $\Box$, $\Diamond$, $\lhd$ and $\rhd$ are denoted $\Diamondblack$, $\blacksquare$, $\blhd$ and $\brhd$, respectively.}
	\end{enumerate}
	We stipulate that $\pdra, \pdla\in \mathcal{F}^*$, that $\rightarrow, \leftarrow\in \mathcal{G}^*$, and moreover, that
	$f^\sharp_i\in\mathcal{G}^*$ if $\varepsilon_f(i) = 1$, and $f^\sharp_i\in\mathcal{F}^*$ if $\varepsilon_f(i) = \partial$. Dually, $g^\flat_i\in\mathcal{F}^*$ if $\varepsilon_g(i) = 1$, and $g^\flat_i\in\mathcal{G}^*$ if $\varepsilon_g(i) = \partial$. The order-type assigned to the additional connectives is predicated on the order-type of their intended interpretations. That is, for any $f\in \mathcal{F}$ and $g\in\mathcal{G}$,
	%each $g^\flat_j\in\mathcal{F}$, for each coordinate $i$ in $f$ or $g$,
	\begin{enumerate}
		\item if $\epsilon_f(i) = 1$, then $\epsilon_{f_i^\sharp}(i) = 1$ and $\epsilon_{f_i^\sharp}(j) = (\epsilon_f(j))^\partial$ for any $j\neq i$.
		\item if $\epsilon_f(i) = \partial$, then $\epsilon_{f_i^\sharp}(i) = \partial$ and $\epsilon_{f_i^\sharp}(j) = \epsilon_f(j)$ for any $j\neq i$.
		\item if $\epsilon_g(i) = 1$, then $\epsilon_{g_i^\flat}(i) = 1$ and $\epsilon_{g_i^\flat}(j) = (\epsilon_g(j))^\partial$ for any $j\neq i$.
		\item if $\epsilon_g(i) = \partial$, then $\epsilon_{g_i^\flat}(i) = \partial$ and $\epsilon_{g_i^\flat}(j) = \epsilon_g(j)$ for any $j\neq i$.
	\end{enumerate}
	
	For instance, if $f$ and $g$ are binary connectives such that $\varepsilon_f = (1, \partial)$ and $\varepsilon_g = (\partial, 1)$, then $\varepsilon_{f^\sharp_1} = (1, 1)$, $\varepsilon_{f^\sharp_2} = (1, \partial)$, $\varepsilon_{g^\flat_1} = (\partial, 1)$ and $\varepsilon_{g^\flat_2} = (1, 1)$.\footnote{Warning: notice that this notation heavily depends from the connective which is taken as primitive, and needs to be carefully adapted to well known cases. For instance, consider the  `fusion' connective $\circ$ (which, when denoted  as $f$, is such that $\varepsilon_f = (1, 1)$). Its residuals
$f_1^\sharp$ and $f_2^\sharp$ are commonly denoted $/$ and
$\backslash$ respectively. However, if $\backslash$ is taken as the primitive connective $g$, then $g_2^\flat$ is $\circ = f$, and
$g_1^\flat(x_1, x_2): = x_2/x_1 = f_1^\sharp (x_2, x_1)$. This example shows
that, when identifying $g_1^\flat$ and $f_1^\sharp$, the conventional order of the coordinates is not preserved, and depends of which connective
is taken as primitive.}
	
	%\begin{remark}
	%Note that ($\top, \bot$), ($\pand, \por$), ($\pdra, \pra$) and ($\pdla, \leftarrow$) are \emph{dual pairs}. Operators in a dual pair have the same arity and the same order type. The pairs ($\pand, \pra$), ($\wedge, \leftarrow$), ($\pdra, \por$) and ($\pdla,\vee$) are \emph{residual pairs} as follows: $\pand \dashv \ \pra$, $\pand \dashv \ \leftarrow$, $\pdra \dashv \por$, and $\pdla \dashv \por$.
	%\end{remark}
	
	\begin{definition}
		For any language $\mathcal{L}_\mathrm{DLE}(\mathcal{F}, \mathcal{G})$, the {\em basic bi-intuitionistic `tense'} $\mathcal{L}_\mathrm{DLE}$-{\em logic} is defined by specializing Definition \ref{def:DLE:logic:general} to the language $\mathcal{L}_\mathrm{DLE}^* = \mathcal{L}_\mathrm{DLE}(\mathcal{F}^*, \mathcal{G}^*)$ %a set of sequents $\phi\vdash\psi$ with $\phi,\psi\in\mathcal{L}_\mathrm{DLE}^*$, which contains the axioms of the DLE-logic $\mathbf{L}_\mathrm{DLE}$, and is closed under rules for DLE-logics plus
		and closing under the following additional rules:
		\begin{enumerate}
			\item residuation rules for lattice connectives:
			$$
			\begin{array}{cccc}
			\AxiomC{$\phi\wedge\psi\vdash \chi$}
			\doubleLine
			\UnaryInfC{$\psi\vdash \phi\rightarrow \chi$}
			\DisplayProof
			&
			\AxiomC{$\phi\wedge\psi\vdash \chi$}
			\doubleLine
			\UnaryInfC{$\phi\vdash \chi\leftarrow \psi$}
			\DisplayProof
			&
			\AxiomC{$\phi\vdash \psi\vee\chi$}
			\doubleLine
			\UnaryInfC{$\psi \pdra \phi \vdash \chi$}
			\DisplayProof
			&
			\AxiomC{$\phi\vdash \psi\vee\chi$}
			\doubleLine
			\UnaryInfC{$\phi \pdla \chi \vdash \psi$}
			\DisplayProof
			\end{array}
			$$
			Notice that the rules for $\rightarrow$ and $\leftarrow$ are interderivable, since $\wedge$ is commutative; similarly, the rules for $\pdra$ and $\pdla$ are interderivable, since $\vee$ is commutative.
			\item Residuation rules for $f\in \mathcal{F}$ and $g\in \mathcal{G}$:
			$$
			\begin{array}{cc}
			\AxiomC{$f(\varphi_1,\ldots,\phi,\ldots, \varphi_{n_f}) \vdash \psi$}
			\doubleLine
			\LeftLabel{$(\epsilon_f(i) = 1)$}
			\UnaryInfC{$\phi\vdash f^\sharp_i(\varphi_1,\ldots,\psi,\ldots,\varphi_{n_f})$}
			\DisplayProof
			&
			\AxiomC{$\phi \vdash g(\varphi_1,\ldots,\psi,\ldots,\varphi_{n_g})$}
			\doubleLine
			\RightLabel{$(\epsilon_g(i) = 1)$}
			\UnaryInfC{$g^\flat_i(\varphi_1,\ldots, \phi,\ldots, \varphi_{n_g})\vdash \psi$}
			\DisplayProof
			\end{array}
			$$
			$$
			\begin{array}{cc}
			\AxiomC{$f(\varphi_1,\ldots,\phi,\ldots, \varphi_{n_f}) \vdash \psi$}
			\doubleLine
			\LeftLabel{$(\epsilon_f(i) = \partial)$}
			\UnaryInfC{$f^\sharp_i(\varphi_1,\ldots,\psi,\ldots,\varphi_{n_f})\vdash \phi$}
			\DisplayProof
			&
			\AxiomC{$\phi \vdash g(\varphi_1,\ldots,\psi,\ldots,\varphi_{n_g})$}
			\doubleLine
			\RightLabel{($\epsilon_g(i) = \partial)$}
			\UnaryInfC{$\psi\vdash g^\flat_i(\varphi_1,\ldots, \phi,\ldots, \varphi_{n_g})$}
			\DisplayProof
			\end{array}
			$$
		\end{enumerate}
		The double line in each rule above indicates that the rule is invertible.
		Let $\mathbf{L}_\mathrm{DLE}^*$ be the minimal bi-intuitionistic `tense' $\mathcal{L}_\mathrm{DLE}$-logic.\footnote{\label{ftn: definition basic logic for expanded language} Hence, for any language $\mathcal{L}_\mathrm{DLE}$, there are in principle two logics associated with the expanded language $\mathcal{L}_\mathrm{DLE}^*$, namely the {\em minimal} $\mathcal{L}_\mathrm{DLE}^*$-logic, which we denote by $\mathbf{L}_\mathrm{DLE}^{\underline{*}}$, and which is obtained by instantiating Definition \ref{def:DLE:logic:general} to the language $\mathcal{L}_\mathrm{DLE}^*$, and the bi-intuitionistic `tense' logic $\mathbf{L}_\mathrm{DLE}^*$, defined above. The logic $\mathbf{L}_\mathrm{DLE}^*$ is the natural logic on the language $\mathcal{L}_\mathrm{DLE}^*$, however it is useful to introduce a specific notation for $\mathbf{L}_\mathrm{DLE}^{\underline{*}}$, given that all the results holding for the minimal logic associated with an arbitrary DLE-language can be instantiated to the expanded language $\mathcal{L}_\mathrm{DLE}^*$ and will then apply to $\mathbf{L}_\mathrm{DLE}^{\underline{*}}$.} For any DLE-language $\mathcal{L}_{\mathrm{DLE}}$, by a {\em tense $\mathrm{DLE}$-logic} we understand any axiomatic extension of the basic tense bi-intuitionistic $\mathcal{L}_{\mathrm{DLE}}$-logic in $\mathcal{L}^*_{\mathrm{DLE}}$.
	\end{definition}
	
	The algebraic semantics of $\mathbf{L}_\mathrm{DLE}^*$ is given by the class of bi-intuitionistic `tense' $\mathcal{L}_\mathrm{DLE}$-algebras, defined as tuples $\bba = (H, \mathcal{F}^*, \mathcal{G}^*)$ such that $H$ is a bi-Heyting algebra\footnote{That is, $H = (D, \rightarrow, \pdla)$ such that both $(D, \rightarrow)$ and $(D^\partial, \pdla)$ are Heyting algebras. In particular, setting $c\leftarrow b: = b\rightarrow c$ and $b\pdra a: = a\pdra b$ for all $a,b,c\in D$, the following equivalences hold \begin{center}
			$a\wedge b\leq c$ iff $b\leq a\rightarrow c$ iff $a\leq c\leftarrow b$,
			$\quad\quad\quad\quad\quad\quad$
			$a\leq b\vee c$ iff $b\pdra a\leq c$ iff $a\pdla c\leq b$.
		\end{center}} and moreover,
		\begin{enumerate}
			
			\item for every $f\in \mathcal{F}$ s.t.\ $n_f\geq 1$, all $a_1,\ldots,a_{n_f}\in D$ and $b\in D$, and each $1\leq i\leq n_f$,
			\begin{itemize}
				\item
				if $\epsilon_f(i) = 1$, then $f(a_1,\ldots,a_i,\ldots a_{n_f})\leq b$ iff $a_i\leq f^\sharp_i(a_1,\ldots,b,\ldots,a_{n_f})$;
				\item
				if $\epsilon_f(i) = \partial$, then $f(a_1,\ldots,a_i,\ldots a_{n_f})\leq b$ iff $a_i\leq^\partial f^\sharp_i(a_1,\ldots,b,\ldots,a_{n_f})$.
			\end{itemize}
			\item for every $g\in \mathcal{G}$ s.t.\ $n_g\geq 1$, any $a_1,\ldots,a_{n_g}\in D$ and $b\in D$, and each $1\leq i\leq n_g$,
			\begin{itemize}
				\item if $\epsilon_g(i) = 1$, then $b\leq g(a_1,\ldots,a_i,\ldots a_{n_g})$ iff $g^\flat_i(a_1,\ldots,b,\ldots,a_{n_g})\leq a_i$.
				\item
				if $\epsilon_g(i) = \partial$, then $b\leq g(a_1,\ldots,a_i,\ldots a_{n_g})$ iff $g^\flat_i(a_1,\ldots,b,\ldots,a_{n_g})\leq^\partial a_i$.
			\end{itemize}
		\end{enumerate}
		
		It is also routine to prove using the Lindenbaum-Tarski construction that $\mathbf{L}_\mathrm{DLE}^*$ (as well as any of its sound axiomatic extensions) is sound and complete w.r.t.\ the class of bi-intuitionistic `tense' $\mathcal{L}_\mathrm{DLE}$-algebras (w.r.t.\ the suitably defined equational subclass, respectively). % and that every consistent DLE$^*$-logic is characterized by its algebras.
		
		\begin{thm}
			\label{th:conservative extension}%\marginnote{Edit the footnote, replace the BDL with DLE and mention the reference to DLE. Reference to footnote about perfect DLEs.}
			The logic $\mathbf{L}_\mathrm{DLE}^*$ is a conservative extension of $\mathbf{L}_\mathrm{DLE}$, i.e., for every $\mathcal{L}_\mathrm{DLE}$-sequent $\phi\vdash\psi$, $\phi\vdash\psi$ is derivable in $\mathbf{L}_\mathrm{DLE}$ iff $\phi\vdash\psi$ is derivable in $\mathbf{L}_\mathrm{DLE}^*$. Moreover, every DLE-logic can be extended conservatively to a DLE$^*$-logic.
		\end{thm}
		\begin{proof}
			We only outline the proof.
			Clearly, every $\mathcal{L}_\mathrm{DLE}$-sequent which is $\mathbf{L}_\mathrm{DLE}$-derivable is also $\mathbf{L}_\mathrm{DLE}^*$-derivable. Conversely, if an $\mathcal{L}_\mathrm{DLE}$-sequent $\phi\vdash\psi$ is not $\mathbf{L}_\mathrm{DLE}$-derivable, then by the completeness of $\mathbf{L}_\mathrm{DLE}$ w.r.t.\ the class of $\mathcal{L}_\mathrm{DLE}$-algebras, there exists an $\mathcal{L}_\mathrm{DLE}$-algebra $\bba$ and a variable assignment $v$ under which $\phi^\bba\not\leq \psi^\bba$. Consider the canonical extension $\bba^\delta$ of $\bba$.\footnote{ \label{def:can:ext}
				The \emph{canonical extension} of a BDL (bounded distributive lattice) $D$ is a complete distributive lattice $D^\delta$ containing $D$ as a sublattice, such that:
				\begin{enumerate}
					\item \emph{(denseness)} every element of $D^\delta$ can be expressed both as a join of meets and as a meet of joins of elements from $D$;
					\item \emph{(compactness)} for all $S,T \subseteq D$, if $\bigwedge S \leq \bigvee T$ in $D^\delta$, then $\bigwedge F \leq \bigvee G$ for some finite sets $F \subseteq S$ and $G\subseteq T$.
				\end{enumerate}
				%\end{definition}
				
				It is well known that the canonical extension of a BDL $D$ is unique up to isomorphism fixing $D$ (cf.\ e.g.\ \cite[Section 2.2]{GeNaVe05}), and that the canonical extension of a BDL is a \emph{perfect} BDL, i.e.\ a complete and completely distributive lattice which is completely join-generated by its completely join-irreducible elements and completely meet-generated by its completely meet-irreducible elements  (cf.\ e.g.\ \cite[Definition 2.14]{GeNaVe05})\label{canext bdl is perfect}. The canonical extension of an
$\mathcal{L}_\mathrm{DLE}$-algebra $\bbA = (D, \mathcal{F}^\bbA, \mathcal{G}^\bbA)$ is the perfect  $\mathcal{L}_\mathrm{DLE}$-algebra (cf.\ Footnote \ref{footnote: perfect DLE})
$\bbA^\delta: = (D^\delta, \mathcal{F}^{\bbA^\delta}, \mathcal{G}^{\bbA^\delta})$ such that $f^{\bbA^\delta}$ and $g^{\bbA^\delta}$ are defined as the
$\sigma$-extension of $f^{\bbA}$ and as the $\pi$-extension of $g^{\bbA}$ respectively, for all $f\in \mathcal{F}$ and $g\in \mathcal{G}$ (cf.\ \cite{sofronie2000duality1, sofronie2000duality2}).
				%An element $x \in \ca$ is \emph{closed} (resp.\ \emph{open}) if it is the meet (resp.\ join) of some subset of $A$. Let $K(\ca)$ (resp.\ $O(\ca)$) be the set of closed (resp.\ open) elements of $\ca$. It is easy to see that the denseness condition in Definition \ref{def:can:ext2.3} implies that $J^{\infty}(\bbas)\subseteq K (A^\delta)$ and $M^{\infty}(\bbas)\subseteq O (A^\delta)$ (cf.\ cite{GeNaVe05}, page 9).
			}
			Since $\bba$ is a subalgebra of $\bba^\delta$, the sequent $\phi\vdash\psi$ is not satisfied in $\bba^\delta$ under the variable assignment $\iota \circ v$ ($\iota$ denoting the canonical embedding $\bba\hookrightarrow \bba^\delta$). Moreover, since $\bba^\delta$ is a perfect $\mathcal{L}_\mathrm{DLE}$-algebra, it is naturally endowed with a structure of bi-intuitionistic `tense' $\mathcal{L}_\mathrm{DLE}$-algebra. Thus, by the completeness of $\mathbf{L}_\mathrm{DLE}^*$ w.r.t.\ the class of bi-intuitionistic `tense' $\mathcal{L}_\mathrm{DLE}$-algebras, the sequent $\phi\vdash\psi$ is not derivable in $\mathbf{L}_\mathrm{DLE}^*$, as required.
		\end{proof}
Notice that the algebraic completeness of the logics $\mathbf{L}_\mathrm{DLE}$ and $\mathbf{L}_\mathrm{DLE}^*$ and the canonical embedding of DLEs into their canonical extensions immediately give completeness of $\mathbf{L}_\mathrm{DLE}$ and $\mathbf{L}_\mathrm{DLE}^*$ w.r.t.\ the appropriate class of perfect DLEs.
		%\end{comment}

	\subsection{The algorithm ALBA, informally}\label{subseq:informal}
	The  contribution of the present paper is an application  of unified correspondence theory \cite{CoPa10, CoGhPa13}, of which the algorithm ALBA is one of the main tools. In the present subsection, we will guide the reader through the main principles which make it work, by means of an example. This presentation is based on analogous illustrations in \cite{CFPS} and \cite{CGPSZ14}.
	
	\bigskip
	
	Let us start with one of the best known examples in correspondence theory, namely $\Diamond \Box p \rightarrow \Box \Diamond p$. It is well known that for every Kripke frame $\mathbb{F} = (W, R)$,
	\[
	\mathbb{F} \Vdash \Diamond \Box p \rightarrow \Box \Diamond p \,\,\,\mbox{ iff }\,\,\, \mathbb{F} \models \,\forall xyz\,(Rxy \wedge Rxz \rightarrow \exists u(Ryu \wedge Rzu)).
	\]
	As is discussed at length in \cite{CoPa10, CoGhPa13}, every piece of argument used to prove this correspondence on frames can be translated by duality to complex algebras (cf.\ \cite[Definition 5.21]{BdRV01}). We will show how this is done in the case of the example above.
	
	As is well known, complex algebras are characterized in purely algebraic terms as complete and atomic BAOs where the modal operations are completely join-preserving. These are also known as \emph{perfect} BAOs \cite[Definition 40, Chapter 6]{HBMoL}.
	%, which are perfect distributive lattices with operators (in the classical case, complete atomic boolean algebras with operators).
	
	First of all, the condition $\mathbb{F} \Vdash \Diamond \Box p \rightarrow \Box \Diamond p$ translates to the complex algebra $\bbA = \mathbb{F}^{+}$ of $\mathbb{F}$  as $\val{\Diamond\Box p} \subseteq \val{\Box \Diamond p}$ for every assignment of $p$ into $A$, so this validity clause can be rephrased as follows:
	\begin{equation}\label{Church:Rosser}
	\bbA \models \forall p [\Diamond\Box p \leq \Box\Diamond p],
	\end{equation}
	where the order $\leq$ is interpreted as set inclusion in the complex algebra.  In perfect BAOs every element is both the join of the completely join-prime elements (the set of which is denoted $\jty(\bbA)$) below it and the meet of the completely meet-prime elements (the set of which is denoted $\mty(\bbA)$) above it\footnote{In BAOs the completely join-prime elements, the completely join-irreducible elements and the atoms coincide. Moreover, the completely meet-prime elements, the completely meet-irreducible elements and the co-atoms coincide.}. Hence, taking some liberties in our use of notation, the condition above can be equivalently rewritten as follows:
	\begin{equation*}
	\bbA\models \forall p [\bigvee\{i\in \jty(\bbA)\mid i \leq\Box\Diamond p\} \leq \bigwedge\{m\in \mty(\bbA)\mid\Box\Diamond p\leq m\}].
	\end{equation*}
	By elementary properties of least upper bounds and greatest lower bounds in posets (cf.\ \cite{DaPr}), this condition is true if and only if every element in the join is less than or equal to every element in the meet; thus, condition \eqref{Church:Rosser} above  can be rewritten as:
	\begin{equation}\label{Eq:First:Approx}
	\bbA\models \forall p \forall \nomi \forall \cnomm [(\nomi \leq\Diamond \Box p \,\,\,\& \,\,\,\Box \Diamond p\leq \cnomm) \Rightarrow \nomi \leq \cnomm ],
	\end{equation}
	where the variables $\nomi$ and $\cnomm$ range over $\jty(\bbA)$ and $\mty(\bbA)$ respectively (following the literature, we will refer to the former variables as {\em nominals}, and  to the latter ones as {\em co-nominals}). Since $\bbA$ is a perfect BAO, the element of $\bbA$ interpreting $\Box p$ is the join of the completely join-prime elements below it. Hence, if $i \in \jty(\bbA)$ and $i \leq \Diamond \Box p$, because $\Diamond$ is completely join-preserving on $\bbA$, we have that
	\[
	i \leq \Diamond(\bigvee \{j \in \jty(\bbA)\mid j\leq  \Box p \}) = \bigvee\{\Diamond j \mid j\in \jty(\bbA)\mbox{ and } j\leq \Box p \},
	\]
	which implies that $i \leq \Diamond j_0$ for some $j_0 \in \jty(\bbA)$ such that $j_0 \leq \Box p$. Hence, we can equivalently rewrite the validity clause above as follows:
	\begin{equation}\label{Eq:DiaApprox}
	\bbA\models \forall p \forall \nomi \forall \cnomm [(\exists \nomj(\nomi\leq\Diamond \nomj \,\,\,\& \,\,\, \nomj\leq \Box p) \,\,\,\& \,\,\, \Box \Diamond p \leq \cnomm) \Rightarrow \nomi \leq \cnomm ],
	\end{equation}
	and then use standard manipulations from first-order logic to pull out quantifiers:
	\begin{equation}\label{Eq:PullOut}
	\bbA\models \forall p \forall \nomi \forall \cnomm \forall \nomj[(\nomi\leq\Diamond \nomj \,\,\,\& \,\,\, \nomj\leq \Box p \,\,\,\& \,\,\,\Box \Diamond p \leq \cnomm) \Rightarrow \nomi \leq \cnomm ].
	\end{equation}
	Now we observe that the operation $\Box$ preserves arbitrary meets in the perfect BAO $\bbA$. By the general theory of adjunction in complete lattices, this is equivalent to $\Box$ being a right adjoint (cf.\  \cite[Proposition 7.34]{DaPr}). It is also well known that the left or lower adjoint (cf.\ \cite[Definition 7.23]{DaPr}) of $\Box$ is the operation $\Diamondblack$, which can be recognized as the backward-looking diamond $P$, interpreted with the converse $R^{-1}$ of the accessibility relation $R$ of the frame $\mathbb{F}$ in the context of tense logic (cf.\ \cite[Example 1.25]{BdRV01} and \cite[Exercise 7.18]{DaPr} modulo translating the notation). Hence the condition above can be equivalently rewritten as:
	\begin{equation}\label{Eq:BoxAdj}
	\bbA \models \forall p \forall \nomi \forall \cnomm \forall \nomj[(\nomi\leq\Diamond \nomj \,\,\,\& \,\,\, \Diamondblack \nomj \leq p \,\,\,\& \,\,\,\Box \Diamond p \leq \cnomm) \Rightarrow \nomi \leq \cnomm ],
	\end{equation}
	and then as follows:
	\begin{equation}
	\label{Before:Ack:Eq}
	\bbA \models \forall \nomi \forall \cnomm \forall \nomj[(\nomi\leq\Diamond \nomj \,\,\,\& \,\,\, \exists p (\Diamondblack \nomj \leq p \,\,\,\& \,\,\,\Box \Diamond p \leq \cnomm)) \Rightarrow \nomi \leq \cnomm ].
	\end{equation}
	At this point we are in a position to eliminate the variable $p$ and equivalently rewrite the previous condition as follows:
	\begin{equation}
	\label{After:Ack:Eq}
	\bbA\models \forall \nomi \forall \cnomm \forall \nomj[(\nomi\leq\Diamond\nomj \,\,\,\& \,\,\, \Box \Diamond \Diamondblack \nomj \leq \cnomm) \Rightarrow \nomi \leq \cnomm ].
	\end{equation}
	Let us justify this equivalence: for the direction from top to bottom, fix an interpretation $V$ of the variables $\nomi, \nomj$, and $\cnomm$ such that $\nomi\leq\Diamond \nomj$ and $\Box \Diamond \Diamondblack \nomj \leq \cnomm$. To prove that $\nomi\leq \cnomm$ holds under  $V$,  consider the variant $V^\ast$ of $V$ such that $V^\ast(p) = \Diamondblack \nomj$. Then it can be easily verified that $V^{\ast}$ witnesses the antecedent of (\ref{Before:Ack:Eq}) under $V$; hence $\nomi\leq \cnomm$ holds under $V$. Conversely,  fix an interpretation $V$ of the variables $\nomi$, $\nomj$ and $\cnomm$ such that
	$\nomi\leq\Diamond \nomj \,\,\,\& \,\,\, \exists p (\Diamondblack \nomj \leq p \,\,\,\& \,\,\,\Box \Diamond p \leq \cnomm)$. Then, by monotonicity,  the antecedent of (\ref{After:Ack:Eq}) holds under $V$, and hence so does $\nomi\leq \cnomm$, as required. This is an instance of the following result, known as {\em Ackermann's lemma} (\cite{Ack35}, see also \cite{Conradie:et:al:SQEMAI}):
	\begin{lemma}\label{Right:Ackermann}
		Fix an arbitrary propositional language $L$. Let $\alpha, \beta(p), \gamma(p)$ be $L$-formulas such that $\alpha$ is $p$-free, $\beta$ is positive and $\gamma$ is negative in $p$. For any assignment $V$ on an $L$-algebra $\mathbb A$, the following are equivalent:
		\begin{enumerate}
			\item  $\bbA, V \models \beta(\alpha/p) \leq \gamma(\alpha/p)$ ;
			\item  there exists a $p$-variant $V^\ast$ of $V$ such that $\bbA, V^\ast \models \alpha \leq p$ and $\ \bbA, V^\ast \models \beta(p)\leq\gamma(p)$,
		\end{enumerate}
		where $\beta(\alpha/p)$ and $\gamma(\alpha/p)$ denote the result of uniformly substituting $\alpha$ for $p$ in $\beta$ and $\gamma$, respectively.
	\end{lemma}
	The proof is essentially the same as \cite[Lemma 4.2]{CoPa10}. Whenever, in a reduction, we reach a shape in which the lemma above (or its order-dual) can be applied, we say that the condition is in {\em Ackermann shape}\label{Ackermann:Shape}.
	
	Taking stock, we note that we have equivalently transformed (\ref{Church:Rosser}) into (\ref{After:Ack:Eq}), which is a condition in which all propositional variables (corresponding to monadic second-order variables) have been eliminated, and all remaining variables range over completely join- and meet-irreducible elements of the complex algebra $\bbA$. Via discrete Stone duality, these elements respectively correspond to singletons and complements of singletons of the Kripke frame from which $\bbA$ arises. Moreover, $\Diamondblack$ is interpreted on Kripke frames using the converse of the same accessibility relation used to interpret $\Box$. Hence, clause (\ref{After:Ack:Eq}) translates equivalently into a condition in the first-order correspondence language of $\mathbb{F}$.

	 To facilitate this translation, we first rewrite (\ref{After:Ack:Eq}) as follows, by reversing the reasoning that brought us from \eqref{Church:Rosser} to \eqref{Eq:First:Approx}:
	\begin{equation}
	\bbA\models \forall \nomj[\Diamond \nomj \leq \Box \Diamond \Diamondblack \nomj].
	\end{equation}
	
	By again applying the fact that $\Box$ is a right adjoint we obtain
	
	\begin{equation}\label{eq:pure: church rosser}
	\bbA\models \forall \nomj[\Diamondblack \Diamond \nomj \leq \Diamond \Diamondblack \nomj].
	\end{equation}
	
	Recalling that $\bbA$ is the complex algebra of $\mathbb{F} = (W,R)$, we can interpret the variable $\nomj$ as an individual variable ranging in the universe $W$ of $\mathbb{F}$, and the operations $\Diamond$ and $\Diamondblack$ as the set-theoretic operations defined on $\mathcal{P}(W)$ by the assignments $X\mapsto R^{-1}[X]$ and $X\mapsto R[X]$ respectively. Hence, clause \eqref{eq:pure: church rosser} above can be equivalently rewritten on the side of the frames  as
	\begin{equation}
	\mathbb{F}\models \forall w( R[R^{-1}[ w ]] \subseteq  R^{-1}[R[ w ]]).
	\end{equation}
	Notice that  $R[R^{-1}[ w ]]$ is the set of all states $x \in W$ which have a predecessor $z$ in common with $w$, while  $R^{-1}[R[ w ]]$ is the set of all states $x \in W$ which have a successor in common with $w$. This can be spelled out as
	\[
	\forall x \forall w ( \exists z (Rzx \wedge Rzw) \rightarrow \exists y (Rxy \wedge Rwy))
	\]
	or, equivalently,
	\[
	\forall z \forall x \forall w ( (Rzx \wedge Rzw) \rightarrow \exists y (Rxy \wedge Rwy))
	\]
	which is the familiar Church-Rosser condition.
	
%	Before moving on, it is worthwhile to observe that it is not a special situation that  we have been able to extract the familiar Church-Rosser condition from clause \eqref{eq:pure: church rosser}. Indeed,  the well known standard translation of classical Sahlqvist correspondence theory can be extended to the ``hybrid'' language comprising the additional variables $\nomj$ and $\cnomm$ and the connectives $\Diamondblack$ and $\blacksquare$, in such a way that  pure expressions in this language (i.e.\ those which are free from proposition variables) correspond to formulas in the first order  language of Kripke frames. In the next subsection we will provide a formal definition of this expanded language.
	
	Finally, the example above   illustrates another important feature of the ALBA-based approach to the computation of first-order correspondents. Namely, ALBA-computations are neatly divided into two stages: the {\em reduction} stage, carried out  from \eqref{Church:Rosser} into \eqref{After:Ack:Eq} in the example above; and the {\em translation} stage, in which the  expressions (equalities and quasi-inequalities) obtained by eliminating all proposition variables from an input inequality are suitably translated into  frame-correspondent language. Only the reduction stage will be relevant to the remainder of the present paper.
	
		\subsection{The algorithm ALBA for $\mathcal{L}_\mathrm{DLE}$-inequalities}
		\label{ssec: ALBA}
		The present subsection reports on the rules and execution of the algorithm ALBA in the setting of $\mathcal{L}_\mathrm{DLE}$.
		ALBA manipulates inequalities and quasi-inequalities\footnote{A {\em quasi-inequality} of $\mathcal{L}_{\mathrm{DLE}}$ is an expression of the form $\bigamp_{i = 1}^n s_i\leq t_i \Rightarrow s\leq t$, where $s_i\leq t_i$ and $s\leq t$ are $\mathcal{L}_{\mathrm{DLE}}$-inequalities for each $i$.} in the expanded language $\mathcal{L}_\mathrm{DLE}^{*+}$, which is built up on the base of the lattice constants $\top, \bot$ and an enlarged set of propositional variables $\mathsf{NOM}\cup \mathsf{CONOM}\cup \mathsf{AtProp}$ (the variables $\mathbf{i}, \mathbf{j}$ in $\mathsf{NOM}$ are referred to as {\em nominals}, and the variables $\mathbf{m}, \mathbf{n}$ in $\mathsf{CONOM}$ as {\em conominals}), closing under the logical connectives of $\mathcal{L}_\mathrm{DLE}^*$. The natural semantic environment of $\mathcal{L}_\mathrm{DLE}^{*+}$ is given by {\em perfect} $\mathcal{L}_\mathrm{DLE}$-{\em algebras}. As already mentioned in the proof of Theorem \ref{th:conservative extension}, these algebras are endowed with a natural structure of bi-intuitionistic `tense' $\mathcal{L}_\mathrm{DLE}$-algebra. Moreover, crucially, perfect $\mathcal{L}_\mathrm{DLE}$-algebras are both completely join-generated by their completely join-irreducible elements and completely meet-generated by their completely meet-irreducible elements.\footnote{\label{footnote: perfect DLE} A distributive lattice is {\em perfect} if it is complete, completely distributive and completely join-generated by the collection of its completely join-prime elements. Equivalently, a distributive lattice is perfect iff it is isomorphic to the lattice of upsets of some poset. A normal DLE is {\em perfect} if $D$ is a perfect distributive lattice, and each $f$-operation (resp.\ $g$-operation) is completely join-preserving (resp.\ meet-preserving) or completely meet-reversing (resp.\ join-reversing) in each coordinate.} This property plays an important part in the algebraic account of the correspondence mechanism (cf.\ discussion in \cite[Section 1.4]{CoGhPa13}). %The interpretation of additional variables is given in perfect distributive lattice expansions.
		Nominals and conominals respectively range over the sets of the completely join-irreducible elements and the completely meet-irreducible elements of perfect DLEs.
		%A canonical extension of a DLE defined below is perfect.

		The version of ALBA relative to $\mathcal{L}_\mathrm{DLE}$ runs as detailed in \cite{CoPa10}. In a nutshell, $\mathcal{L}_\mathrm{DLE}$-inequalities are equivalently transformed into the conjunction of one or more $\mathcal{L}_\mathrm{DLE}^{*+}$ quasi-inequalities, with the aim of eliminating propositional variable occurrences via the application of Ackermann rules. We refer the reader to \cite{CoPa10} for details. In what follows, we illustrate how ALBA works, while at the same time we introduce its rules. The proof of the soundness and invertibility of the general rules for the DLE-setting is similar to the one provided in \cite{CoPa10, CoGhPa13}. ALBA manipulates input inequalities $\phi\leq\psi$ and proceeds in three stages:
		
		\textbf{First stage: preprocessing and first approximation.} %\marginnote{ in addition to the standard preprocessing, let us add some rules which collapse skeleton terms containing $+\bot$ or $-\top$ into $\bot$ and collapse PIA terms containing $-\bot$ or $+\top$ into $\top$. mention that this is not done in standard ALBA since there it is not strictly needed for f.o. corresp. but this step is relevant for the management of occurrences of constants on the wrong side, since it allows to transform e.g.\ a problematic premise into a tautology and make it hence disappear}
		ALBA preprocesses the input inequality $\phi\leq \psi$ by performing the following steps
		exhaustively in the signed generation trees $+\phi$ and $-\psi$:
		\begin{enumerate}
			\item
			\begin{enumerate}
				\item Push down, towards variables, occurrences of $+\land$, by distributing each of them over their children nodes labelled with $+\lor$ which are not in the scope of PIA nodes;
				\item Push down, towards variables, occurrences of $-\lor$, by distributing each of them over their children nodes labelled with $-\land$ which are not in the scope of PIA nodes;
				\item Push down, towards variables, occurrences of $+f$ for any $f\in \mathcal{F}$, by distributing each such occurrence over its $i$th child node whenever the child node is labelled with $+\lor$ (resp.\ $-\land$) and is not in the scope of PIA nodes, and whenever $\epsilon_f(i)=1$ (resp.\ $\epsilon_f(i)=\partial$);				
				\item Push down, towards variables, occurrences of $-g$ for any $g\in \mathcal{G}$, by distributing each such occurrence over its $i$th child node whenever the child node is labelled with $-\land$ (resp.\ $+\lor$) and is not in the scope of PIA nodes, and whenever $\epsilon_g(i)=1$ (resp.\ $\epsilon_g(i)=\partial$).
				
				%and $-g$ for any $g\in \mathcal{G}$ for $\epsilon_f(i)=1$, $+\land$, $-g$ for $\epsilon_g(i) = \partial$, %$+\diamdot, -\rhddot$by distributing each of them over their $i$th nodes labelled with $+\lor$ which are not in the scope of PIA nodes, or
				%\item Push down, towards variables, occurrences of $-g$ for $\epsilon_g(i)=1$, $-\lor$ and $+ f$ for $\epsilon_f(i) = \partial$, $-\boxdot, +\lhddot$ by distributing them over nodes labelled with $-\land$ which are not in the scope of PIA nodes.
			\end{enumerate}
			\item Apply the splitting rules:
			$$\infer{\alpha\leq\beta\ \ \ \alpha\leq\gamma}{\alpha\leq\beta\land\gamma}
			\qquad
			\infer{\alpha\leq\gamma\ \ \ \beta\leq\gamma}{\alpha\lor\beta\leq\gamma}
			$$
			\item Apply the monotone and antitone variable-elimination rules:
			$$\infer{\alpha(\perp)\leq\beta(\perp)}{\alpha(p)\leq\beta(p)}
			\qquad
			\infer{\beta(\top)\leq\alpha(\top)}{\beta(p)\leq\alpha(p)}
			$$
			for $\beta(p)$ positive in $p$ and $\alpha(p)$ negative in $p$.
		\end{enumerate}

\begin{remark}
\label{rmk: management of top and bottom on the wrong side}
 The standard ALBA preprocessing can be supplemented with the application of additional rules which replace SLR-nodes (resp.\ SRR-nodes) of the form $\circledast(\gamma_1,\ldots,\bot^{\epsilon_\circledast(i)}, \ldots, \gamma_m)$ (resp.\ $\circledast(\gamma_1,\ldots,\top^{\epsilon_\circledast(i)}, \ldots, \gamma_m)$) with $\bot$ (resp.\ $\top$). %and SRR nodes skeleton terms containing $+\bot$ or $-\top$ into $\bot$ and collapse PIA terms containing $-\bot$ or $+\top$ into $\top$. mention that
Although clearly sound, these rules have not been included in other ALBA settings such as \cite{CoPa10, CFPS}, since they are not strictly needed for the computation of first-order correspondents. However, in the present setting, ALBA is used for a different purpose than the one it was originally designed for. Allowing these rules to be applied during the preprocessing will address the problem of the occurrences of constants in the `wrong' position,\footnote{As we will see, in the context of analytic inductive inequalities, occurrences of $+\bot$ or $-\top$ as skeleton nodes and occurrences of $-\bot$ or $+\top$ as PIA-nodes are problematic. Indeed, in the context of the procedure which transforms inequalities into equivalent structural rules (cf.\ Sections \ref{sec:extended classes} and \ref{ssec:type5}), these logical constants would occur within certain sequents in positions (antecedent or succedent) in which they are not the interpretation of the corresponding structural constant. This fact would block the smooth transformation of logical axioms containing them into structural rules.} since it allows to transform e.g.\ a problematic premise into a tautology and make it hence disappear. These ideas will be expanded on in Sections \ref{ssec:type3} and \ref{ssec:type5}.

Another step of the preprocessing which, although sound, is not included in standard executions of ALBA concerns the exhaustive application of the following distribution rules:
 \begin{enumerate}
				\item[(a')] Push down, towards variables, occurrences of $-\land$ in the scope of PIA-nodes which are not Skeleton-nodes, by distributing each of them over their children nodes labelled with $-\lor$; %which are not in the scope of PIA nodes;
				\item[(b')] Push down, towards variables, occurrences of $+\lor$ in the scope of PIA-nodes which are not Skeleton-nodes, by distributing each of them over their children nodes labelled with $+\land$; %which are not in the scope of PIA nodes;
				\item[(c')] Push down, towards variables, occurrences of $-f$ for any $f\in \mathcal{F}$, by distributing each such occurrence over its $i$th child node whenever the child node is labelled with $-\lor$ (resp.\ $+\land$), and whenever $\epsilon_f(i)=1$ (resp.\ $\epsilon_f(i)=\partial$);
				\item[(d')] Push down, towards variables, occurrences of $+g$ for any $g\in \mathcal{G}$, by distributing each such occurrence over its $i$th child node whenever the child node is labelled with $+\land$ (resp.\ $-\lor$), and whenever $\epsilon_g(i)=1$ (resp.\ $\epsilon_g(i)=\partial$).				 
				%and $-g$ for any $g\in \mathcal{G}$ for $\epsilon_f(i)=1$, $+\land$, $-g$ for $\epsilon_g(i) = \partial$, %$+\diamdot, -\rhddot$by distributing each of them over their $i$th nodes labelled with $+\lor$ which are not in the scope of PIA nodes, or
				%\item Push down, towards variables, occurrences of $-g$ for $\epsilon_g(i)=1$, $-\lor$ and $+ f$ for $\epsilon_f(i) = \partial$, $-\boxdot, +\lhddot$ by distributing them over nodes labelled with $-\land$ which are not in the scope of PIA nodes.
			\end{enumerate} Applied to PIA-terms, this additional step has the effect of surfacing all occurrences of $+\wedge$ and $-\vee$ up to the root of each PIA-term (so as to form a connected block of nodes including the root which are all labelled $+\wedge$ or all labelled $-\vee$). In this position, these occurrences can be all regarded as Skeleton nodes. Hence, after this step, no occurrences of $+\wedge$ and $-\vee$ will remain in the PIA subterms.\footnote{\label{footnote: def definite PIA} PIA subterms $\ast s$ in which no nodes $+\wedge$ and $-\vee$ occur are referred to as {\em definite}.} Notice that applying this step to $(\Omega, \epsilon)$-inductive terms produces $(\Omega, \epsilon)$-inductive terms each PIA-subterm of which contains at most one $\varepsilon$-critical variable occurrence.
\end{remark}
		
		Let $\mathsf{Preprocess}(\phi\leq\psi)$ be the finite set $\{\phi_i\leq\psi_i\mid 1\leq i\leq n\}$ of inequalities obtained after the exhaustive application of the previous rules. We proceed separately on each of them, and hence, in what follows, we focus only on one element $\phi_i\leq\psi_i$ in $\mathsf{Preprocess}(\phi\leq\psi)$, and we drop the subscript. Next, the following {\em first approximation rule} is applied {\em only once} to every inequality in $\mathsf{Preprocess}(\phi\leq\psi)$:
		$$\infer{\nomi_0\leq\phi\ \ \ \psi\leq \cnomm_0}{\phi\leq\psi}
		$$
		Here, $\nomi_0$ and $\cnomm_0$ are a nominal and a conominal respectively. The first-approximation
		step gives rise to systems of inequalities $\{\nomi_0\leq\phi_i, \psi_i\leq \cnomm_0\}$ for each inequality in $\mathsf{Preprocess}(\phi\leq\psi)$. Each such system is called an {\em initial
			system}, and is now passed on to the reduction-elimination cycle.
		
		\textbf{Second stage: reduction-elimination cycle.} The goal of the reduction-elimination cycle is to eliminate all propositional variables from the systems
		received from the preprocessing phase. The elimination of each variable is effected by an
		application of one of the Ackermann rules given below. In order to apply an Ackermann rule, the
		system must have a specific shape. The adjunction, residuation, approximation, and splitting rules are used to transform systems into this shape. The rules of the reduction-elimination cycle, viz.\ the adjunction, residuation, approximation, splitting, and Ackermann rules, will be collectively called the {\em reduction} rules.
		
		\textbf{Residuation rules.} Here below we provide the residuation rules relative to each $f\in \mathcal{F}$ and $g\in \mathcal{G}$ of arity at least $1$: for each $1\leq h\leq n_f$ and each $1\leq k\leq n_g$:
		\begin{center}
			\begin{tabular}{cc}
				\AxiomC{$f(\psi_1,\ldots,\psi_h, \ldots, \psi_{n_f})\leq \chi$}
				\LeftLabel{($\varepsilon_f(h) = 1$)}
				\UnaryInfC{$\psi_h\leq f_h^{\sharp}(\psi_1,\ldots,\chi, \ldots, \psi_{n_f})$}
				\DisplayProof
				&
				\AxiomC{$f(\psi_1,\ldots,\psi_h, \ldots, \psi_{n_f})\leq \chi$}
				\RightLabel{($\varepsilon_f(h) = \partial$)}
				\UnaryInfC{$f_h^{\sharp}(\psi_1,\ldots,\chi, \ldots, \psi_{n_f})\leq \psi_h$}
				\DisplayProof
				\\
			\end{tabular}
		\end{center}
		
		\begin{center}
			\begin{tabular}{cc}
				\AxiomC{$\chi \leq g(\psi_1,\ldots,\psi_k, \ldots, \psi_{n_g})$}
				\LeftLabel{($\varepsilon_g(k) = \partial$)}
				\UnaryInfC{$\psi_k \leq g_k^{\flat}(\psi_1,\ldots,\chi, \ldots, \psi_{n_g})$}
				\DisplayProof
				&
				\AxiomC{$\chi \leq g(\psi_1,\ldots,\psi_k, \ldots, \psi_{n_g})$}
				\RightLabel{($\varepsilon_g(k) = 1$)}
				\UnaryInfC{$g_k^{\flat}(\psi_1,\ldots,\chi, \ldots, \psi_{n_g})\leq \psi_k$}
				\DisplayProof
				\\
			\end{tabular}
		\end{center}
		
		%\textbf{Adjunction rules.}
		%\begin{prooftree}
		%\AxiomC{${\diamdot} \phi\leq \psi $}\UnaryInfC{$\phi\leq {\boxdotb} \psi$}
		%\AxiomC{$\phi\leq \boxdot \psi $}\UnaryInfC{${\diamdotb} \phi\leq \psi$}
		%\AxiomC{${\lhddot} \phi\leq \psi $}\UnaryInfC{${\lhddotb} \psi\leq \phi$}
		%\noLine\TrinaryInfC{}
		%\AxiomC{$\phi\leq {\rhddot} \psi $}\UnaryInfC{$\psi\leq {\rhddotb}\phi$} \noLine\UnaryInfC{} \noLine\BinaryInfC{}
		%\end{prooftree}
		
		\textbf{Approximation rules.} Here below we provide the approximation rules\footnote{The version of the approximation rules given in \cite{CoPa10,PaSoZh15r,CGPSZ14} is slightly different from but equivalent to that of the approximation rules reported on here. That formulation is motivated by the need of enforcing the invariance of certain topological properties for the purpose of proving the canonicity of the inequalities on which ALBA succeeds. In this context, we do not need to take these constraints into account, and hence we can take this more flexible version of the approximation rules as primitive, bearing in mind that when proving canonicity one has to take a formulation analogous to that in in \cite{CoPa10,PaSoZh15r,CGPSZ14} as primitive.} relative to each $f\in \mathcal{F}$ and $g\in \mathcal{G}$ of arity at least $1$: for each $1\leq h\leq n_f$ and each $1\leq k\leq n_g$,
		\begin{comment}
		Approximation rules for the additional connectives in DLE$^*$ are completely analogous to those of the DML-connectives in cite{CoPa10}, namely:
		\begin{prooftree}
		\AxiomC{$\mathbf{i}\leq{\diamdot} \phi$}\UnaryInfC{$\mathbf{j}\leq
		\phi\quad \mathbf{i}\leq{\diamdot} \mathbf{j}$}
		
		\AxiomC{$\boxdot \phi\leq \mathbf{m} $}\UnaryInfC{$ \phi\leq
		\mathbf{n}\quad \boxdot\mathbf{n}\leq \mathbf{m}$}
		
		\AxiomC{$\mathbf{i}\leq {\lhddot} \phi $}\UnaryInfC{$\phi\leq \mathbf{m}\quad
		\mathbf{i}\leq {\lhddot}\mathbf{m}$} \noLine\TrinaryInfC{}
		
		\AxiomC{${\rhddot} \phi\leq \mathbf{m} $}\UnaryInfC{$\mathbf{i}\leq \phi\quad
		{\rhddot}\mathbf{i}\leq \mathbf{m} $} \noLine\UnaryInfC{}
		\noLine\BinaryInfC{}
		\end{prooftree}
		\end{comment}
		\begin{center}
			\begin{tabular}{cc}
				\AxiomC{$\nomi\leq f(\psi_1,\ldots,\psi_h, \ldots, \psi_{n_f})$}
				\LeftLabel{$(\varepsilon_f(h) = 1)$}
				\UnaryInfC{$\nomi\leq f(\psi_1,\ldots,\nomj, \ldots, \psi_{n_f})\quad\nomj\leq \psi_h$}
				\DisplayProof
				&
				\AxiomC{$g(\psi_1,\ldots,\psi_k, \ldots, \psi_{n_g})\leq \cnomm$}
				\RightLabel{$(\varepsilon_g(k) = 1)$}
				\UnaryInfC{$g(\psi_1,\ldots,\cnomn, \ldots, \psi_{n_g})\leq \cnomm\quad\psi_k\leq\cnomn$}
				\DisplayProof
				\\
				& \\
				\AxiomC{$\nomi\leq f(\psi_1,\ldots,\psi_h, \ldots, \psi_{n_f})$}
				\LeftLabel{$(\varepsilon_f(h) = \partial)$}
				\UnaryInfC{$\nomi\leq f(\psi_1,\ldots,\cnomn, \ldots, \psi_{n_f})\quad \psi_k\leq\cnomn$}
				\DisplayProof
				&
				\AxiomC{$g(\psi_1,\ldots,\psi_k, \ldots, \psi_{n_g})\leq \cnomm$}
				\RightLabel{$(\varepsilon_g(k) = \partial)$}
				\UnaryInfC{$g(\psi_1,\ldots,\nomj, \ldots, \psi_{n_g})\leq \cnomm\quad \nomj\leq \psi_h$}
				\DisplayProof
			\end{tabular}
		\end{center}
		where the variable $\nomj$ (resp.\ $\cnomn$) is a nominal (resp.\ a conominal). %if $\epsilon_f(k)=1$ (resp.\ $\epsilon_g(k)=1$), and is a conominal (resp.\ a nominal) if $\epsilon_f(k)=\partial$ (resp.\ $\epsilon_g(k)=\partial$). Moreover, $\leq^{\epsilon_k}$ denotes $\leq$ if $\epsilon_k = 1$, and denotes $\geq$ if $\epsilon_k = \partial$. The leftmost inequalities in each rule above will be referred to as the \emph{side condition}.
		The nominals and conominals introduced by the approximation rules must be {\em fresh}, i.e.\ must not already occur in the system before applying the rule.
		
		%Each approximation rule transforms a given system $S\cup\{s\leq t\}$ into systems $S\cup\{s_1\leq t_1\}$ and $S\cup\{s_2\leq t_2, s_3\leq t_3\}$, the first of which containing only the side condition (in which no propositional variable occurs), and the second one containing the instances of the two remaining lower inequalities.
		
		\textbf{Ackermann rules.} These rules are the core of ALBA, since their application eliminates proposition variables. As mentioned earlier, all the preceding steps are aimed at equivalently rewriting the input system into one of a shape in which the Ackermann rules can be applied. An important feature of Ackermann rules is that they are executed on the whole set of inequalities in which a given variable occurs, and not on a single inequality.
		\begin{center}
			\AxiomC{$\bigamp \{ \alpha_i \leq p \mid 1 \leq i \leq n \} \amp \bigamp \{ \beta_j(p)\leq \gamma_j(p) \mid 1 \leq j \leq m \} \; \Rightarrow \; \nomi \leq \cnomm$}
			\RightLabel{$(RAR)$}
			\UnaryInfC{$\bigamp \{ \beta_j(\bigvee_{i=1}^n \alpha_i)\leq \gamma_j(\bigvee_{i=1}^n \alpha_i) \mid 1 \leq j \leq m \} \; \Rightarrow \; \nomi \leq \cnomm$}
			\DisplayProof
		\end{center}
		where $p$ does not occur in $\alpha_1, \ldots, \alpha_n$, $\beta_{1}(p), \ldots, \beta_{m}(p)$ are positive in $p$, and $\gamma_{1}(p), \ldots, \gamma_{m}(p)$ are negative in $p$.
		
		\begin{center}
			\AxiomC{$\bigamp \{ p \leq \alpha_i \mid 1 \leq i \leq n \} \amp \bigamp \{ \beta_j(p)\leq \gamma_j(p) \mid 1 \leq j \leq m \} \; \Rightarrow \; \nomi \leq \cnomm$}
			\RightLabel{$(LAR)$}
			\UnaryInfC{$\bigamp \{ \beta_j(\bigwedge_{i=1}^n \alpha_i)\leq \gamma_j(\bigwedge_{i=1}^n \alpha_i) \mid 1 \leq j \leq m \} \; \Rightarrow \; \nomi \leq \cnomm$}
			\DisplayProof
		\end{center}
		where $p$ does not occur in $\alpha_1, \ldots, \alpha_n$, $\beta_{1}(p), \ldots, \beta_{m}(p)$ are negative in $p$, and $\gamma_{1}(p), \ldots, \gamma_{m}(p)$ are positive in $p$.
		
		\textbf{Third stage: output.}
		If there was some system in the second stage from which not all occurring propositional variables could be eliminated through the application of the reduction rules, then ALBA reports failure and terminates. Else, each system $\{\nomi_0\leq\phi_i, \psi_i\leq \cnomm_0\}$ obtained from $\mathsf{Preprocess}(\varphi\leq \psi)$ has been reduced to a system, denoted $\mathsf{Reduce}(\varphi_i\leq \psi_i)$, containing no propositional variables. Let ALBA$(\varphi\leq \psi)$ be the set of quasi-inequalities \begin{center}{\Large{\&}}$[\mathsf{Reduce}(\varphi_i\leq \psi_i) ]\Rightarrow \nomi_0 \leq \cnomm_0$\end{center} for each $\varphi_i \leq \psi_i \in \mathsf{Preprocess}(\varphi\leq \psi)$.
		
		Notice that all members of ALBA$(\varphi\leq \psi)$ are free of propositional variables. ALBA returns ALBA$(\varphi\leq \psi)$ and terminates. An inequality $\varphi\leq \psi$ on which ALBA succeeds will be called an ALBA-{\em inequality}.

		The proof of the following theorem is a straightforward generalization of \cite[Theorem 8.1]{CoPa10}, and hence its proof is omitted.
		\begin{thm}[Correctness]\label{albacorrect}
			If ALBA succeeds on a $\mathcal{L}_{\mathrm{DLE}}$-inequality $\varphi\leq\psi$, then for every perfect    $\mathcal{L}_{\mathrm{DLE}}$-algebra $\bbA$, $$\bbA\models\varphi\leq\psi\quad\mbox{iff}\quad\bbA\models\mathrm{ALBA}(\varphi\leq\psi).$$
		\end{thm}
		
				\subsection{Inductive inequalities}
				%A DLE-inequality %(resp.\ DLE$^*$-inequality)
				%is an expression of the form $s\leq t$ where $s,t\in\mathcal{L}_\mathrm{DLE}$ %(resp.\ $s,t\in\mathcal{L}_\mathrm{DLE}^*$),
				%which is essentially a sequent in algebraic form.
				
				In the present subsection, we will report on the definition of {\em inductive} $\mathcal{L}_\mathrm{DLE}$-inequalities
				%(resp.\ $\mathcal{L}_\mathrm{DLE}^*$),
				on which the algorithm ALBA is guaranteed to succeed %equivalently transform into one (or the conjunction of more) pure quasi-inequalities in an expanded language. For more details,
				(cf.\ \cite{CoGhPa13,CoPa10}).
				
				\begin{definition}[\textbf{Signed Generation Tree}]
					\label{def: signed gen tree}
					The \emph{positive} (resp.\ \emph{negative}) {\em generation tree} of any $\mathcal{L}_\mathrm{DLE}$-term $s$ is defined by labelling the root node of the generation tree of $s$ with the sign $+$ (resp.\ $-$), and then propagating the labelling on each remaining node as follows:
					\begin{itemize}
						%\item The root node $+s$ (resp.\ $-s$) is the root node of the positive (resp.\ negative) generation tree of $s$ signed with + (resp.\ $-$).
						\item For any node labelled with $ \lor$ or $\land$, assign the same sign to its children nodes.
						%\item If a node is labelled with $\lhd$, $\rhd$, assign the opposite sign to its child node.
						\item For any node labelled with $h\in \mathcal{F}\cup \mathcal{G}$ of arity $n_h\geq 1$, and for any $1\leq i\leq n_h$, assign the same (resp.\ the opposite) sign to its $i$th child node if $\varepsilon_h(i) = 1$ (resp.\ if $\varepsilon_h(i) = \partial$).
					\end{itemize}
					Nodes in signed generation trees are \emph{positive} (resp.\ \emph{negative}) if are signed $+$ (resp.\ $-$).\footnote{\label{footnote: why not adopt terminology of CiGaTe} The terminology used in \cite{ciabattoni2008axioms} regarding `positive' and `negative connectives' has not been adopted in the present paper to avoid confusion with positive and negative nodes in signed generation trees.}
				\end{definition}
				
				Signed generation trees will be mostly used in the context of term inequalities $s\leq t$. In this context we will typically consider the positive generation tree $+s$ for the left-hand side and the negative one $-t$ for the right-hand side. We will also say that a term-inequality $s\leq t$ is \emph{uniform} in a given variable $p$ if all occurrences of $p$ in both $+s$ and $-t$ have the same sign, and that $s\leq t$ is $\epsilon$-\emph{uniform} in a (sub)array $\vec{p}$ of its variables if $s\leq t$ is uniform in $p$, occurring with the sign indicated by $\epsilon$, for every $p$ in $\vec{p}$\footnote{\label{footnote:uniformterms}The following observation will be used at various points in the remainder of the present paper: if a term inequality $s(\vec{p},\vec{q})\leq t(\vec{p},\vec{q})$ is $\epsilon$-uniform in $\vec{p}$ (cf.\ discussion after Definition \ref{def: signed gen tree}), then the validity of $s\leq t$ is equivalent to the validity of $s(\overrightarrow{\top^{\epsilon(i)}},\vec{q})\leq t(\overrightarrow{\top^{\epsilon(i)}},\vec{q})$, where $\top^{\epsilon(i)}=\top$ if $\epsilon(i)=1$ and $\top^{\epsilon(i)}=\bot$ if $\epsilon(i)=\partial$. }.
				
				For any term $s(p_1,\ldots p_n)$, any order type $\epsilon$ over $n$, and any $1 \leq i \leq n$, an \emph{$\epsilon$-critical node} in a signed generation tree of $s$ is a leaf node $+p_i$ with $\epsilon_i = 1$ or $-p_i$ with $\epsilon_i = \partial$. An $\epsilon$-{\em critical branch} in the tree is a branch from an $\epsilon$-critical node. The intuition, which will be built upon later, is that variable occurrences corresponding to $\epsilon$-critical nodes are \emph{to be solved for}, according to $\epsilon$.
				
				For every term $s(p_1,\ldots p_n)$ and every order type $\epsilon$, we say that $+s$ (resp.\ $-s$) {\em agrees with} $\epsilon$, and write $\epsilon(+s)$ (resp.\ $\epsilon(-s)$), if every leaf in the signed generation tree of $+s$ (resp.\ $-s$) is $\epsilon$-critical.
				In other words, $\epsilon(+s)$ (resp.\ $\epsilon(-s)$) means that all variable occurrences corresponding to leaves of $+s$ (resp.\ $-s$) are to be solved for according to $\epsilon$. We will also write $+s'\prec \ast s$ (resp.\ $-s'\prec \ast s$) to indicate that the subterm $s'$ inherits the positive (resp.\ negative) sign from the signed generation tree $\ast s$. Finally, we will write $\epsilon(\gamma) \prec \ast s$ (resp.\ $\epsilon^\partial(\gamma) \prec \ast s$) to indicate that the signed subtree $\gamma$, with the sign inherited from $\ast s$, agrees with $\epsilon$ (resp.\ with $\epsilon^\partial$).
				\begin{definition}
					\label{def:good:branch}
					Nodes in signed generation trees will be called \emph{$\Delta$-adjoints}, \emph{syntactically left residual (SLR)}, \emph{syntactically right residual (SRR)}, and \emph{syntactically right adjoint (SRA)}, according to the specification given in Table \ref{Join:and:Meet:Friendly:Table}.
					A branch in a signed generation tree $\ast s$, with $\ast \in \{+, - \}$, is called a \emph{good branch} if it is the concatenation of two paths $P_1$ and $P_2$, one of which may possibly be of length $0$, such that $P_1$ is a path from the leaf consisting (apart from variable nodes) only of PIA-nodes, and $P_2$ consists (apart from variable nodes) only of Skeleton-nodes. %\footnote{These classes are grouped together into the super-classes \emph{Skeleton} and \emph{PIA} as indicated in the table. This organization is motivated and discussed in \cite{CFPS} and \cite{CoGhPa13}. %to establish a connection with analogous terminology in \cite{BeHovB12}.}
					\begin{table}[\here]
						\begin{center}
							\begin{tabular}{| c | c |}
								\hline
								Skeleton &PIA\\
								\hline
								$\Delta$-adjoints & SRA \\
								\begin{tabular}{ c c c c c c}
									$+$ &$\vee$ &$\wedge$ &$\phantom{\lhd}$ & &\\
									$-$ &$\wedge$ &$\vee$\\
									\hline
								\end{tabular}
								&
								\begin{tabular}{c c c c }
									$+$ &$\wedge$ &$g$ & with $n_g = 1$ \\
									$-$ &$\vee$ &$f$ & with $n_f = 1$ \\
									\hline
								\end{tabular}
								\\
								SLR &SRR\\
								\begin{tabular}{c c c c }
									$+$ & $\wedge$ &$f$ & with $n_f \geq 1$\\
									$-$ & $\vee$ &$g$ & with $n_g \geq 1$ \\
								\end{tabular}
								&\begin{tabular}{c c c c}
									$+$ &$\vee$ &$g$ & with $n_g \geq 2$\\
									$-$ & $\wedge$ &$f$ & with $n_f \geq 2$\\
								\end{tabular}
								\\
								\hline
							\end{tabular}
						\end{center}
						\caption{Skeleton and PIA nodes for $\mathrm{DLE}$.}\label{Join:and:Meet:Friendly:Table}
						\vspace{-1em}
					\end{table}
				\end{definition}

				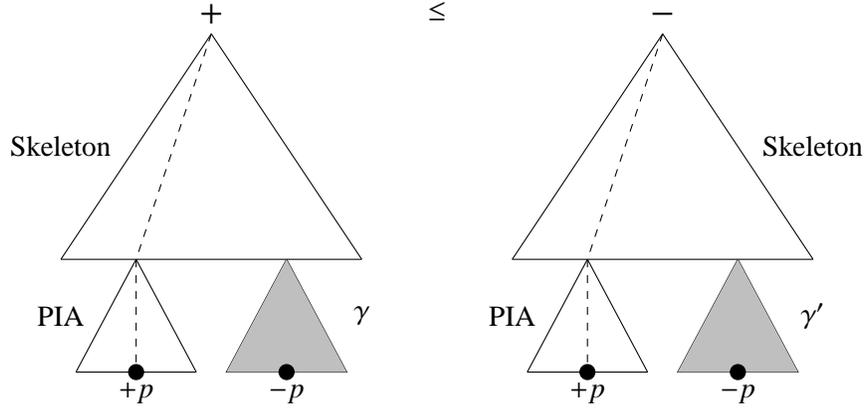
\begin{figure*}
					\begin{center}
						\begin{tikzpicture}
						\draw (-5,-1.5)   -- (-3,1.5) node[above]{\Large$+$} ;
						\draw (-5,-1.5) -- (-1,-1.5) ;
						\draw (-3,1.5) -- (-1,-1.5);
						\draw (-5,0) node{Skeleton}  ;
						\draw[dashed] (-3,1.5) -- (-4,-1.5);
						%\draw[dashed] (-3,1.5) -- (-2,-1.5);
						\draw (-4,-1.5) --(-4.8,-3);
						\draw (-4.8,-3) --(-3.2,-3);
						\draw (-3.2,-3) --(-4,-1.5);
						\draw[dashed] (-4,-1.5) -- (-4,-3);
						\draw[fill] (-4,-3) circle[radius=.1] node[below]{$+p$};
						\draw
						(-2,-1.5) -- (-2.8,-3) -- (-1.2,-3) -- (-2,-1.5);
						\fill[lightgray]
						(-2,-1.5) -- (-2.8,-3) -- (-1.2,-3);
						\draw (-1,-2.25)node{$\gamma$};
						\draw (-5,-2.25) node{PIA}  ;
						\draw (0,1.8) node{$\leq$};
						\draw (5,-1.5)   -- (3,1.5) node[above]{\Large$-$} ;
						\draw (5,-1.5) -- (1,-1.5) ;
						\draw (3,1.5) -- (1,-1.5);
						\draw (5,0) node{Skeleton}  ;
						%\draw[dashed] (3,1.5) -- (4,-1.5);
						\draw[dashed] (3,1.5) -- (2,-1.5);
						\draw (2,-1.5) --(2.8,-3);
						\draw (2.8,-3) --(1.2,-3);
						\draw (1.2,-3) --(2,-1.5);
						\draw[dashed] (2,-1.5) -- (2,-3);
						\draw[fill] (2,-3) circle[radius=.1] node[below]{$+p$};
						\draw
						(4,-1.5) -- (4.8,-3) -- (3.2,-3) -- (4, -1.5);
						\fill[lightgray]
						(4,-1.5) -- (4.8,-3) -- (3.2,-3) -- (4, -1.5);
						\draw (5,-2.25)node{$\gamma'$};
						\draw (1,-2.25) node{PIA} ;
						\draw[fill] (-2,-3) circle[radius=.1] node[below]{$-p$};
						\draw[fill] (4,-3) circle[radius=.1] node[below]{$-p$};
						\end{tikzpicture}
						\caption{A schematic representation of inductive inequalities.}
					\end{center}
				\end{figure*}

				\begin{definition}[Inductive inequalities]
					\label{Inducive:Ineq:Def}
					For any order type $\epsilon$ and any irreflexive and transitive relation $<_\Omega$ on $p_1,\ldots p_n$, the signed generation tree $*s$ $(* \in \{-, + \})$ of a term $s(p_1,\ldots p_n)$ is \emph{$(\Omega, \epsilon)$-inductive} if
					\begin{enumerate}
						\item for all $1 \leq i \leq n$, every $\epsilon$-critical branch with leaf $p_i$ is good (cf.\ Definition \ref{def:good:branch});
						\item every $m$-ary SRR-node occurring in the critical branch is of the form $ \circledast(\gamma_1,\dots,\gamma_{j-1},\beta,\gamma_{j+1}\ldots,\gamma_m)$, where for any $h\in\{1,\ldots,m\}\setminus j$: %$\gamma \in \{\gamma_1,\ldots,\gamma_{j-1},\gamma_{j+1},\ldots,\alpha_m\}$
						\begin{enumerate}
							\item $\epsilon^\partial(\gamma_h) \prec \ast s$ (cf.\ discussion before Definition \ref{def:good:branch}), and
							%\item $\epsilon^\partial(\ast \gamma)$, and
							%
							\item $p_k <_{\Omega} p_i$ for every $p_k$ occurring in $\gamma_h$ and for every $1\leq k\leq n$.
						\end{enumerate}

						%for every SRR node $+g'(\vec\phi, \vec\psi, \theta)$ ($g'\in \mathcal{G}\cup\{\vee\}$) or $-f'(\vec\psi, \vec\phi, \theta)$ ($f'\in \mathcal{F}\cup\{\wedge\}$) %with $n_{f'}, n_{g'}\geq 2$
						%				on an $\epsilon$-critical branch running through $\theta$,\marginnote{$f'$ and $g'$ are just for the sake of including also $+\vee$ and $-\wedge$. Is this notation too disgusting?}
						%				\begin{itemize}
						%					\item $+\phi$ is $\epsilon^{\partial}$-uniform for every $\phi$ in $\vec \phi$, and $+\psi$ is $\epsilon$-uniform for every $\psi$ in $\vec \psi$, and
						%					\item $p_j <_{\Omega} p_i$ for every $p_j$ occurring in any term in $\vec\phi $ or in $\vec\psi$.
						%				\end{itemize}
					\end{enumerate}
					
					We will refer to $<_{\Omega}$ as the \emph{dependency order} on the variables. An inequality $s \leq t$ is \emph{$(\Omega, \epsilon)$-inductive} if the signed generation trees $+s$ and $-t$ are $(\Omega, \epsilon)$-inductive. An inequality $s \leq t$ is \emph{inductive} if it is $(\Omega, \epsilon)$-inductive for some $\Omega$ and $\epsilon$.
				\end{definition}

				In what follows, we will find it useful to refer to formulas $\phi$ such that only PIA nodes occur in $+\phi$ (resp.\ $-\phi$) as {\em positive} (resp.\ {\em negative}) {\em PIA-formulas}, and to formulas $\xi$ such that only Skeleton nodes occur in $+\xi$ (resp.\ $-\xi$) as {\em positive} (resp.\ {\em negative}) {\em Skeleton-formulas}\label{page: positive negative PIA}.
				
				\medskip
		The proof of the following theorem is a straightforward generalization of \cite[Theorem 10.11]{CoPa10}, and hence its proof is omitted.
		\begin{thm}\label{Thm:ALBA:Success:Inductive}
			For any  language $\mathcal{L}_\mathrm{DLE}$, its corresponding version of ALBA succeeds on all inductive $\mathcal{L}_\mathrm{DLE}$-inequalities, which are hence canonical\footnote{An $\mathcal{L}_\mathrm{DLE}$-inequality $s\leq t$ is {\em canonical} if the class of $\mathcal{L}_\mathrm{DLE}$-algebras defined by $s\leq t$ is closed under the construction of canonical extension (cf.\ Footnote \ref{def:can:ext}).} and their corresponding logics are complete w.r.t.\ the elementary classes of relational structures\footnote{Such are those introduced in \cite{sofronie2000duality1, sofronie2000duality2}.} defined by their first-order correspondents.
		\end{thm}%\marginnote{Apostolos: This Theorem 17 was Minghui's issue with the canonical. We already mention the ALBA paper. As for the Viorica Sofronie-Stokkermans reference, I found the following: ``\emph{Duality and Canonical Extensions of Bounded Distributive Lattices with Operators and Applications to the Semantics of Non-Classical Logics}'' Parts I and II. If this is the one that we needed then they are added to the references as sofronie2000duality1 and sofronie2000duality2}

		\section{Display calculi for $\mathbf{L}_\mathrm{DLE}$ and $\mathbf{L}_\mathrm{DLE}^*$}
		\label{sec:display calculi DL DL ast}
		In the present section, we introduce the basic proof-theoretic environment of our treatment, given by the display calculi $\mathbf{DL}^*$ and $\mathbf{DL}$ for the logics $\mathbf{L}_\mathrm{DLE}$ and $\mathbf{L}_\mathrm{DLE}^*$ associated with any given language $\mathcal{L}_\mathrm{DLE}(\mathcal{F}, \mathcal{G})$. % in the language $\mathcal{L}_\mathrm{DEL}^*$, which is the extension of $\mathbf{L}_\mathrm{DLE}$ by adding rules for adjoints and residuals of connectives in $\mathbf{L}_\mathrm{DLE}$.
		%We will show the display calculus for $\mathbf{L}_\mathrm{DLE}^*$
		We also show some of their basic properties. %show that they enjoy Belnap-style cut elimination and subformula property. %Hence a display calculus $\mathbf{L}_\mathrm{DLE}$ will be obtained by dropping rules for additional operators. We refer to cite{Belnap,Wa98,Kracht,Gore1,Restall} for basic concepts and results in display calculi.
		
		\subsection{Language and rules}
		The present subsection is aimed at simultaneously introducing the display calculi $\mathbf{DL}^*$ and $\mathbf{DL}$ for $\mathbf{L}_\mathrm{DLE}^*$ and $\mathbf{L}_\mathrm{DLE}$, respectively. As is usual of existing logical systems which the present framework intends to capture (e.g.\ intuitionistic and bi-intuitionistic logics, or modal and tense logics \cite{Go00}), the languages manipulated by these calculi are built up using {\em one and the same} set of structural terms, and differ only in the set of operational term constructors. In the tables below, each structural symbol in the upper rows corresponds to one or two logical (or operational) symbols. The idea, which will be made precise later on, is that each structural connective can be interpreted as the corresponding left-hand (resp.\ right-hand) side logical connective (if it exists) when occurring in antecedent (resp.\ consequent) position.
		\begin{itemize}
			\item Structural symbols for lattice operators:
			\begin{center}
				\begin{tabular}{|r|c|c|c|c|c|c|c|c|}
					\hline
					\scriptsize{Structural symbols} & \mc{2}{c|}{I} & \mc{2}{c|}{$;$} & \mc{2}{c|}{$>$} & \mc{2}{c|}{$<$}\\
					\hline
					\scriptsize{Operational symbols} & $\top$ & $\bot$ & $\pand$ & $\por$ & $(\pdra)$ & $(\pra)$ & $(\pdla)$ & $(\leftarrow)$\\
					\hline
				\end{tabular}
			\end{center}
			
			\item Structural symbols for any $f\in \mathcal{F}$ and $g\in \mathcal{G}$:
			\begin{center}
				\begin{tabular}{|r|c|c|c|c|}
					\hline
					\scriptsize{Structural symbols} & \mc{2}{c|}{$H$} & \mc{2}{c|}{$K$} \\
					\hline
					\scriptsize{Operational symbols} & $f$ & $\phantom{f}$ & $\phantom{g}$ & $g$ \\
					\hline
				\end{tabular}
			\end{center}
			
			\item Structural symbols for any $f_i^\sharp, g_h^\flat \in (\mathcal{F}^*\cup \mathcal{G}^*)\setminus (\mathcal{F}\cup \mathcal{G})$, and any $0\leq i\leq n_f$ and $0\leq h\leq n_g$:
			\begin{center}
				\begin{tabular}{|r|c|c|c|c|c|c|c|c|}
					\hline
					\scriptsize{Structural symbols} & \mc{2}{c|}{$H_i$ ($\varepsilon_f(i) = 1$)} & \mc{2}{c|}{$H_i$ ($\varepsilon_f(i) = \partial$)} & \mc{2}{c|}{$K_h$ ($\varepsilon_g(h) = 1$)} & \mc{2}{c|}{$K_h$ ($\varepsilon_g(h) = \partial$)} \\
					\hline
					\scriptsize{Operational symbols} & \ \ \ \ \ \ \ \, & $ (f_i^\sharp)$ & \,\,$(f_i^\sharp)\,\, $ &  & \ \,$(g_h^\flat)$\ \, & $\phantom{(g_h^\flat)}$ & $\phantom{(g_h^\flat)}$ \ \,& \,$(g_h^\flat)\,$\\
					\hline
				\end{tabular}
			\end{center}
		\end{itemize}
		Some operational symbols above appear in brackets as a reminder that, unlike their associated structural symbols, they might occur only in the language and in the calculus for $\mathbf{L}_\mathrm{DLE}^*$.
		\begin{remark} If $f\in \mathcal{F}$ and $g\in \mathcal{G}$ form a dual pair,\footnote{Examples of dual pairs are $(\top, \bot)$, $(\wedge, \vee)$, $(\pdra, \pra)$, $(\pdla, \leftarrow)$, and $(\Diamond, \Box)$ where $\Diamond$ is defined as $\neg\Box\neg$.} %\footnote{The connectives $f\in \mathcal{F}$ and $g\in \mathcal{G}$ form a {\em dual pair} if $n_f = n_g = n\geq 1$, and for every DLE-relational structure in which $f$ and $g$ are interpreted by means of the $n+1$-ary relations $R$ and $S$ respectively, $R^{-1}[X_1,\ldots, X_n] = (S^{-1}[X_1^c,\ldots, X_n^c])^c$ for every $n$-tuple $(X_1,\ldots, X_n)$ of potential interpretants of proposition variables. For any set $W$ and any $n+1$-ary relation $R$ on $W$, we let $R^{-1}[X_1,\ldots, X_n]: = \{y\mid R(y, x_1,\ldots, x_n)$ for some $x_i\in X_i, 1\leq i\leq n\}$.},\marginnote{check definition of dual pair in the footnote. This is very cumbersome, but the thing is that I don't want to give it in terms of boolean negation, since e.g.\ it is also true in the distributive case}
 then $n_f = n_g$ and $\varepsilon_f = \varepsilon_g$. Then $f$ and $g$ can be assigned one and the same structural operator, as follows:
			
			\begin{center}
				\begin{tabular}{|r|c|c|}
					\hline
					\scriptsize{Structural symbols} & \mc{2}{c|}{$H$} \\
					\hline
					\scriptsize{Operational symbols} & $f$ & $g$ \\
					\hline
				\end{tabular}
			\end{center}
			Moreover, for any $1\leq i\leq n_f = n_g$, the residuals $f_i^\sharp$ and $g_i^\flat$ are dual to one another. Hence they can also be assigned one and the same structural connective as follows:
			
			\begin{center}
				\begin{tabular}{|r|c|c|c|c|}
					\hline
					\scriptsize{Structural symbols} & \mc{2}{c|}{$H_i$ ($\varepsilon_f(i) = \varepsilon_g(i) = 1$)} & \mc{2}{c|}{$H_i$ ($\varepsilon_f(i) = \varepsilon_g(i) = \partial$)} \\
					\hline
					\scriptsize{Operational symbols} & \ \ \,\,\,$(g_i^\flat)$\ \ \,\,\,& $(f_i^\sharp)$ & \ \ \,\,\,$(f_i^\sharp)$\ \ \,\,\, & $(g_i^\flat)$\\
					\hline
				\end{tabular}
			\end{center}
			%Notice that ($f, g$) and ($g^\flat, f^\sharp$) are (coordinate-wise) \emph{dual pairs}, while the operators ($f, f^\sharp$) and ($g^\flat, g$) are (coordinatewise) \emph{residual pairs} as $f \dashv f^\sharp$ and $g^\flat \dashv g$.
		\end{remark}
		%A {\em structure} is an expression built from $\mathcal{L}_\mathrm{DLE}^*$-terms via structural operators. We use capital letters $X, Y, Z$ to denote structures. A sequent is of the form $X\vdash Y$ where $X,Y$ are structures. $X$ is called {\em antecedent}, and $Y$ the {\em succedent}.
		
		\begin{definition}
			\label{def: DL and DL star}
			The display calculi $\mathbf{DL}^*$ and $\mathbf{DL}$ consist of the following display postulates, structural rules, and operational
			rules:\footnote{\label{ftn: display calculus DL star minus}The display calculus associated with the basic DLE-logic $\mathbf{L}^{\underline{*}}$ (cf.\ footnote \ref{ftn: definition basic logic for expanded language}) in the expanded language $\mathcal{L}_\mathrm{DLE}^*$ is denoted by $\mathbf{DL}^{\underline{*}}$, and is defined by instantiating the definition of $\mathbf{DL}$ to the expanded language $\mathcal{L}_\mathrm{DLE}^*$.}
			\begin{enumerate}
				\item Identity and cut:
				
				\begin{center}
					\begin{tabular}{cc c}
						
						$p \vdash p$
						
						& &
						\AX$X \fCenter A$
						\AX$A \fCenter Y$
						\BI$X \fCenter Y$
						\DisplayProof
						
					\end{tabular}
				\end{center}
				\item Display postulates for lattice connectives:
				\begin{center}
					\begin{tabular}{cccc}
						\AX$X \,; Y \fCenter Z$
						\doubleLine
						\UI$Y \fCenter X > Z$
						\DisplayProof
						&
						\AX$Z \fCenter X \,; Y$
						\doubleLine
						\UI$X > Z \fCenter Y$
						\DisplayProof
						&
						\AX$X \,; Y \fCenter Z$
						\doubleLine
						\UI$X \fCenter Z < Y$
						\DisplayProof
						&
						\AX$Z \fCenter X \,; Y$
						\doubleLine
						\UI$Z < Y \fCenter X$
						\DisplayProof
					\end{tabular}
				\end{center}
				\item Display postulates for $f\in \mathcal{F}$ and $g\in \mathcal{G}$: for any $1\leq i\leq n_f$ and $1\leq h\leq n_g$,
				\begin{center}
					\begin{tabular}{cc}
						\AXC{$H\, (X_1, \ldots, X_i, \ldots, X_{n_f}) \fCenter Y$}
						\doubleLine
						\LeftLabel{$(\varepsilon_f(i) = 1)$}
						\UIC{$X_i \fCenter H_i\, (X_1, \ldots, Y, \ldots, X_{n_f})$}
						\DisplayProof
						&
						\AXC{$Y \fCenter K\, (X_1 \ldots, X_h, \ldots X_{n_g})$}
						\doubleLine
						\RightLabel{$(\varepsilon_g(h) = 1)$}
						\UIC{$K_h\, (X_1, \ldots, Y, \ldots, X_{n_g}) \fCenter X_h$}
						\DisplayProof \\
						
						& \\
						
						\AXC{$H\, (X_1, \ldots, X_i, \ldots, X_{n_f}) \fCenter Y$}
						\doubleLine
						\LeftLabel{$(\varepsilon_f(i) = \partial)$}
						\UIC{$H_i\, (X_1, \ldots, Y, \ldots, X_{n_f}) \fCenter X_i $}
						\DisplayProof
						&
						\AXC{$Y \fCenter K\, (X_1, \ldots, X_h, \ldots, X_{n_g})$}
						\doubleLine
						\RightLabel{$(\varepsilon_g(h) = \partial)$}
						\UIC{$ X_h\fCenter K_h\, (X_1, \ldots, Y, \ldots, X_{n_g})$}
						\DisplayProof \\
					\end{tabular}
				\end{center}
				Notice that the display postulates for all the connectives in $\mathcal{F}^*\cup \mathcal{G}^*$ are derivable from the display postulates above. The rules for the case of connectives in the dual pairs are obtained by replacing $K$ for $H$ in the corresponding rules above.

%%%

				\item Necessitation for $f\in \mathcal{F}$ and $g\in \mathcal{G}$: for any $1\leq k \leq n_f$ and $1\leq h\leq n_g$,
				\begin{center}
					\begin{tabular}{c}
						\bottomAlignProof
						\AxiomC{$\Big(X_i \vdash Y_i \quad Y_j \vdash X_j \mid i \neq k, 1\leq i, j\leq n_f, \varepsilon_f(i) = 1\mbox{ and } \varepsilon_f(j) = \partial \Big) \quad\quad\quad\quad X_k \vdash \textrm{I}_k$}
						\LeftLabel{$(\varepsilon_f(k) = 1)$}
						\UnaryInfC{$X_k \vdash H_k (X_1,\ldots, X_{k-1}, \textrm{I}, X_{k+1}, \ldots, X_{n_f})$}
						\DisplayProof \\
					\end{tabular}
				\end{center}

				\begin{center}
					\begin{tabular}{c}
						\bottomAlignProof
						\AxiomC{$\Big(X_i \vdash Y_i \quad Y_j \vdash X_j \mid j \neq k, 1\leq i, j\leq n_f, \varepsilon_f(i) = 1\mbox{ and } \varepsilon_f(j) = \partial \Big) \quad\quad\quad\quad \textrm{I}_k \vdash X_k$}
						\LeftLabel{$(\varepsilon_f(k) = \partial)$}
						\UnaryInfC{$H_k (X_1,\ldots, X_{k-1}, \textrm{I}, X_{k+1}, \ldots, X_{n_f}) \vdash X_k$}
						\DisplayProof \\
					\end{tabular}
				\end{center}

%%%

				\begin{center}
					\begin{tabular}{c}
						\bottomAlignProof
						\AxiomC{$\Big(X_j \vdash Y_j \quad Y_i \vdash X_i \mid i \neq h, 1\leq i, j\leq n_f, \varepsilon_f(i) = 1\mbox{ and } \varepsilon_f(j) = \partial \Big) \quad\quad\quad\quad \textrm{I}_h \vdash X_h$}
						\LeftLabel{$(\varepsilon_g(h) = 1)$}
						\UnaryInfC{$K_h (X_1,\ldots, X_{h-1}, \textrm{I}, X_{h+1}, \ldots, X_{n_g}) \vdash X_h$}
						\DisplayProof \\
					\end{tabular}
				\end{center}

				\begin{center}
					\begin{tabular}{c}
						\bottomAlignProof
						\AxiomC{$\Big(X_j \vdash Y_j \quad Y_i \vdash X_i \mid j \neq h, 1\leq i, j\leq n_f, \varepsilon_f(i) = 1\mbox{ and } \varepsilon_f(j) = \partial \Big) \quad\quad\quad\quad X_h \vdash \textrm{I}_h$}
						\LeftLabel{$(\varepsilon_g(h) = \partial)$}
						\UnaryInfC{$X_h \vdash K_h (X_1,\ldots, X_{h-1}, \textrm{I}, X_{h+1}, \ldots, X_{n_g})$}
						\DisplayProof \\
					\end{tabular}
				\end{center}

%%%

				\item Structural rules encoding the distributive lattice axiomatization:
				\begin{center}
					\begin{tabular}{rlcrl}
						\AX$X \fCenter Y$
						\doubleLine
						\LeftLabel{\fns$\textrm{I}_{L}$}
						\UI$\textrm{I}\,; X \fCenter Y$
						\DisplayProof
						&
						\AX$Y \fCenter X$
						\doubleLine
						\RightLabel{\fns$\textrm{I}_{R}$}
						\UI$Y \fCenter X\,; \textrm{I}$
						\DisplayProof
						& &
						\AX$Y \,; X \fCenter Z$
						\LeftLabel{\fns$E_L$}
						\UI$X \,; Y \fCenter Z $
						\DisplayProof
						&
						\AX$Z \fCenter X \,; Y$
						\RightLabel{\fns$E_R$}
						\UI$Z \fCenter Y \,; X$
						\DisplayProof
						\\
						&\\
						\AX$Y \fCenter Z$
						\LeftLabel{\fns$W_L$}
						\UI$X\,; Y \fCenter Z$
						\DisplayProof
						&
						\AX$Z \fCenter Y$
						\RightLabel{\fns$W_R$}
						\UI$Z \fCenter Y\,; X$
						\DisplayProof
						& &
						\AX$X \,; X \fCenter Y$
						\LeftLabel{\fns$C_L$}
						\UI$X \fCenter Y $
						\DisplayProof
						&
						\AX$Y \fCenter X \,; X$
						\RightLabel{\fns$C_R$}
						\UI$Y \fCenter X$
						\DisplayProof
						\\
						&\\
						\mc{2}{c}{
							\AX$X \,; (Y \,; Z) \fCenter W$
							\doubleLine
							\LeftLabel{\fns$A_{L}$}
							\UI$(X \,; Y) \,; Z \fCenter W $
							\DisplayProof}
						& &
						\mc{2}{c}{
							\AX$W \fCenter (Z \,; Y) \,; X$
							\doubleLine
							\RightLabel{\fns$A_{R}$}
							\UI$W \fCenter Z \,; (Y \,; X)$
							\DisplayProof}
					\end{tabular}
				\end{center}
				\item Introduction rules for the propositional (BDL and bi-intuitionistic) connectives:
				\begin{center}
					\begin{tabular}{rlrl}
						\AXC{\phantom{$\gbot \fCenter \textrm{I}$}}
						\LeftLabel{\fns$\bot_L$}
						\UI$\bot \fCenter \textrm{I}$
						\DisplayProof
						&
						\AX$X \fCenter \textrm{I}$
						\RightLabel{\fns$\bot_R$}
						\UI$X \fCenter \bot$
						\DisplayProof
						&
						\AX$\textrm{I} \fCenter X$
						\LeftLabel{\fns$\top_L$}
						\UI$\top \fCenter X$
						\DisplayProof
						&
						\AXC{\phantom{$\textrm{I} \fCenter \top$}}
						\RightLabel{\fns$\top_R$}
						\UI$\textrm{I} \fCenter \top$
						\DisplayProof
						\\
						& & & \\
						\AX$A \,; B \fCenter X$
						\LeftLabel{\fns$\pand_L$}
						\UI$A \pand B \fCenter X$
						\DisplayProof
						&
						\AX$X \fCenter A$
						\AX$Y \fCenter B$
						\RightLabel{\fns$\pand_R$}
						\BI$X \,; Y \fCenter A \pand B$
						\DisplayProof
						&
						\AX$A \fCenter X$
						\AX$B \fCenter Y$
						\LeftLabel{\fns$\por_L$}
						\BI$A \por B \fCenter X \,; Y$
						\DisplayProof
						&
						\AX$X \fCenter A \,; B$
						\RightLabel{\fns$\por_R$}
						\UI$X \fCenter A \por B$
						\DisplayProof
						\\
						& \\
						\AX$X \fCenter A$
						\AX$B \fCenter Y$
						\LeftLabel{\fns$\pra_L$}
						\BI$A \pra B \fCenter X > Y$
						\DisplayProof
						&
						\AX$X \fCenter A > B$
						\RightLabel{\fns$\pra_R$}
						\UI$X \fCenter A \pra B$
						\DisplayProof
						&
						\AX$A > B \fCenter Z$
						\LeftLabel{\fns$\pdra_L$}
						\UI$A \pdra B \fCenter Z$
						\DisplayProof
						&
						\AX$A \fCenter X$
						\AX$Y \fCenter B$
						\RightLabel{\fns$\pdra_R$}
						\BI$X > Y \fCenter A \pdra B$
						\DisplayProof
						\\
						& \\
						\AX$X \fCenter A$
						\AX$B \fCenter Y$
						\LeftLabel{\fns$\leftarrow_L$}
						\BI$A \leftarrow B \fCenter X < Y$
						\DisplayProof
						&
						\AX$X \fCenter A < B$
						\RightLabel{\fns$\leftarrow_R$}
						\UI$X \fCenter A \leftarrow B$
						\DisplayProof
						&
						\AX$A < B \fCenter Z$
						\LeftLabel{\fns$\pdla_L$}
						\UI$A \pdla B \fCenter Z$
						\DisplayProof
						&
						\AX$A \fCenter X$
						\AX$Y \fCenter B$
						\RightLabel{\fns$\pdla_R$}
						\BI$X < Y \fCenter A \pdla B$
						\DisplayProof
					\end{tabular}
				\end{center}
				In the presence of the exchange rules $E_L$ and $E_R$, the structural connective $<$ and the corresponding operational connectives $\pdla$ and $\pla$ are redundant.
				\item Introduction rules for $f\in\mathcal{F}$ and $g\in\mathcal{G}$:
				\begin{center}
					\begin{tabular}{c c}
						\bottomAlignProof
						\AxiomC{$H (A_1,\ldots, A_{n_f})\vdash X$}
						\LeftLabel{\fns$f_L$}
						\UnaryInfC{$f(A_1,\ldots, A_{n_f})\vdash X$}
						\DisplayProof
						&
						\bottomAlignProof
						\AxiomC{$X \vdash K(A_1,\ldots, A_{n_g})$}
						\RightLabel{\fns$g_R$}
						\UnaryInfC{$X \vdash g(A_1,\ldots, A_{n_g})$}
						\DisplayProof
					\end{tabular}
					\begin{tabular}{c}
						\bottomAlignProof
						\AxiomC{$\Big(X_i \vdash A_i \quad A_j \vdash X_j \mid 1\leq i, j\leq n_f, \varepsilon_f(i) = 1\mbox{ and } \varepsilon_f(j) = \partial \Big)$}
						\LeftLabel{\fns$f_R$}
						\UnaryInfC{$H(X_1,\ldots, X_{n_f})\vdash f(A_1,\ldots, A_{n})$}
						\DisplayProof
					\end{tabular}
					\begin{tabular}{c c c }
						\bottomAlignProof
						\AxiomC{$\Big( A_i \vdash X_i \quad X_j \vdash A_j \,\mid\, 1\leq i, j\leq n_g, \varepsilon_g(i) = 1\mbox{ and } \varepsilon_g(j) = \partial \Big)$}
						\LeftLabel{\fns$g_L$}
						\UnaryInfC{$g(A_1,\ldots, A_{n_g}) \vdash K (X_1,\ldots, X_{n})$}
						\DisplayProof
					\end{tabular}
				\end{center}
				%where $i$-monotone ($j$-antitone) means that the operator is order preserving (order reversing) at the $i$-th ($j$-th) coordinate.
				In particular, if $f$ and $g$ are $0$-ary (i.e.\ they are constants), the rules $f_R$ and $g_L$ above reduce to the axioms ($0$-ary rule) $H\vdash f$ and $g\vdash K$.
				
				\item Only for $\mathrm{DL}^*$, introduction rules for each $f_i^\sharp, g_h^\flat\in (\mathcal{F}^*\cup\mathcal{G}^*)\setminus (\mathcal{F}\cup \mathcal{G})$:
				
				\begin{enumerate}

					\item If $\varepsilon_f(i) = 1$ and $\varepsilon_g(h) = 1$,
					\begin{center}
						\begin{tabular}{c c}
							\bottomAlignProof
							\AxiomC{$K_h (A_1,\ldots, A_{n_g})\vdash X$}
							\LeftLabel{\fns${g_h^\flat}_L$}
							\UnaryInfC{$g_h^\flat(A_1,\ldots, A_{n_f})\vdash X$}
							\DisplayProof
							&
							\bottomAlignProof
							\AxiomC{$X \vdash H_i(A_1,\ldots, A_{n_g})$}
							\RightLabel{\fns${f_i^\sharp}_R$}
							\UnaryInfC{$X \vdash f_i^\sharp(A_1,\ldots, A_{n_g})$}
							\DisplayProof
						\end{tabular}
						\begin{tabular}{c}
							\bottomAlignProof
							\bottomAlignProof
							\AxiomC{$\Big(X_{\ell} \vdash A_{\ell} \quad A_m \vdash X_m \mid 1\leq \ell, m\leq n_g, \varepsilon_{g_h^\flat}(\ell) = 1\mbox{ and } \varepsilon_{g_h^\flat}(m) = \partial \Big)$}
							\LeftLabel{\fns${g_h^\flat}_R$}
							\UnaryInfC{$K_h(X_1,\ldots, X_{n_g})\vdash g_h^\flat(A_1,\ldots, A_{n_g})$}
							\DisplayProof
						\end{tabular}
						\begin{tabular}{c c c }
							\bottomAlignProof
							\AxiomC{$\Big( A_\ell \vdash X_\ell \quad X_m \vdash A_m \,\mid\, 1\leq \ell, m\leq n_g, \varepsilon_{f_i^\sharp}(\ell) = 1\mbox{ and } \varepsilon_{f_i^\sharp}(m) = \partial \Big)$}
							\LeftLabel{\fns${f_i^\sharp}_L$}
							\UnaryInfC{$f_i^\sharp(A_1,\ldots, A_{n_g}) \vdash H_i (X_1,\ldots, X_{n_g})$}
							\DisplayProof
						\end{tabular}
					\end{center}
					%where $i$-monotone ($j$-antitone) means that the operator is order preserving (order reversing) at the $i$-th ($j$-th) coordinate.
					%In particular, if $f$ and $g$ are $0$-ary (i.e.\ they are constants), the rules above reduce to the axioms ($0$-ary rule) $H\vdash f$ and $g\vdash K$.
					
					\item If $\varepsilon_f(i) = \partial$ and $\varepsilon_g(h) = \partial$,
					\begin{center}
						\begin{tabular}{c c}
							\bottomAlignProof
							\AxiomC{$H_i (A_1,\ldots, A_{n_f})\vdash X$}
							\LeftLabel{\fns${f_i^\sharp}_L$}
							\UnaryInfC{$f_i^\sharp(A_1,\ldots, A_{n_f})\vdash X$}
							\DisplayProof
							&
							\bottomAlignProof
							\AxiomC{$X \vdash K_h(A_1,\ldots, A_{n_g})$}
							\RightLabel{\fns${g_h^\flat}_R$}
							\UnaryInfC{$X \vdash g_h^\flat(A_1,\ldots, A_{n_g})$}
							\DisplayProof
						\end{tabular}
						\begin{tabular}{c}
							\bottomAlignProof
							\bottomAlignProof
							\AxiomC{$\Big(X_{\ell} \vdash A_{\ell} \quad A_m \vdash X_m \mid 1\leq \ell, m\leq n_f, \varepsilon_{f_i^\sharp}(\ell) = 1\mbox{ and } \varepsilon_{f_i^\sharp}(m) = \partial \Big)$}
							\LeftLabel{\fns${f_i^\sharp}_R$}
							\UnaryInfC{$H_i(X_1,\ldots, X_{n_f})\vdash f_i^\sharp(A_1,\ldots, A_{n_f})$}
							\DisplayProof
						\end{tabular}
						\begin{tabular}{c c c }
							\bottomAlignProof
							\AxiomC{$\Big( A_\ell \vdash X_\ell \quad X_m \vdash A_m \,\mid\, 1\leq \ell, m\leq n_g, \varepsilon_{g_h^\flat}(\ell) = 1\mbox{ and } \varepsilon_{g_h^\flat}(m) = \partial \Big)$}
							\LeftLabel{\fns${g_h^\flat}_L$}
							\UnaryInfC{$g_h^\flat(A_1,\ldots, A_{n_g}) \vdash K_h (X_1,\ldots, X_{n_g})$}
							\DisplayProof
						\end{tabular}
					\end{center}
					%where $i$-monotone ($j$-antitone) means that the operator is order preserving (order reversing) at the $i$-th ($j$-th) coordinate.
					%In particular, if $f$ and $g$ are $0$-ary (i.e.\ they are constants), the rules above reduce to the axioms ($0$-ary rule) $H\vdash f$ and $g\vdash K$.
					
				\end{enumerate}
			\end{enumerate}

		\end{definition}
		
		A display calculus enjoys the {\em full display property} (resp.\ the {\em relativized display property}) if for every (derivable) sequent $X \vdash Y$ and every substructure $Z$ of either $X$ or $Y$, the sequent $X \vdash Y$ can be equivalently transformed, using the rules of the system, into a sequent which is either of the form $Z \vdash W$ or of the form $W \vdash Z$, for some structure $W$. %In the first case, $Z$ occurs is \emph{displayed in precedent position}, and in the second case, $Z$ is \emph{displayed in succedent position}.
		A routine check will show that the display calculi $\mathbf{DL}$ and $\mathbf{DL}^*$ both enjoy the relativized display property, and moreover, if $\mathcal{F}$ and $\mathcal{G}$ are such that for every $f\in \mathcal{F}$ the dual of $f$ is in $\mathcal{G}$ and for every $g\in \mathcal{G}$ the dual of $g$ is in $\mathcal{F}$, then $\mathbf{DL}$ and $\mathbf{DL}^*$ both enjoy the full display property. The proof of these facts is omitted.
		
		\begin{prop}
\label{prop: DL has relativized display property}
			The display calculi $\mathbf{DL}$ and $\mathbf{DL}^*$ enjoy the relativized display property, and under the assumption above on $\mathcal{F}$ and $\mathcal{G}$ they enjoy the full display property.
		\end{prop}
		
		\subsection{Soundness, completeness, conservativity}
		\label{ssec:soundness}
		
		\paragraph{Soundness.} Let us expand on how to interpret structures and sequents in the language manipulated by the calculi $\mathbf{DL}$ and $\mathbf{DL}^*$ in any perfect $\mathcal{L}_\mathrm{DLE}$-algebra $\bba$ (cf.\ Footnote \ref{footnote: perfect DLE}). Structures will be translated into formulas, and formulas will be interpreted as elements of $\bba$. In order to translate structures as formulas, structural terms need to be translated as formulas, %To this effect,
		%as mentioned in Section \ref{ssec:DisplayLogic},
		%structural connectives are associated with either one or two logical connectives
as is specified in Definition \ref{def:rsls} below. To this effect, any given occurrence of a structural connective in a sequent is translated as (one or the other of) its associated logical connective(s), as reported in
		Table \ref{pic:1}, provided its operational counterpart relative to its position (antecedent or succedent) exists. % of the (sub)structure rooted at that given occurrence is compatible with the stipulations, according to which side of the sequent the given occurrence can be displayed on as main connective, % on and on its position within the generation tree of the structure %{\em pairs} of logical connectives %(some of which do not belong to the original logical signature; however, the special properties of perfect distributive lattices guarantee the existence of the semantic interpretation of these additional logical connectives),
		\begin{table}
			\begin{center}
				\begin{tabular}{c||c|c c}
					Structural & $\quad$ if in precedent$\quad$ & $\quad$ if in succedent$\quad$\\
					$\quad$ connective $\quad$ \ &\ position \ &\ position & \\[1pt]
					\hline
					I & $\top$ & $\bot$ & \\[1pt]
					$A\,; B$ & $A\wedge B$ & $A \vee B$ & \\[2pt]
					$A > B$ \ & $A \pdra B$ & $A \rightarrow B$ & \\[2pt]
					$H(\overline{A})$ & $f(\overline{A})$ & & \\[4pt]
					$K(\overline{A})$ & & $g(\overline{A})$ & \\[4pt]
					$H_i(\overline{A})$ & & $f_i^\sharp(\overline{A})$ & if $\varepsilon_f(i) = 1$ \\
					$H_i(\overline{A})$ & $f_i^\sharp(\overline{A})$ & & if $\varepsilon_f(i) = \partial$ \\
					$K_h(\overline{A})$ & $g_h^\flat(\overline{A})$ & & if $\varepsilon_g(h) = 1$ \\[4pt]
					$K_h(\overline{A})$ & & $g_h^\flat(\overline{A})$ & if $\varepsilon_g(h) = \partial$ \\[4pt]
				\end{tabular}
			\end{center}
			\caption{Translation of structural connectives into logical connectives}\label{pic:1}
		\end{table}
		Clearly, not all structural terms will in general have a translation as formulas. This motivates the following definition:
		
		\begin{definition}\label{def:sidedstructures}
			A structural term $S$ is \emph{left-sided} (resp.\ \emph{right-sided}) if in its positive (resp.\ negative) signed generation tree,\footnote{Signed generation trees of structural terms are defined analogously to signed generation trees of logical terms. Logical formulas label the leaves of the signed generation trees of structural terms.} every positive node is labelled with a structural connective which is associated with a logical connective when occurring in antecedent position, and every negative node is labelled with a structural connective which is associated with a logical connective when occurring in succedent position.
		\end{definition}
		Clearly, if every structural connective is associated with some logical connectives both when occurring in antecedent position and when occurring in succedent position, as is the case e.g.\ when $\mathcal{F}$ and $\mathcal{G}$ bijectively correspond via conjugation, every structural term is both left-sided and right-sided.
		\begin{definition}\label{def:rsls}
			For every left-sided (resp.\ right-sided) structural term $S$, let $l(S)$ (resp.\ $r(S)$) denote the formula associated with $S$ and defined inductively according to Table \ref{pic:1}.
		\end{definition}
		
		Structural sequents $S\vdash T$ such that $S$ is left-sided and $T$ is right-sided are those translatable as formula-sequents $l(S)\vdash r(T)$. These sequents in turn are interpreted in any $\mathcal{L}_{\mathrm{DLE}}$-algebra $\bba$ in the standard way. Hence, for any assignment $v: \mathsf{AtProp}\rightarrow \bba$, we denote by $\val{\cdot}_v$ the unique homomorphic extension of $v$ to the formula algebra,  interpret sequents $l(S)\vdash r(T)$  as inequalities $\val{l(S)}_v\leq \val{r(T)}_v$ and rules $(S_i\vdash T_i\mid i\in I)/S\vdash T$ as implications of the form ``if $\val{l(S_i)}_v\leq \val{r(T_i)}_v$ for every $i\in I$, then $\val{l(S)}_v\leq \val{r(T)}_v$''.

		Under these stipulations, it is routine to check that all axioms and rules of the calculi $\mathbf{DL}$ and $\mathbf{DL}^*$ are satisfied under any assignment. Hence, it is immediate to prove, by induction on the depth of the derivation tree, that
		
		\begin{prop}
			If $S\vdash T$ is $\mathbf{DL}$-derivable (resp.\ $\mathbf{DL}^*$-derivable), then $S$ is left-sided, $T$ is right-sided and $\val{l(S)}_v\leq \val{r(T)}_v$ is satisfied on every perfect $\mathcal{L}_\mathrm{DLE}$-algebra $\bba$ and under any assignment $v: \mathsf{AtProp}\rightarrow \bba$.
		\end{prop}
		
		\paragraph{Completeness.} At the end of Section \ref{ssec:expanded language}, we outlined the proof of the completeness of $\mathbf{L}_\mathrm{DLE}$ and $\mathbf{L}_\mathrm{DLE}^*$ w.r.t.\ perfect $\mathcal{L}_\mathrm{DLE}$-algebras. Hence, to show that $\mathbf{DL}$ and $\mathbf{DL}^*$ are complete w.r.t.\ perfect $\mathcal{L}_\mathrm{DLE}$-algebras, it is enough to show that the axioms and rules of $\mathbf{L}_\mathrm{DLE}$ (resp.\ $\mathbf{L}_\mathrm{DLE}^*$) are derivable in $\mathbf{DL}$ (resp.\ $\mathbf{DL}^*$). These verifications are routine. For instance, let $f\in \mathcal{F}$ be binary and s.t.\ $\varepsilon_f = (1, \partial)$. Then the following sequents are derivable in $\mathbf{DL}$:
		\[ f_1^\sharp(A, C)\wedge f_1^\sharp(B, C) \vdash f_1^\sharp(A\wedge B, C) \quad \quad f_1^\sharp(A, B)\wedge f_1^\sharp(A, C)\vdash f_1^\sharp(A, B\wedge C) \]
		\[f_2^\sharp(A\vee B, C)\vdash f_2^\sharp(A, C)\vee f_2^\sharp(B, C) \quad \quad f_2^\sharp(A, B\wedge C) \vdash f_2^\sharp(A, B)\vee f_2^\sharp(A, C). \]
By way of example, a derivation for $f_1^\sharp (A, C) \wedge f_1^\sharp (B, C) \fCenter f_1^\sharp (A \wedge B, C)$ is reported below. %is derivable using Weakening, Contraction and Display (or, analogously, Display and additive logical rules for meet):

\begin{center}
\AX$A \fCenter A$
\AX$C \fCenter C$
\BI$f_1^\sharp (A, C) \fCenter H_1 [ A, C ]$
\UI$f_1^\sharp (A, C) \,; f_1^\sharp (B, C) \fCenter H_1 [ A, C ]$
\UI$f_1^\sharp (A, C) \wedge f_1^\sharp (B, C) \fCenter H_1 [ A, C ]$

\UI$H [ f_1^\sharp (A, C) \wedge f_1^\sharp (B, C), C ] \fCenter A$

\AX$B \fCenter B$
\AX$C \fCenter C$
\BI$f_1^\sharp (B, C) \fCenter H_1 [ B, C ]$
\UI$f_1^\sharp (A, C) \,; f_1^\sharp (B, C) \fCenter H_1 [ B, C ]$
\UI$f_1^\sharp (A, C) \wedge f_1^\sharp (B, C) \fCenter H_1 [ B, C ]$
\UI$H [ f_1^\sharp (A, C) \wedge f_1^\sharp (B, C), C ] \fCenter B$

\BI$H [ f_1^\sharp (A, C) \wedge f_1^\sharp (B, C), C ] \,; H [ f_1^\sharp (A, C) \wedge f_1^\sharp (B, C), C ] \fCenter A \wedge B$
\UI$H [ f_1^\sharp (A, C) \wedge f_1^\sharp (B, C), C ] \fCenter A \wedge B$
\UI$f_1^\sharp (A, C) \wedge f_1^\sharp (B, C) \fCenter H_1 [ A \wedge B, C ]$
\UI$f_1^\sharp (A, C) \wedge f_1^\sharp (B, C) \fCenter f_1^\sharp (A \wedge B, C)$
\DisplayProof
\end{center}

		\paragraph{Conservativity.} Let $A\vdash B$ be a $\mathbf{DL}^*$-derivable sequent in the language of $\mathbf{DL}$ (i.e., no operational connective in
		$(\mathcal{F}^*\cup \mathcal{G}^*)\setminus (\mathcal{F}\cup \mathcal{G})$ occurs in the sequent). Hence, by the soundness of $\mathbf{DL}^*$ w.r.t.\
		perfect $\mathcal{L}_\mathrm{DLE}$-algebras, the inequality $A\leq B$ is valid on these algebras. By the completeness of $\mathbf{L}_\mathrm{DLE}$ w.r.t.\
		perfect $\mathcal{L}_\mathrm{DLE}$-algebras, the inequality $A\leq B$ is derivable in $\mathbf{L}_\mathrm{DLE}$, which implies, by the syntactic completeness of $\mathbf{DL}$ w.r.t.\ $\mathbf{L}_\mathrm{DLE}$, that $A\vdash B$ is $\mathbf{DL}$-derivable, as required.

\subsection{Cut elimination and subformula property}
		The calculi $\mathbf{DL}$ and $\mathbf{DL^*}$ are proper display calculi, and hence, by Theorem \ref{thm:meta}, they enjoy Belnap-style cut-elimination and subformula property. % by checking that conditions C$_1$--C$_8$, discussed in Section \ref{PS:ssec:DisplayLogic}.
		\begin{thm}
			The calculi $\mathbf{DL}$ and $\mathbf{DL}^*$ are proper display calculi. %hence they enjoy cut-elimination and the subformula property.
		\end{thm}
		\begin{proof}
			The conditions C$_1$--C$_7$ can be straightforwardly verified by inspection on the rules. As to C$_8$, cf.\ Fact \ref{app:c8} in the Appendix. %\marginnote{Maybe we give one case here as an example?}
		\end{proof}

\subsection{Properly displayable $\mathcal{L}_{\mathrm{DLE}}$-logics}
\begin{definition}
\label{def: properly displayable logic}
For any DLE-language $\mathcal{L}_{\mathrm{DLE}}$, an $\mathcal{L}_{\mathrm{DLE}}$-logic (cf.\ Definition \ref{def:DLE:logic:general}) is {\em properly displayable} (resp.\ {\em specially displayable}) if it is exactly captured by a display calculus obtained by adding analytic rules (resp.\ special rules)---cf.\ Definition \ref{def:analytic} (resp.\ Definition \ref{def:special})---to the calculus $\mathbf{DL}$ for $\mathcal{L}_{\mathrm{DLE}}$.
\end{definition}
		
		\section{Primitive inequalities and special rules}
		\label{sec:primitive special main strategy}
		%\marginnote{mention that with this procedure we are going to obtain special rules in the strict hard-coded sense}
		In \cite[Theorem 16]{Kracht}, Kracht showed that primitive formulas %\footnote{Actually, Kracht's primitive formulas are defined also in terms of the left adjoint $\Diamondblack$ of $\Box$.}
of basic normal/tense modal logic on a classical propositional base can be equivalently transformed into (a set of) special structural rules satisfying the defining conditions of proper display calculi (cf.\ Subsection \ref{PS:para:CanonicalCutElimination}). In the present section, we extend this result to any language $\mathcal{L}_\mathrm{DLE}$. We base this extension on the notion of primitive inequalities. Namely, %Moreover, we set the stage for extending the class of primitive inequalities to a significantly larger class.
in Subsection \ref{ssec:left right primitive}, we introduce the class of (left- and right-)primitive inequalities in any language $\mathcal{L}_\mathrm{DLE}$ (cf.\ Definition \ref{def:primitive}), and show (cf.\ Lemma \ref{lemma:pri}) that these inequalities can be equivalently (and effectively) transformed into special structural rules (cf.\ in the restricted sense of Definition \ref{def:special}). We also show that special structural rules can be equivalently (and effectively) transformed into primitive inequalities.
		In Subsection \ref{ssec:order theoreic prop of primitive}, we identify the crucial order-theoretic feature induced by the syntactic shape of definite primitive inequalities (cf.\ Lemma \ref{lemma:primitive scattered is operator}), on the basis of which a special ALBA-type reduction for definite primitive inequalities is given (cf.\ Proposition \ref{prop:type1}).
		In Subsection \ref{ssec:main strategy}, we take stock of the previous results and outline the way they will be further extended in Section \ref{sec:extended classes}.
		
		%The present section is aimed at identifying classes of $\mathcal{L}_\mathrm{DLE}$-inequalities which properly extend primitive inequalities and which can be equivalently (and effectively) transformed into `good' structural rules. Our general method makes crucial use of the fact that the languages of $\mathbf{DL}$ and $\mathbf{DL}^*$ are built using {\em the same} set of structural connectives: this method consists in equivalently transforming given $\mathcal{L}_\mathrm{DLE}$-inequalities into {\em primitive} $\mathcal{L}_\mathrm{DLE}^*$-inequalities. By the result in the previous section, each primitive $\mathcal{L}_\mathrm{DLE}^*$-inequality can be equivalently transformed into a set of `good' rules in the (structural) language of $\mathbf{DL}^*$, which, as observed above, coincides with the structural language of $\mathbf{DL}$. The equivalent transformation of input $\mathcal{L}_\mathrm{DLE}$-inequalities into primitive $\mathcal{L}_\mathrm{DLE}^*$-inequalities is effected using the algorithm ALBA described in Section \ref{}.

		\subsection{Left-primitive and right-primitive inequalities and special rules}
		\label{ssec:left right primitive}

		In what follows, for each connective $f\in \mathcal{F}$ and $g\in \mathcal{G}$, we will write $f(\vec p, \vec q)$ and $g(\vec p, \vec q)$, stipulating that
		$\varepsilon_f(p) = \varepsilon_g(p) = 1$ for all $p$ in $\vec p$, and $\varepsilon_f(q) = \varepsilon_g(q) = \partial$ for all $q$ in $\vec q$. Moreover, we write e.g.\ $f(\vec u/ \vec p, \vec v/\vec q)$ to indicate that the arrays $\vec u$ and $\vec p$ (resp.\ $\vec v$ and $\vec q$) have the same length $n$ (resp.\ $m$) and that, for each $1\leq i \leq n$ (resp.\ for each $1\leq j \leq m$), the formula $u_i$ (resp.\ $v_j$) has been uniformly substituted in $f$ for the variable $p_i$ (resp.\ $q_j$).

		\begin{dfn}[Primitive inequalities]
\label{def:primitive}
For any language $\mathcal{L}_\mathrm{DLE} = \mathcal{L}_\mathrm{DLE}(\mathcal{F}, \mathcal{G})$, the {\em left-primitive} $\mathcal{L}_\mathrm{DLE}$-formulas $\psi$ and {\em right-primitive} $\mathcal{L}_\mathrm{DLE}$-formulas $\phi$ are defined by simultaneous recursion as follows:
			$$\psi:=p \mid \top \mid \lor\mid \land \mid f(\vec \psi/\vec p, \vec \phi/\vec q),$$
			$$\phi:=p \mid \bot \mid \land \mid \lor \mid g(\vec \phi/\vec p, \vec \psi/\vec q).$$
			A left-primitive (resp.\ right-primitive) $\mathcal{L}_\mathrm{DLE}$-formula is {\em definite} if there are no occurrences of $+\vee$ or $-\wedge$ (resp.\ $+\wedge$ or $-\vee$) in its positive generation tree.
			An $\mathcal{L}_\mathrm{DLE}$-inequality $s_1\leq s_2$ is {\em left-primitive} (resp.\ {\em right-primitive}) if both $s_1$ and $s_2$ are left-primitive (resp.\ right-primitive) formulas and moreover:
			\begin{enumerate}
				\item each proposition variable in $s_1$ (resp.\ $s_2$) occurs at most once, in which case we say that $s_1$ (resp.\ $s_2$) is {\em scattered}.
				%\item $s_2$ (resp.\ $s_1$) is definite.
				\item $s_1$ and $s_2$ have the same order-type relative to the variables they have in common.
				\item $s_2$ (resp.\ $s_1$) is $\epsilon$-uniform w.r.t.\ some order-type $\epsilon$ on its occurring variables.
			\end{enumerate}
			Sometimes, the scattered side of a primitive inequality will be referred to as its {\em head} and the other one as its \emph{tail}.
		\end{dfn}
	It immediately follows from the axiomatization of the basic logic $\mathbf{L}_\mathrm{DLE}$ that left-primitive (resp.\ right-primitive) $\mathcal{L}_\mathrm{DLE}$-formulas can be equivalently written in disjunction (resp.\ conjunction) normal form of definite left-primitive (resp.\ right-primitive) formulas.
		The condition that the head of a primitive inequality is scattered implies that the head is $\epsilon$-uniform for the order type $\epsilon$ of its occurring variables in e.g.\ its positive generation tree. Notice that the definition above does not exclude the possibility that some variables which do not occur in the head of a primitive inequality might occur in its tail. However, item 3 of the definition above requires the tail to be uniform in these variables. This observation will be helpful later on in the treatment of these cases (cf.\ proof of Lemma \ref{lemma:pri}).  %
%This implies in particular that the inequality is uniform in these variables, and hence, as mentioned in footnote \ref{footnote:uniformterms}, any such primitive inequality $s\leq t$ can be transformed into some inequality $s'\leq t'$ in which these variables do not occur.
%If some variables occur in $s$ which do not occur in $t$, then item 3 of Definition \ref{def:primitive} guarantees that $s$, and hence the whole inequality, is uniform in these variables. Hence,  as discussed in Footnote \ref{footnote:uniformterms},  $s\leq t$ can be transformed into some inequality $s'\leq t$ in which each positive  (resp.\ negative) occurrence of these variables has been suitably replaced by  $\top$ (resp.\ $\bot$). The assumption that each term $s_i$ is definite right-primitive implies that each term in which the substitution has been effected is equivalent to $\top$, and hence can be removed from the conjunction normal form.	
%Hence, from now on we will assume w.l.o.g.\ that any variable occurring in the tail of any primitive inequality occurs also in its head. %Hence, item 2 of the definition above implies that also the non-head term of a primitive inequality is monotone
	
\begin{remark}\label{rem:primitive-skeleton-PIA}
The notion of primitive terms provides the first and most basic connection of unified correspondence theory to the characterization problem of the properly displayable DLE-logics (cf.\ Definition \ref{def: properly displayable logic}). Indeed, it can be easily verified by direct inspection that left-primitive terms are both positive Skeleton-terms and negative PIA-terms (cf.\ discussion after Definition \ref{Inducive:Ineq:Def}), and right-primitive terms are both positive PIA-terms and negative Skeleton-terms. In principle, not all positive PIA-terms (or negative Skeleton terms) are right-primitive, since $-\bot$ and $+\top$ are allowed to occur in their positive generation tree, while they are not allowed to occur in $+s$ for any right-primitive term $s$. Likewise, not all negative PIA-terms (or positive Skeleton-terms) are left-primitive, since $+\bot$ and $-\top$ are allowed to occur in their positive generation tree, while they are not allowed to occur in $+s$ for any left-primitive term $s$. %such that it does not contain $\top$ in positive position and $\bot$ in negative position. Henceforth we will assume that this is really the case since as we noted Remark \ref{rmk: management of top and bottom on the wrong side}, if such an occurrence takes place the whole term can be eliminated.
\end{remark}		
		\begin{example}
			\label{ex:FS1right primitive}
			Let $\mathcal{L}_\mathrm{DLE}(\mathcal{F}, \mathcal{G})$ be s.t.\ $\mathcal{F} = \{\Diamond\}$ and $\mathcal{G} = \{\rightarrow, \Box\}$. Of the following Fischer Servi inequalities (cf.\ \cite{Se77, Se81}),
			\[\Diamond(q\rightarrow p)\leq \Box q\rightarrow \Diamond p \quad \quad \Diamond q\rightarrow \Box p\leq \Box (q\rightarrow p),\]
			the second one is right-primitive, whereas the first one is neither right- nor left-primitive.
		\end{example}
		%In what follows, we adopt the convention that $f(\vec p, \vec q)$ and $ g(\vec p, \vec q)$ are s.t.\ $\varepsilon_f(p) = \varepsilon_g(p) = 1$ for every $p\in \vec p$ and $\varepsilon_f(q) = \varepsilon_g(q) = \partial$ for every $q\in \vec q$.
		Early on, in Definition \ref{def:rsls}, left-sided and right-sided structural terms were associated with formulas. In fact, it is not difficult to show, by induction on the shape of left-sided and right-sided structural terms, that the set of definite left-primitive (resp.\ right-primitive) formulas (cf.\ Definition \ref{def:primitive}) is exactly the image of the map $l$ (resp.\ $r$). The inverse maps of $l$ and $r$ are defined as follows:
		\begin{definition}[Structures associated with definite primitive formulas]
			\label{def:structure from primitive}
			Any definite left-primitive formula $s$ and any definite right-primitive formula $t$ is associated with structures $S=l^{-1}(s)$ and $T=r^{-1}(t)$ respectively, by the following simultaneous induction on $s$ and $t$.
			\begin{center}
				\begin{tabular}{l l}
					\begin{tabular}{l}
						if $s = p$ then $S: = \zeta(p)$ \\
						if $s = \top$ then $S: = \mathrm{I}$\\
						if $s = s_1\wedge s_2$ then $S = S_1\ ;\ S_2$\\
						if $s = f(\vec {s'}/\vec p, \vec {t'}/\vec q)$ then $S: = H(\vec {S'}, \vec {T'})$\\
					\end{tabular}
					&
					\begin{tabular}{l}
						if $t = p$ then $T: = \zeta(p)$\\
						if $t = \bot$ then $T: = \mathrm{I}$\\
						if $t = t_1\vee t_2$ then $T = T_1\ ;\ T_2$\\
						if $t = g(\vec {t'}/\vec p, \vec {s'}/\vec q)$ then $T: = K(\vec {T'}, \vec {S'})$\\
					\end{tabular}
				\end{tabular}
			\end{center}
			where $\zeta$ is an injective map from $\mathsf{AtProp}$ to the set of structural variables.
		\end{definition}
		
		\begin{lemma}\label{lemma:pri}
			Every left-primitive (resp.\ right-primitive) inequality $s\leq t$ is semantically equivalent to a set of special structural rules in the display calculus $\mathbf{DL}$.
		\end{lemma}
		\begin{proof}
			Assume that $s\leq t$ is right-primitive, and that both $s$ and $t$ are in conjunction normal form, that is, $s = \bigwedge_{i \leq n} s_i$ and
			$t = \bigwedge_{j \leq k} t_j$ where $s_i$ and $t_j$ are definite right-primitive formulas for any $i\leq n$ and any $j\leq k$. If some variables occur in $s$ which do not occur in $t$, then item 3 of Definition \ref{def:primitive} guarantees that $s$, and hence the whole inequality, is uniform in these variables. Hence,  as discussed in Footnote \ref{footnote:uniformterms}, the inequality $s\leq t$ can be transformed into some inequality $s'\leq t$ in which each positive  (resp.\ negative) occurrence of these variables has been suitably replaced by  $\top$ (resp.\ $\bot$). The assumption that each term $s_i$ is definite right-primitive implies that each term in which the substitution has been effected is equivalent to $\top$, and hence can be removed from the conjunction normal form. If the substitution has been effected on each $s_i$, then the inequality $s\leq t$ is equivalent to $\top\leq t$, which can be equivalently transformed into the $0$-ary rule $\mathrm{I}\vdash T$, where $T: = r^{-1}(t)$ as in Definition \ref{def:structure from primitive}, which is immediately verified to be analytic.
Assume now that all the variables which occur in $s$ occur as well in $t$.
 The following chain of equivalences is sound on any $\mathcal{L}_\mathrm{DLE}$-algebra $\bbA$:
			\begin{center}
				
				\begin{tabular}{c l l}
					& $\forall \vec p[s\leq t]$ \\
					iff & $\forall \vec p\forall p[p\leq s\Rightarrow p\leq t]$ & ($p$ fresh proposition variable)\\
					iff & $\forall \vec p\forall p[p\leq \bigwedge_{i \leq n} s_i\Rightarrow p\leq \bigwedge_{j \leq k} t_j]$\\
					iff & $\forall \vec p\forall p[\bigamp_{i\leq n} p\leq s_i\Rightarrow \bigamp_{j \leq k} p\leq t_j]$\\
					iff & $ \bigamp_{j \leq k}\Big (\forall \vec p\forall p[\bigamp_{i\leq n} p\leq s_i\Rightarrow p\leq t_j]\Big ).$ &\\
					
				\end{tabular}
			\end{center}
			Recalling  the definition of satisfaction of rules of $\mathbf{DL}$ on algebras (cf.\ Subsection \ref{ssec:soundness}), the chain of equivalences above proves that for every perfect $\mathcal{L}_\mathrm{DLE}$-algebra $\bba$, the validity of $s\leq t$ on $\bba$ is equivalent to the simultaneous validity on $\bba$ of the following rules:
			\[
			\Big ( \mbox{\AxiomC{$(X\vdash S_i\mid i\leq n)$}
				\UnaryInfC{$X\vdash T_j$}
				\DisplayProof}
			\mid j\leq k\Big)
			\]
			where for every $i\leq n$ and $j\leq k$, the structures $S_i$ and $T_j$ are the ones associated with $s_i$ and $t_j$ respectively, as indicated in Definition \ref{def:structure from primitive}.
			With a similar argument, it can be shown that if $s\leq t$ is left-primitive and both $s$ and $t$ are in disjunction normal form (that is, $s = \bigvee_{i \leq n} s_i$ and
			$t = \bigvee_{j \leq k} t_j$ where $s_i$ and $t_j$ are definite left-primitive formulas for any $i\leq n$ and any $j\leq k$), the validity of $s\leq t$ on $\bba$ is equivalent to the simultaneous validity on $\bba$ of the following rules:
			\[ \Big ( \mbox{\AxiomC{$(T_j\vdash Y\mid j\leq k)$}
				\UnaryInfC{$S_i\vdash Y$}
				\DisplayProof}
			\mid i\leq n\Big),\]
			where for every $i\leq n$ and $j\leq k$, the structures $S_i$ and $T_j$ are the ones associated with $s_i$ and $t_j$ respectively, as indicated in Definition \ref{def:structure from primitive}.
			It remains to be shown that these rules are analytic, i.e.\ that they satisfy conditions C$_1$-C$_7$. Condition C$_1$ follows from the assumption that all the variables which occur in the tail occur as well in the head. C$_5$ imposes restrictions on the introduction of formulas, and hence is vacuously true on structural rules. Conditions C$_2$, C$_6$, and C$_7$ are immediate. Condition C$_3$ follows from the requirement that every proposition variable occurs only once in the head of a primitive inequality. Finally, condition C$_4$ follows from the requirement that the formulas have the same order-type on the variables they have in common.
			\end{proof}
		
		Notice that the rules obtained from primitive inequalities in the way described above have the following special cases:
		\begin{itemize}
			\item if $s\leq t$ is a left-primitive (resp.\ right-primitive) inequality such that $t$ (resp.\ $s$) is definite, then the corresponding set of rules consists of {\em unary} rules;
			
			\item if $s\leq t$ is a left-primitive (resp.\ right-primitive) inequality $s\leq t$ such that $s$ (resp.\ $t$) is definite, then
			the corresponding set of rules consists of {\em one single} rule;
			
			\item if $s\leq t$ is a left-primitive (resp.\ right-primitive) inequality $s\leq t$ such that both $s$ and $t$ are definite, then the
			the corresponding set of rules consists of {\em one single} {\em unary} rule.
		\end{itemize}
		
		\begin{comment}
		It remains to be checked that, when added to the display calculi $\mathbf{DL}$ or $\mathbf{DL}^*$, the rules generated by the Lemma \ref{lemma:pri} preserve the satisfaction of conditions C$_1$-C$_8$ for the resulting calculus. Notice that item 1 of Definition \ref{def:primitive} guarantees that the rules satisfy condition C$_3$ (non-proliferation of parameters), and item 2 guarantees their enjoying condition C$_4$ (position-alikeness of congruent parameters). This yields the following result:
		\begin{prop}\label{prop:primitive properly displ}
		For any language $\mathcal{L}_\mathrm{DLE}$, every axiomatic extension of $\mathbf{L}_\mathrm{DLE}$ (resp.\ $\mathbf{L}_\mathrm{DLE}^*$) by primitive inequalities is properly displayable.
		\end{prop}
		
		\begin{proof}
		We need to show that the newly added rules do not falsify conditions C$_1$-C$_8$. Conditions C$_1$, C$_5$, and C$_8$ impose restrictions on the introduction of formulas, and hence are not affected by the new rules, which are structural. Conditions C$_2$, C$_6$, and C$_7$ are immediate. Condition C$_3$ follows from the requirement that every proposition variable occurs only once in the head of the primitive inequality . Finally, condition C$_4$ follows from the requirement that the formulas have the same order-type on the variables they have in common.
		\end{proof}
		\end{comment}
		
		The other direction is also true:
		
		\begin{lemma}\label{lemma:spetopri} %\marginnote{We need to have a notation for the translation presented in section \ref{ssec:soundness}. Here use the valuation to represent this translation.}
			%\marginnote{the lemma giving the direction from special rules to primitive inequalities should be added here Add footnote saying that because our environment takes inequalities as first-class citizens, the inverse direction always holds and we do not need to assume that the logic has modus ponens as is done in cite{CiRa14}}
			Every special structural rule in the language of $\mathbf{DL}$ is semantically equivalent to some left-primitive or right-primitive inequality.\footnote{Notice that translating rules as axioms of the original DLE-language instead of as inequalities (as done e.g.\ in \cite[Theorem 4.5]{CiRa14}) is possible only if the basic logic has an implication-type connective with modus ponens. In the present logical setting this is not possible in general.}
		\end{lemma}
		\begin{proof}
			Let us treat the case in which the special rule is of the form \[
			\mbox{\AxiomC{$(X\vdash S_i\mid i\leq n)$}\RightLabel{$\rho$}
				\UnaryInfC{$X\vdash T$}
				\DisplayProof},
			\]
			where $X$ does not occur in any $S_i$ nor in $T$. Let $l(X)=p$ and let $\vec{q}$ be the variables that appear in $r(S_i)$ and $r(T)$.
			As discussed in Section \ref{ssec:soundness}, the semantic validity of the rule above can be expressed as follows: $$\forall p\forall\vec{q}[\bigamp_{1\leq i\leq n}(p\leq r(S_i))\implies p\leq r(T)].$$ The fact that $X$ does not occur in any $S_i$ nor in $T$ implies that $p$ does not occur in $r(S_i)$ and $r(T)$. Then the above quasi-inequality can be equivalently rewritten as follows: $$\forall\vec{q}[\bigwedge_{i\leq n} r(S_i)\leq r(T)].$$ The inequality between brackets is right-primitive: indeed, similarly to what has been discussed above Definition \ref{def:structure from primitive} it is not difficult to show that $\bigwedge_{i\leq n} r(S_i)$ and $r(T)$ are right-primitive terms. Moreover, the assumption that $\rho$ is special implies that it is analytic, and hence $\rho$ satisfies conditions C$_1$-C$_7$. Condition C$_3$ guarantees that $r(T)$ is scattered and hence item 1.\ of Definition \ref{def:primitive} is satisfied. Condition C$_1$ guarantees Condition C$_4$ guarantees that $\bigwedge_{i\leq n} r(S_i)$ and $r(T)$ are uniform w.r.t.\ the same order-type and hence items 2.\ and 3.\ are satisfied.
		\end{proof}
		
		\begin{example}
			\label{FS2 from axiom to rule}
			Let $\mathcal{F} = \{\Diamond\}$ and $\mathcal{G} = \{\rightarrow, \Box\}$. The logical connectives of the display calculi $\mathbf{DL}$ and $\mathbf{DL}^*$ associated with the basic $\mathcal{L}_\mathrm{DLE}(\mathcal{F}, \mathcal{G})$-logic can be represented synoptically as follows:
			
			\begin{center}
				\begin{tabular}{|r|c|c|c|c|c|c|c|c|c|c|}
					\hline
					\scriptsize{Structural symbols} & \mc{2}{c|}{I} & \mc{2}{c|}{$;$} & \mc{2}{c|}{$>$} & \mc{2}{c|}{$\circ$} & \mc{2}{c|}{$\bullet$}\\
					\hline
					\scriptsize{Operational symbols} & $\top$ & $\bot$ & $\pand$ & $\por$ & $(\pdra)$ & \ $\pra$ \ & $\Diamond$ & $\Box$ & $(\Diamondblack)$ & $(\blacksquare)$\\
					\hline
				\end{tabular}
			\end{center}
			Below we illustrate schematically how to apply the procedure above to the Fischer Servi inequality
			$\Diamond q\rightarrow \Box p\leq \Box (q\rightarrow p)$, which is right-primitive (cf.\ Example \ref{ex:FS1right primitive}):
			
			\[\Diamond q\rightarrow \Box p\leq \Box (q\rightarrow p) \quad \rightsquigarrow \quad \frac{x\vdash \Diamond q\rightarrow \Box p}{x\vdash \Box (q\rightarrow p)} \quad \rightsquigarrow \quad \frac{X\vdash \circ Z > \circ Y}{X\vdash \circ (Z > Y)}.\]
		\end{example}

		\subsection{Order-theoretic properties of primitive inequalities}
		\label{ssec:order theoreic prop of primitive}

		The following lemma identifies the most important order-theoretic feature induced by the syntactic shape of primitive inequalities.
		Notice that, by definition, any scattered term $s$ is monotone, hence $s$ can be associated with an order-type on its variables, which is denoted $\varepsilon_s$. In these cases, we will sometimes write $s(\vec p, \vec q)$ with the convention that $\varepsilon_s(p) = 1$ for any $p$ in $\vec p$, and $\varepsilon_s(q) = \partial$ for any $q$ in $\vec q$. Also, in what follows we will find it convenient to represent an array $\vec{s}=(s_1,\ldots,s_n)$ as $(\overrightarrow{s_{-i}}, s_i)$ for $1\leq i\leq n$, where $\overrightarrow{s_{-i}}:=(s_1,\ldots, s_{i-1}, s_{i+1},\ldots,s_n)$. Finally, we write e.g.\ $s(\vec u/ \vec p)$ to indicate that the arrays $\vec u$ and $\vec p$ have the same length $n$ and that, for each $1\leq i \leq n$, the formula $u_i$ has been uniformly substituted in $s$ for the variable $p_i$.
		\begin{lemma}
			\label{lemma:primitive scattered is operator}
			For every language $\mathcal{L}_\mathrm{DLE}$, any definite and scattered left-primitive (resp.\ right-primitive) $\mathcal{L}_\mathrm{DLE}$-term $s$ and any $\mathcal{L}_\mathrm{DLE}$-algebra $\bba$, the term function $s^\bba: \bba^{\varepsilon_s}\to \bba$ is a (dual) operator, and if $\bba$ is perfect, then $s^\bba: \bba^{\varepsilon_s}\to \bba$ is a complete (dual) operator.\footnote{An operation on a lattice $\bba$ is an {\em operator} (resp.\ a {\em dual operator}) if it preserves finite joins (resp.\ meets) in each coordinate. Notice that this condition includes the preservation of the empty join $\bot$ (resp.\ the empty meet $\top$). An operation on a complete lattice is a {\em complete operator} (resp.\ a {\em complete dual operator}) if it preserves all joins (resp.\ meets) in each coordinate.}
		\end{lemma}
		\begin{proof}
			By induction on the complexity of $s$. If $s$ is a constant or a proposition variable, the verification of the statement is immediate. Let $s= f(\vec{u},\vec{v}) = f(\vec{u}/\vec {p},\vec{v}/\vec {q})$. The assumptions that $s$ is definite, left-primitive and scattered and those on the order-type of $s$ imply that each $u$ in $\vec u$ is definite, left-primitive, and scattered, and each $v$ in $\vec v$ is definite, right-primitive, and scattered. Hence, by induction hypothesis, the term function $u^\bba:\bba^{\varepsilon_s}\to\bba$ is an operator for each $u$, and $v^\bba:\bba^{\varepsilon^\partial_s}\to\bba$ is a dual operator for each $v$. Let $r$ be a variable occurring in $s$, and assume that the only occurrence of $r$ belongs to a subterm $u_i$. If $\varepsilon_s(r) = 1$, then $\varepsilon_{u_i}(r)=1$, and hence
			\begin{center}
				\begin{tabular}{r c l l}
					$f(\vec{u}_{-i},u_i[(\bigvee_{j\in I}\phi_j)/r],\vec{v})$ & $=$ & $f(\vec{u}_{-i}, (\bigvee_{j\in I}u_i[\phi_j/r]),\vec{v})$ & (induction hypothesis)\\
					& $=$ & $\bigvee_{j\in I}f(\vec{u}_{-i},u_i[\phi_j/r],\vec{v})$.\\
				\end{tabular}
			\end{center}
			
			If $\varepsilon_s(r) = \partial$, then $\varepsilon_{u_i}(r)=\partial$, hence
			\begin{center}
				\begin{tabular}{r c l l}
					$f(\vec{u}_{-i},u_i[(\bigwedge_{j\in I}\phi_j)/r],\vec{v})$ & $=$ & $f(\vec{u}_{-i}, (\bigvee_{j\in I}u_i[\phi_j/r]),\vec{v})$ & (induction hypothesis)\\
					& $=$ & $\bigvee_{j\in I}f(\vec{u}_{-i},u_i[\phi_j/r],\vec{v})$.\\
				\end{tabular}
			\end{center}
			The remaining cases can be proven with similar arguments.
			%
			% If $p$ occurs in $t_i$, we have $f(\vec{r},t_i[p_1\lor p_2/p],\vec{t}_{-i})=f(\vec{r},t_i[p_1/p]\land t_i[p_2/p],\vec{t}_{-i})=f(\vec{r},t_i[p_1/p],\vec{t}_{-i})\lor f(\vec{r},t_i[p_2/p],\vec{t}_{-i})$. If $\bba$ is perfect, then $s^\bba$ becomes complete.
		\end{proof}
		
		\begin{cor}\label{cor:rule:type:1}
			The following rules are sound and invertible in perfect DLEs, and derivable in ALBA for any definite scattered left-primitive term $s(\vec{p},\vec{q})$ and definite scattered right-primitive term $t(\vec{p},\vec{q})$:
			\begin{center}
				\begin{tabular}{ccc}
					\AxiomC{$\nomj\leq s(\vec{p},\vec{q})$}\LeftLabel{(Approx($s$))}
					\UnaryInfC{$\nomj\leq s(\vec{\nomi},\vec{\cnomm})\quad\vec{\nomi}\leq\vec{p}\quad\vec{q}\leq\vec{\cnomm}$}
					\DisplayProof
					&&
					\AxiomC{$t(\vec{p},\vec{q})\leq\cnomm$}\RightLabel{(Approx($t$))}
					\UnaryInfC{$t(\vec{\cnomn},\vec{\nomi})\leq\cnomm\quad\vec{p}\leq\vec{\cnomm}\quad \vec{\nomi}\leq\vec{q}$}
					\DisplayProof \\
				\end{tabular}
			\end{center}
		\end{cor}
		\begin{proof}
			The first part of the statement is an immediate consequence of Lemma \ref{lemma:primitive scattered is operator}. The second part can be straightforwardly shown by induction on $s$ and $t$. The details of the proof are omitted.
		\end{proof}
		\begin{prop}\label{prop:type1}
			For every language $\mathcal{L}_\mathrm{DLE}$ any left-primitive $\mathcal{L}_\mathrm{DLE}$-inequality $s(\vec p, \vec q)\leq s'(\vec p, \vec q)$ and any right-primitive $\mathcal{L}_\mathrm{DLE}$-inequality $t'(\vec p, \vec q)\leq t(\vec p, \vec q)$,
			
			\begin{enumerate}
				\item if $s(\vec p, \vec q)$ is definite, then the following are equivalent for every perfect DLE $\bba$:
				\begin{enumerate}
					\item $\bba\models s(\vec p, \vec q)\leq s'(\vec p, \vec q)$;
					\item $\bba\models s(\vec \nomi, \vec \cnomm)\leq s'(\vec \nomi, \vec \cnomm)$.
				\end{enumerate}
				\item If $t(\vec p, \vec q)$ is definite, then the following are equivalent for every perfect DLE $\bba$:
				\begin{enumerate}
					\item $\bba\models t'(\vec p, \vec q)\leq t(\vec p, \vec q)$;
					\item $\bba\models t'(\vec \cnomm, \vec \nomi)\leq t(\vec \cnomm, \vec \nomi)$.
				\end{enumerate}
			\end{enumerate}
		\end{prop}
		\begin{proof}
			We only prove 1, the proof of item 2 being order dual. By the assumptions and Corollary \ref{cor:rule:type:1}, the following chain of equivalences can be obtained via an ALBA reduction and hence is sound on perfect DLEs:
			\begin{center}
				\begin{tabular}{c l l}
					& $\forall \vec p\forall \vec q[s(\vec p, \vec q)\leq s'(\vec p, \vec q)]$ & \\
					iff & $\forall \vec p\forall \vec q\forall \nomj[\nomj\leq s(\vec p, \vec q)\Rightarrow \nomj\leq s'(\vec p, \vec q)]$ & \\
					iff & $\forall \vec p\forall \vec q\forall \nomj\forall\nomi\forall\cnomm[(\vec \nomi\leq \vec p \ \& \ \vec q\leq \vec \cnomm \ \&\ \nomj\leq s(\vec \nomi, \vec \cnomm))\Rightarrow \nomj\leq s'(\vec p, \vec q)]$ & (Approx($s$)) \\
					iff & $\forall \nomj\forall \vec \nomi\forall \vec \cnomm[\nomj\leq s(\vec \nomi, \vec \cnomm)\Rightarrow \nomj\leq s'(\vec \nomi, \vec \cnomm)]$ & (Ackermann, $s, s'$ same order type)\\
					iff & $\forall \vec \nomi\forall \vec \cnomm[s(\vec \nomi, \vec \cnomm)\leq s'(\vec \nomi, \vec \cnomm)]$ & \\
				\end{tabular}
			\end{center}
		\end{proof}
		
		\begin{remark}
			\label{rmk:prop type1 generalized to non definite}
			Proposition \ref{prop:type1} can be straightforwardly generalized to primitive inequalities the heads of which are not definite. For any such inequality, the preprocessing stage of ALBA %(omitting the application of the monotone/antitone variable-elimination rules, cf.\ Section \ref{})
produces a set of definite primitive inequalities with definite heads, to each of which Proposition \ref{prop:type1} can then be applied separately. Notice that the preprocessing does not affect the order-type of the occurring variables. Then, one can reverse the preprocessing steps and transform the set of pure definite primitive inequalities into a substitution instance of the input primitive inequality in which proposition variables have been suitably substituted for nominals and conominals.
		\end{remark}
		
		\begin{example}\label{ex: FS1 }
			Let us illustrate the reduction strategy of the proposition above by applying it to the right-primitive Fischer Servi inequality discussed in Examples \ref{ex:FS1right primitive} and \ref{FS2 from axiom to rule} (cf.\ \cite[Lemma 27]{MaPaSa14}).

			\begin{center}
				\begin{tabular}{r c l l}
					& & $\forall q\forall p[\Diamond q\rightarrow \Box p\leq \Box (q\rightarrow p) ]$ &\\
					& iff & $\forall q\forall p\forall \nomi\forall \cnomm[(\nomi\leq \Diamond q\rightarrow \Box p \ \& \ \Box (q\rightarrow p)\leq \cnomm)\Rightarrow \nomi\leq \cnomm] $ &\\
					& iff & $\forall q\forall p\forall \nomi\forall \cnomm\forall \cnomn[(\nomi\leq \Diamond q\rightarrow \Box p \ \& \ \Box (q\rightarrow \cnomn)\leq \cnomm \ \&\ p\leq \cnomn)\Rightarrow \nomi\leq \cnomm] $ &\\
					& iff & $\forall q\forall \nomi\forall \cnomm\forall \cnomn[(\nomi\leq \Diamond q\rightarrow \Box \cnomn \ \& \ \Box (q\rightarrow \cnomn)\leq \cnomm)\Rightarrow \nomi\leq \cnomm] $ &\\
					& iff & $\forall q\forall \nomi\forall \cnomm\forall \cnomn\forall \nomj[(\nomi\leq \Diamond q\rightarrow \Box \cnomn \ \& \ \Box (\nomj\rightarrow \cnomn)\leq \cnomm\ \&\ \nomj\leq q)\Rightarrow \nomi\leq \cnomm] $ &\\
					& iff & $\forall \nomi\forall \cnomm\forall \cnomn\forall \nomj[(\nomi\leq \Diamond \nomj\rightarrow \Box \cnomn \ \& \ \Box (\nomj\rightarrow \cnomn)\leq \cnomm)\Rightarrow \nomi\leq \cnomm] $ &\\
					& iff & $\forall \nomi\forall \cnomn\forall \nomj[\nomi\leq \Diamond \nomj\rightarrow \Box \cnomn \Rightarrow \forall \cnomm[ \Box (\nomj\rightarrow \cnomn)\leq \cnomm\Rightarrow \nomi\leq \cnomm]] $ &\\
					& iff & $\forall \nomi\forall \cnomn\forall \nomj[\nomi\leq \Diamond \nomj\rightarrow \Box \cnomn \Rightarrow \nomi\leq \Box (\nomj\rightarrow \cnomn)] $ &\\
					& iff & $\forall \cnomn\forall \nomj[\Diamond \nomj\rightarrow \Box \cnomn \leq \Box (\nomj\rightarrow \cnomn)]. $ &\\
				\end{tabular}
			\end{center}
		\end{example}

		\subsection{Special rules via ALBA: main strategy} %\marginnote{needs editing so as to take into account sections 7 and 8}
		\label{ssec:main strategy}
		Before moving on to the next section, in the present subsection we take stock of the facts we have collected so far, and spell out their role in the context of the method we will apply in the following section. %which will be developed in its most part in the following section,
		%The following section is
		%and which is
		%aimed at
		This method is to extend the class of primitive inequalities in any given language $\mathcal{L}_\mathrm{DLE}$ to classes of inequalities each element of which can be equivalently (and effectively) transformed into (a set of) special structural rules, hence giving rise to specially displayable DLE-logics (cf.\ Definition \ref{def: properly displayable logic}).
		This method is based on the simple but crucial observation that the languages of the display calculi $\mathbf{DL}$, $\mathbf{DL}^*$, and $\mathbf{DL}^{\underline{*}}$ (cf.\ Definition \ref{def: DL and DL star} and Footnote \ref{ftn: display calculus DL star minus}) are built using {\em the same} set of structural connectives. For each language $\mathcal{L}_\mathrm{DLE}$, we are going to identify classes of {\em non-primitive} $\mathcal{L}_\mathrm{DLE}$-inequalities which can be equivalently and effectively transformed into (conjunctions of) {\em primitive} inequalities in the expanded language $\mathcal{L}_\mathrm{DLE}^*$ (cf.\ Section \ref{ssec:expanded language}). By Lemma \ref{lemma:pri} %and Proposition \ref{prop:primitive properly displ}
 applied to $\mathcal{L}_\mathrm{DLE}^*$, each primitive $\mathcal{L}_\mathrm{DLE}^*$-inequality can then be equivalently transformed into a set of special structural rules in the language of $\mathbf{DL}^{\underline{*}}$, which, as observed above, coincides with the structural language of $\mathbf{DL}$. %The is effected using the algorithm ALBA, described in Subsection \ref{ssec: ALBA}.
		
		Proposition \ref{prop:type1} provides a key step in the procedure to equivalently transform input $\mathcal{L}_\mathrm{DLE}$-inequalities into primitive $\mathcal{L}_\mathrm{DLE}^*$-inequalities. Indeed, it guarantees that each definite primitive $\mathcal{L}_\mathrm{DLE}^*$-inequality %, which, by Proposition \ref{prop:type1} applied to $\mathcal{L}_\mathrm{DLE}^*$,
		is equivalent to a ``substitution instance of itself'' in which all the nominals and conominals have been uniformly substituted for proposition variables, as illustrated by the right-hand vertical equivalence in the diagram below:

		\begin{center}
			\begin{tabular}{l c l}\label{table:U:shape}
				$\mathbb{A}\models s(\vec p, \vec q)\leq s'(\vec p, \vec q)$ & &$\mathbb{A}\models \bigamp \Big \{ s_i^*(\vec p, \vec q)\leq {s'}_i^*(\vec p, \vec q)\mid i\in I\Big \}$\\
				&& \\
				$\ \ \ \ \Updownarrow \ \mathrm{Theorems \ref{albacorrect} and \ref{Thm:ALBA:Success:Inductive}}$ & & $\ \ \ \ \Updownarrow \ \mathrm{Proposition}$ \ref{prop:type1}\\
				&& \\
				$\mathbb{A}\models \bigamp \Big \{ s_i^*(\vec\nomi, \vec \cnomm)\leq {s'_i}^*(\vec\nomi, \vec \cnomm)\mid i\in I\Big \}$
				&\ \ \ $\Large{=}$ \ \ \ &$\mathbb{A}\models \bigamp \Big \{s_i^*(\vec\nomi, \vec \cnomm)\leq {s'_i}^*(\vec\nomi, \vec \cnomm)\mid i\in I\Big \}$\\
				
			\end{tabular}
		\end{center}
		Our task in the following section will be %motivates the strategy we are going to implement in the following section. Namely, we will
		to perform ALBA-reductions aimed at equivalently transforming $\mathcal{L}_\mathrm{DLE}$-inequalities into sets of definite {\em pure} primitive $\mathcal{L}_\mathrm{DLE}^*$-inequalities, so as to provide the left-hand side leg of the diagram above.
		
		%\begin{cor}
		%Every definite left-primitive or definite right-primitive inequality can be equivalently transformed into a primitive inequality.
		%\end{cor}

		%%%%%%%%%%%%%%%%%%%%%
		
		\section{Extending the class of primitive inequalities}\label{sec:extended classes}
%\marginnote{mention that this section provides a finer analysis which is aimed at obtaining directly special rules in the restricted sense; the following section which provides the most general case is independent from this section and hence this can be skipped and the reader can go directly to the next section if only interested in that. mention also that this paper is intended for two very different readerships, and hence this detour can be useful to see examples and get familiar with ALBA}
		
		%The present section is aimed at identifying classes of $\mathcal{L}_\mathrm{DLE}$-inequalities
		
In the present section, we introduce a hierarchy of classes of $\mathcal{L}_\mathrm{DLE}$-inequalities which properly extend primitive inequalities, and which can
be equivalently (and effectively) transformed into sets of special structural rules (cf.\ Definition \ref{def:special}), via %the uniform application of
progressively more complex ALBA-reduction strategies. %which can be equivalently transformed into primitive ones via the algorithm ALBA described in section 2.
The classes of inequalities treated in the present section are all proper subclasses of the class of analytic inductive inequalities (cf.\ Definition
\ref{def:type5}), which is the most general, and which, in Section \ref{sec:analytic}, will be also shown to capture analytic rules modulo equivalence. However,
the procedure described in Section \ref{sec:analytic} does not deliver special rules in the restricted sense of Definition \ref{def:special} in general, whereas
the finer analysis provided in the present section is guaranteed to yield special rules in this restricted sense (cf.\ Remark
\ref{rmk:two understandings of special}) in each instance in which it is applicable. Thus, unlike the general procedure, the procedure described in the present
section provides a direct and fully mechanized way\footnote{In Section \ref{sec:special rules as expressive as analytic}, we will show that in fact, all DLE-logics
axiomatized by analytic inductive inequalities are specially displayable. However, the general procedure, derived from the results in Sections \ref{sec:analytic} and \ref{sec:special rules as expressive as analytic}, to extract special rules from analytic inductive inequalities is indirect, as it consists of more than one back-and-forth toggle between inequalities and rules. }
%Most importantly, as discussed at the end of Section \ref{ssce:Krachtswhatever}, one equivalence step in this procedure is not `fully mechanized', in the sense
%that it is still an open question whether it can be derived using ALBA or $\mathbf{DL}$. Notice that we would face this open question also if we take Definition
%\ref{def:special} rather than the less restrictive alternative described in Remark \ref{rmk:two understandings of special} as the definition of special rule.
to obtain specially displayable DLE-logics (cf.\ Definition \ref{def: properly displayable logic}). Section \ref{sec:analytic} is independent from the present section, hence the reader is not constrained to read the present section before the next. Finally, the present paper is intended for two very different readerships; in this respect, the present section, which is the richest in examples of the whole paper, can be useful to the reader who wishes to become familiar with ALBA reductions.

Throughout the present section, we adopt the convention that $f(\vec p, \vec q)$ and $ g(\vec p, \vec q)$ are s.t.\ $\varepsilon_f(p) = \varepsilon_g(p) = 1$ for every $p\in \vec p$ and $\varepsilon_f(q) = \varepsilon_g(q) = \partial$ for every $q\in \vec q$. For any sequence of formulas $\vec{\psi}=(\psi_1,\ldots,\psi_n)$ and any $1\leq i\leq n$, we let $\overrightarrow{\psi_{-i}}:=(\psi_1,\ldots,\psi_{i-1},\psi_{i+1},\ldots,\psi_n)$.

		\subsection{Type 2: allowing multiple occurrences of critical variables}
		\label{ssec:type 2}
		By definition, each proposition letter in the head of a primitive inequality is required to occur at most once (that is, the head of primitive inequalities is required to be scattered). The present subsection is aimed at showing that this condition can be relaxed. % and that ine on the head of left-primitive and right-primitive inequalities the head of which is not scattered can be equivalently transformed into (sets of) primitive inequalities.

		\begin{definition}[Quasi-primitive inequalities]
			\label{def:quasi primitive}
			An inequality $s_1\leq s_2$ is {\em quasi left-primitive} (resp.\ {\em quasi right-primitive}) if both $s_1$ and $s_2$ are monotone (w.r.t.\ some order-type $\varepsilon_{s_i}$) and left-primitive (resp.\ right-primitive) formulas, and moreover $s_1$ and $s_2$ have the same order-type relative to the variables they have in common.
		\end{definition}
		
		The definition above differs from Definition \ref{def:primitive} in that the requirement that the head be scattered is dropped. %We find it useful to introduce the following notation: the symbol $s[\vec{p},\vec{q}]$ indicates that each variable $p$ or $q$ comes together with its multiplicity $m_p$ or $m_q$. % whereas, as usual, $s(\vec{p},\vec{q})$, %that is, the number of occurrences of each variable in $s$ %treated as different, hence $\vec{p}$ and $\vec{q}$ can contain the same variable multiple times.
		
		%This is in contrast to the symbol $s(\vec{p},\vec{q})$ where each variable is treated as one object and hence $\vec{p}$ and $\vec{q}$ contain each variable only once.
		\begin{remark}
			In what follows, we are going to provide an effective procedure to equivalently transform quasi-primitive inequalities into pure primitive inequalities. We will restrict our focus to quasi-primitive inequalities with {\em definite} head (cf.\ Proposition \ref{lem:type2}). Indeed, during the pre-processing stage of the execution of ALBA, each quasi-primitive inequality with non-definite head can be equivalently transformed into (the conjunction of) a set of quasi-primitive inequalities with definite head, on each of which the procedure described below can be effected in parallel. Thus, this restriction is without loss of generality. % discussed in the analogous case We can do this without loss of generality: In the case where the inequality is not definite, we simply apply the splitting rule when needed, to obtain definite inequalities.
		\end{remark}
		\begin{definition}\label{def:scattered:transform}
			For every left-primitive (resp.\ right-primitive) formula $s(\vec{p},\vec{q})$, a \emph{scattered transform} of $s$ is a scattered left-primitive (resp.\ right-primitive) term $s^\ast(\vec{p'},\vec{q'})$ for which there exists a substitution $\sigma: \mathsf{AtProp}(s^*)\to \mathsf{AtProp}(s)$ such that $s(\vec{p},\vec{q})=\sigma(s^\ast(\vec{p'},\vec{q'}))$.
		\end{definition}
		Clearly, we can always assume without loss of generality that $s(\vec{p},\vec{q})$ and $s^\ast(\vec{p'},\vec{q'})$ share no variables. In particular,
		in the following lemma, we will find it useful to consider scattered transforms which are pure, i.e.\ of the form $s^\ast(\vec{\nomi},\vec{\cnomm})$ or $s^\ast(\vec{\cnomm},\vec{\nomi})$, and such that their associated substitution $\sigma$ maps
		nominals and conominals to proposition variables in a suitable way according to their polarity. This can always be done without loss of generality.
		\begin{lemma}\label{cor:rule:type:2}
			The following rules are sound and invertible in perfect DLEs and are derivable in ALBA:
			\begin{enumerate}
				\item for any definite quasi left-primitive term $s(\vec{p},\vec{q})$,
				\begin{center}
					\begin{tabular}{ccc}
						\AxiomC{$\nomj\leq s(\vec{p},\vec{q})$}\RightLabel{(Approx$_{\sigma}$($s$))}
						\UnaryInfC{$\nomj\leq s^{\ast}(\overrightarrow{\nomi},\overrightarrow{\cnomm})\quad\overrightarrow{\bigvee \sigma^{-1}[p]}\leq \vec p\quad \vec q\leq\overrightarrow{\bigwedge \sigma^{-1}[q]}$}
						\DisplayProof
					\end{tabular}
				\end{center}
				
				where, for every $p$ in $\vec p$ and every $q$ in $\vec q$, every variable in $\sigma^{-1}[p]$ is a (fresh) nominal, and
				every variable in $\sigma^{-1}[q]$ is a (fresh) conominal, and $s^\ast$ is the scattered transform of $s$ induced by $\sigma$.
				
				% $\vec p$, : =\{\nomi'_{p'}\mid p'\in \sigma^{-1}(p)\}$ and $\mathbf{M}_q=\{\cnomm'_{q'}\mid q'\leq\cnomm'\mbox{ for }\sigma(q')=q\}$, and
				\item For any definite quasi right-primitive term $t(\vec{p},\vec{q})$:
				\begin{center}
					\begin{tabular}{ccc}
						\AxiomC{$t(\vec{p},\vec{q})\leq\cnomm$}\RightLabel{(Approx$_{\sigma}$($t$))}
						 \UnaryInfC{$t^{\ast}(\overrightarrow{\cnomn},\overrightarrow{\nomi})\leq\cnomm\quad\vec{p}\leq\overrightarrow{\bigwedge\sigma^{-1}[p]}\quad\overrightarrow{\bigvee\sigma^{-1}[q]}\leq\vec{q}$}
						\DisplayProof \\
					\end{tabular}
				\end{center}
				where, for every $p$ in $\vec p$ and every $q$ in $\vec q$, every variable in $\sigma^{-1}[p]$ is a (fresh) conominal, and
				every variable in $\sigma^{-1}[q]$ is a (fresh) nominal,  and $t^\ast$ is the scattered transform of $t$ induced by $\sigma$.
				%where $\bigvee\nomi'=\bigvee\{\nomi'_{p'} \mid p'\in \sigma^{-1}(p)\}$
				%
				% $\bigvee\nomi'=\bigvee\{\nomi'\mid\nomi'\leq q'\mbox{ for }\sigma(q')=q\}$ and $\bigwedge\cnomn'=\bigwedge\{\cnomn'\mid p'\leq\cnomn'\mbox{ for }\sigma(p')=p\}$.
			\end{enumerate}
		\end{lemma}
		
		\begin{proof}
			
			We only prove item 1, item 2 being order-dual.
			
			\begin{center}
				\begin{tabular}{ccc}
					\AxiomC{$\nomj\leq s(\vec{p},\vec{q})$}\RightLabel{(Definition \ref{def:scattered:transform})}
					\UnaryInfC{$\nomj\leq \sigma(s^{\ast}(\overrightarrow{\nomi},\overrightarrow{\cnomm}))$}\RightLabel{(definition of substitution)}
					\UnaryInfC{$\nomj\leq s^{\ast}(\overrightarrow{\sigma(\nomi)},\overrightarrow{\sigma(\cnomm)})$}\RightLabel{(Approx($s^\ast$))}
					\UnaryInfC{$\nomj\leq s^\ast(\overrightarrow{\nomi},\overrightarrow{\cnomm})\quad\quad\overrightarrow{\nomi}\leq\overrightarrow{\sigma(\nomi)}\quad\quad\overrightarrow{\sigma(\cnomm)}\leq\overrightarrow{\cnomm}$}\RightLabel{(reverse splitting rule)}
					\UnaryInfC{$\nomj\leq s^{\ast}(\overrightarrow{\nomi},\overrightarrow{\cnomm})\quad\quad\overrightarrow{\bigvee \sigma^{-1}[p]}\leq\vec{p}\quad\quad\vec{q}\leq\overrightarrow{\bigwedge\sigma^{-1}[q]}$}
					\DisplayProof
				\end{tabular}
			\end{center}
		\end{proof}
		The following proposition and its proof provide an effective procedure to equivalently transform quasi-primitive inequalities with definite head into pure primitive inequalities.
		\begin{prop}\label{lem:type2}
			
			For every quasi left-primitive inequality $s(\vec{p},\vec{q})\leq s'(\vec{p},\vec{q})$ such that $s$ is definite and every quasi right-primitive inequality $t'(\vec{p},\vec{q})\leq t(\vec{p},\vec{q})$ such that $t$ is definite,
			\begin{enumerate}
				\item the following are equivalent for every perfect $\mathcal{L}_\mathrm{DLE}$ algebra $\mathbb{A}$:
				\begin{enumerate}
					\item $\bba\models s(\vec p, \vec q)\leq s'(\vec p, \vec q)$;
					\item $\bba\models s^\ast(\vec \nomi, \vec \cnomm)\leq s'(\overrightarrow{\bigvee\sigma^{-1}[p]}, \overrightarrow{\bigwedge\sigma^{-1}[q]})$,
				\end{enumerate}
				
				where $s^\ast$ is a pure scattered transform of $s$ witnessed by a map $\sigma: \mathsf{Prop}(s^*)\to \mathsf{Prop}(s)$ such that, for every $p$ in $\vec p$ and every $q$ in $\vec q$,
				every variable in $\sigma^{-1}[p]$ is a nominal and every variable in $\sigma^{-1}[q]$ is a conominal.
				\item The following are equivalent for every perfect $\mathcal{L}_\mathrm{DLE}$ algebra $\mathbb{A}$:
				\begin{enumerate}
					\item $\bba\models t'(\vec p, \vec q)\leq t(\vec p, \vec q)$;
					\item $\bba\models t'(\overrightarrow{\bigwedge\sigma^{-1}[p]}, \overrightarrow{\bigvee\sigma^{-1}[q]})\leq t^\ast(\vec \cnomm, \vec \nomi)$,
				\end{enumerate}
				
				where $t^\ast$ is a pure scattered transform of $t$ witnessed by a map $\sigma: \mathsf{Prop}(t^*)\to \mathsf{Prop}(t)$ such that, for every $p$ in $\vec p$ and every $q$ in $\vec q$,
				every variable in $\sigma^{-1}[p]$ is a conominal and every variable in $\sigma^{-1}[q]$ is a nominal.\end{enumerate}
		\end{prop}
		
		\begin{proof}
			We only prove item 1, item 2 being order-dual. The assumptions and Lemma \ref{cor:rule:type:2} guarantee that the following ALBA reduction is sound:
			\begin{center}
				\begin{tabular}{c l l}
					& $\forall \vec p\forall \vec q[s(\vec p, \vec q)\leq s'(\vec p, \vec q)]$ & \\
					iff & $\forall \vec p\forall \vec q\forall \nomj[\nomj\leq s(\vec p, \vec q)\Rightarrow \nomj\leq s'(\vec p, \vec q)]$ & \\
					iff & $\forall \vec p\forall \vec q\forall \nomj\forall\overrightarrow{\nomi}\forall\overrightarrow{\cnomm}[(\nomj\leq s^{\ast}(\overrightarrow{\nomi},\overrightarrow{\cnomm})\ \&\ \overrightarrow{\bigvee\sigma^{-1}[p]}\leq\vec{p}\ \&\ \vec{q}\leq\overrightarrow{\bigwedge\sigma^{-1}[q]})\Rightarrow \nomj\leq t(\vec p, \vec q)]$~(Approx$_\sigma$(s))\\
					iff & $\forall \nomj\forall \vec \nomi\forall \vec \cnomm[\nomj\leq s^{\ast}(\overrightarrow{\nomi},\overrightarrow{\cnomm})\Rightarrow \nomj\leq s'(\overrightarrow{\bigvee\sigma^{-1}[p]}, \overrightarrow{\bigwedge \sigma^{-1}[q]})]$~(Ackermann, $s, s'$ same order type)\\
					iff & $\forall \vec \nomi\forall \vec \cnomm[s^{\ast}(\overrightarrow{\nomi},\overrightarrow{\cnomm})\leq s'(\overrightarrow{\bigvee\sigma^{-1}[p]}, \overrightarrow{\bigwedge \sigma^{-1}[q]})].$ & \\
				\end{tabular}
			\end{center}
		\end{proof}

		\paragraph{A concrete instantiation of the method.}
		
		Let $\mathcal{F} = \{\cdot, \Diamond\}$ and $\mathcal{G} = \varnothing$, where $\cdot$ is binary and of order type $(1, 1)$. The inequality $\Diamond\Diamond p\cdot\Diamond p\leq \Diamond p$ is quasi left-primitive and definite, and fails to be left-primitive because its head (the term on the left-hand side) is not scattered. Firstly, we run ALBA on this inequality, so as to equivalently transform it into a {\em pure} non-definite left-primitive inequality as follows:
		
		\begin{center}
			\begin{tabular}{c l l l}
				& $\forall p[\Diamond\Diamond p\cdot\Diamond p\leq \Diamond p]$ &\\
				
				iff & $\forall p\forall\nomj\forall \cnomm[(\nomj\leq \Diamond\Diamond p\cdot\Diamond p \ \&\ \Diamond p\leq \cnomm)\Rightarrow\nomj\leq \cnomm]$ &\\
				
				%iff & $\forall p\forall\nomj[(\nomj\leq\Diamond\Diamond p\ \&\ \nomj\leq\Diamond p)\Rightarrow\nomj\leq p]$\\
				
				iff & $\forall p\forall\nomj\forall\cnomm\forall\nomi[(\nomj\leq \Diamond\Diamond \nomi\cdot \Diamond p \ \&\ \nomi\leq p \ \&\ \Diamond p\leq \cnomm)\Rightarrow\nomj\leq \cnomm]$ &\\
				
				iff & $\forall p\forall\nomj\forall\cnomm\forall\nomi\forall\nomh [(\nomj\leq \Diamond\Diamond \nomi\cdot \Diamond \nomh \ \&\ \nomi\leq p\ \&\ \nomh\leq p \ \&\ \Diamond p\leq \cnomm)\Rightarrow\nomj\leq \cnomm]$\\
				
				iff & $\forall p\forall\nomj\forall \cnomm\forall\nomi \forall\nomh[(\nomj\leq \Diamond\Diamond \nomi\cdot \Diamond \nomh
				\ \&\ \nomi\lor \nomh\leq p \ \&\ \Diamond p\leq \cnomm)\Rightarrow\nomj\leq \cnomm]$ & (reverse splitting rule)\\
				
				%iff & $\forall p\forall\nomj\forall\nomi\forall\nomi'[(\nomj\leq \Diamond \nomi'\ \&\ \nomi\lor\nomi'\leq p)\Rightarrow\nomj\leq \Diamond\Diamond p]$\\
				
				iff & $\forall\nomj\forall \cnomm\forall\nomi\forall\nomh[(\nomj\leq \Diamond\Diamond \nomi\cdot \Diamond \nomh \ \&\ \Diamond (\nomi\lor \nomh)\leq \cnomm)\Rightarrow\nomj\leq \cnomm]$\\
				
				iff & $\forall\nomj\forall\nomi\forall\nomh[\nomj\leq \Diamond\Diamond \nomi\cdot \Diamond \nomh \Rightarrow \forall \cnomm[ \Diamond (\nomi\lor \nomh)\leq \cnomm\Rightarrow\nomj\leq \cnomm]]$\\
				
				iff & $\forall\nomj\forall\nomi\forall\nomh[\nomj\leq \Diamond\Diamond \nomi\cdot \Diamond \nomh \Rightarrow \nomj\leq \Diamond (\nomi\lor \nomh)]$\\
				
				iff & $\forall\nomi\forall\nomh[\Diamond\Diamond \nomi\cdot \Diamond \nomh \leq \Diamond (\nomi\lor \nomh)]$\\

				%iff & $\forall\nomi\forall\nomi'[\nomi\land\Diamond \nomi'\leq \Diamond\Diamond (\nomi\lor\nomi')\lor\Diamond (\nomi\lor\nomi')\lor (\nomi\lor\nomi')]$.\\
				
				%iff & $\forall r\forall q[r\land\Diamond q\leq \Diamond\Diamond (r\lor q)\lor\Diamond (r\lor q)\lor (r\lor q)]$
				
			\end{tabular}
		\end{center}
		
		By Proposition \ref{prop:type1}, %and Remark \ref{rmk:prop type1 generalized to non definite},
		the pure left-primitive inequality $\Diamond\Diamond \nomi\cdot \Diamond \nomh \leq \Diamond (\nomi\lor \nomh)$ is equivalent on perfect $\mathcal{L}_\mathrm{DLE}(\mathcal{F}, \mathcal{G})$-algebras to the left-primitive inequality
		$\Diamond\Diamond p_1\cdot \Diamond p_2 \leq \Diamond (p_1\lor p_2)$, which, via ALBA-distribution rule, is equivalent to the following inequality in disjunction normal form: \[\Diamond\Diamond p_1\cdot \Diamond p_2 \leq \Diamond p_1\lor \Diamond p_2.\]
		
		If we specify the non-lattice fragment of the language of the associated calculus $\mathbf{DL}$ as follows:
		\begin{center}
			\begin{tabular}{|r|c|c|c|c|c|c|c|c|c|c|}
				\hline
				\scriptsize{Structural symbols} & \mc{2}{c|}{$\circ$} & \mc{2}{c|}{$\bullet$} & \mc{2}{c|}{$\odot$} & \mc{2}{c|}{$\backslash\!\backslash_{\odot}$} & \mc{2}{c|}{$/\!/_{\odot}$}\\
				\hline
				\scriptsize{Operational symbols} & $\Diamond$ & $\phantom{\Diamond}$ & $\phantom{(\blacksquare)}$ & $(\blacksquare)$ & $\cdot$ & $\phantom{\cdot}$ & $\phantom{(\backslash_{\odot})}$ & $(\backslash_{\odot})$ & $\phantom{(/_{\odot})}$ & $(/_{\odot})$\\
				\hline
			\end{tabular}
		\end{center}
		
		then, applying the procedure indicated in the proof of Lemma \ref{lemma:pri}, the inequality above can be transformed into a structural rule in the language above as follows:
		\[ \Diamond\Diamond p_1\cdot \Diamond p_2 \leq \Diamond p_1\vee \Diamond p_2\quad \rightsquigarrow\quad \frac{\Diamond p_1\vdash z\quad \Diamond p_2\vdash z}{\Diamond\Diamond p_1\cdot \Diamond p_2 \vdash z}\quad \rightsquigarrow\quad \frac{\circ X \vdash Z\quad \circ Y \vdash Z}{\circ\circ X\odot \circ Y \vdash Z}.\]
		
		%\[ \Diamond\Diamond p_1\cdot \Diamond p_2 \leq \Diamond p_2\quad \rightsquigarrow\quad \frac{\Diamond p_2\vdash z}{\Diamond\Diamond p_1\cdot \Diamond p_2 \vdash z}\quad \rightsquigarrow\quad \frac{\circ Y \vdash Z}{\circ\circ X\odot \circ Y \vdash Z}.\]
		
		\paragraph{Monotone terms in quasi-primitive inequalities.}
		The head of primitive inequalities is scattered, hence monotone (w.r.t.\ some order-type). In defining quasi-primitive inequalities, we have dropped the former requirement but kept the latter. Before moving on, let us illustrate why by means of an example.
		Let $\mathcal{F} = \{\cdot, \Diamond, {\lhd}\}$ and $\mathcal{G} = \varnothing$, where $\cdot$ is binary and of order type $(1, 1)$, and $\lhd$ is unary and of order-type $(\partial)$. The inequality $p\cdot{\lhd} p\leq \Diamond p$ is not quasi-primitive, since its head $p\cdot{\lhd} p$ is not monotone. Actually, this inequality behaves like a primitive inequality, in that Proposition \ref{prop:type1} can be generalized to cover such an inequality; indeed
		
		\begin{center}
			\begin{tabular}{c l}
				& $\forall p[p\cdot{\lhd} p\leq \Diamond p]$\\
				iff & $\forall p\forall \nomi\forall \cnomm[(\nomi\leq p\cdot{\lhd} p \ \&\ \Diamond p\leq \cnomm)\Rightarrow \nomi\leq \cnomm]$\\
				iff & $\forall p\forall \nomi\forall \cnomm\forall \nomj[(\nomi\leq \nomj\cdot{\lhd} p \ \&\ \nomj\leq p\ \&\ \Diamond p\leq \cnomm)\Rightarrow \nomi\leq \cnomm]$\\
				
				iff & $\forall \nomi\forall \cnomm\forall \nomj[(\nomi\leq \nomj\cdot{\lhd} \nomj \ \&\ \Diamond \nomj\leq \cnomm)\Rightarrow \nomi\leq \cnomm]$\\
				iff & $\forall \nomi\forall \nomj[\nomi\leq \nomj\cdot{\lhd} \nomj \Rightarrow \forall \cnomm[ \Diamond \nomj\leq \cnomm\Rightarrow \nomi\leq \cnomm]]$\\
				iff & $\forall \nomi\forall \nomj[\nomi\leq \nomj\cdot{\lhd} \nomj \Rightarrow \nomi\leq \Diamond \nomj]$\\
				iff & $\forall \nomj[\nomj\cdot{\lhd} \nomj \leq \Diamond \nomj]$.\\
			\end{tabular}
		\end{center}
		However, this is not good news. Indeed, this reduction does not help to solve the main problem of this inequality, namely the fact that if we apply the procedure described in the proof of Lemma \ref{lemma:pri} to this inequality, we obtain a rule which violates condition C$_4$ (position-alikeness of parameters). %We will come back to this example in Section \ref{ssec:type5}.

		\subsection{Type 3: allowing PIA-subterms}
		\label{ssec:type3}
		In Sections \ref{ssec:order theoreic prop of primitive} and \ref{ssec:type 2}, we have generalized Kracht's notion of primitive inequalities, first by making this notion apply uniformly to any $\mathcal{L}_\mathrm{DLE}$-signature, and then by dropping the requirement that the heads of inequalities be scattered. Moreover, we have identified the main order-theoretic features induced by the syntactic shape of definite scattered primitive formulas, and, thanks to this identification, we have started to see ALBA at work on primitive and quasi-primitive inequalities. However, so far we have not discussed why ALBA was guaranteed to succeed on any primitive or quasi-primitive inequality in the first place. More in general, we have not yet made use of the second tool of unified correspondence theory: the possibility of identifying Sahlqvist and inductive type of inequalities in any $\mathcal{L}_\mathrm{DLE}$-signature.
		
		So let us start the present subsection by analyzing (quasi-)primitive inequalities as inductive inequalities (cf.\ Definition \ref{Inducive:Ineq:Def}). Indeed, it can be easily verified by direct inspection that all primitive inequalities are a very special subclass of inductive $\mathcal{L}_\mathrm{DLE}$-inequalities. Specifically, as observed earlier (cf.\ Remark \ref{rem:primitive-skeleton-PIA}), all non-leaf nodes in the generation tree $+s$ (resp.\ $-s$) of a (quasi) left-primitive (resp.\ right-primitive) formula $s$ are Skeleton nodes. %\marginnote{Discuss also the nodes in the tail and the PIAs}
This guarantees that, if $+s$ is also monotone w.r.t.\ some order-type $\varepsilon_s$, then all the variables at the leaves of such a generation tree (which are $\varepsilon_s$-critical) can be solved for, and moreover (together with the condition on the order-type in Definition \ref{def:quasi primitive}), that an ALBA reduction on a (quasi-)primitive inequality is guaranteed to reach Ackermann shape using {\em only} approximation and splitting rules after the preprocessing stage.
		
		A natural question arising at this point is whether or not all inductive inequalities can be transformed via ALBA into (conjunctions of) pure primitive inequalities, as outlined in Subsection \ref{ssec:main strategy}. We can already answer this question in the negative, as the following example shows. Let $\mathcal{F} = \{\Diamond\}$ and $\mathcal{G} = \{\Box\}$, and consider the inequality $\Diamond p\leq \Diamond\Box p$, which is Sahlqvist for the order-type $(1)$ and `McKinsey-type' for the order-type $(\partial)$, and is neither left-primitive nor right-primitive. There is only one successful reduction strategy for ALBA, which consists in solving for the positive occurrence of $p$ as follows:
		
		\begin{center}
			\begin{tabular}{c l}
				& $\forall p[\Diamond p\leq \Diamond\Box p]$\\
				iff & $\forall p\forall \nomi\forall\cnomm[(\nomi\leq \Diamond p \ \&\ \Diamond\Box p\leq \cnomm)\Rightarrow \nomi\leq \cnomm]$\\
				iff & $\forall p\forall \nomi\forall\cnomm\forall \nomj[(\nomi\leq \Diamond \nomj \ \&\ \nomj\leq p\ \&\ \Diamond\Box p\leq \cnomm)\Rightarrow \nomi\leq \cnomm]$\\
				iff & $\forall \nomi\forall\cnomm\forall \nomj[(\nomi\leq \Diamond \nomj \ \&\ \Diamond\Box \nomj\leq \cnomm)\Rightarrow \nomi\leq \cnomm]$\\
				iff & $\forall \nomi\forall \nomj[\nomi\leq \Diamond \nomj \Rightarrow \forall\cnomm[ \Diamond\Box \nomj\leq \cnomm\Rightarrow \nomi\leq \cnomm]]$\\
				iff & $\forall \nomi\forall \nomj[\nomi\leq \Diamond \nomj \Rightarrow \nomi\leq \Diamond\Box \nomj]$\\
				iff & $\forall \nomj[\Diamond \nomj \leq \Diamond\Box \nomj].$\\
				
			\end{tabular}
		\end{center}
		Clearly, this reduction fails to improve the situation, since it leaves the troublemaking side $\Diamond\Box p$ untouched. In contrast to this example, consider
		the inequality $\Diamond\Box p\leq \Diamond p$, which is again neither left-primitive nor right-primitive, but is Sahlqvist for both order-types $(1)$ and $(\partial)$. Solving for the troublemaking side we obtain:
		\begin{center}
			\begin{tabular}{c l}
				& $\forall p[\Diamond\Box p\leq \Diamond p]$\\
				iff & $\forall p\forall \nomi\forall\cnomm[(\nomi\leq \Diamond\Box p \ \&\ \Diamond p\leq \cnomm)\Rightarrow \nomi\leq \cnomm]$\\
				iff & $\forall p\forall \nomi\forall\cnomm\forall \nomj[(\nomi\leq \Diamond \nomj \ \&\ \nomj\leq \Box p\ \&\ \Diamond p\leq \cnomm)\Rightarrow \nomi\leq \cnomm]$\\
				iff & $\forall p\forall \nomi\forall\cnomm\forall \nomj[(\nomi\leq \Diamond \nomj \ \&\ \Diamondblack\nomj\leq p\ \&\ \Diamond p\leq \cnomm)\Rightarrow \nomi\leq \cnomm]$\\
				iff & $\forall \nomi\forall\cnomm\forall \nomj[(\nomi\leq \Diamond \nomj \ \&\ \Diamond \Diamondblack\nomj\leq \cnomm)\Rightarrow \nomi\leq \cnomm]$\\
				iff & $\forall \nomi\forall \nomj[\nomi\leq \Diamond \nomj \Rightarrow \forall\cnomm[ \Diamond\Diamondblack \nomj\leq \cnomm\Rightarrow \nomi\leq \cnomm]]$\\
				iff & $\forall \nomi\forall \nomj[\nomi\leq \Diamond \nomj \Rightarrow \nomi\leq \Diamond\Diamondblack \nomj]$\\
				iff & $\forall \nomj[\Diamond \nomj \leq \Diamond\Diamondblack \nomj],$\\
			\end{tabular}
		\end{center}
		from which the usual steps (Proposition \ref{prop:type1} and Lemma \ref{lemma:pri}) yield the rule \[\frac{\circ\bullet X\vdash Y}{\circ X\vdash Y}\]
		% In the present section we are going to extend our fragment by allowing also the PIA nodes in the formula.
		
		These ideas motivate the following
		\begin{definition}[Very restricted analytic inductive inequalities]
			\label{def:type3}
			For any order type $\epsilon$ and any irreflexive and transitive relation $\Omega$ on the variables $p_1,\ldots p_n$, the signed generation tree $\ast s$ ($\ast\in \{+, -\}$) of a term $s(p_1,\ldots p_n)$ is \emph{restricted analytic $(\Omega, \epsilon)$-inductive} if
			
			\begin{enumerate}
				\item $\ast s$ is $(\Omega, \epsilon)$-inductive (cf.\ Definition \ref{Inducive:Ineq:Def});
				\item every branch of $\ast s$ is good (cf.\ Definition \ref{def:good:branch});
				
				\item every maximal $\epsilon^{\partial}$-uniform subtree of $\ast s$ %$+\phi$ or $-\psi$
				occurs as an immediate subtree of an SRR node of some $\epsilon$-critical branch of $\ast s$;
				%$+g(\vec\phi, \vec\psi, \theta)$ or $-f(\vec\psi, \vec\phi, \theta)$ with $n_f, n_g\geq 2$.
			\end{enumerate}
			
			\noindent An inequality $s \leq t$ is \emph{very restricted left-analytic $(\Omega, \epsilon)$-inductive} (resp.\ \emph{very restricted right-analytic $(\Omega, \epsilon)$-inductive}) if
			\begin{enumerate}
				\item $+s$ (resp.\ $-t$) (which we refer to as the head of the inequality) is restricted analytic $(\Omega, \epsilon)$-inductive;
				\item $-t$ (resp.\ $+s$) is $\epsilon^\partial$-uniform, and
				\item $t$ is left-primitive (resp.\ $s$ is right-primitive) (cf.\ Definition \ref{def:primitive}).
			\end{enumerate}
			An inequality $s \leq t$ is \emph{very restricted analytic inductive} if it is very restricted (right-analytic or left-analytic) $(\Omega, \epsilon)$-inductive for some $\Omega$ and $\epsilon$.
		\end{definition}
		
		\begin{remark}
			\label{remark:discussion type 3}
			The syntactic shape specified in the definition above can be intuitively understood with the help of the following picture, which illustrates the `left-analytic' case:
			
			
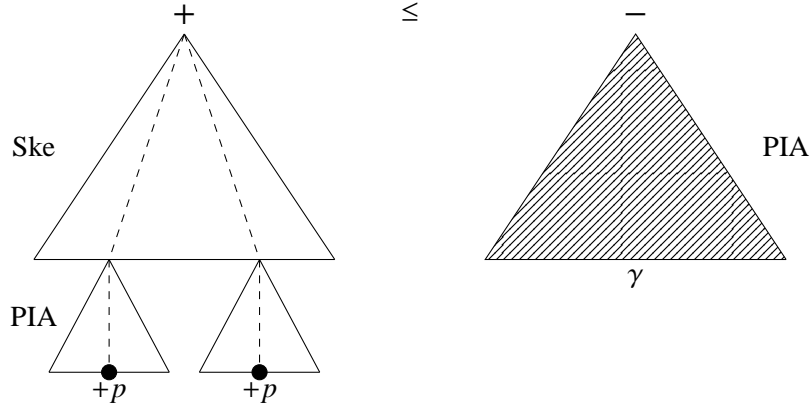
\begin{center}
				\begin{tikzpicture}
				\draw (-5,-1.5) -- (-3,1.5) node[above]{\Large$+$} ;
				\draw (-5,-1.5) -- (-1,-1.5) ;
				\draw (-3,1.5) -- (-1,-1.5);
				\draw (-5,0) node{Ske} ;
				\draw (5,0) node{PIA} ;
				\draw (3,-1.75) node{$\gamma$} ;
				\draw[dashed] (-3,1.5) -- (-4,-1.5);
				\draw[dashed] (-3,1.5) -- (-2,-1.5);
				\draw (-4,-1.5) --(-4.8,-3);
				\draw (-4.8,-3) --(-3.2,-3);
				\draw (-3.2,-3) --(-4,-1.5);
				\draw (-2,-1.5) --(-2.8,-3);
				\draw (-2.8,-3) --(-1.2,-3);
				\draw (-1.2,-3) --(-2,-1.5);
				\draw[dashed] (-4,-1.5) -- (-4,-3);
				\draw[dashed] (-2,-1.5) -- (-2,-3);
				\draw[fill] (-4,-3) circle[radius=.1] node[below]{$+p$};
				\draw[fill] (-2,-3) circle[radius=.1] node[below]{$+p$};
				\draw (-5,-2.25) node{PIA} ;
				\draw (0,1.8) node{$\leq$};
				\draw
				(5,-1.5) -- (3,1.5) node[above]{\Large$-$} -- (1,-1.5) -- (5,-1.5);
				\fill[pattern=north east lines]
				(5,-1.5) -- (3,1.5) -- (1,-1.5) ;
				\end{tikzpicture}
			\end{center}
As the picture shows, this syntactic shape requires that each $\varepsilon$-critical occurrence is a leaf of the head of the inequality. Moreover, the definition of restricted analytic $(\Omega, \epsilon)$-inductive signed generation tree implies that every maximal PIA-subtree contains at least one (but possibly more) $\varepsilon$-critical variable occurrence. Further, the requirement that every branch be good implies that every maximal $\varepsilon^\partial$-subtree $\gamma$ of every PIA-structure consists also exclusively of PIA-nodes. Moreover, the requirement that these subtrees be attached to their main PIA-subtree by means of an SRR-node lying on a critical branch guarantees that these subtrees will be incorporated in the minimal valuation subtree of the critical occurrence at the leaf of that critical branch.

Finally, exhaustively applying the distribution rules (a')-(c') described in Remark \ref{rmk: management of top and bottom on the wrong side} %we mentioned two types of ALBA-transformations which are not typically applied in ALBA reductions aimed at calculating first-order-correspondents of input inequalities. One transformation in particular concerns the exhaustive application of distribution rules in PIA-subterms. Applying this step
to any restricted analytic inductive term produces a restricted analytic inductive term, every maximal PIA-subterm of which is definite (cf.\ Footnote \ref{footnote: def definite PIA}) and contains {\em exactly one} $\varepsilon$-critical variable occurrence.
		\end{remark}
		
		\begin{example}\label{ex: frege as type 3}
			Let $\mathcal{F} = \varnothing$ and $\mathcal{G} = \{\rhu\}$, with $\rhu$ binary and of order-type $(\partial, 1)$.
			As observed in \cite{CoPa10}, the Frege inequality
			
			\[p\rhu(q\rhu r)\leq (p\rhu q)\rhu (p\rhu r)\] %\marginnote{Type 3 in the order-type run below, and Type 4 in the order-type 111.}
			
			\noindent is not Sahlqvist for any order type, but is $(\Omega, \epsilon)$-inductive, e.g.\ for $r <_\Omega p <_\Omega q$ and $\epsilon(p, q, r) =(1, 1, \partial)$, and is also
			%the order-type $\varepsilon_{p}=\varepsilon_{q}=\varepsilon_{r}=1$ and dependency order $p<_{\Omega}q<_{\Omega}r$) and .
			 very restricted right-analytic $(\Omega, \epsilon)$-inductive for the same $\Omega$ and $\epsilon$, as can be seen from the signed generation trees below: % $r <_\Omega p <_\Omega q$ and $\epsilon(p, q, r) =(1, 1, \partial)$

		\begin{center}
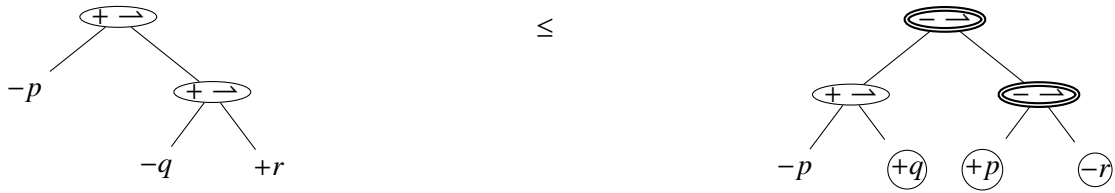
			\begin{multicols}{3}
				\begin{tikzpicture}
				\tikzstyle{level 1}=[level distance=1cm, sibling distance=2.5cm]
				\tikzstyle{level 2}=[level distance=1cm, sibling distance=1.5cm]
				\tikzstyle{level 3}=[level distance=1 cm, sibling distance=1.5cm]
				\node[PIA] {$+\rightharpoonup$}
				child{node{$-p$}}
				child{node[PIA]{$+\rightharpoonup$}
					child{node{$-q$}}
					child{node{$+r$}}}
				;
				
				\end{tikzpicture}
				
				\columnbreak
				$\leq$
				\columnbreak
				\begin{tikzpicture}
				\tikzstyle{level 1}=[level distance=1cm, sibling distance=2.5cm]
				\tikzstyle{level 2}=[level distance=1cm, sibling distance=1.5cm]
				\tikzstyle{level 3}=[level distance=1 cm, sibling distance=1.5cm]
				\node[Ske]{$-\rightharpoonup$}
				child{node[PIA]{$+\rightharpoonup$}
					child{node{$-p$}}
					child{node{\circled{$+q$}}}}
				child{node[Ske]{$-\rightharpoonup$}
					child{node{\circled{$+p$}}}
					child{node{\circled{$-r$}}}}
				;
				\end{tikzpicture}
			\end{multicols}
			%\caption{Signed generation tree for $\Box (p \rightarrow  q) \to \Box (\Box p\to\Box q)$}
			%\label{fig:fisher-servi}
		\end{center}
In the picture above, the circled variable occurrences are the $\epsilon$-critical ones, the doubly circled nodes are the Skeleton ones and the single-circle ones are PIA. 		
\end{example}
		%We say that a formula is $\epsilon$-primitive (inductive) if it is left- (resp.\ right-) primitive (inductive) and $\epsilon=1$ (resp.\ $\epsilon=\partial$).
		
		%\begin{remark}
		%Recall that in the definition of inductive terms (cf.\ Definition \ref{Inducive:Ineq:Def}), only critical branches are required to be good. While in the definition of primitive inductive inequalities, every branch (both critical and non-critical) is required to be good. This is due to the fact that after the substitution of the minimal valuation, the right hand side of the inequality should still be primitive. By item 2 and 3, $+\phi$s are $\epsilon^{\partial}$-right primitive and $+\psi$s are $\epsilon$-left primitive, which guarantees the requirement after the substitution.
		%\end{remark}

Below, we introduce an auxiliary definition which is a simplified version of \cite[Definition 5.1]{CFPS} and is aimed at effectively calculating the residuals of definite positive and negative PIA formulas (cf.\ discussion after Definition \ref{Inducive:Ineq:Def} and Footnote \ref{footnote: def definite PIA}) w.r.t.\ a given variable occurrence $x$. %of adjunctions is for computing the ``minimal/maximal valuation'':
		The intended meaning of notation such as $\phi(!x, \oz)$ is that the variable $x$ occurs {\em exactly once} in the formula $\phi$.
		\begin{definition} %\marginnote{Say somewhere here that we always preproccess so that meets and joins are moves as far as possible to the "top" Then we assume in most cases that w.l.o.g.\ everything is definite}
			%For $\ox$ of arity $n$, for each order-type $\varepsilon$ over $\ox$, and each $1 \leq i \leq n$, we define maps $\mathsf{LA}^{\varepsilon}_i$ and $\mathsf{RA}^{\varepsilon}_i$, sending $(\ox, \oz)$- and $(\ox, \oz)$-formulas into $(u, \oz)$- and $(u, \oz)$-formulas respectively, by the following simultaneous recursion:
			%
			For every definite positive PIA $\mathcal{L}_{\mathrm{DLE}}$-formula $\phi = \phi(!x, \oz)$, and any definite negative PIA $\mathcal{L}_{DLE}$-formula $\psi = \psi(!x, \oz)$ such that $x$ occurs in them exactly once, the $\mathcal{L}^+_\mathrm{DLE}$-formulas $\mathsf{LA}(\phi)(u, \oz)$ and $\mathsf{RA}(\psi)(u, \oz)$ (for $u \in Var - (x \cup \oz)$) are defined by simultaneous recursion as follows:
			
			\begin{center}
				\begin{tabular}{r c l}
					%$\mathsf{LA}^{\varepsilon}(\top)$ &= &$\bot$ \\
					$\mathsf{LA}(x)$ &= &$u$;\\
					%$\mathsf{LA}^{\varepsilon}(x_j)$ &= &$\bot^{\varepsilon_j}$ when $i \neq j$;\\
					$\mathsf{LA}(\Box \phi(x, \oz))$ &= &$\mathsf{LA}(\phi)(\Diamondblack u, \overline{z})$;\\
					%$\mathsf{LA}(\phi_1(x, \oz) \wedge \phi_2(\oz))$ &= &$\mathsf{LA}(\phi_1)(u, \overline{z})$;\\
					$\mathsf{LA}(\psi(\oz) \rightarrow \phi(x, \oz))$ &= &$\mathsf{LA}(\phi)(u \wedge \psi(\oz), \oz)$;\\
					$\mathsf{LA}(\phi_1(\oz) \vee \phi_2(x, \oz))$ &= &$\mathsf{LA}(\phi_2)(u - \phi_1(\oz), \oz)$;\\
					$\mathsf{LA}(\psi(x, \oz)\rightarrow\phi(\oz))$ &= &$\mathsf{RA}(\psi)(u \rightarrow \phi(\oz), \oz)$;\\
					$\mathsf{LA}(g(\overrightarrow{\phi_{-j}(\oz)},\phi_j(x,\oz), \overrightarrow{\psi(\oz)}))$ &= &$\mathsf{LA}(\phi_j)(g^{\flat}_{j}(\overrightarrow{\phi_{-j}(\oz)},u, \overrightarrow{\psi(\oz)} ), \oz)$;\\
					$\mathsf{LA}(g(\overrightarrow{\phi(\oz)}, \overrightarrow{\psi_{-j}(\oz)},\psi_j(x,\oz)))$ &= &$\mathsf{RA}(\psi_j)(g^{\flat}_{j}(\overrightarrow{\phi(\oz)}, \overrightarrow{\psi_{-j}(\oz)},u), \oz)$;\\
					&&\\
					%$\mathsf{RA}^{\varepsilon}(\bot)$ &= &$\top$ \\
					$\mathsf{RA}(x)$ &= &$u$;\\
					%$\mathsf{RA}^{\varepsilon}(x_j)$ &= &$\top^{\varepsilon_j}$ when $i \neq j$;\\
					$\mathsf{RA}(\Diamond \psi(x, \oz))$ &= &$\mathsf{RA}(\psi)(\blacksquare u, \overline{z})$;\\
					%$\mathsf{RA}(\psi_1(x, \oz) \vee \psi_2(\oz))$ &= &$\mathsf{RA}(\psi_1)(u, \overline{z})$;\\
					$\mathsf{RA}(\psi(x, \oz) - \phi(\oz))$ &= &$\mathsf{RA}(\psi)(\phi(\oz) \vee u, \oz)$;\\
					$\mathsf{RA}(\psi_1(\oz) \wedge \psi_2(x, \oz))$ &= &$\mathsf{RA}(\psi_2)(\psi_1(\oz) \rightarrow u, \oz)$;\\
					$\mathsf{RA}(\psi(\oz) - \phi(x, \oz))$ &= &$\mathsf{LA}(\phi)(\psi(\oz) - u, \oz)$;\\
					$\mathsf{RA}(f(\overrightarrow{\psi_{-j}(\oz)},\psi_j(x,\oz), \overrightarrow{\phi(\oz)}))$ &= &$\mathsf{RA}(\psi_j)(f^{\sharp}_{j}(\overrightarrow{\psi_{-j}(\oz)},u, \overrightarrow{\phi(\oz)} ), \oz)$;\\
					$\mathsf{RA}(f(\overrightarrow{\psi(\oz)}, \overrightarrow{\phi_{-j}(\oz)},\phi_j(x,\oz)))$ &= &$\mathsf{LA}(\phi_j)(f^{\sharp}_{j}(\overrightarrow{\psi(\oz)}, \overrightarrow{\phi_{-j}(\oz)},u), \oz)$.\\
				\end{tabular}
			\end{center}
		\end{definition}

		\begin{lem}\label{type3lem}
			For all definite positive PIA $\mathcal{L}_{\mathrm{DLE}}$-formulas $\phi_1(!x, \oz)$, $\phi_2(!x, \oz)$, and all definite negative PIA $\mathcal{L}_{\mathrm{DLE}}$-formulas $\psi_1(!x, \oz)$, $\psi_2(!x, \oz)$ such that the variable $x$ occurs in them exactly once,
			\begin{enumerate}
				\item if $+x\prec +\phi_1$, then the following rule is derivable in ALBA:
				\begin{center}
					\AxiomC{$\chi\leq \phi_1(x, \oz)$}\LeftLabel{(LA($\phi_1$))}
					\UnaryInfC{$\mathsf{LA}(\phi_1)(\chi/u, \oz)\leq x$}
					\DisplayProof
				\end{center}
				and moreover, $\mathsf{LA}(\phi_1)(u, \oz)$ is a definite negative PIA $\mathcal{L}_{\mathrm{DLE}}^\ast$-formula.
				\item if $-x\prec +\phi_2$, then the following rule is derivable in ALBA:
				\begin{center}
					\AxiomC{$\chi\leq \phi_2(x, \oz)$}\LeftLabel{(LA($\phi_2$))}
					\UnaryInfC{$x\leq \mathsf{LA}(\phi_2)(\chi/u, \oz)$}
					\DisplayProof
				\end{center}
				and moreover, $\mathsf{LA}(\phi_2)(u, \oz)$ is a definite positive PIA $\mathcal{L}_{\mathrm{DLE}}^\ast$-formula.
				\item if $+x\prec +\psi_1$, then the following rule is derivable in ALBA:
				\begin{center}
					\AxiomC{$\psi_1(x, \oz)\leq \chi$}\LeftLabel{(RA($\psi_1$))}
					\UnaryInfC{$x\leq \mathsf{RA}(\psi_1)(\chi/u, \oz)$}
					\DisplayProof
				\end{center}
				and moreover, $\mathsf{RA}(\psi_1)(u, \oz)$ is a definite positive PIA $\mathcal{L}_{\mathrm{DLE}}^\ast$-formula.
				\item if $-x\prec +\psi_2$, then the following rule is derivable in ALBA:
				\begin{center}
					\AxiomC{$\psi_2(x, \oz)\leq \chi$}\LeftLabel{(RA($\psi_2$))}
					\UnaryInfC{$\mathsf{RA}(\psi_2)(\chi/u, \oz)\leq x$}
					\DisplayProof
				\end{center}
				and moreover, $\mathsf{RA}(\psi_2)(u, \oz)$ is a definite negative PIA $\mathcal{L}_{\mathrm{DLE}}^\ast$-formula.
				
			\end{enumerate}
		\end{lem}
		\begin{proof}
			By simultaneous induction on the shapes of $\phi_1, \phi_2, \psi_1$ and $\psi_2$. The case in which they coincide with $x$ immediately follows from the definitions involved. As to the inductive step, we only illustrate the case in which $\phi_1$ is of the form $g(\overrightarrow{\phi(\oz)}, \overrightarrow{\psi_{-j}(\oz)},\psi_{2,j}(x,\oz))$ for some array $\overrightarrow{\phi(\oz)}$ of formulas which are positive PIA, some array $\overrightarrow{\psi_{-j}(\oz)}$ of formulas which are negative PIA, and some negative PIA formula $\psi_{2,j}(x,\oz)$ such that $-x\prec +\psi_{2,j}$. Then, by induction hypothesis, the following rule is derivable in ALBA for every formula $\chi'$:
			\begin{center}
				\AxiomC{$\psi_{2, j}(x, \oz)\leq \chi'$}\LeftLabel{(RA($\psi_{2, j}$))}
				\UnaryInfC{$\mathsf{RA}(\psi_{2,j})(\chi'/u', \oz)\leq x$}
				\DisplayProof
			\end{center}
			Moreover, by definition,
			\begin{equation}
				\label{eq:unraveling LA phi1}
				\mathsf{LA}(g(\overrightarrow{\phi(\oz)}, \overrightarrow{\psi_{-j}(\oz)},\psi_{2,j}(x,\oz)))= \mathsf{RA}(\psi_{2,j})(g^{\flat}_{j}(\overrightarrow{\phi(\oz)}, \overrightarrow{\psi_{-j}(\oz)},u)/u', \oz).\end{equation}
			Hence, we can show that RA($\phi_1$) is a derivable ALBA-rule as follows: for every formula $\chi$,
			\begin{center}
				\begin{tabular}{ccc}
					\AxiomC{$\chi\leq g(\overrightarrow{\phi(\oz)}, \overrightarrow{\psi_{-j}(\oz)},\psi_{2,j}(x,\oz))$}\RightLabel{(Residuation)}
					\UnaryInfC{$\psi_{2,j}(x,\oz)\leq g^{\flat}_{j}(\overrightarrow{\phi(\oz)}, \overrightarrow{\psi_{-j}(\oz)},\chi/u)$}\RightLabel{RA($\psi_{2, j}$)}
					\UnaryInfC{$\mathsf{RA}(\psi_{2,j})(g^{\flat}_{j}(\overrightarrow{\phi(\oz)}, \overrightarrow{\psi_{-j}(\oz)},\chi/u)/u', \oz)\leq x$}\RightLabel{(Identity \eqref{eq:unraveling LA phi1})}
					\UnaryInfC{$\mathsf{LA}(g(\overrightarrow{\phi(\oz)}, \overrightarrow{\psi_{-j}(\oz)},\psi_{2,j}(x,\oz)))(\chi/u, \oz)\leq x$}
					\DisplayProof \\
				\end{tabular}
			\end{center}
			To see that $\mathsf{LA}(\phi_1)(u, \oz)$ is a definite negative PIA-formula, one needs to show that every positive (resp.\ negative) node in $ + \mathsf{LA}(\phi_1)(u, \oz)$ is labeled with a connective from $\mathcal{F}^\ast$ (resp.\ $\mathcal{G}^\ast$). This follows from the identity \eqref{eq:unraveling LA phi1}, the second part of the induction hypothesis (stating that $\mathsf{RA}(\psi_{2,j})(u', \oz)$ is definite negative PIA), the fact that $\mathsf{RA}(\psi_{2,j})(u', \oz)$ is negative in $u'$, the fact that $g^{\flat}_{j}\in \mathcal{G}^\ast$ (and its corresponding node in $+ \mathsf{LA}(\phi_1)(u, \oz)$ is signed $-$, as we have just remarked), the fact that the order-type of $g^{\flat}_{j}$ is the same as the order-type of $g$, the fact that every formula in $\overrightarrow{\phi(\oz)}$ is positive PIA, and for each $\phi$ in $\overrightarrow{\phi(\oz)}$, \[ -\phi\prec - g^{\flat}_{j}(\overrightarrow{\phi(\oz)}, \overrightarrow{\psi_{-j}(\oz)},\chi/u)\prec +\mathsf{LA}(\phi_1)(u, \oz),\]
			and finally, the fact that every formula in $\overrightarrow{\psi_{-j}(\oz)}$ is negative PIA, and for each $\psi$ in $\overrightarrow{\psi_{-j}(\oz)}$, \[ +\psi\prec - g^{\flat}_{j}(\overrightarrow{\phi(\oz)}, \overrightarrow{\psi_{-j}(\oz)},\chi/u)\prec +\mathsf{LA}(\phi_1)(u, \oz).\]
		\end{proof}

		\begin{thm}\label{type3:theorem}
			Every very restricted left-analytic (resp.\ right-analytic) inductive $\mathcal{L}_{\mathrm{DLE}}$-inequality can be equivalently transformed, via an ALBA-reduction, into a set of pure left-primitive (resp.\ right-primitive) $\mathcal{L}_{\mathrm{DLE}}^\ast$-inequalities.
		\end{thm}
		\begin{proof}
			We only consider the case of the inequality $s\leq t$ being very restricted left-analytic $(\Omega,\epsilon)$-inductive, since the proof of the right-analytic case is dual. By assumption, $t$ is a negative PIA formula (cf.\ page \pageref{page: positive negative PIA}). Observe preliminarily that we can assume w.l.o.g.\ that there are no occurrences of $+\bot$ and $-\top$ in $+t$. Indeed, modulo exhaustive application of distribution rules, $t$ can be equivalently written as the disjunction of definite negative PIA terms $t_i$. If $+\bot$ or $-\top$ occurred in $+t_i$ for some $i$, the exhaustive application of the rules which identify each $+f'(\phi_1,\ldots,\bot^{\varepsilon_{f'}}(i),\ldots,\phi_{n_{f'}})$ with $+\bot$ for every $f'\in \mathcal{F}\cup\{\wedge\}$ and $-g'(\phi_1,\ldots,\top^{\varepsilon_{g'}}(i),\ldots,\phi_{n_{g'}})$ with $-\top$ for every $g'\in \mathcal{G}\cup\{\vee\}$ would identify $t_i$ with $\bot$. Hence the offending subterm can be removed from the disjunction. Hence (cf.\ Remark \ref{rem:primitive-skeleton-PIA}), we can assume w.l.o.g.\ that $t$ is left-primitive.

By assumption, $s: = \xi(\vec{\phi}/\vec{x},\vec{\psi}/\vec{y})$, where $\xi(\vec{!x}, \vec{!y})$ is a positive Skeleton-formula---cf.\ page \pageref{page: positive negative PIA}---which is scattered, monotone in $\vec{x}$ and antitone in $\vec{y}$. Moreover, the formulas in $\vec{\phi}$ are positive PIA, and the formulas in $\vec{\psi}$ are negative PIA. %hence right-primitive inductive, and they consist only of PIA nodes. Using Corollary \ref{cor:rule:type:3} we obtain:
			Modulo exhaustive application of distribution and splitting rules of the standard ALBA preprocessing,\footnote{The applications of splitting rules at this stage give rise to a set of inequalities, each of which can be treated separately. In the remainder of the proof, we focus on one of them.} we can assume w.l.o.g.\ that the scattered positive Skeleton formula $\xi$ is also definite. Modulo exhaustive application of the additional rules which identify $+f'(\phi_1,\ldots,\bot^{\varepsilon_{f'}}(i),\ldots,\phi_{n_{f'}})$ with $+\bot$ for every $f'\in \mathcal{F}\cup\{\wedge\}$ and $-g'(\phi_1,\ldots,\top^{\varepsilon_{g'}}(i),\ldots,\phi_{n_{g'}})$ with $-\top$ for every $g'\in \mathcal{G}\cup\{\vee\}$, which would reduce $s\leq t$ to a tautology, we can assume w.l.o.g.\ that there are no occurrences of $+\bot$ and $-\top$ in $+\xi$. Hence (cf.\ Remark \ref{rem:primitive-skeleton-PIA}) we can assume w.l.o.g.\ that $\xi$ is scattered, definite and left-primitive. Therefore, the derived rule Approx($\xi$) (cf.\ Corollary \ref{cor:rule:type:1}) is applicable, which justifies the last equivalence in the following chain:
			
			\begin{center}
				\begin{tabular}{c l l}
					& $\forall \vec{p}[\xi(\vec{\phi}/\vec{x},\vec{\psi}/\vec{y})\leq t(\vec{p})]$\\
					iff & $\forall \vec{p}\forall \nomj\forall \cnomn[(\nomj\leq \xi(\vec{\phi}/\vec{x},\vec{\psi}/\vec{y})\ \&\ t(\vec{p})\leq\cnomn)\Rightarrow \nomj\leq \cnomn]$\\
					iff & $\forall \vec{p}\forall \nomj\forall \cnomn\forall \vec{\nomi}\forall\vec{\cnomm}[(\nomj\leq \xi(\vec{\nomi}/\vec{x},\vec{\cnomm}/\vec{y})\ \&\ \vec{\nomi}\leq \vec{\phi}\ \&\ \vec{\psi}\leq \vec{\cnomm}\ \&\ t(\vec{p})\leq\cnomn)\Rightarrow \nomj\leq \cnomn]$. & (Approx($\xi$))\\
				\end{tabular}
			\end{center}
			By assumption, in each inequality $\vec{\nomi}\leq \vec{\phi}$ and $\vec{\psi}\leq \vec{\cnomm}$ there is at least one $\varepsilon$-critical variable occurrence. Modulo exhaustive application of distribution rules (a')-(c') of Remark \ref{rmk: management of top and bottom on the wrong side} and splitting rules, we can assume w.l.o.g.\ that each $\phi$ in $\vec{\phi}$ (resp.\ $\psi$ in $\vec{\psi}$) is a definite positive (resp.\ negative) PIA-formula, which has exactly one $\varepsilon$-critical variable occurrence. That is, if $\vec{p_1}$ and $\vec{p_2}$ respectively denote the subarrays of $\vec{p}$ such that $\varepsilon(p_1) = 1$ for each $p_1$ in $\vec{p_1}$ and $\varepsilon(p_2) = \partial$ for each $p_2$ in $\vec{p_2}$, then each $\phi$ in $\vec{\phi}$ is either of the form $\phi_1(p_1/!x, \vec{p'}/\oz)$ with $+x\prec +\phi_1$, or of the form $\phi_2(p_2/!x, \vec{p'}/\oz)$ with $-x\prec +\phi_2$. Similarly, each $\psi$ in $\vec{\psi}$ is either of the form $\psi_1(p_2/!x, \vec{p'}/\oz)$ with $+x\prec -\psi_1$, or of the form $\psi_2(p_1/!x, \vec{p'}/\oz)$ with $-x\prec -\psi_2$. Recall that each $\phi$ in $\vec{\phi}$ is definite positive PIA and each $\psi$ in $\vec{\psi}$ is definite negative PIA. Hence, we can assume w.l.o.g.\ that there are no occurrences of $-\bot$ and $+\top$ in $+\phi$. Indeed, otherwise, the exhaustive application of the additional rules which identify $-f'(\phi_1,\ldots,\bot^{\varepsilon_{f'}}(i),\ldots,\phi_{n_{f'}})$ with $-\bot$ for every $f'\in \mathcal{F}\cup\{\wedge\}$ and $+g'(\phi_1,\ldots,\top^{\varepsilon_{g'}}(i),\ldots,\phi_{n_{g'}})$ with $+\top$ for every $g'\in \mathcal{G}\cup\{\vee\}$, would reduce all offending inequalities to tautological inequalities of the form $\nomi\leq \top$ which can then be removed. Likewise, we can assume w.l.o.g.\ that there are no occurrences of $+\bot$ and $-\top$ in $+\psi$. This shows (cf.\ Remark \ref{rem:primitive-skeleton-PIA}) that we can assume w.l.o.g.\ that each $\phi$ in $\vec{\phi}$ (resp.\ $\psi$ in $\vec{\psi}$) is a right-primitive (resp.\ left-primitive) term. Moreover, by Lemma \ref{type3lem}, the suitable derived adjunction rule among LA($\phi_1$), LA($\phi_2$), RA($\psi_1$), RA($\psi_2$) is applicable to each formula, yielding:
			
			\begin{center}
				\begin{tabular}{c l l}
					& $\forall \vec{p_1}\forall \vec{p_2}\forall \nomj\forall \cnomn\forall \vec{\nomi}\forall\vec{\cnomm}[(\nomj\leq \xi(\vec{\nomi}/\vec{x},\vec{\cnomm}/\vec{y})\ \&\ \vec{\nomi}\leq \vec{\phi}\ \&\ \vec{\psi}\leq \vec{\cnomm}\ \&\ t(\vec{p_1}, \vec{p_2})\leq\cnomn)\Rightarrow \nomj\leq \cnomn]$ & \\
					iff & $\forall \vec{p_1}\forall \vec{p_2}\forall \nomj\forall \cnomn\forall \vec{\nomi}\forall\vec{\cnomm}[(\nomj\leq \xi(\vec{\nomi}/\vec{x},\vec{\cnomm}/\vec{y})\ \&\ \overrightarrow{\mathsf{LA}(\phi_1)(\nomi/u)}\leq \vec{p_1}\ \&\ \vec{p_2}\leq \overrightarrow{\mathsf{LA}(\phi_2)(\nomi/u)}\ \&\ $\\
					&$\overrightarrow{\mathsf{RA}(\psi_2)(\cnomm/u)}\leq \vec{p_1}\ \&\ \vec{p_2}\leq \overrightarrow{\mathsf{RA}(\psi_1)(\cnomm/u)}\ \&\ t(\vec{p_1}, \vec{p_2})\leq\cnomn)\Rightarrow \nomj\leq \cnomn]$. & \\ %(RA($\phi$) and LA($\psi$))\\
				\end{tabular}
			\end{center}
			The assumptions made above imply that $t(\vec{p_1}, \vec{p_2})$ is monotone in each variable in $\vec{p_1}$ and antitone in each variable $\vec{p_2}$. Hence, the quasi-inequality above is simultaneously in Ackermann shape w.r.t.\ all variables.\footnote{The formulas $\mathsf{LA}(\phi_1)(\nomi/u)$, $\mathsf{LA}(\phi_2)(\nomi/u)$, $\mathsf{RA}(\psi_1)(\cnomm/u)$, and $\mathsf{RA}(\psi_2)(\cnomm/u)$ do not need to be pure, and in general they are not. However, the assumptions and the general theory of ALBA guarantee that they are $\varepsilon^\partial$-uniform and free of the variable the `minimal valuation' of which they are part of. The reader is referred to \cite{CoPa10} for an expanded treatment of this point.} Applying the Ackermann rule repeatedly in the order indicated by $\Omega$ yields the following pure quasi-inequality:
			
			\begin{center}
				\begin{tabular}{c l l}
					& $\forall \nomj\forall \cnomn\forall \vec{\nomi}\forall\vec{\cnomm}[(\nomj\leq \xi(\vec{\nomi}/\vec{x},\vec{\cnomm}/\vec{y})\ \&\ t(\vec{P_1}/\vec{p_1}, \vec{P_2}/\vec{p_2})\leq\cnomn)\Rightarrow \nomj\leq \cnomn],$ \\
				\end{tabular}
			\end{center}
			where $P_1$ and $P_2$ denote the pure $\mathcal{L}_{\mathrm{DLE}}^\ast$-terms obtained by applying the Ackermann-substitution. For instance, for every $\Omega$-minimal $p_1$ in $\vec{p_1}$, \[P_1: = \bigvee_i \mathsf{LA}(\phi_1^{(i)})(\nomi/u)\vee \bigvee_j \mathsf{RA}(\psi_2^{(j)})(\cnomm/u),\]
			and for every $\Omega$-minimal $p_2$ in $\vec{p_2}$, \[P_2: = \bigwedge_i \mathsf{LA}(\phi_2^{(i)})(\nomi/u)\wedge \bigwedge_j \mathsf{RA}(\psi_1^{(j)})(\cnomm/u).\]
			In the clauses above, the indexes $i$ and $j$ count the number of critical occurrences of the given variable $p_1$ (resp.\ $p_2$) in PIA-subterms of type $\phi_1$ and $\psi_2$ (resp.\ $\phi_2$ and $\psi_1$). The pure quasi-inequality above can be equivalently transformed into one pure inequality as follows:
			\begin{center}
				\begin{tabular}{c l l}
					& $\forall \nomj\forall \cnomn\forall \vec{\nomi}\forall\vec{\cnomm}[(\nomj\leq \xi(\vec{\nomi}/\vec{x},\vec{\cnomm}/\vec{y})\ \&\ t(\vec{P_1}/\vec{p_1}, \vec{P_2}/\vec{p_2})\leq\cnomn)\Rightarrow \nomj\leq \cnomn]$ \\
					iff & $\forall \nomj\forall \vec{\nomi}\forall\vec{\cnomm}[(\nomj\leq \xi(\vec{\nomi}/\vec{x},\vec{\cnomm}/\vec{y})\Rightarrow \forall\cnomn[ t(\vec{P_1}/\vec{p_1}, \vec{P_2}/\vec{p_2})\leq\cnomn\Rightarrow \nomj\leq \cnomn]]$ \\
					iff & $\forall \nomj\forall \vec{\nomi}\forall\vec{\cnomm}[\nomj\leq \xi(\vec{\nomi}/\vec{x},\vec{\cnomm}/\vec{y})\Rightarrow \nomj\leq t(\vec{P_1}/\vec{p_1}, \vec{P_2}/\vec{p_2})]$ \\
					iff & $\forall \vec{\nomi}\forall\vec{\cnomm}[\xi(\vec{\nomi}/\vec{x},\vec{\cnomm}/\vec{y})\leq t(\vec{P_1}/\vec{p_1}, \vec{P_2}/\vec{p_2})].$ \\
				\end{tabular}
			\end{center}

			To finish the proof, it remains to be shown that the inequality in the last clause above is left-primitive. This is a rather simple proof by induction on the maximum length of chains in $\Omega$. The base case, when $\Omega$ is the discrete order (hence $P_1$ and $P_2$ are of the form displayed above), immediately follows from the polarity of $\xi$ and $t$ in $\vec{p_1}$ and $\vec{p_2}$, and by Lemma \ref{type3lem}. The inductive step is routine.
		\end{proof}
		\paragraph{The Frege axiom in a pre-Heyting algebra setting.}\label{par:Frege type 3} Let $\mathcal{F} = \varnothing$ and $\mathcal{G} = \{\rhu\}$, with $\rhu$ binary and of order-type $(\partial, 1)$. The logical connectives of the display calculi $\mathbf{DL}$ and $\mathbf{DL}^*$ arising from the basic $\mathcal{L}_\mathrm{DLE}(\mathcal{F}, \mathcal{G})$-logic can be represented synoptically as follows:
		
		\begin{center}
			\begin{tabular}{|r|c|c|c|c|c|c|c|c|c|c|}
				\hline
				\scriptsize{Structural symbols} & \mc{2}{c|}{I} & \mc{2}{c|}{$;$} & \mc{2}{c|}{$>$} & \mc{2}{c|}{$\succ$} & \mc{2}{c|}{\circled{$\bullet$}}\\
				\hline
				\scriptsize{Operational symbols} & $\top$ & $\bot$ & $\pand$ & $\por$ & $(\pdra)$ & \ $(\pra)$\  & $\phantom{\rhu}$ & $\rhu$ & $(\bullet)$ & $\phantom{\bullet}$\\
				\hline
			\end{tabular}
		\end{center}
		As mentioned in Example \ref{ex: frege as type 3},
		
		\[p\rhu(q\rhu r)\leq (p\rhu q)\rhu (p\rhu r)\]%\marginnote{Type 3 in the order-type run below, and Type 4 in the order-type 111.}

		%the order-type $\varepsilon_{p}=\varepsilon_{q}=\varepsilon_{r}=1$ and dependency order $p<_{\Omega}q<_{\Omega}r$) and .
		\noindent is strictly right-primitive $(\Omega, \epsilon)$-inductive for $r <_\Omega p <_\Omega q$ and $\epsilon(p, q, r) =(1, 1, \partial)$. Executing ALBA according to this choice of $\Omega$ and $\epsilon$, we obtain:
		
		%(which, as observed in cite{CoPa10}, is strictly right-primitive $(\Omega, \epsilon)$-inductive (e.g.\ for the order-type $\varepsilon_{p}=\varepsilon_{q}=\varepsilon_{r}=1$ and dependency order $p<_{\Omega}q<_{\Omega}r$) and is not Sahlqvist for any order type. In a Heyting algebra, the Frege axiom is a tautology. However, in weaker settings, such as the one of pre-Heyting algebras,
		
		%Clearly, the axiom above is also not primitive.\marginnote{Notice that the running below is not according to this order-type.}
		
		%For every perfect lattice $\bba$, let $\to$ be a binary operation on $\bba$ such that for every$u, v, w\in \bba$ $$u\bullet v\leq w \ \mbox{ iff } u\leq v\to w.$$ (here $p, q, r$ range on arbitrary elements of $\bba$,$m$ ranges in $\mty(\bba)$ and $i, j, h$ range in $\jty(\bba)$):
		\begin{center}
			\begin{tabular}{r c l l}
				& & $\forall p\forall q\forall r[p\rhu(q\rhu r)\leq (p\rhu q)\rhu (p\rhu r)]$ &\\
				%& iff & $\forall p \forall q\forall r[\bigvee \{j: j\leq p\rhu(q\rhu r)\}\leq \bigwedge \{(p\rhu q)\rhu (p\rhu m): r\leq m\}]$&\\
				& iff & $\forall p \forall q\forall r\forall \nomj \forall \cnomm[(\nomj\leq p\rhu(q\rhu r)\ \& \ (p\rhu q)\rhu (p\rhu r)\leq \cnomm)\Rightarrow \nomj\leq \cnomm]$&\\
				& iff & $\forall p \forall q\forall r\forall \nomj \forall \cnomm\forall \cnomn[(\nomj\leq p\rhu(q\rhu r) \ \&\ (p\rhu q)\rhu (p\rhu \cnomn)\leq \cnomm \ \&\ r\leq \cnomn)\Rightarrow \nomj\leq \cnomm]$&\\
				& iff & $\forall p \forall q\forall \nomj \forall \cnomm\forall \cnomn[(\nomj\leq p\rhu(q\rhu \cnomn) \ \&\ (p\rhu q)\rhu (p\rhu \cnomn)\leq \cnomm)\Rightarrow \nomj\leq \cnomm]$&\\
				& iff & $\forall p \forall q\forall \nomj \forall \cnomm\forall \cnomn\forall \nomi[(\nomj\leq p\rhu(q\rhu \cnomn) \ \&\ (p\rhu q)\rhu (\nomi\rhu \cnomn)\leq \cnomm \ \&\ \nomi\leq p)\Rightarrow \nomj\leq \cnomm]$&\\
				& iff & $ \forall q\forall \nomj \forall \cnomm\forall \cnomn\forall \nomi[(\nomj\leq \nomi\rhu(q\rhu \cnomn) \ \&\ (\nomi\rhu q)\rhu (\nomi\rhu \cnomn)\leq \cnomm)\Rightarrow \nomj\leq \cnomm]$&\\
				& iff & $ \forall q\forall \nomj \forall \cnomm\forall \cnomn\forall \nomi\forall \nomh[(\nomj\leq \nomi\rhu(q\rhu \cnomn) \ \&\ \nomh\rhu (\nomi\rhu \cnomn)\leq \cnomm \ \&\ \nomh\leq \nomi\rhu q)\Rightarrow \nomj\leq \cnomm]$&\\
				& iff & $ \forall q\forall \nomj \forall \cnomm\forall \cnomn\forall \nomi\forall \nomh[(\nomj\leq \nomi\rhu(q\rhu \cnomn) \ \&\ \nomh\rhu (\nomi\rhu \cnomn)\leq \cnomm \ \&\ \nomi \bullet \nomh\leq q)\Rightarrow \nomj\leq \cnomm]$&\\
				& iff & $\forall \nomj \forall \cnomm\forall \cnomn\forall \nomi\forall \nomh[(\nomj\leq \nomi\rhu((\nomi \bullet \nomh)\rhu \cnomn) \ \&\ \nomh\rhu (\nomi\rhu \cnomn)\leq \cnomm)\Rightarrow \nomj\leq \cnomm]$&\\
				& iff & $\forall \nomj \forall \cnomn\forall \nomi\forall \nomh[\nomj\leq \nomi\rhu((\nomi \bullet \nomh)\rhu \cnomn) \Rightarrow \forall \cnomm[\nomh\rhu (\nomi\rhu \cnomn)\leq \cnomm\Rightarrow \nomj\leq \cnomm]]$&\\
				& iff & $\forall \nomj \forall \cnomn\forall \nomi\forall \nomh[\nomj\leq \nomi\rhu((\nomi \bullet \nomh)\rhu \cnomn) \Rightarrow \nomj\leq\nomh\rhu (\nomi\rhu \cnomn)]$&\\
				& iff & $\forall \cnomn\forall \nomi\forall \nomh[\nomi\rhu((\nomi \bullet \nomh)\rhu \cnomn)\leq\nomh\rhu (\nomi\rhu \cnomn)],$&\\
			\end{tabular}
		\end{center}
		The last inequality above is a pure right-primitive $\mathcal{L}_{DLE}^*$-inequality, and by Proposition \ref{prop:type1} is equivalent on perfect DLE-algebras to
		\[p\rhu((q\bullet p)\rhu r)\leq p\rhu (q\rhu r).\]
		By applying the usual procedure, we obtain the following rule:
		\[p\rhu((q\bullet p)\rhu r)\leq p\rhu (q\rhu r)\quad \rightsquigarrow\quad \frac{x\vdash p\rhu((q\bullet p)\rhu r)}{x\vdash p\rhu (q\rhu r)}\quad \rightsquigarrow\quad \frac{X\vdash W\succ((Y\circled{$\bullet$} W)\succ Z)}{X\vdash W\succ (Y\succ Z)}\]

		\subsection{Type 4: allowing both sides of inequalities to be non-primitive}
		\label{ssec:type4}
		In all syntactic shapes of inequalities treated so far, the tail has been required to be primitive. This requirement is dropped in the syntactic shape treated in the present subsection. Let us start with a motivating example:
		
		\paragraph{The Church-Rosser inequality.}
		Let $\mathcal{F} = \{\Diamond\}$ and $\mathcal{G} = \{\Box\}$. The $\mathcal{L}_{\mathrm{DLE}}(\mathcal{F}, \mathcal{G})$-inequality $\Diamond\Box p\leq\Box\Diamond p$ is neither very restricted left-analytic inductive nor very restricted right-analytic inductive, given that neither side is primitive. However, the following ALBA reduction succeeds in transforming it into a pure left-primitive $\mathcal{L}_{\mathrm{DLE}}^\ast$-inequality: %\marginnote{this reduction needs to be edited}
		\begin{center}
			\begin{tabular}{c l l}
				
				& $\forall p[\Diamond\Box p\leq\Box\Diamond p]$\\
				
				iff & $\forall p[\Diamondblack\Diamond\Box p\leq\Diamond p]$& (Adjunction)\\
				
				iff & $\forall p\forall\nomi\forall\cnomm[\nomi\leq\Diamondblack\Diamond\Box p\ \&\ \Diamond p\leq\cnomm\Rightarrow\nomi\leq\cnomm]$&(First approximation)\\
				
				iff & $\forall p\forall\nomi\forall\nomj\forall\cnomm[\nomi\leq\Diamondblack\Diamond\nomj\ \&\ \nomj\leq\Box p\ \&\ \Diamond p\leq\cnomm\Rightarrow\nomi\leq\cnomm]$&(Approximation rules for $\Diamond$ and $\Diamondblack$)\\
				
				iff & $\forall p\forall\nomi\forall\nomj\forall\cnomm[\nomi\leq\Diamondblack\Diamond\nomj\ \&\ \Diamondblack\nomj\leq p\ \&\ \Diamond p\leq\cnomm\Rightarrow\nomi\leq\cnomm]$&(Adjunction)\\
				
				iff & $\forall\nomi\forall\nomj\forall\cnomm[\nomi\leq\Diamondblack\Diamond\nomj\ \&\ \Diamond \Diamondblack\nomj\leq\cnomm\Rightarrow\nomi\leq\cnomm]$&(Ackermann lemma)\\
				
				iff & $\forall\nomj[\Diamondblack\Diamond\nomj\leq\Diamond \Diamondblack\nomj].$\\
				
			\end{tabular}
		\end{center}
		Notice that this reduction departs in significant ways from the standard ALBA executions as described in Section \ref{ssec: ALBA}, in that we have applied an adjunction rule other than a splitting rule {\em before} the first approximation step, that is, as part of the preprocessing, and to a {\em Skeleton} node. This rule application is sound, but would be redundant if our goal was restricted to calculating first-order correspondents of input formulas. Notice that this rule application succeeded in transforming the input inequality into the inequality $\Diamondblack\Diamond\Box p\leq\Diamond p$, which is very restricted left-analytic inductive (cf.\ Definition \ref{def:type3}), and thus can be treated as indicated in the previous subsection. This example illustrates the ideas on which the treatment of the following class of inequalities is based:
		
		% So far we have , the we have only allowed for more intricate structure on one side of the inequality enforcing that the right-hand side (resp.\ left-hand side) of the inequality is left (resp.\ right) primitive. Before presenting the fully general case we present a form of inequality on which both sides are non-primitive, but which can be turned into a strictly primitive inductive inequality. This translation is not related to the ALBA algorithm.
		\begin{definition}[Restricted analytic inductive inequalities]
			\label{def:type4}
			For any order type $\epsilon$ and any irreflexive and transitive relation $\Omega$ on the variables $\vec{p}$,
			the signed generation tree $\ast s$ ($\ast\in \{+, -\}$) of a term $s(p_1,\ldots p_n)$ is \emph{analytic $(\Omega, \epsilon)$-inductive} if
			
			\begin{enumerate}%\marginnote{We still need to sort out the occurences of top and bottom}
				\item $\ast s$ is $(\Omega, \epsilon)$-inductive (cf.\ Definition \ref{Inducive:Ineq:Def});
				\item every branch of $\ast s$ is good (cf.\ Definition \ref{def:good:branch}).
			\end{enumerate}
			An inequality $s \leq t$ in $\vec{p}$ is \emph{restricted left-analytic} (resp.\ {\em right-analytic}) $(\Omega, \epsilon)$-{\em inductive} if
			\begin{enumerate}
				\item $+s$ (resp.\ $-t$) is restricted analytic $(\Omega, \epsilon)$-inductive (cf.\ Definition \ref{def:type3}) and $-t$ (resp.\ $+s$) is analytic $(\Omega, \epsilon)$-inductive;
				\item there exists exactly one $\epsilon^\partial$-uniform PIA subtree in $-t$ (resp.\ in $+s$) the root of which is attached to the Skeleton of $-t$ (resp.\ $+s$).
				%\item $t$ (resp.\ $s$) is of the form $\xi(\gamma/!x, \vec{\phi}/\oz)$ for some negative (resp.\ positive) Skeleton formula $\xi(!x, \oz)$, some formula $\phi$ which is negative (resp.\ positive) PIA if $-x\prec -\xi$ (i.e.\ if $+x\prec +\xi$) and positive (resp.\ negative) PIA if $+x\prec -\xi$ (i.e.\ if $-x\prec +\xi$), and some formulas $\vec{\phi}$ each of which is negative (resp.\ positive) PIA if $-z\prec -\xi$ (i.e.\ if $+z\prec +\xi$) and positive (resp.\ negative) PIA if $+z\prec -\xi$ (i.e.\ if $-z\prec +\xi$);
				%\item $\ast \gamma \prec -\xi$ is $\epsilon^\partial$-uniform;
				%\item $\ast \phi\prec -\xi$ is not $\epsilon^\partial$-uniform for every $\phi$ in $\vec{\phi}$.
				
			\end{enumerate}
			An inequality $s \leq t$ is \emph{restricted analytic inductive} if it is restricted (right-analytic or left-analytic) $(\Omega, \epsilon)$-inductive for some $\Omega$ and $\epsilon$.
		\end{definition}

		\begin{remark}%\marginnote{introduce and discuss the notation $\xi(\phi/!x, \vec{\phi}/\oz)$ withthe help of the picture}
			The syntactic shape specified in the definition above can be intuitively understood with the help of the following picture, which illustrates the `left-analytic' case:
			
			
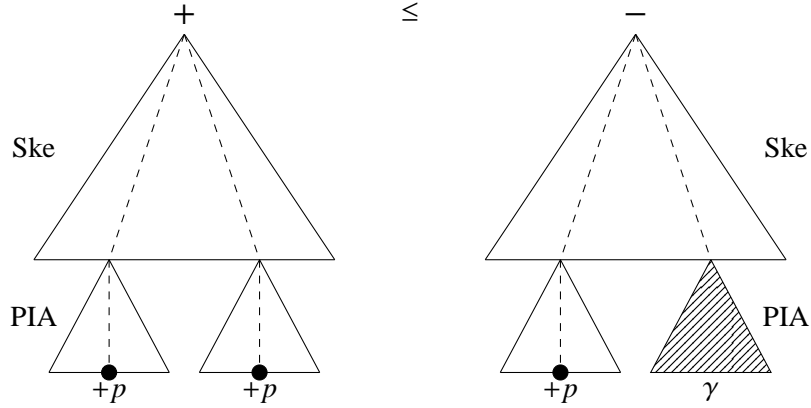
\begin{center}
				\begin{tikzpicture}
				\draw (-5,-1.5) -- (-3,1.5) node[above]{\Large$+$} ;
				\draw (-5,-1.5) -- (-1,-1.5) ;
				\draw (-3,1.5) -- (-1,-1.5);
				\draw (-5,0) node{Ske} ;
				\draw[dashed] (-3,1.5) -- (-4,-1.5);
				\draw[dashed] (-3,1.5) -- (-2,-1.5);
				\draw (-4,-1.5) --(-4.8,-3);
				\draw (-4.8,-3) --(-3.2,-3);
				\draw (-3.2,-3) --(-4,-1.5);
				\draw (-2,-1.5) --(-2.8,-3);
				\draw (-2.8,-3) --(-1.2,-3);
				\draw (-1.2,-3) --(-2,-1.5);
				\draw[dashed] (-4,-1.5) -- (-4,-3);
				\draw[dashed] (-2,-1.5) -- (-2,-3);
				\draw[fill] (-4,-3) circle[radius=.1] node[below]{$+p$};
				\draw[fill] (-2,-3) circle[radius=.1] node[below]{$+p$};
				\draw (-5,-2.25) node{PIA} ;
				\draw (0,1.8) node{$\leq$};
				\draw (5,-1.5) -- (3,1.5) node[above]{\Large$-$} ;
				\draw (5,-1.5) -- (1,-1.5) ;
				\draw (3,1.5) -- (1,-1.5);
				\draw (5,0) node{Ske} ;
				\draw[dashed] (3,1.5) -- (4,-1.5);
				\draw[dashed] (3,1.5) -- (2,-1.5);
				\draw (2,-1.5) --(2.8,-3);
				\draw (2.8,-3) --(1.2,-3);
				\draw (1.2,-3) --(2,-1.5);
				\draw[dashed] (2,-1.5) -- (2,-3);
				\draw[fill] (2,-3) circle[radius=.1] node[below]{$+p$};
				\draw
				(4,-1.5) -- (4.8,-3) -- (3.2,-3) -- (4, -1.5);
				\fill[pattern=north east lines]
				(4,-1.5) -- (4.8,-3) -- (3.2,-3) -- (4, -1.5);
				\draw (4,-3.25)node{$\gamma$};
				\draw (5,-2.25) node{PIA} ;
				\end{tikzpicture}
			\end{center}
			
			As the picture shows, similarly to the very restricted analytic inductive inequalities, this syntactic shape forbids the root of any $\epsilon^\partial$-uniform subtree to be attached directly to the skeleton of the head of the inequality. However, in contrast to the very restricted analytic inductive inequalities, critical branches can appear now in the tail of the inequality. Finally, there exists a unique $\epsilon^\partial$-uniform subtree whose root is attached to the skeleton of the tail of the inequality. In the lemma below, we will denote the tail of a restricted left-analytic (resp.\ right-analytic) inductive inequality by $\xi(\gamma/!x, \vec{\psi}/\oz)$, where $\xi(!x,\oz)$ is a negative (resp.\ positive) skeleton term, %hence $\xi(!x,\oz)$ is a right-primitive (resp.\ left-primitive),
 and $\gamma$ denotes the unique $\epsilon^\partial$-uniform PIA subtree attached to the skeleton, and for each $\psi\in\vec{\psi}$ with $\ast\psi\prec-\xi$ is a PIA subtree that contains a critical branch. %each $\varepsilon$-critical occurrence is a leaf of the head of the inequality. Moreover, the definition of restricted analytic $(\Omega, \epsilon)$-inductive signed generation tree prescribes that every PIA-subtree contains at least one $\varepsilon$-critical variable occurrence. Further, the requirement that every branch be good implies that every maximal $\varepsilon^\partial$-subtree $\gamma$ of every PIA-structure consists also exclusively of PIA-nodes. Finally, the requirement that these subtrees be attached to the main PIA-subtree by means of an SRR-node lying on a critical branch guarantees that these subtrees will be incorporated in the minimal valuation subtree of the critical occurrence at the leaf of that critical branch.
		\end{remark}
		\paragraph{BNF presentation of analytic $(\Omega,\epsilon)$-inductive terms.} In what follows, we adopt the following conventions: when writing e.g.\ $g(\vec x, \vec y, \vec z, \vec w)$, we understand that the arrays of variables are of different lengths, which can be possibly $0$, and moreover $g$ is monotone in $\vec x$ and $\vec z$ and is antitone in $\vec y$ and $\vec w$. Let us first introduce the BNF presentation of the $\epsilon^\partial$-uniform PIA terms $\gamma$ and $\chi$, which are substituted for positive and negative placeholder variables in the skeleton of analytic inductive terms respectively. This implies that we will only be interested in the signed generation trees $+\gamma$ and $-\chi$. Moreover, we will use the letter $p$ (or $\vec p$) to indicate those variables which are assigned to $1$ by $\epsilon$, and the letter $q$ (or $\vec q$) for those which are assigned to $\partial$.
		$$\gamma:= q \mid \bot \mid \top \mid\gamma\lor\gamma\mid \gamma\land\gamma \mid g(\vec \gamma/\vec x, \vec \chi/\vec y),$$
		$$\chi:= p \mid \top \mid\bot\mid \chi\land\chi \mid \chi\lor\chi \mid f(\vec \chi/\vec x, \vec \gamma/\vec y).$$
		
		Next, let us introduce the BNF presentation of the non $\epsilon^\partial$-uniform PIA terms $\phi$ and $\psi$, which are substituted for positive and negative placeholder variables in the skeleton of analytic inductive terms respectively. This implies that we will only be interested in the signed generation trees $+\phi$ and $-\psi$. Let $\mathsf{PosPIA}$ and $\mathsf{NegPIA}$ respectively denote the sets of the $\phi$- and $\psi$-terms. In addition, we will need---and define by simultaneous induction---the function $$CVar:\mathsf{PosPIA}\cup\mathsf{NegPIA}\to\wp(Var)$$ which maps each $\phi$ and $\psi$ to the set of variables of which there are critical occurrences in $\phi$ and $\psi$.
		
		$$\phi:= p \mid \top\land \phi\mid \bot\land \phi\mid \phi\land\phi\mid \gamma\vee \phi \mid g(\vec \gamma/\vec x, \vec \chi/\vec y,\phi/z)\mid g(\vec \gamma/\vec x, \vec \chi/\vec y,\psi/w)$$
		$$\psi:= q \mid \bot\lor\psi\mid \top\land \psi\mid \psi\lor\psi \mid \chi\wedge \psi \mid f(\vec \chi/\vec x, \vec \gamma/\vec y,\psi/z)\mid f(\vec \chi/\vec x, \vec \gamma/\vec y,\phi/w).$$
		
		In the two presentations above, the construction of the terms which have $g$ or $f$ as their main connectives is subject to the condition that all the variables in $CVar(\phi)$ (resp.\ $CVar(\psi)$)---where $\phi$ and $\psi$ denote the immediate subformulas as indicated above---are common upper bounds of the variables occurring in $\vec{\gamma}$ and $\vec{\chi}$ w.r.t.\ $\Omega$.
		
		\begin{center}
			\begin{tabular}{r c l c r c l}
				$CVar(p)$ & $=$ & $\{p\}$ & $\quad$ & $CVar(q)$ & $=$ & $\{q\}$\\
				$CVar(\phi_1\land\phi_2)$&$=$&$CVar(\phi_1)\cup CVar(\phi_2)$ & &	$CVar(\psi_1\lor\psi_2)$&$=$&$CVar(\psi_1)\cup CVar(\psi_2)$\\
				$CVar(g(\vec \gamma/\vec x, \vec \chi/\vec y,\phi/z))$&$=$&$CVar(\phi)$&& $CVar(f(\vec \chi/\vec x, \vec \gamma/\vec y,\psi/z))$&$=$&$CVar(\psi)$\\
				$CVar(g(\vec \gamma/\vec x, \vec \chi/\vec y,\psi/w))$&$=$&$CVar(\psi)$&& $CVar(f(\vec \chi/\vec x, \vec \gamma/\vec y,\phi/w))$&$=$&$CVar(\phi)$
			\end{tabular}	
		\end{center}
		
		Finally, let us introduce the BNF presentation of the analytic inductive terms $s$ and $t$, which are to occur on the left-hand side and right-hand side of inequalities respectively. This implies that we will only be interested in the sign generation trees $+s$ and $-t$.

		$$s:=\gamma \mid \phi \mid s\lor s\mid s\land s \mid f(\vec s/\vec x, \vec t/\vec y),$$
		$$t:=\chi \mid \psi \mid t\land t \mid t\lor t \mid g(\vec t/\vec x, \vec s/\vec y).$$

		\begin{lemma} %\marginnote{edit the lemma. we need to first reach the definite shape}
\label{type4:lemma}
For any $\mathcal{L}_{\mathrm{DLE}}$-inequality $s\leq t$,
			\begin{enumerate}
				\item if $s\leq t = \xi(\gamma/!x, \vec{\psi}/\oz)$ is restricted left-analytic $(\Omega,\epsilon)$-inductive such that $\xi$ is definite and $-x\prec -\xi$ (resp.\ $+x\prec -\xi$), the adjunction rule LA($\xi$) is applicable and yields the equivalent inequality $\mathsf{LA}(\xi)(s/u,\vec{\psi})\leq \gamma$ (resp.\ $\gamma\leq \mathsf{LA}(\xi)(s/u,\vec{\psi})$), which is very restricted left-analytic (resp.\ right-analytic) $(\Omega,\epsilon)$-inductive.
				\item if $\xi(\gamma/!x, \vec{\psi}/\oz) = s\leq t$ is restricted right-analytic $(\Omega,\epsilon)$-inductive such that $\xi$ is definite and $+x\prec +\xi$ (resp.\ $-x\prec +\xi$), the adjunction rule RA($\xi$) is applicable and yields the equivalent inequality $\gamma\leq \mathsf{RA}(\xi)(t/u,\vec{\psi})$ (resp.\ $\mathsf{RA}(\xi)(t/u,\vec{\psi})\leq \gamma$), which is very restricted right-analytic (resp.\ left-analytic) $(\Omega,\epsilon)$-inductive.
			\end{enumerate}
			% By applying consecutive applications of adjoints and residuations in an almost strictly primitive inductive inequality we get a strictly primitive inductive inequality.
		\end{lemma}
		\begin{proof}
			We only show the first item in the case $-x\prec -\xi$, the remaining cases being similar. %As discussed in the proof of Theorem \ref{type3:theorem}, we can assume w.l.o.g.\ that $\xi(!x,\oz)$, which is definite positive PIA, is also right-primitive.
The assumptions imply (cf.\ Lemma \ref{type3lem}) that the rule $LA(\xi)$ is applicable to $s\leq t$ so as to obtain the inequality $\mathsf{LA}(\xi)(s/u,\vec{\psi})\leq \gamma$, and that $\mathsf{LA}(\xi)(s/u,\oz)$ is a definite negative PIA formula. Since the polarities of $\oz$ do not change under the application of adjunction rules and the polarity of $u$ is positive, in $\mathsf{LA}(\xi)(s/u,\vec{\psi})$ the subtree of each $\psi\in\vec{\psi}$ remains a PIA subtree with at least one critical branch, and the branches running through $s$ remain good. Hence, $+\mathsf{LA}(\xi)(s/u,\vec{\psi})$ is $(\Omega,\epsilon)$-inductive, all of its branches are good, and all of its maximal $\epsilon^\partial$-uniform PIA subtrees occur as immediate subtrees of SRR nodes of some $\epsilon$-critical branches. That is, $\mathsf{LA}(\xi)(s/u,\vec{\psi})$ is a restricted analytic inductive term. Furthermore, $\gamma$ is negative PIA, and $-\gamma$ is $\epsilon^\partial$-uniform by assumption. From the above observations it follows that $\mathsf{LA}(\xi)(s/u,\vec{\psi})\leq \gamma$ is a very restricted left-analytic $(\Omega,\epsilon)$-inductive inequality.
			%We prove this simultaneously for almost restricted right-analytic inductive and almost restricted left-analytic inductive, by induction on the complexity of $\xi(!x,\oz)$. The base case where $\xi=\phi$ is by definition a very restricted analytic inductive. We only show the induction step for the first case, the other being order dual. So let $s\leq t$ be an almost strictly left-primitive $(\Omega,\epsilon)$-inductive inequality, and let $t=g(t_1,\ldots,t_n,t_{n+1})$. where each $t_i$ for $1\leq i\leq n$ is primitive and $t_{n+1}$ contains the subterm $t'$. Since a subterm of a primitive term is also primitive, $t_{n+1}$ is also almost primitive. Without loss of generality assume that $\epsilon_f(n+1)=\partial$. Then by applying residuation we have: $t_{n+1}\leq g^\flat(t_1,\ldots,t_n,s)$. This is an almost strictly left-primitive inequality and induction hypothesis yields the result.
		\end{proof}
		
		%Below we present some examples of almost strictly primitive inequalities
\begin{cor}
\label{type4:cor}
			Every restricted left-analytic (resp.\ right-analytic) inductive $\mathcal{L}_{\mathrm{DLE}}$-inequality can be equivalently transformed, via an ALBA-reduction, into a set of pure left-primitive (resp.\ right-primitive) $\mathcal{L}_{\mathrm{DLE}}^\ast$-inequalities.
\end{cor}		
\begin{proof}
We only consider the case of the inequality $s\leq t = \xi(\gamma/!x, \vec{\psi}/\oz)$ being restricted left-analytic $(\Omega,\epsilon)$-inductive, and $+x\prec -\xi$, the remaining cases being similar. Modulo exhaustive application of distribution and splitting rules of the standard ALBA preprocessing,\footnote{The applications of splitting rules at this stage give rise to a set of inequalities, each of which can be treated separately. In the remainder of the proof, we focus on one of them.} we can assume w.l.o.g.\ that the negative Skeleton formula $\xi$ is also definite. By Lemma \ref{type4:lemma}, the adjunction rule LA($\xi$) is applicable and yields the equivalent inequality $\gamma\leq \mathsf{LA}(\xi)(s/u,\vec{\psi})$, which is very restricted right-analytic $(\Omega,\epsilon)$-inductive. Hence the statement follows by Theorem \ref{type3:theorem}.
\end{proof}		
		
		\paragraph{The Frege inequality, again.} Early on (cf.\ page \pageref{par:Frege type 3}), we have discussed the Frege inequality as an example of very restricted right-analytic $(\Omega, \epsilon)$-inductive inequality for $r <_\Omega p <_\Omega q$ and $\epsilon(p, q, r) =(1, 1, \partial)$. Here below, we provide an alternative solving strategy based on the fact that the Frege inequality is also a restricted left-analytic $(\Omega, \epsilon)$-inductive inequality for $\epsilon(p, q, r) = (1, 1, 1)$.
		%\marginnote{ Type 4 in the order-type (1, 1, 1). This paragraph needs to be revised}
		\begin{center}
			\begin{tabular}{c l l}
				
				& $\forall p\forall q\forall r[p\rhu(q\rhu r)\leq (p\rhu q)\rhu (p\rhu r)]$\\
				
				iff & $\forall p\forall q\forall r[(p\rhu q)\bullet (p\rhu(q\rhu r))\leq p\rhu r]$&(Residuation)\\
				
				iff & $\forall p\forall q\forall r[ p\bullet ((p\rhu q)\bullet (p\rhu(q\rhu r)))\leq r]$&(Residuation)\\
				
				iff & $\forall p\forall q\forall r\forall \nomi\forall\cnomm[\nomi\leq p\bullet ((p\rhu q)\bullet (p\rhu(q\rhu r)))\ \&\ r\leq\cnomm\Rightarrow\nomi\leq\cnomm]$&(First approximation)\\
				
				iff & $\forall p\forall q\forall r\forall \nomi\forall \nomj\forall \nomk\forall \nomh\forall\cnomm[\nomi\leq \nomj\bullet (\nomk\bullet \nomh)\ \&\ $\\
				
				& $\nomj\leq p\ \&\ \nomk\leq p\rhu q\ \&\ \nomh\leq p\rhu(q\rhu r)\ \&\ r\leq\cnomm\Rightarrow\nomi\leq\cnomm]$&(Approximation)\\
				
				iff & $\forall p\forall q\forall r\forall \nomi\forall \nomj\forall \nomk\forall \nomh\forall\cnomm[\nomi\leq \nomj\bullet (\nomk\bullet \nomh)\ \&\ $\\
				
				& $\nomj\leq p \ \&\ p\bullet \nomk\leq q\ \&\ q\bullet (p\bullet \nomh)\leq r\ \&\ r\leq\cnomm\Rightarrow\nomi\leq\cnomm]$&(Residuation)\\
				
				iff & $\forall q\forall r\forall \nomi\forall \nomj\forall \nomk\forall \nomh\forall\cnomm[\nomi\leq \nomj\bullet (\nomk\bullet \nomh)\ \&\ $\\
				
				& $\nomj\bullet \nomk\leq q\ \&\ q\bullet (\nomj\bullet \nomh)\leq r\ \&\ r\leq\cnomm\Rightarrow\nomi\leq\cnomm]$&(Ackermann lemma)\\
				
				iff & $\forall r\forall \nomi\forall \nomj\forall \nomk\forall \nomh\forall\cnomm[\nomi\leq \nomj\bullet (\nomk\bullet \nomh)\ \&\ (\nomj\bullet \nomk)\bullet (\nomj\bullet \nomh)\leq r\ \&\ r\leq\cnomm\Rightarrow\nomi\leq\cnomm]$&(Ackermann lemma)\\
				
				iff & $\forall \nomi\forall \nomj\forall \nomk\forall \nomh\forall\cnomm[\nomi\leq \nomj\bullet (\nomk\bullet \nomh)\ \&\ (\nomj\bullet \nomk)\bullet (\nomj\bullet \nomh)\leq\cnomm\Rightarrow\nomi\leq\cnomm]$&(Ackermann lemma)\\
				
				iff & $\forall \nomj\forall \nomk\forall \nomh[\nomj\bullet (\nomk\bullet \nomh)\leq(\nomj\bullet \nomk)\bullet (\nomj\bullet \nomh)]$.\\
				
			\end{tabular}
		\end{center}
The last inequality above is a pure left-primitive $\mathcal{L}_{DLE}^*$-inequality, and by Proposition \ref{prop:type1} is equivalent on perfect DLE-algebras to
		\[p\bullet (q\bullet r)\leq(p\bullet q)\bullet (p\bullet r).\]
		By applying the usual procedure, we obtain the following rule:
		\[p\bullet (q\bullet r)\leq(p\bullet q)\bullet (p\bullet r)\quad \rightsquigarrow\quad \frac{(p\bullet q)\bullet (p\bullet r)\vdash y}{p\bullet (q\bullet r)\vdash y}\quad \rightsquigarrow\quad \frac{(X\circled{$\bullet$} Y)\circled{$\bullet$} (X\circled{$\bullet$} Z)\vdash W}{X\circled{$\bullet$} (Y\circled{$\bullet$}Z)\vdash W}\]

\paragraph{The non-primitive Fischer Servi inequality.} For the $\mathcal{L}_\mathrm{DLE}$-setting specified as in Example \ref{FS2 from axiom to rule}, the Fischer Servi inequality $\Diamond(p\rightarrow q)\leq \Box p\rightarrow \Diamond q$ is restricted right-analytic $(\Omega, \epsilon)$-inductive w.r.t.\ the discrete order $\Omega$ and $\epsilon(p, q) = (1, \partial)$. Let us apply the procedure indicated in the proof of Corollary \ref{type4:cor} to it:

\begin{center}
	\begin{tabular}{r c l l}
		& & $\forall p\forall q[\Diamond(p\rightarrow q)\leq \Box p\rightarrow \Diamond q]$ &\\
 & iff & $\forall p\forall q[p\rightarrow q\leq \blacksquare(\Box p\rightarrow \Diamond q)]$ & (Adjunction)\\
		& iff & $\forall p\forall q\forall \nomi\forall \cnomm[(\nomi\leq p\rightarrow q \ \& \ \blacksquare(\Box p\rightarrow \Diamond q)\leq \cnomm)\Rightarrow \nomi\leq \cnomm] $ & (First approximation)\\
		& iff & $\forall p\forall q\forall \nomi\forall m\forall \nomj\forall \cnomn[(\nomi\leq p\rightarrow q \ \& \ \blacksquare(\nomj\rightarrow \cnomn)\leq \cnomm \ \&\ \nomj\leq \Box p \ \&\ \Diamond q\leq \cnomn )\Rightarrow \nomi\leq \cnomm] $ & (Approximation)\\
 & iff & $\forall p\forall q\forall \nomi\forall m\forall \nomj\forall \cnomn[(\nomi\leq p\rightarrow q \ \& \ \blacksquare(\nomj\rightarrow \cnomn)\leq \cnomm \ \&\ \Diamondblack\nomj\leq p \ \&\ q\leq \blacksquare \cnomn )\Rightarrow \nomi\leq \cnomm] $ & (Adjunction)\\
& iff & $\forall \nomi\forall m\forall \nomj\forall \cnomn[(\nomi\leq \Diamondblack\nomj\rightarrow \blacksquare \cnomn \ \& \ \blacksquare(\nomj\rightarrow \cnomn)\leq \cnomm)\Rightarrow \nomi\leq \cnomm] $ & (Ackermann)\\
& iff & $\forall \nomj\forall \cnomn[\Diamondblack \nomj\rightarrow \blacksquare \cnomn\leq \blacksquare( \nomj\rightarrow \cnomn)]. $ &\\
	\end{tabular}
\end{center}
The last inequality above is a pure right-primitive $\mathcal{L}_{DLE}^*$-inequality, and by Proposition \ref{prop:type1} is equivalent on perfect DLE-algebras to
		\[\Diamondblack p\rightarrow \blacksquare q\leq \blacksquare( p\rightarrow q).\]
		By applying the usual procedure, we obtain the following rule:
		\[\Diamondblack p\rightarrow \blacksquare q\leq \blacksquare( p\rightarrow q)\quad \rightsquigarrow\quad \frac{x\vdash
\Diamondblack p\rightarrow \blacksquare q}{x\vdash \blacksquare( p\rightarrow q)}\quad \rightsquigarrow\quad \frac{X\vdash
\bullet Y > \bullet Z}{X\vdash \bullet( Y > Z)}\]

\paragraph{The `transitivity' axiom.} For the $\mathcal{L}_\mathrm{DLE}$-setting in which we discussed the Frege inequality (cf.\ page \pageref{par:Frege type 3}), the inequality $(p\rhu q)\bullet (q\rhu r)\leq p\rhu r$ is restricted right-analytic $(\Omega, \epsilon)$-inductive w.r.t.\ the order $p<_{\Omega} q$ and $\epsilon(p, q, r) = (1, \partial, \partial)$. Let us apply the procedure indicated in the proof of Corollary \ref{type4:cor} to it:

\begin{center}
	\begin{tabular}{r c l l}
		& & $\forall p\forall q\forall r[(p\rhu q)\bullet (q\rhu r)\leq p\rhu r]$&\\
 & iff & $\forall p\forall q\forall r[ q\rhu r\leq (p\rhu q)\rhu(p\rhu r)]$& (Adjunction)\\
		& iff & $\forall pqr\forall \nomi\forall \cnomm[(\nomi\leq q\rhu r\ \& \ (p\rhu q)\rhu(p\rhu r)\leq \cnomm )\Rightarrow \nomi\leq \cnomm]$ & (First approximation)\\
		& iff & $\forall pqr\forall \nomi\cnomm \nomh\nomj\cnomn[(\nomi\leq q\rhu r\ \& \ \nomj\rhu(\nomh\rhu \cnomn)\leq \cnomm \ \&\ \nomh\leq p\ \&\ r\leq \cnomn\ \&\ \nomj\leq p\rhu q)\Rightarrow \nomi\leq \cnomm]$ & (Approximation)\\
& iff & $\forall q\forall \nomi\forall \cnomm\forall \nomh\forall\nomj\forall\cnomn[(\nomi\leq q\rhu \cnomn\ \& \ \nomj\rhu(\nomh\rhu \cnomn)\leq \cnomm \ \&\ \nomj\leq \nomh\rhu q)\Rightarrow \nomi\leq \cnomm]$ & (Ackermann)\\
& iff & $\forall q\forall \nomi\forall \cnomm\forall \nomh\forall\nomj\forall\cnomn[(\nomi\leq q\rhu \cnomn\ \& \ \nomj\rhu(\nomh\rhu \cnomn)\leq \cnomm \ \&\ \nomh\bullet \nomj\leq q)\Rightarrow \nomi\leq \cnomm]$ & (Adjunction)\\
& iff & $\forall \nomi\forall \cnomm\forall \nomh\forall\nomj\forall\cnomn[(\nomi\leq (\nomh\bullet \nomj)\rhu \cnomn\ \& \ \nomj\rhu(\nomh\rhu \cnomn)\leq \cnomm)\Rightarrow \nomi\leq \cnomm]$ & (Ackermann)\\
& iff & $\forall \nomh\forall\nomj\forall\cnomn[(\nomh\bullet \nomj)\rhu \cnomn\leq\nomj\rhu(\nomh\rhu \cnomn)]$. & \\
\end{tabular}
\end{center}
The last inequality above is a pure right-primitive $\mathcal{L}_{DLE}^*$-inequality, and by Proposition \ref{prop:type1} is equivalent on perfect DLE-algebras to
		\[(p\bullet q)\rhu r\leq q\rhu(p\rhu r).\]
		By applying the usual procedure, we obtain the following rule:
		\[(p\bullet q)\rhu r\leq q\rhu(p\rhu r)\quad \rightsquigarrow\quad \frac{x\vdash
(p\bullet q)\rhu r}{x\vdash q\rhu(p\rhu r)}\quad \rightsquigarrow\quad \frac{X\vdash
(Y\circled{$\bullet$} Z)\succ W}{X\vdash Z\succ(Y\succ W)}\]

\section{Analytic inductive inequalities and analytic rules}\label{sec:analytic}
In the present section, we address the most general syntactic shape considered in the paper: in the following subsection we define the class of analytic inductive inequalities, and show that each of them can be equivalently transformed into (a set of) analytic structural rules (which are in fact quasi-special). In Subsection \ref{sec:rulestoineq}, we also show that any analytic rule is semantically equivalent to some analytic inductive inequality. Thus, the DLE-logics axiomatized by means of analytic inductive inequalities are exactly the properly displayable ones.
\subsection{From analytic inductive inequalities to quasi-special rules}
\label{ssec:type5}
 Let us start with a motivating example:

\paragraph{The pre-linearity axiom.} Let $\mathcal{F} = \varnothing$, $\mathcal{G} = \{\rhu\}$ where $\rhu$ is binary and of order-type $(\partial, 1)$.
\begin{center}
	\begin{tabular}{|c|c|c|c|c|c|}
		\hline
		\mc{2}{|c|}{I} & \mc{2}{c|}{$;$} & \mc{2}{c|}{$\succ$} \\
		\hline
		$\top$ & $\bot$ & $\pand$ & $\por$ & $\phantom{\rhu}$ & $\rhu$ \\
		\hline
	\end{tabular}
\end{center}

\noindent The following inequality
\[\top\leq ( p\rhu q)\vee (q\rhu p) \]
is not restricted analytic inductive for any order-type: indeed, all the non-leaf nodes of the right-hand are Skeleton, and the PIA subterms are reduced to the variables. The inequality above is not restricted right-analytic for any order-type $\varepsilon$, since the right-hand side contains $\varepsilon^\partial$-uniform PIA-subterms attached to the skeleton. It is not restricted left-analytic for any order-type $\varepsilon$, since the right-hand side contains more than one $\varepsilon^\partial$-uniform PIA-subterm.

We have not found an ALBA-reduction suitable to extend the strategy of the previous section so as to equivalently transform the inequality above into one or more primitive inequalities. However, the following ALBA reduction, exclusively based on applications of a modified (inverted) Ackermann rule (the soundness of which is proved in Lemma \ref{lem:inverted:ackermann} below) and adjunction rules, transforms the inequality above into a quasi-inequality which gives rise to an analytic (in fact quasi-special, cf.\ Definition \ref{def:quasispecial}) structural rule.

\begin{center}
	\begin{tabular}{r l l}
		& $\forall p\forall q[\top\leq ( p\rhu q)\vee (q\rhu p) ]$ &\\
		& $\forall p\forall q\forall \vec{r}[(r_1\leq p \ \&\ q\leq r_2\ \&\ r_3\leq q\ \&\ p\leq r_4) \Rightarrow \top\leq ( r_1\rhu r_2)\vee (r_3\rhu r_4) ]$ &\\
		& $\forall q\forall \vec{r}[(r_1\leq r_4 \ \&\ q\leq r_2\ \&\ r_3\leq q) \Rightarrow \top\leq ( r_1\rhu r_2)\vee (r_3\rhu r_4) ]$ &\\
		& $\forall \vec{r}[(r_1\leq r_4 \ \&\ r_3\leq r_2) \Rightarrow \top\leq ( r_1\rhu r_2)\vee (r_3\rhu r_4) ].$ &\\
		
	\end{tabular}
\end{center}

The last quasi-inequality above expresses the validity of the following quasi-special structural rule on perfect DLEs:
\begin{center}
\AxiomC{$X\vdash W$}\AxiomC{$Z\vdash Y$}
\BinaryInfC{$\mathrm{I}\vdash (X\succ Y)\, ;\, (Z\succ W)$}
\DisplayProof
\end{center}
%\noindent Clearly, the rule above is analytic but not special.
\noindent We will see that the solving strategy applied to the example above can be applied to the following class of inequalities:
\begin{definition}[Analytic inductive inequalities]
	\label{def:type5}
	For every order type $\epsilon$ and every irreflexive and transitive relation $\Omega$ on the variables $p_1,\ldots p_n$,
	an inequality $s \leq t$ is \emph{ analytic $(\Omega, \epsilon)$-inductive} if $+s$ and $-t$ are both $(\Omega, \epsilon)$-analytic inductive (cf.\ Definition \ref{def:type4}). An inequality $s \leq t$ is \emph{analytic inductive} if is $(\Omega, \epsilon)$-analytic inductive for some $\Omega$ and $\epsilon$.
\end{definition}

\begin{remark}
	\label{remark:discussion type 5} %\marginnote{note that in the picture somehow the positive skeletons and negative pia are left primitive and the negative skeleton and positive pia are right primitive.}
	The syntactic shape specified in the definition above can be intuitively understood with the help of the following picture:
	
	
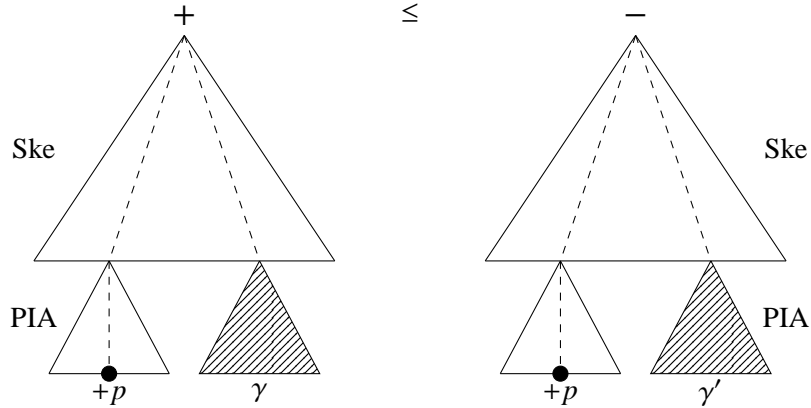
\begin{center}
		\begin{tikzpicture}
		\draw (-5,-1.5) -- (-3,1.5) node[above]{\Large$+$} ;
		\draw (-5,-1.5) -- (-1,-1.5) ;
		\draw (-3,1.5) -- (-1,-1.5);
		\draw (-5,0) node{Ske} ;
		\draw[dashed] (-3,1.5) -- (-4,-1.5);
		\draw[dashed] (-3,1.5) -- (-2,-1.5);
		\draw (-4,-1.5) --(-4.8,-3);
		\draw (-4.8,-3) --(-3.2,-3);
		\draw (-3.2,-3) --(-4,-1.5);
		\draw[dashed] (-4,-1.5) -- (-4,-3);
		\draw[fill] (-4,-3) circle[radius=.1] node[below]{$+p$};
		\draw
		(-2,-1.5) -- (-2.8,-3) -- (-1.2,-3) -- (-2,-1.5);
		\fill[pattern=north east lines]
		(-2,-1.5) -- (-2.8,-3) -- (-1.2,-3);
		\draw (-2,-3.25)node{$\gamma$};
		\draw (-5,-2.25) node{PIA} ;
		\draw (0,1.8) node{$\leq$};
		\draw (5,-1.5) -- (3,1.5) node[above]{\Large$-$} ;
		\draw (5,-1.5) -- (1,-1.5) ;
		\draw (3,1.5) -- (1,-1.5);
		\draw (5,0) node{Ske} ;
		\draw[dashed] (3,1.5) -- (4,-1.5);
		\draw[dashed] (3,1.5) -- (2,-1.5);
		\draw (2,-1.5) --(2.8,-3);
		\draw (2.8,-3) --(1.2,-3);
		\draw (1.2,-3) --(2,-1.5);
		\draw[dashed] (2,-1.5) -- (2,-3);
		\draw[fill] (2,-3) circle[radius=.1] node[below]{$+p$};
		\draw
		(4,-1.5) -- (4.8,-3) -- (3.2,-3) -- (4, -1.5);
		\fill[pattern=north east lines]
		(4,-1.5) -- (4.8,-3) -- (3.2,-3) -- (4, -1.5);
		\draw (4,-3.25)node{$\gamma'$};
		\draw (5,-2.25) node{PIA} ;
		\end{tikzpicture}
	\end{center}

	As the picture shows,
	the difference between analytic inductive inequalities and restricted analytic inductive inequalities is that, in the latter, there can be exactly one $\varepsilon^\partial$-uniform subterm attached to the skeleton of the inequality, %every maximal subterm consisting only of PIA nodes has to contain an $\epsilon$-critical branch,
	while in the former this requirement is dropped. %Hence the only restriction we require in this case is that every branch, regardless of whether it is critical or not, is a good branch.

\end{remark}
Below, we discuss a slightly modified version of the Ackermann rule, which will be used in the proof of Proposition \ref{prop:type5}.
\begin{lemma}
	\label{lem:inverted:ackermann}
	Let $s(\vec{!q})$ and $t(\vec{!r})$ be $\mathcal{L}_\mathrm{DLE}$-terms such that the propositional variables in the arrays $\vec{!q}$ and $\vec{!r}$ are disjoint. Let $\varepsilon$ be the unique order-type on the concatenation $\vec{!q}\oplus\vec{!r}$ with respect to which $s\leq t$ is $\varepsilon$-uniform. Let $\vec{\alpha}$ and $\vec{\beta}$ be arrays of DLE-terms of the same length as $\vec{q}$ and $\vec{r}$ respectively, and such that no variable in $\vec{q}\oplus \vec{r}$ occurs in any $\alpha$ or $\beta$.
	Then the following are equivalent for any perfect DLE $\bbA$:
	\begin{enumerate}
		\item $\bbA\models s(\vec{\alpha}/\vec{q})\leq t(\vec{\beta}/\vec{r})$;
		\item $\bbA\models \forall\vec{q}\forall \vec{r}[\bigamp_{q\in \vec{q}, r\in \vec{r}}(q\leq^{\varepsilon(q)} \alpha\ \&\ \beta\leq^{\varepsilon(r)} r )\Rightarrow s(\vec{!q})\leq t(\vec{!r})]$,
	\end{enumerate}
	where $\leq^1: = \leq$ and $\leq^\partial: = \geq$.
\end{lemma}
\begin{proof}
	Let us assume 1. To show 2., fix an interpretation $v$ of the variables in $\bbA$ such that $(\bbA, v)\models q\leq^{\varepsilon(q)} \alpha$ and $ (\bbA, v)\models \beta \leq^{\varepsilon(r)} r$ for each $q$ in $\vec{q}$ and $r$ in $\vec{r}$. Hence, assumption 1.\ and the $\varepsilon$-uniformity of $s\leq t$ imply that $\val{s(\vec{q})}_v\leq \val{s(\vec{\alpha}/\vec{q})}_v\leq \val{t(\vec{\beta}/\vec{r})}_v\leq \val{t(\vec{r})}_v$, which proves that $(\bbA, v)\models s(\vec{q})\leq t(\vec{r})$, as required. Conversely, fix a valuation $v$, and notice that the truth of the required condition $(\bbA, v)\models s(\vec{\alpha}/\vec{q})\leq t(\vec{\beta}/\vec{r})$ does not depend on where $v$ maps the variables in $\vec{q}\oplus\vec{r}$, since none of these variables occurs in $s(\vec{\alpha}/\vec{q})\leq t(\vec{\beta}/\vec{r})$. Hence, it is enough to show that $(\bbA, v')\models s(\vec{\alpha}/\vec{q})\leq t(\vec{\beta}/\vec{r})$ for some $\vec{q}\oplus\vec{r}$-variant $v'$ of $v$.
	Let $v'$ be the $\vec{q}\oplus\vec{r}$-variant of $v$ such that $\val{q_i}_{v'}: = \val{\alpha_i}_v= \val{\alpha_i}_{v'}$ and $\val{r_j}_{v'}: = \val{\beta_j}_v= \val{\beta_j}_{v'}$. Then clearly, $(\bbA, v')\models q\leq^{\varepsilon(q)} \alpha$ and $ (\bbA, v')\models \beta \leq^{\varepsilon(r)} r$. By assumption 2, this implies that $(\bbA, v')\models s(\vec{q})\leq t(\vec{r})$, which is equivalent by construction to the required $(\bbA, v')\models s(\vec{\alpha}/\vec{q})\leq t(\vec{\beta}/\vec{r})$.
\end{proof}
\begin{remark}
	\label{rem:polarity}
	Notice that, in the quasi-inequality in item 2 of the statement of the lemma above, each variable $q$ in $\vec{q}$ and $r$ in $\vec{r}$ occurs twice, i.e.\ once in exactly one inequality in the antecedent and once in the conclusion of the quasi-inequality. These two occurrences have the same polarity in the two inequalities. For example, if $q$ is in $\vec{q_1}$ and $\varepsilon(q) = \partial$, then $q$ occurs negatively in the conclusion of the quasi-inequality, and also negatively in the inequality $q\leq^{\varepsilon(q)}\phi$, which can be rewritten as $\phi\leq q$. The remaining cases are analogous and left to the reader.
\end{remark}
\begin{prop}
	\label{prop:type5} Every analytic $(\Omega, \epsilon)$-inductive inequality can be equivalently transformed, via an ALBA-reduction, into a set of quasi-special structural rules.
\end{prop}
\begin{proof}
	By assumption, $s\leq t$ is of the form \[\xi_1(\vec{\phi}_1/\vec{x}_1,\vec{\psi}_1/\vec{y}_1, \vec{\gamma}_1/\vec{z}_1, \vec{\chi}_1/\vec{w}_1)\leq \xi_2(\vec{\psi}_2/\vec{x}_2,\vec{\phi}_2/\vec{y}_2, \vec{\chi}_2/\vec{z}_2,\vec{\gamma}_2/\vec{w}_2),\] where
	$\xi_1(\vec{!x}_1, \vec{!y}_1, \vec{!z}_1, \vec{!w}_1)$ and $\xi_2(\vec{!x}_2, \vec{!y}_2, \vec{!z}_2, \vec{!w}_2)$ respectively are a positive and a negative Skeleton-formula (cf.\ page \pageref{page: positive negative PIA}) which are scattered, monotone in $\vec{x}$ and $\vec{z}$ and antitone in $\vec{y}$ and $\vec{w}$. Moreover, the formulas in $\vec{\phi}$ and $\vec{\gamma}$ are positive PIA, and the formulas in $\vec{\psi}$ and $\vec{\chi}$ are negative PIA. Finally, every $\phi$ and $\psi$ contains at least one $\varepsilon$-critical variable, whereas all $+\gamma$ and $-\chi$ are $\varepsilon^\partial$-uniform. %hence right-primitive inductive, and they consist only of PIA nodes. Using Corollary \ref{cor:rule:type:3} we obtain:
Modulo exhaustive application of distribution and splitting rules of the standard ALBA preprocessing,\footnote{The applications of splitting rules at this stage give rise to a set of inequalities, each of which can be treated separately. In the remainder of the proof, we focus on one of them.} we can assume w.l.o.g.\ that the scattered Skeleton formulas $\xi_1$ and $\xi_2$ are also definite. Modulo exhaustive application of the additional rules which identify $+f'(\phi_1,\ldots,\bot^{\varepsilon_{f'}}(i),\ldots,\phi_{n_{f'}})$ with $+\bot$ for every $f'\in \mathcal{F}\cup\{\wedge\}$ and $-g'(\phi_1,\ldots,\top^{\varepsilon_{g'}}(i),\ldots,\phi_{n_{g'}})$ with $-\top$ for every $g'\in \mathcal{G}\cup\{\vee\}$, which would reduce $s\leq t$ to a tautology, we can assume w.l.o.g.\ that there are no occurrences of $+\bot$ and $-\top$ in $+\xi_1$ and $-\xi_2$. Hence (cf.\ Remark \ref{rem:primitive-skeleton-PIA}) we can assume w.l.o.g.\ that $\xi_1$ (resp.\ $\xi_2$) is scattered, definite and left-primitive (resp.\ right-primitive).
	%
	%Hence, Lemma is applicable, %derived rule Approx($\xi$) (cf.\ Corollary \ref{cor:rule:type:1}) is applicable,
	%which justifies
	The following equivalence is justified by Lemma \ref{lem:inverted:ackermann} (in what follows, we write e.g.\ $\vec{q}_{1, 6}\leq \vec{\phi}$ to represent concisely both $\vec{q}_1\leq \vec{\phi}_1$ and $\vec{q}_{6}\leq \vec{\phi}_2$):
	
	\begin{center}
		\begin{tabular}{c l l}
			& $\forall \vec{p}[s\leq t]$\\
			iff & $\forall \vec{p}\forall \vec{q}[(\vec{q}_{1, 6}\leq \vec{\phi}\ \&\ \vec{\psi}\leq \vec{q}_{2, 5}\ \&\ \vec{q}_{3,8}\leq \vec{\gamma}\ \&\ \vec{\chi}\leq \vec{q}_{4, 7})\Rightarrow \xi_1(\vec{q}_1, \vec{q}_2, \vec{q}_3, \vec{q}_4)\leq \xi_2(\vec{q}_5, \vec{q}_6, \vec{q}_7, \vec{q}_8)].$\\
			%iff & $\forall \vec{p}\forall \nomj\forall \cnomn\forall \vec{\nomi}\forall\vec{\cnomm}[(\nomj\leq \xi(\vec{\nomi}/\vec{x},\vec{\cnomm}/\vec{y}, \vec{p})\ \&\ \vec{\nomi}\leq \vec{\phi}\ \&\ \vec{\psi}\leq \vec{\cnomm}\ \&\ t(\vec{p})\leq\cnomn)\Rightarrow \nomj\leq \cnomn]$. & (Approx($\xi$))\\
		\end{tabular}
	\end{center}
	By assumption, in each inequality $q\leq \phi$ (resp.\ $\psi\leq q$) in $\vec{q}_{1, 6}\leq \vec{\phi}$ (resp.\ $\vec{\psi}\leq \vec{q}_{2, 5}$) there is at least one $\varepsilon$-critical occurrence of some variable in $\vec{p}$.
Modulo exhaustive application of distribution rules (a')-(c') of Remark \ref{rmk: management of top and bottom on the wrong side} and splitting rules, we can assume w.l.o.g.\ that each $\phi$ in $\vec{\phi}$ (resp.\ $\psi$ in $\vec{\psi}$) is a definite positive (resp.\ negative) PIA-formula, which has exactly one $\varepsilon$-critical variable occurrence. That is, if $\vec{p_1}$ and $\vec{p_2}$ respectively denote the subarrays of $\vec{p}$ such that $\varepsilon(p_1) = 1$ for each $p_1$ in $\vec{p_1}$ and $\varepsilon(p_2) = \partial$ for each $p_2$ in $\vec{p_2}$, then each $\phi$ in $\vec{\phi}$ is either of the form $\phi_+(p_1/!x, \vec{p'}/\oz)$ with $+x\prec +\phi_+$, or of the form $\phi_-(p_2/!x, \vec{p'}/\oz)$ with $-x\prec +\phi_-$. Similarly, each $\psi$ in $\vec{\psi}$ is either of the form $\psi_+(p_1/!x, \vec{p'}/\oz)$ with $+x\prec -\psi_+$, or of the form $\psi_-(p_2/!x, \vec{p'}/\oz)$ with $-x\prec -\psi_-$. Recall that each $\phi$ in $\vec{\phi}$ is definite positive PIA and each $\psi$ in $\vec{\psi}$ is definite negative PIA. Hence, we can assume w.l.o.g.\ that there are no occurrences of $-\bot$ and $+\top$ in each $+\phi$. Indeed, otherwise, the exhaustive application of the additional rules which identify $-f'(\phi_1,\ldots,\bot^{\varepsilon_{f'}}(i),\ldots,\phi_{n_{f'}})$ with $-\bot$ for every $f'\in \mathcal{F}\cup\{\wedge\}$ and $+g'(\phi_1,\ldots,\top^{\varepsilon_{g'}}(i),\ldots,\phi_{n_{g'}})$ with $+\top$ for every $g'\in \mathcal{G}\cup\{\vee\}$, would reduce all offending inequalities to tautological inequalities of the form $q\leq \top$ or $\bot\leq q$ which can then be removed. Likewise, we can assume w.l.o.g.\ that there are no occurrences of $+\bot$ and $-\top$ in each $+\psi$. This shows (cf.\ Remark \ref{rem:primitive-skeleton-PIA}) that we can assume w.l.o.g.\ that each $\phi$ in $\vec{\phi}$ (resp.\ $\psi$ in $\vec{\psi}$) is a right-primitive (resp.\ left-primitive) term. Likewise, we can assume w.l.o.g.\ that each $\gamma$ in $\vec{\gamma}$ (resp.\ $\chi$ in $\vec{\chi}$) is right-primitive (resp.\ left-primitive).
By Lemma \ref{type3lem}, the suitable derived adjunction rule among LA($\phi_+$), LA($\phi_-$), RA($\psi_+$), RA($\psi_-$) is applicable to each $\phi$ and $\psi$, yielding:
	\begin{center}
		\begin{tabular}{c l l}
			& $\forall \vec{p}\forall \vec{q}[(\vec{q}_{1, 6}\leq \vec{\phi}\ \&\ \vec{\psi}\leq \vec{q}_{2, 5}\ \&\ \vec{q}_{3,8}\leq \vec{\gamma}\ \&\ \vec{\chi}\leq \vec{q}_{4, 7})\Rightarrow \xi_1(\vec{q}_1, \vec{q}_2, \vec{q}_3, \vec{q}_4)\leq \xi_2(\vec{q}_5, \vec{q}_6, \vec{q}_7, \vec{q}_8)]$\\
			
			iff & $\forall \vec{p_1}\forall \vec{p_2}\forall \vec{q}[(\overrightarrow{\mathsf{LA}(\phi_+)(q/u)}\leq \vec{p_1}\ \&\ \vec{p_2}\leq \overrightarrow{\mathsf{LA}(\phi_-)(q/u)}\ \&\ \overrightarrow{\mathsf{RA}(\psi_+)(q/u)}\leq \vec{p_1}\ \& $\\
			&$ \vec{p_2}\leq \overrightarrow{\mathsf{RA}(\psi_-)(q/u)}\ \&\ \vec{q}_{3,8}\leq \vec{\gamma}\ \&\ \vec{\chi}\leq \vec{q}_{4, 7})\Rightarrow \xi_1(\vec{q}_1, \vec{q}_2, \vec{q}_3, \vec{q}_4)\leq \xi_2(\vec{q}_5, \vec{q}_6, \vec{q}_7, \vec{q}_8)].$\\
		\end{tabular}
\end{center}
Notice that, when applying the adjunction/residuation rules, the polarity of subterms which are parametric in the rule application remains unchanged. Hence, the assumption that there are no occurrences of $-\bot$ and $+\top$ in each $+\phi$ and there are no occurrences of $+\bot$ and $-\top$ in each $+\psi$ implies that that there are no occurrences of $-\bot$ and $+\top$ in each $+\mathsf{LA}(\phi_-)(q/u)$ and $+\mathsf{RA}(\psi_-)(q/u)$, which are then shown to be right-primitive, and there are no occurrences of $+\bot$ and $-\top$ in each $+\mathsf{LA}(\phi_+)(q/u)$ and $+\mathsf{RA}(\psi_+)(q/u)$, which are then shown to be left-primitive.
The assumptions made above imply that each $\gamma$ is antitone in each variable in $\vec{p_1}$ and monotone in each variable in $\vec{p_2}$, while each $\chi$ is monotone in each variable in $\vec{p_1}$ and antitone in each variable in $\vec{p_2}$. Hence, the quasi-inequality above is simultaneously in Ackermann shape w.r.t.\ all variables in $\vec{p}$.\footnote{The formulas $\mathsf{LA}(\phi_+)(q/u)$, $\mathsf{LA}(\phi_-)(q/u)$, $\mathsf{RA}(\psi_+)(q/u)$, and $\mathsf{RA}(\psi_-)(q/u)$ do not need to be free of all variables in $\vec{p}$, and in general they are not. However, the assumptions and the general theory of ALBA guarantee that they are $\varepsilon^\partial$-uniform and free of the specific $p$-variable the `minimal valuation' of which they are part of. The reader is referred to \cite{CoPa10} for an expanded treatment of this point.} Applying the Ackermann rule repeatedly in the order indicated by $\Omega$ yields the following quasi-inequality, free of variables in $\vec{p}$:
	\begin{equation}
		\label{eq:analytic quasieq}
		\forall \vec{q}[(\vec{q}_{3, 8}\leq \vec{\gamma}(\vec{P_1}/\vec{p_1}, \vec{P_2}/\vec{p_2})\ \&\ \vec{\chi}(\vec{P_1}/\vec{p_1}, \vec{P_2}/\vec{p_2})\leq\vec{q}_{4, 7})\Rightarrow \xi_1(\vec{q}_1, \vec{q}_2, \vec{q}_3, \vec{q}_4)\leq \xi_2(\vec{q}_5, \vec{q}_6, \vec{q}_7, \vec{q}_8)],
	\end{equation}
	where $P_1$ and $P_2$ denote the $\mathcal{L}_{\mathrm{DLE}}^\ast$-terms, with variables in $\vec{q}_1, \vec{q}_2, \vec{q}_5, \vec{q}_6$, obtained by applying the Ackermann-substitution. For instance, for every $\Omega$-minimal $p_1$ in $\vec{p_1}$, \[P_1: = \bigvee_i \mathsf{LA}(\phi_+^{(i)})(q/u)\vee \bigvee_j \mathsf{RA}(\psi_+^{(j)})(q/u),\]
	and for every $\Omega$-minimal $p_2$ in $\vec{p_2}$, \[P_2: = \bigwedge_i \mathsf{LA}(\phi_-^{(i)})(q/u)\wedge \bigwedge_j \mathsf{RA}(\psi_-^{(j)})(q/u).\]
	In the clauses above, the indexes $i$ and $j$ count the number of critical occurrences of the given variable $p_1$ (resp.\ $p_2$) in PIA-subterms of type $\phi_+$ and $\psi_+$ (resp.\ $\phi_-$ and $\psi_-$).
	
	Let us show that the quasi-inequality \eqref{eq:analytic quasieq} represents the validity in perfect DLEs of some analytic (in fact quasi-special, cf.\ Definition \ref{def:quasispecial}) structural rule of the calculus $\mathbf{DL}$. We have already observed above that $\xi_1$ is left-primitive and $\xi_2$ is right-primitive. Hence, the conclusion of the quasi-inequality \eqref{eq:analytic quasieq} can be understood as the semantic interpretation of some structural sequent (cf.\ Definition \ref{def:structure from primitive}). To see that each inequality in the antecedent of \eqref{eq:analytic quasieq} is also the interpretation of some structural sequent, it is enough to show that every $\gamma(\vec{P_1}/\vec{p_1}, \vec{P_2}/\vec{p_2})$ is right-primitive, and every $\chi(\vec{P_1}/\vec{p_1}, \vec{P_2}/\vec{p_2})$ is left-primitive. Indeed, if this is the case, then we can apply distribution rules exhaustively so as to surface the $+\lor$ and $-\land$, and then apply splitting rules to obtain definite left-primitive and right-primitive inequalities. By Definition \ref{def:structure from primitive} each of these inequalities will be the interpretation of some structural sequent.
	
	This is a rather simple proof by induction on the maximum length of chains in $\Omega$. The base case, when $\Omega$ is the discrete order (hence $P_1$ and $P_2$ are of the form displayed above), immediately follows from the observation, made above, that each $\gamma$ is right-primitive, antitone in each variable in $\vec{p_1}$ and monotone in each variable $\vec{p_2}$, while each $\chi$ is left-primitive, monotone in each variable in $\vec{p_1}$ and antitone in each variable $\vec{p_2}$, and by Lemma \ref{type3lem}. The inductive step is routine.
	
	Let us show that the rule so obtained is analytic (cf.\ Definition \ref{def:analytic}), that is, it satisfies conditions C$_1$-C$_7$. As to C$_1$, notice that each variable $q$ in $\vec{q}_i$ for $1\leq i\leq 8$ appears in some inequality in the antecedent of the initial quasi-inequality, and has not been eliminated in any ensuing transformations. This implies that each $q$ gives rise to a parametric structural variable $X$ which occurs in some premise and in the conclusion. Condition C$_2$ is guaranteed by construction: indeed, the congruence relation is defined as the transitive closure of the relation identifying only the occurrences of the structural variable $X$ corresponding to one variable $q$. Condition C$_3$ is also guaranteed by construction, given that each variable $q$ occurs exactly once in $\xi_1(\vec{q}_1, \vec{q}_2, \vec{q}_3, \vec{q}_4)\leq \xi_2(\vec{q}_5, \vec{q}_6, \vec{q}_7, \vec{q}_8)$. Condition C$_4$ follows from Remark \ref{rem:polarity}, and the fact that adjunction rules and usual Ackermann rule preserve the polarity of the variables. Condition C$_5$ vacuously holds, since all constituents of structural rules are parametric. Conditions C$_6$ and C$_7$ are immediate.

Finally, observe that the rule we have obtained is in fact quasi-special. Indeed, the variables $\vec{q_3}, \vec{q_4}, \vec{q_7}, \vec{q_8}$ are fresh, and each of them occurs only once in the premises. %\marginnote{Recall the definition of quasi-special rules; edit the proof to make it clear that it is quasi-special}
\end{proof}

\subsection{From analytic rules to analytic inductive inequalities}
\label{sec:rulestoineq}
In the previous section, we introduced the syntactic shape of analytic inductive $\mathcal{L}_{\mathrm{DLE}}$-inequalities for any language $\mathcal{L}_{\mathrm{DLE}}$, and showed that these inequalities can be effectively transformed via ALBA into a set of analytic structural rules of the associated display calculus $\mathbf{DL}_{\mathrm{DLE}}$.
In the present section, we show that having this shape is also a necessary condition.

\begin{lemma}\label{quasitoineq}
Let $s(\vec{q},\vec{r})$, $t(\vec{q},\vec{r})$, $\overrightarrow{\alpha(\vec{q},\vec{r})}$ and $\overrightarrow{\beta(\vec{q},\vec{r})}$ be $\mathcal{L}_{\mathrm{DLE}}$-terms such that %the variables in $\vec{q}$ (resp.\ $\vec{r}$) occur positively (resp.\ negatively) in $+s$ and $-t$. Let $\vec{\alpha},\vec{\beta}$ be arrays of $\mathcal{L}_{\mathrm{DLE}}$-terms such that
$t$ and each $\alpha$ are monotone in $\vec{r}$ and antitone in $\vec{q}$, and $s$ and each $\beta$ are monotone in $\vec{q}$ and antitone in $\vec{r}$. Then the following are equivalent for any DLE $\bbA$:
\begin{enumerate}
\item $\bbA\models s(\vec{q}\land\vec{\alpha}/\vec{q},\vec{r}\lor\vec{\beta}/\vec{r})\leq t(\vec{q}\land\vec{\alpha}/\vec{q},\vec{r}\lor\vec{\beta}/\vec{r})$;
\item $\bbA\models \forall\vec{q}\forall \vec{r}[(\vec{q}\leq\vec{\alpha}\ \&\ \vec{\beta}\leq\vec{r})\Rightarrow s(\vec{q},\vec{r})\leq t(\vec{q},\vec{r})]$
\end{enumerate}

%\item $\bbA\models \forall\vec{q}\forall \vec{r}[\bigamp_{q\in \vec{q}, r\in \vec{r}}(q\leq^{\varepsilon(q)} \alpha\ \&\ \beta\leq^{\varepsilon(r)} r )\Rightarrow s(\vec{!q})\leq t(\vec{!r})]
\end{lemma}
\begin{proof}
Assume item 1. To show item 2, fix a valuation $v$ such that $(\bbA, v)\models\vec{q}\leq\vec{\alpha}$ and $(\bbA, v)\models\vec{\beta}\leq\vec{r}$. Hence, $(\bbA, v)\models \vec{q}\land\vec{\alpha}=\vec{q}$ and $(\bbA, v)\models\vec{r}\lor\vec{\beta}=\vec{r}$. By item 1, $(\bbA, v)\models s(\vec{q}\land\vec{\alpha}/\vec{q},\vec{r}\lor\vec{\beta}/\vec{r})\leq t(\vec{q}\land\vec{\alpha}/\vec{q},\vec{r}\lor\vec{\beta}/\vec{r})$, which is equivalent to $(\bbA, v)\models s(\vec{q},\vec{r})\leq t(\vec{q},\vec{r})$, as required.

Conversely, assume item 2 and fix a valuation $v$. Clearly, $(\bbA, v)\models \vec{q}\wedge \vec{\alpha}\leq \vec{\alpha}$ and $(\bbA, v)\models \vec{\beta} \leq \vec{r}\vee \vec{\beta}$. Since each $\alpha$ (resp.\ $\beta$) is monotone (resp.\ antitone) in $\vec{r}$ and antitone (resp.\ monotone) in $\vec{q}$, this implies that \[(\bbA, v)\models \overrightarrow{\alpha(\vec{q},\vec{r})}\leq \overrightarrow{\alpha((\vec{q}\wedge \vec{\alpha})/\vec{q},(\vec{r}\vee \vec{\beta})/\vec{r})}\quad\quad (\bbA, v)\models \overrightarrow{\beta((\vec{q}\wedge \vec{\alpha})/\vec{q},(\vec{r}\vee \vec{\beta})/\vec{r})}\leq \overrightarrow{\beta(\vec{q},\vec{r})}, \]
which immediately entail that %for all $q\in\vec{q},r\in\vec{r},\alpha\in\vec{\alpha}$ and $\beta\in\vec{\beta}$,
\[(\bbA, v)\models \vec{q}\wedge \overrightarrow{\alpha(\vec{q},\vec{r})}\leq \overrightarrow{\alpha((\vec{q}\wedge \vec{\alpha})/\vec{q},(\vec{r}\vee \vec{\beta})/\vec{r})}\quad\quad (\bbA, v)\models \overrightarrow{\beta((\vec{q}\wedge \vec{\alpha})/\vec{q},(\vec{r}\vee \vec{\beta})/\vec{r})}\leq \overrightarrow{\beta(\vec{q},\vec{r})}\vee \vec{r}. \]
%$\vec{q}$ are antitone on $\vec{\alpha}$ and monotone on $\vec{\beta}$ and $\vec{r}$ are monotone on $\vec{\alpha}$ and antitone on $\vec{\beta}$. The above observations imply that
%$\val{\alpha}_v\leq\val{\alpha[\vec{q}\land\vec{\alpha}/\vec{q},\vec{r}\lor\vec{\beta}/\vec{r}]}_v$ and $\val{\beta[\vec{q}\land\vec{\alpha}/\vec{q},\vec{r}\lor\vec{\beta}/\vec{r}]}_v\leq\val{\beta}_v$ for all $\alpha\in\vec{\alpha}$ and $\beta\in\vec{\beta}$. Hence $$\val{q}_v\land\val{\alpha}_v\leq\val{\alpha[\vec{q}\land\vec{\alpha}/\vec{q},\vec{r}\lor\vec{\beta}/\vec{r}]}_v$$ and $$\val{\beta[\vec{q}\land\vec{\alpha}/\vec{q},\vec{r}\lor\vec{\beta}/\vec{r}]}_v\leq\val{r}_v\lor\val{\beta}_v$$ for every $q\in\vec{q},r\in\vec{r},\alpha\in\vec{\alpha}$ and $\beta\in\vec{\beta}$.
%
Let $v'$ be the $\vec{q}\oplus\vec{r}$-variant of $v$ such that $\overrightarrow{v'(q)}: = \overrightarrow{v(q)}\land\overrightarrow{\val{\alpha}_v}$ and $\overrightarrow{v'(r)}: = \overrightarrow{v(r)}\lor\overrightarrow{\val{\beta}_v}$. By definition, the conditions above are equivalent to
\[(\bbA, v')\models \vec{q}\leq \overrightarrow{\alpha(\vec{q},\vec{r})}\quad\quad (\bbA, v')\models \overrightarrow{\beta(\vec{q},\vec{r})}\leq \vec{r}. \]
Hence, by assumption 2, we can conclude that $(\bbA, v')\models s\leq t$, which is equivalent to $(\bbA, v)\models s(\vec{q}\land\vec{\alpha}/\vec{q},\vec{r}\lor\vec{\beta}/\vec{r})\leq t(\vec{q}\land\vec{\alpha}/\vec{q},\vec{r}\lor\vec{\beta}/\vec{r})$, as required.
\end{proof}

\begin{prop}\label{prop:rulestoineq}
For any language $\mathcal{L}_\mathrm{DLE}$, every analytic rule in the language of the corresponding calculus $\mathbf{DL}$ is semantically equivalent to some analytic inductive $\mathcal{L}^\ast_\mathrm{DLE}$-inequality.
\end{prop}
\begin{proof}
Modulo application of display postulates, any analytic rule can be represented as follows:
\begin{center}
\begin{tabular}{ccc}
\AxiomC{$(S^i_j\vdash Y^i\mid 1\leq i\leq n\textrm{ and } 1\leq j\leq n_i)\quad (X^k\vdash T_\ell^k\mid 1\leq k\leq m\textrm{ and } 1\leq\ell\leq m_k)$}
\UnaryInfC{$(S\vdash T)[Y^i]^{suc}[X^k]^{pre}$}
\DisplayProof \\
\end{tabular}
\end{center}
 where $Y^i$ and $X^k$ are structural variables and $S^i_j$, $T^k_\ell$, $S$ and $T$ are structural terms.
As discussed in Section \ref{ssec:soundness}, for every perfect $\mathcal{L}_\mathrm{DLE}$-algebra $\bbA$, the validity of the rule above on $\bbA$ is equivalent to the validity on $\bbA$ of the following quasi-inequality
%$$\forall i,j,k,l(\val{l(\Gamma^i_j)}_v\leq\val{p_i'}_v\qquad\val{q_k'}_v\leq\val{r(\Delta^k_l)}_v)\implies \val{l(S)}_v\leq\val{r(T)}_v $$
%for every valuation $v$. This is equivalent to
%$$\forall i,k(\val{\bigvee_j l(\Gamma^i_j)}_v\leq\val{p_i'}_v\qquad\val{q_k'}_v\leq\val{\bigwedge_l r(\Delta^k_l)}_v)\implies \val{l(S)}_v\leq\val{r(T)}_v $$
%for every valuation $v$. That is the rule is equivalent to the quasi-inequality
$$\bigamp_{\begin{smallmatrix}
1\leq i\leq n\\
1\leq k\leq m
\end{smallmatrix}
}\left(\bigvee_{1\leq j\leq n_i} s^i_j\leq p_i'\quad\&\quad q_k'\leq \bigwedge_{1\leq \ell\leq m_k}t^k_\ell\right)\Rightarrow s\leq t$$
where $p_i': = r(Y^i)$, $q_k': = l(X^k)$, for each $i$ and $k$, and $s: = l(S)$, $t: =r(T)$, and $s_j^i: = l(S^i_j)$, and $t^k_\ell: =r(T^k_\ell)$ for each $j$ and $\ell$. Let $s_i: = \bigvee_{1\leq j\leq n_i} s^i_j$ and $t_k: = \bigwedge_{1\leq \ell\leq m_k}t^k_\ell$.

By Lemma \ref{quasitoineq}, the validity of the quasi-inequality above is equivalent to the validity of the following inequality, where $\vec{p'}: = (p_i')_i$, $\vec{s}: = (s_i)_i$, $\vec{q'}: = (q'_k)_k$ and $\vec{t}: = (t_k)_k$:

\begin{equation}
s((\vec{p'}\lor \vec{s})/\vec{p'},(\vec{q'}\land \vec{t})/\vec{q'})\leq t((\vec{p'}\lor \vec{s})/\vec{p'},(\vec{q'}\land \vec{t})/\vec{q'}).
\label{eq:ruletoineq}
\end{equation}
To finish the proof, we need to show that the inequality above is analytic $(\Omega,\epsilon)$-inductive for some $\Omega$ and $\epsilon$.
Let $\vec{p}$, $\vec{q}$ be the variables in the inequality $s\leq t$, different from the variables in $\vec{p'}$ and $\vec{q'}$, and occurring in $s\vdash t$ in antecedent and succedent position respectively (by C$_4$ they are disjoint). Clearly, $s(\vec{p},\vec{q},\vec{p'},\vec{q'})$ is left-primitive, and hence is positive skeleton, and $t(\vec{p},\vec{q},\vec{p'},\vec{q'})$ is right-primitive, and hence is negative skeleton. Condition C$_3$ implies that $X^k$ and $Y^i$ are in antecedent and succedent position respectively in $S\vdash T$, and hence $s$ (resp.\ $t$) is monotone in $\vec{q'}$ (resp.\ in $\vec{p'}$) and antitone in $\vec{p'}$ (resp.\ $\vec{q'}$). %$+q_k'\prec +l(S)$, $+q_k'\prec -r(T)$, $-p_i'\prec +l(S)$ and $-p_i'\prec -r(T)$.
Moreover, $s_i$ is left-primitive, and hence is negative PIA for every $i$, and $t_k$ is right-primitive, and hence is positive PIA for every $k$. These observations immediately yield that every branch in the inequality \eqref{eq:ruletoineq} is good, and in particular, $\vec{s}$ and $\vec{t}$ are the PIA-parts.

 Next, let $\epsilon$ be the order-type which assigns all $p$ in $\vec{p}$ and $q'$ in $\vec{q'}$ to $1$ and all $q$ in $\vec{q}$ and $q$ in $\vec{p'}$ to $\partial$. Let $\Omega$ be the discrete order. To show that the inequality \eqref{eq:ruletoineq} is analytic $(\Omega, \epsilon)$-inductive, it is enough to show that all terms in $\vec{s}$ and $\vec{t}$ are $\epsilon^\partial$-uniform.

%We claim that the inequality \ref{eq:ruletoineq} is analytic $(\Omega,\epsilon)$-inductive.
 Since any $p$ in $\vec{p}$ corresponds to a structural variable antecedent position, $+p\prec + s_i$ and $+p\prec -t_k$ for all $i$ and $k$, hence $-p\prec - s_i$ and $-p\prec +t_k$ for all $i$ and $k$. This shows that $\vec{s}$ and $\vec{t}$ are $\epsilon^\partial$-uniform in any $p$ in $\vec{p}$. Similar arguments relative to the variables in $\vec{q}$, $\vec{p'}$ and $\vec{q'}$ complete the proof. %in $s_i$ and $t_k$ and because -as we noted above- every occurrence of $p_i'$ and $q_k'$ in $s\leq t$, $\vec{s}\leq \vec{p'}$ and $\vec{q'}\leq\vec{t}$ are in negative and positive position respectively, all the PIA subtrees in the inequality \eqref{eq:ruletoineq}, i.e.\ all terms in $\vec{s}$ and $\vec{t}$ are $\epsilon^\partial$-uniform. Hence since every critical occurrence of a variable is directly connected to the skeleton, we immediately get that the inequality is inductive.
\end{proof}

\begin{remark}\label{remark:quasi-special}
	If the rule
	\begin{center}
		\begin{tabular}{ccc}
			\AxiomC{$(S^i_j\vdash Y^i\mid 1\leq i\leq n\textrm{ and } 1\leq j\leq n_i)\quad (X^k\vdash T_\ell^k\mid 1\leq k\leq m\textrm{ and } 1\leq\ell\leq m_k)$}
			\UnaryInfC{$(S\vdash T)[Y^i]^{suc}[X^k]^{pre}$}
			\DisplayProof \\
		\end{tabular}
	\end{center}
	is quasi-special, then, in order to transform it into an analytic inequality as in the proof of the proposition above, we can use Lemma \ref{lem:inverted:ackermann} rather than Lemma \ref{quasitoineq}, which yields the inequality
	
	\begin{equation}
	s(\vec{s}/\vec{p'},\vec{t}/\vec{q'})\leq t(\vec{s}/\vec{p'},\vec{t}/\vec{q'}),
	\label{eq:qsruletoineq}
	\end{equation}
	which is equivalent to \eqref{eq:ruletoineq}. Indeed, all variables in $\vec{p'}$ occur only in positive position and can hence be equivalently replaced by $\bot$ and all variables in $\vec{q'}$ occur only in negative position and can be equivalently replaced by $\top$, yielding
	
	\begin{equation}
	s((\vec{\bot}\lor\vec{s})/\vec{p'},(\vec{\top}\land\vec{t})/\vec{q'})\leq t((\vec{\bot}\lor\vec{s})/\vec{p'},(\vec{\top}\land\vec{t})/\vec{q'}),
	\end{equation}
 which is equivalent to \eqref{eq:qsruletoineq}. We will come back to this observation in the following section.%By stipulating that we always use the more special procedure whenever this is available we see that there is a perfect correspondence between quasi-primitives inequalities and quasi-special rules.

\end{remark}

\section{Special rules are as expressive as analytic rules}
\label{sec:special rules as expressive as analytic}
In \cite{Kracht}, Kracht states without proof that every analytic rule in the display calculus for the classical basic tense logic $Kt$ is equivalent to a special rule (see also the discussion in \cite[Section 5.1]{CiRa14}). A proof of this fact is presented in \cite{CiRa14}, where it is shown that, in classical tense logic, every axiom which is obtained from an analytic rule of the display calculus is equivalent to a primitive axiom. In the present section, we extend this result from classical tense logic to any \textrm{DLE}-logic. Namely, we show, using ALBA, that every analytic inductive inequality in any \textrm{DLE}-language is equivalent to some primitive inequality in the corresponding \textrm{DLE}*-language. We will proceed in two steps: in Section \ref{ssec:quasi-primitive}, we will present an intermediate subclass of analytic inductive inequalities, referred to as quasi-primitive inequalities, and show that any analytic inductive inequality can be equivalently transformed into some quasi-primitive inequality. Then, in Section \ref{ssce:Krachtswhatever}, we will prove that every quasi-primitive inequality is equivalent to some primitive inequality.

These results imply that special structural rules (cf.\ Definition \ref{def:special}) are as expressive
as analytic rules (cf.\ Definition \ref{def:analytic}). Hence, for any language $\mathcal{L}_\mathrm{DLE}$, any properly displayable DLE-logic is specially
displayable. Notice that this fact does not imply that any properly displayable $\mathcal{L}_\mathrm{DLE}$-logic can be axiomatized by means of primitive
$\mathcal{L}_\mathrm{DLE}$-inequalities, since the required primitive inequalities pertain to the language $\mathcal{L}^\ast_\mathrm{DLE}$. However, this fact does
imply that any properly displayable $\mathcal{L}^\ast_\mathrm{DLE}$-logic can be axiomatized by means of primitive
$\mathcal{L}^\ast_\mathrm{DLE}$-inequalities.

\subsection{Quasi-special rules and quasi-special inductive inequalities}
\label{ssec:quasi-primitive}

Let us take stock of what was presented in Sections \ref{ssec:type5} and \ref{sec:rulestoineq}. Taken together, Proposition \ref{prop:rulestoineq} and \ref{prop:type5} immediately imply that every analytic rule is equivalent to a quasi-special rule. Furthermore, any analytic inductive inequality derived from an analytic rule has a special shape: every critical branch consists only of Skeleton nodes, leaving all PIA subtrees to be $\epsilon^\partial$-uniform. This motivates the following definition:
\begin{definition}\label{def:quasianalyticinductive}
 For every analytic $(\Omega,\varepsilon)$-inductive inequality $s\leq t$, if every $\varepsilon$-critical branch of the signed generation trees $+s$ and $-t$ consists solely of skeleton nodes, then $s\leq t$ is a \emph{quasi-special inductive inequality}. Such an inequality is {\em definite} if none of its Skeleton nodes is $+\vee $ or $-\wedge$.
\end{definition}

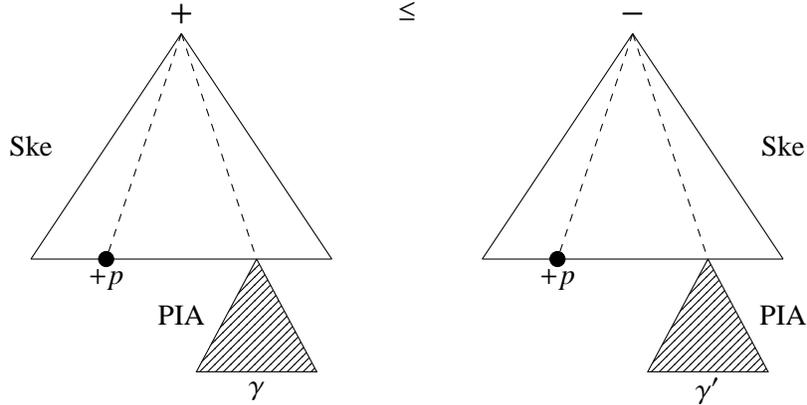
\begin{center}
	\begin{tikzpicture}
	\draw (-5,-1.5) -- (-3,1.5) node[above]{\Large$+$} ;
	\draw (-5,-1.5) -- (-1,-1.5) ;
	\draw (-3,1.5) -- (-1,-1.5);
	\draw (-5,0) node{Ske} ;
	\draw[dashed] (-3,1.5) -- (-4,-1.5);
	\draw[dashed] (-3,1.5) -- (-2,-1.5);
	\draw[fill] (-4,-1.5) circle[radius=.1] node[below]{$+p$};
	\draw
	(-2,-1.5) -- (-2.8,-3) -- (-1.2,-3) -- (-2,-1.5);
	\fill[pattern=north east lines]
	(-2,-1.5) -- (-2.8,-3) -- (-1.2,-3);
	\draw (-2,-3.25)node{$\gamma$};
	\draw (-3,-2.25) node{PIA} ;
	\draw (0,1.8) node{$\leq$};
	\draw (5,-1.5) -- (3,1.5) node[above]{\Large$-$} ;
	\draw (5,-1.5) -- (1,-1.5) ;
	\draw (3,1.5) -- (1,-1.5);
	\draw (5,0) node{Ske} ;
	\draw[dashed] (3,1.5) -- (4,-1.5);
	\draw[dashed] (3,1.5) -- (2,-1.5);
	\draw[fill] (2,-1.5) circle[radius=.1] node[below]{$+p$};
	\draw
	(4,-1.5) -- (4.8,-3) -- (3.2,-3) -- (4, -1.5);
	\fill[pattern=north east lines]
	(4,-1.5) -- (4.8,-3) -- (3.2,-3) -- (4, -1.5);
	\draw (4,-3.25)node{$\gamma'$};
	\draw (5,-2.25) node{PIA} ;
	\end{tikzpicture}
\end{center}

Definite quasi-special inductive inequalities and quasi-special rules entertain the same privileged relation with each other as the one entertained by definite primitive inequalities and special rules. Indeed, translating into an inequality the rule obtained from a definite quasi-special inductive inequality leads to the original inequality (cf.\ Remark \ref{remark:quasi-special}). Notice that these are exactly the inequalities that have this property, since the inequality that is obtained by Proposition \ref{prop:rulestoineq} is always definite quasi-special inductive. Since every analytic inductive inequality is equivalent to a set of analytic rules (in fact quasi-special rules) and every analytic rule is equivalent to a definite quasi-special inductive inequality, is it clear that every analytic inductive inequality is equivalent to a set of definite quasi-special inductive inequalities.

\subsection{Quasi-special inductive inequalities are equivalent to primitive inequalities}
\label{ssce:Krachtswhatever}

The following propositions generalize \cite[Lemma 5.12]{CiRa14}. %shows that every quasi-special inductive inequality is equivalent to a primitive inequality:
%\marginnote{mention that the following two propositions together generalize lemma??? in Revantha TOCL}
\begin{prop}\label{lemma:analytictospecial}
	Let $\xi_1(!\vec{x},!\vec{y},!\vec{z},!\vec{w})$ be a definite positive Skeleton formula and $\xi_2(\vec{x},\vec{y})$ be a positive Skeleton formula such that $+\vec{x},+\vec{z}\prec+\xi_1$, $-\vec{y},-\vec{w}\prec+\xi_1$ and $-\vec{x},+\vec{y}\prec-\xi_2$. Let $\overrightarrow{\gamma(\vec{p},\vec{q})}$ be an array of positive PIA-formulas such that $-\vec{p},+\vec{q}\prec+\gamma$ and let $\overrightarrow{\chi(\vec{p},\vec{q})}$ be an array of negative PIA formulas such that $-\vec{p},+\vec{q}\prec-\chi$. Then the following are equivalent:
	\begin{enumerate}
		\item $\forall\vec{p}\forall \vec{q}[\xi_1(\vec{p}/\vec{x},\vec{q}/\vec{y},\vec{\gamma}/\vec{z},\vec{\chi}/\vec{w})\leq\xi_2(\vec{p}/\vec{x},\vec{q}/\vec{y})]$;
	\item \[\forall\vec{p}\forall\vec{q}\forall\vec{p'}\forall\vec{q'}\left[ \begin{tabular}{r c l}
	$\xi_1(\vec{p}/\vec{x},\vec{q}/\vec{y},\vec{p'}/\vec{z},\vec{q'}/\vec{w})$ & $\!\leq\!$ & $\xi_2(\vec{p}/\vec{x},\vec{q}/\vec{y})\lor$\\
	&&$(\bigvee_{z_k\in\vec{z}}\xi_1(\vec{p}/\vec{x},\vec{q}/\vec{y},\vec{p'}_{-k}/\vec{z}_{-k},p'_k\pdla\gamma_k/z_k,\vec{q'}/\vec{w}))\lor$\\
	& & $(\bigvee_{w_\ell\in\vec{w}}\xi_1(\vec{p}/\vec{x},\vec{q}/\vec{y},\vec{p'}/\vec{z},\vec{q'}_{-\ell}/\vec{w}_{-\ell},\chi_\ell\to q'_\ell/w_\ell))$
	\end{tabular}\right].\]
\end{enumerate}

\end{prop}
\begin{proof}
The inequality in item 1 of the statement can be equivalently transformed via ALBA into the following quasi-inequality:
	
	\begin{equation}\label{eq:alba1}
\forall\nomi\forall\vec{\nomj}\forall\vec{\nomj'}\forall\vec{\cnomn}\forall\vec{\cnomn'}\forall\cnomm[(\overrightarrow{\nomj'}\leq\overrightarrow{\gamma(\vec{\nomj},\vec{\cnomn})}\ \amp\ \overrightarrow{\chi(\vec{\nomj},\vec{\cnomn})}\leq\overrightarrow{\cnomn'}\ \amp\ \xi_2(\vec{\nomj},\vec{\cnomn})\leq\cnomm)\Rightarrow \xi_1(\vec{\nomj},\vec{\cnomn},\vec{\nomj'},\vec{\cnomn'})\leq\cnomm].
	\end{equation}
Likewise, the inequality in item 2 can be equivalently transformed via ALBA into the following quasi-inequality:
	\begin{equation}\label{eq:alba2}\forall\nomi\forall\vec{\nomj}\forall\vec{\nomj'}\forall\vec{\cnomn}\forall\vec{\cnomn'}\forall\cnomm\left[\left(
	\begin{tabular}{c}
	$\bigamp_k\xi_1(\vec{\nomj},\vec{\cnomn},\vec{\nomj'}_{-k},\nomj'_k\pdla\gamma_k(\vec{\nomj},\vec{\cnomn}),\vec{\cnomn'})\leq\cnomm$ \\ $\bigamp_\ell \xi_1(\vec{\nomj},\vec{\cnomn},\vec{\nomj'},\vec{\cnomn'}_{-\ell},\chi_\ell(\vec{\nomj},\vec{\cnomn})\to \cnomn'_\ell)\leq\cnomm$ \\ $\xi_2(\vec{\nomj},\vec{\cnomn})\leq\cnomm$
	\end{tabular}\right)
	\Rightarrow \xi_1(\vec{\nomj},\vec{\cnomn},\vec{\nomj'},\vec{\cnomn'})\leq\cnomm\right].
	\end{equation}
To finish the proof, it is enough to show that conditions \eqref{eq:alba1} and \eqref{eq:alba2} are equivalent. Assume condition \eqref{eq:alba2} and let $\vec{\nomj}$ $\vec{\cnomn}$, $\cnomm$, $\nomj'_k$ and $\cnomn'_\ell$ be such that the following inequalities hold for any $\gamma_k$ in $\overrightarrow{\gamma}$ and $\chi_\ell$ in $\overrightarrow{\chi}$:
\begin{equation}\label{eq:2to1first}
\nomj'_k\leq\gamma_k(\vec{\nomj},\vec{\cnomn})\quad \chi_\ell(\vec{\nomj},\vec{\cnomn})\leq\cnomn'_\ell\quad \xi_2(\vec{\nomj},\vec{\cnomn})\leq\cnomm.
\end{equation} By applying adjunction all inequalities above but the last one become $$\nomj'_k\pdla\gamma_k(\vec{\nomj},\vec{\cnomn})=\bot\quad\top=\chi_\ell(\vec{\nomj},\vec{\cnomn})\to\cnomn'_\ell.$$
	These equalities imply that $$\bigamp_k\xi_1(\vec{\nomj},\vec{\cnomn},\vec{\nomj'}_{-k},\nomj'_k\pdla\gamma_k(\vec{\nomj},\vec{\cnomn}),\vec{\cnomn'})=\bot\quad \bigamp_\ell\xi_1(\vec{\nomj},\vec{\cnomn},\vec{\nomj'},\vec{\cnomn'}_{-\ell},\chi_\ell(\vec{\nomj},\vec{\cnomn})\to \cnomn'_\ell)=\bot.$$ Indeed, by assumption, $\xi_1$ is a definite positive Skeleton formula such that any variable in it occurs at most once. Hence, $\xi_1$ is a definite and scattered left-primitive formula. By Lemma \ref{lemma:primitive scattered is operator}, the term function induced by $\xi_1$ is an operator, and hence $\xi_1$ preserves $\bot$ in its positive coordinates and reverses $\top$ in its negative coordinates. This finishes the proof that all $\vec{\nomj}$ $\vec{\cnomn}$, $\cnomm$, $\nomj'_k$ and $\cnomn'_\ell$ satisfying conditions \eqref{eq:2to1first} satisfy also the premises of the quasi-inequality \eqref{eq:alba2}, namely:
	
	$$\bigamp_k\xi_1(\vec{\nomj},\vec{\cnomn},\vec{\nomj'}_{-k},\nomj'_k\pdla\gamma_k(\vec{\nomj},\vec{\cnomn}),\vec{\cnomn'})\leq\cnomm\quad\bigamp_\ell \xi_1(\vec{\nomj},\vec{\cnomn},\vec{\nomj'},\vec{\cnomn'}_{-\ell},\chi_\ell(\vec{\nomj},\vec{\cnomn})\to \cnomn'_\ell)\leq\cnomm\quad \xi_2(\vec{\nomj},\vec{\cnomn})\leq\cnomm.$$
By assumption \eqref{eq:alba2}, we conclude that $\xi_1(\vec{\nomj},\vec{\cnomn},\vec{\nomj'},\vec{\cnomn'})\leq\cnomm$, as required.
	
Conversely, assume condition \eqref{eq:alba1} and let $\vec{\nomj}$ $\vec{\cnomn}$, $\cnomm$, $\nomj'_k$ and $\cnomn'_\ell$ be such that the following inequalities hold for any $k$ and $\ell$ as above:
\begin{equation}\label{eq:1to2first}
\xi_1(\vec{\nomj},\vec{\cnomn},\vec{\nomj'}_{-k},\nomj'_k\pdla\gamma_k(\vec{\nomj},\vec{\cnomn}),\vec{\cnomn'})\leq\cnomm\quad \xi_1(\vec{\nomj},\vec{\cnomn},\vec{\nomj'},\vec{\cnomn'}_{-\ell},\chi_\ell(\vec{\nomj},\vec{\cnomn})\to \cnomn'_\ell)\leq\cnomm\quad \xi_2(\vec{\nomj},\vec{\cnomn})\leq\cnomm.
\end{equation} By applying the appropriate residuation rules, all but the last inequality above can be equivalently written as follows:
$$\nomj'_k\pdla\gamma_k(\vec{\nomj},\vec{\cnomn})\leq \mathsf{RA}(\xi_1)(\vec{\nomj},\vec{\cnomn},\vec{\nomj'}_{-k},\cnomm/u,\vec{\cnomn'})\quad\mathsf{RA}(\xi_1)(\vec{\nomj},\vec{\cnomn},\vec{\nomj'},\vec{\cnomn'}_{-\ell},\cnomm/u)\leq\chi_\ell(\vec{\nomj},\vec{\cnomn})\to \cnomn'_\ell$$ and by applying residuation once again we obtain for every $k$ and $\ell$:
$$\nomj'_k\leq \gamma_k(\vec{\nomj},\vec{\cnomn})\lor\mathsf{RA}(\xi_1)(\vec{\nomj},\vec{\cnomn},\vec{\nomj'}_{-k},\cnomm/u,\vec{\cnomn'})\quad \mathsf{RA}(\xi_1)(\vec{\nomj},\vec{\cnomn},\vec{\nomj'},\vec{\cnomn'}_{-\ell},\cnomm/u)\land\chi_\ell(\vec{\nomj},\vec{\cnomn})\leq \cnomn'_\ell.$$ Since each $\nomj'_k$ and each $\cnomn'_\ell$ is a nominal and a conominal respectively, they are interpreted as join-prime and meet-prime elements respectively. If $\nomj'_k\leq\gamma_k(\vec{\nomj},\vec{\cnomn})$ and $\chi_\ell(\vec{\nomj},\vec{\cnomn})\leq\cnomn'_\ell$ for all $k$ and $\ell$, then the antecedent of \ref{eq:alba1} is satisfied and hence we conclude $\xi_1(\vec{\nomj},\vec{\cnomn},\vec{\nomj'},\vec{\cnomn'})\leq\cnomm$. If $\nomj'_k\leq\mathsf{RA}(\xi_1)(\vec{\nomj},\vec{\cnomn},\vec{\nomj'}_{-k},\cnomm/u,\vec{\cnomn'})$ or $\mathsf{RA}(\xi_1)(\vec{\nomj},\vec{\cnomn},\vec{\nomj'},\vec{\cnomn'}_{-\ell},\cnomm/u)\leq\cnomn'_\ell$ for some $\nomj'_k$ or $\cnomn'_\ell$, then by applying the appropriate residuation rule we immediately obtain that $\xi_1(\vec{\nomj},\vec{\cnomn},\vec{\nomj'},\vec{\cnomn'})\leq\cnomm$.
\end{proof}
The following proposition is order-dual to the previous one, hence its proof is omitted.
\begin{prop}\label{lemma:analytictospecialdual}
	Let $\xi_1(\vec{x},\vec{y})$ be a negative Skeleton formula and $\xi_2(!\vec{x},!\vec{y},!\vec{z},!\vec{w})$ be a definite negative Skeleton formula such that $-\vec{x},+\vec{y}\prec+\xi_1$, $+\vec{x},+\vec{z}\prec-\xi_2$ and $-\vec{y},-\vec{w}\prec-\xi_2$. Let $\overrightarrow{\gamma(\vec{p},\vec{q})}$ be an array of positive PIA-formulas such that $-\vec{p},+\vec{q}\prec+\gamma$ and let $\overrightarrow{\chi(\vec{p},\vec{q})}$ be an array of negative PIA formulas such that $-\vec{p},+\vec{q}\prec-\chi$. Then the following are equivalent:
	\begin{enumerate}
		\item $\forall\vec{p}\forall \vec{q}[\xi_1(\vec{p}/\vec{x},\vec{q}/\vec{y})\leq\xi_2(\vec{p}/\vec{x},\vec{q}/\vec{y},\vec{\gamma}/\vec{z},\vec{\chi}/\vec{w})]$
		\item \[\forall\vec{p}\forall\vec{q}\forall\vec{p'}\forall\vec{q'}\left[ \begin{tabular}{l c l}
			 $\xi_1(\vec{p}/\vec{x},\vec{q}/\vec{y})\land$ & &\\
			 $(\bigwedge_{z_k\in\vec{z}}\xi_2(\vec{p}/\vec{x},\vec{q}/\vec{y},\vec{p'}_{-k}/\vec{z}_{-k},p'_k\pdla\gamma_k/z_k,\vec{q'}/\vec{w}))\land$ & & \\
			 $(\bigwedge_{w_\ell\in\vec{w}}\xi_2(\vec{p}/\vec{x},\vec{q}/\vec{y},\vec{p'}/\vec{z},\vec{q'}_{-\ell}/\vec{w}_{-\ell},\chi_\ell\to q'_\ell/w_\ell))$ & $\leq$ & $\xi_2(\vec{p}/\vec{x},\vec{q}/\vec{y},\vec{p'}/\vec{z},\vec{q'}/\vec{w})$
		\end{tabular}\right].\]
	\end{enumerate}
	\end{prop}

%Hence the following result is a consequence of Propositions \ref{lemma:analytictospecial} and \ref{lemma:analytictospecialdual}:

\begin{cor}
For any language $\mathcal{L}_\mathrm{DLE}$, every analytic structural rule in the language of the corresponding display calculus $\mathbf{DL}$ %$\mathcal{L}^*_\mathrm{DLE}$
can be equivalently transformed into %a primitive inequality and hence a
some special structural rule in the same language.
\end{cor}
\begin{proof}
As discussed at the beginning of Section \ref{ssec:quasi-primitive}, any analytic structural rule in $\mathbf{DL}$ is equivalent to a definite quasi-special inequality in $\mathcal{L}^*_\mathrm{DLE}$. It is easy to see that every definite quasi-special inequality can be transformed into one inequality of the form of item (1) in Propositions \ref{lemma:analytictospecial} or \ref{lemma:analytictospecialdual}. This transformation is effected by applying suitable residuation rules so as to reduce one side of the given inequality to an $\epsilon^\partial$-uniform PIA subterm (analogously to the treatment of the Type 4 inequalities discussed in Section \ref{ssec:type4}). Hence, either Propositions \ref{lemma:analytictospecial} or \ref{lemma:analytictospecialdual} is applicable, yielding an equivalent inequality as in item 2 of the propositions mentioned above.
Finally, the inequality in item 2 of Proposition \ref{lemma:analytictospecial} (resp.\ \ref{lemma:analytictospecialdual}) is definite left-primitive (resp.\ right-primitive). Hence, the statement follows by Lemma \ref{lemma:pri}. %In particular, if the quasi-primitive inequality is obtained by an analytic rule, after applying the residuation rule, Conditions $C_1$-$C_7$ guarantee that the inequality will be in one of the form of item (1) of Propositions \ref{lemma:analytictospecial} and \ref{lemma:analytictospecialdual}.
\end{proof}
%\marginnote{I'm not really sure that this is the case. After all, the equivalence of conditions \eqref{eq:alba1} and \eqref{eq:alba2} is shown using only
%adjunction and substitution, so these steps are very much ALBA-type. But how to establish a difference between what we are doing in section 6 and what we are
% doing here?}

%The proof of Proposition \ref{lemma:analytictospecial} (resp.\ \ref{lemma:analytictospecialdual}) is not internal to ALBA.
%Indeed, using ALBA we were able to transform the items of the statement into conditions \eqref{eq:alba1} and \eqref{eq:alba2}, and then prove their
%equivalence in the metatheory. It is an open problem whether this step can also be performed internally to the calculus ALBA. This would amount to introducing additional ALBA rules,
%of which Propositions \ref{lemma:analytictospecial} and \ref{lemma:analytictospecialdual} provide the soundness, and then show that these additional rules are
%derivable in ALBA.

\section{Two methodologies: a sketch of a comparison}
\label{sec:comparison}

The generalizations of Kracht's results presented in Sections \ref{sec:primitive special main strategy}--\ref{sec:special rules as expressive as analytic} are
alternative to those proposed in \cite{CiRa13, CiRa14}, and the aim of the present section is connecting and comparing these two generalizations.
%\marginnote{CiRa13: Agata Ciabattoni, Revantha Ramanayake: Structural Extensions of Display Calculi: A General Recipe. WoLLIC 2013: 81-95}
Such a comparison is not straightforward, since the methodologies the two generalizations rely on are different: while the treatment in \cite{CiRa13, CiRa14} relies purely on proof-theoretic notions and is therefore {\em internal} to proof theory, the present is {\em external}, in that is based on a theory (unified correspondence) originating in the model theory of modal logic, developed independently of proof theory, and whose connections with proof theory have not been systematically explored before. As to the basic settings for the two generalizations, the basic setting of the treatment in \cite{CiRa13, CiRa14} is given by the so called {\em amenable calculi} (the definition of which is reported on in Definition \ref{def:amenable calculus} below), which are defined for an arbitrary logical signature by means of conditions concerning the {\em performances} of the calculus (requiring e.g.\ that sequents of certain shapes be derivable) rather than the specific shape of the rules of the calculus. For any logical language, and any amenable calculus $\mathcal{C}$, the class of axioms which is proven to give rise to analytic structural rules is defined parametrically in $\mathcal{C}$, as a certain subcollection of the set $\mathcal{I}_2(\mathcal{C})$ of those ``formulae $A$ whose logical connectives can be eliminated by applying the invertible logical rules [of $\mathcal{C}$] to the premises of those rules obtained by applying some invertible rules to $\mathbf{I}\vdash A$ followed by [the Ackermann lemma]''. The subcollection just mentioned is the one of {\em acyclic} formulas, which is defined again taking $\mathcal{C}$ as a parameter. In the present paper, the basic environment is given by the class of perfect DLEs, which provides the common semantic environment for both the language of ALBA and for display calculi. In this setting, the logical connectives pertaining to the `expansion' of the lattice signature are classified into two sets $\mathcal{F}$ and $\mathcal{G}$, according to the order-theoretic properties of their algebraic interpretations. Hence, any DLE-signature is uniquely determined by the sets of logical connectives/function symbols $\mathcal{F}$ and $\mathcal{G}$, which are taken as parameters of the language $\mathcal{L}_{\mathrm{DLE}} = \mathcal{L}_{\mathrm{DLE}}(\mathcal{F}, \mathcal{G})$. The display calculus $\mathbf{DL}$, the language and rules of the appropriate version of ALBA, and the inductive $\mathcal{L}_{\mathrm{DLE}}$-inequalities are then defined parametrically in $\mathcal{F}$ and $\mathcal{G}$ and are hence unique for each choice of $\mathcal{F}$ and $\mathcal{G}$. In Appendix \ref{appedix:amenable},
we sketch the proof that, for each $\mathcal{F}$ and $\mathcal{G}$, the associated display calculus $\mathbf{DL}$ is amenable, and in Appendix \ref{appedix:analytic and I2}, we show that acyclic inequalities in $\mathcal{I}_2(\mathbf{DL})$ can be identified with analytic inductive inequalities.

Notwithstanding their different and mutually independent starting points, once a concrete setting is defined which provides a common ground for the application of the two methodologies, it is not difficult to recognize striking similarities between the algorithm defined in \cite{CiRa13, CiRa14} for computing analytic structural rules from input analytic inductive inequalities and the ALBA-based procedure illustrated in Section \ref{ssec:type5}. In what follows, we are not giving a formal proof establishing systematic connections between the two procedures, and limit ourselves to illustrating them by means of an example.

\paragraph{Generalized Church-Rosser inequality.} Let $\mathcal{F} = \{\cdot, \Diamond, {\lhd}\}$, $\mathcal{G} = \{\star, \Box, {\rhd}\}$, where $\cdot$ and $\star$ are binary and of order-type $(1, 1)$, $\Diamond$ and $\Box$ are unary and of order-type $(1)$, and $\lhd$ and $\rhd$ are unary and of order-type $(\partial)$. The logical and structural connectives of the display calculi $\mathbf{DL}$ and
$\mathbf{DL}^\ast$ associated with the basic $\mathcal{L}_{\mathrm{DLE}}(\mathcal{F}, \mathcal{G})$-logic can be represented synoptically as follows (we omit the residuals of the binary connectives since they are not relevant to the present discussion):

\begin{comment}
\begin{center}
	\begin{tabular}{|c|c|c|c|c|c|c|c|c|c|c|c|c|c|c|c|c|c|c|c|c|c|c|c|}
		\hline
		\mc{2}{|c|}{I} & \mc{2}{c|}{$;$} & \mc{2}{c|}{$\Scdot$} & \mc{2}{c|}{$\Sstar$}
		& \mc{2}{c|}{$\SDiamond$} & \mc{2}{c|}{$\SBox$} & \mc{2}{c|}{$\Slhd$} & \mc{2}{c|}{$\Srhd$} & \mc{2}{c|}{$\SDiamondblack$} & \mc{2}{c|}{$\Sblacksquare$} & \mc{2}{c|}{$\Sblacktriangleleft$} & \mc{2}{c|}{$\Sblacktriangleright$} \\
		\hline
		$\top$ & $\bot$ & $\pand$ & $\por$ & $\cdot$ & $\phantom{\cdot}$ & $\phantom{\star}$ & $\star$ & $\Diamond$ & $\phantom{\Diamond}$ & $\phantom{\Box}$ & $\Box$ & $\lhd$ & $\phantom{\lhd}$ & $\phantom{\rhd}$ & $\rhd$ & $\Diamondblack$ & $\phantom{\Diamondblack}$ & $\phantom{\blacksquare}$ & $\blacksquare$ & $\blacktriangleleft$ & $\phantom{\blacktriangleleft}$ & $\phantom{\blacktriangleright}$ & $\blacktriangleright$\\
		\hline
	\end{tabular}
\end{center}
\end{comment}

\begin{center}
	\begin{tabular}{|c|c|c|c|c|c|c|c|c|c|c|c|c|c|c|c|c|c|c|c|c|c|c|c|}
		\hline
		\mc{2}{|c|}{I} & \mc{2}{c|}{$;$} & \mc{2}{c|}{\dcircled{.}} & \mc{2}{c|}{\circled{$\star$}}
		& \mc{2}{c|}{\circled{$\Diamond$}} & \mc{2}{c|}{\circled{$\Box$}} & \mc{2}{c|}{\circled{$\lhd$}} & \mc{2}{c|}{\circled{$\rhd$}} & \mc{2}{c|}{\circled{$\Diamondblack$}} & \mc{2}{c|}{\circled{$\blacksquare$}} & \mc{2}{c|}{\circled{$\blacktriangleleft$}} & \mc{2}{c|}{\circled{$\blacktriangleright$}} \\
		\hline
		$\top$ & $\bot$ & $\pand$ & $\por$ & $\cdot$ & $\phantom{\cdot}$ & $\phantom{\star}$ & $\star$ & $\Diamond$ & $\phantom{\Diamond}$ & $\phantom{\Box}$ & $\Box$ & $\lhd$ & $\phantom{\lhd}$ & $\phantom{\rhd}$ & $\rhd$ & $\Diamondblack$ & $\phantom{\Diamondblack}$ & $\phantom{\blacksquare}$ & $\blacksquare$ & $\blacktriangleleft$ & $\phantom{\blacktriangleleft}$ & $\phantom{\blacktriangleright}$ & $\blacktriangleright$\\
		\hline
	\end{tabular}
\end{center}

\noindent Consider the following analytic inductive inequality:
\[\Box p\cdot {\rhd} p\leq \Diamond p\star {\lhd} p. \]
%is not restricted analytic inductive for any order-type: indeed, both for the order-type $(1)$ and for $(\partial)$ there are two $\varepsilon^\partial$-PIA subtrees attached to the skeleton of the inequality.
Let us implement the procedure illustrated in Section \ref{ssec:type5} on the inequality above:
\begin{center}
	\begin{tabular}{r l l}
		& $\forall p[\Box p\cdot {\rhd} p\leq \Diamond p\star {\lhd} p]$\\
		iff & $\forall p\forall\vec{q}[(q_1\leq \Box p\ \&\ q_2\leq {\rhd} p\ \&\ \Diamond p\leq q_3\ \&\ {\lhd} p \leq q_4)\Rightarrow q_1\cdot q_2\leq q_3\star q_4]$\\
		iff & $\forall p\forall\vec{q}[(\Diamondblack q_1\leq p\ \&\ q_2\leq {\rhd} p\ \&\ \Diamond p\leq q_3\ \&\ {\blacktriangleleft}q_4 \leq p)\Rightarrow q_1\cdot q_2\leq q_3\star q_4]$\\
		iff & $\forall p\forall\vec{q}[(\Diamondblack q_1\vee {\blacktriangleleft} q_4\leq p\ \&\ q_2\leq {\rhd} p\ \&\ \Diamond p\leq q_3)\Rightarrow q_1\cdot q_2\leq q_3\star q_4]$\\
		iff & $\forall\vec{q}[(q_2\leq {\rhd} (\Diamondblack q_1\vee {\blacktriangleleft}q_4)\ \&\ \Diamond (\Diamondblack q_1\vee {\blacktriangleleft}q_4)\leq q_3)\Rightarrow q_1\cdot q_2\leq q_3\star q_4]$.\\
	\end{tabular}
\end{center}
The last quasi-inequality above expresses the validity on perfect DLEs of the following quasi-special structural rule:
\begin{equation}\label{eq:rule gen churchrosser alba}
	\frac{Y\vdash {\circled{$\rhd$}} (\circled{$\Diamondblack$} X\ ; \ {\circled{$\blacktriangleleft$}}W)\quad \quad \circled{$\Diamond$} (\circled{$\Diamondblack$} X\ ;\ {\circled{$\blacktriangleleft$}}W)\vdash Z}{X  \dcircled{.} Y\vdash Z\circled{$\star$} W}\end{equation}

\noindent Let us apply the procedure described in \cite{CiRa13, CiRa14} to the calculus $\mathbf{DL}$ and the sequent
\[\Box p\cdot {\rhd} p\vdash \Diamond p\star {\lhd} p.\] \noindent We start by exhaustively applying in reverse all invertible rules of $\mathbf{DL}$ which are applicable to the sequent. These rules are:
\begin{center}
	\begin{tabular}{c c}
		\AxiomC{$A\dcircled{.} B\vdash Z$}
		\UnaryInfC{$A\cdot B\vdash Z$}
		\DisplayProof &
		\AxiomC{$X\vdash A $\circled{$\star$} $B$}
		\RightLabel{.}
		\UnaryInfC{$X\vdash A\star B$}
		\DisplayProof\\
	\end{tabular}
\end{center}
\noindent This yields the following sequent: \[\Box p\dcircled{.} {\rhd} p\vdash \Diamond p\circled{$\star$} {\lhd} p.\]
At this point, the procedure in \cite{CiRa13, CiRa14} calls for the display of the subformulas on which it is not possible to apply invertible rules as a-parts or s-parts of the premises of the rule-to be. The equivalence of the rule below to the sequent above is guaranteed by the Ackermann lemma:
\begin{center}
	\AxiomC{$X\vdash \Box p\quad \quad Y\vdash {\rhd} p\quad\quad \Diamond p\vdash Z\quad\quad {\lhd} p\vdash W$}
	\RightLabel{.}
	\UnaryInfC{$X\dcircled{.} Y\vdash Z$\circled{$\star$} $W$}
	\DisplayProof
\end{center}
\noindent On each of the premises of the rule above, more invertible rules of $\mathbf{DL}$ can be applied in reverse, namely the following ones:
\begin{center}
	\begin{tabular}{c c c c}
		\AxiomC{$X\vdash$ \circled{$\Box$} $A$}
		\UnaryInfC{$X\vdash \Box A$}
		\DisplayProof
		&
		\AxiomC{$X\vdash$\circled{$\rhd$} $A$}
		\UnaryInfC{$X\vdash {\rhd} A$}
		\DisplayProof
		&
		\AxiomC{\circled{$\Diamond$} $A\vdash Y$}
		\UnaryInfC{$\Diamond A \vdash Y$}
		\DisplayProof
		&
		\AxiomC{\circled{$\lhd$} $A\vdash Y$}
		\RightLabel{.}
		\UnaryInfC{${\lhd} A \vdash Y$}
		\DisplayProof
		\\
	\end{tabular}
\end{center}
Applying them exhaustively yields
\begin{center}
	\AxiomC{$X\vdash$ \circled{$\Box$} $p\quad \quad Y\vdash$ \circled{$\rhd$}$ p\quad\quad$ \circled{$\Diamond$}$ p\vdash Z\quad\quad$ \circled{$\lhd$}$ p\vdash W$}
	\RightLabel{.}
	\UnaryInfC{$X\dcircled{.} Y\vdash Z$\circled{$\star$}$W$}
	\DisplayProof
\end{center}
Modulo replacing $p$ with a fresh structural variable $V$, the rule above satisfies conditions C$_2$-C$_7$ but fails to satisfy C$_1$. To transform it into an analytic rule, one needs to first display all occurrences of the variable $p$, by suitably applying the following display postulates:

\begin{center}
	\begin{tabular}{c c c c}
		\AxiomC{$X\vdash$ {\circled{$\Box$}} $Y$}
		\UnaryInfC{{\circled{$\Diamondblack$}} $X\vdash Y$}
		\DisplayProof
		&
		\AxiomC{$X\vdash$ {\circled{$\rhd$}} $Y$}
		\UnaryInfC{$Y\vdash$ {\circled{$\blacktriangleright$}} $X$}
		\DisplayProof
		&
		\AxiomC{{\circled{$\Diamond$}}$ X\vdash Y$}
		\UnaryInfC{$ X \vdash$ {\circled{$\blacksquare$}}$Y$}
		\DisplayProof
		&
		\AxiomC{{\circled{$\lhd$}} $X\vdash Y$}
		\RightLabel{.}
		\UnaryInfC{{\circled{$\blacktriangleleft$}} $Y \vdash X$}
		\DisplayProof
		\\
	\end{tabular}
\end{center}
\noindent This step yields the following rule:
\begin{center}
	\AxiomC{\circled{$\Diamondblack$}$X\vdash p\quad \quad p\vdash$ \circled{$\blacktriangleright$} $Y\quad\quad p\vdash$ \circled{$\blacksquare$} $Z\quad\quad$ {\circled{$\blacktriangleleft$}} $W\vdash p$ }
	\RightLabel{.}
	\UnaryInfC{$X\dcircled{.} Y\vdash Z$\circled{$\star$}$W$}
	\DisplayProof
\end{center}
\noindent Eliminating $p$ by means of all the possible applications of cut on the premises yields:
\begin{center}
	\AxiomC{ \circled{$\Diamondblack$} $X\vdash$ \circled{$\blacktriangleright$} $Y\quad\quad$ \circled{$\Diamondblack$} $X\vdash$ \circled{$\blacksquare$} $Z\quad \quad${\circled{$\blacktriangleleft$}} $W\vdash$ \circled{$\blacktriangleright$} $Y\quad\quad$ {\circled{$\blacktriangleleft$}} $W\vdash$ \circled{$\blacksquare$} $Z$}
	\RightLabel{.}
	\UnaryInfC{$X\dcircled{.} Y\vdash Z$\circled{$\star$} $W$}
	\DisplayProof
\end{center}
The rule above is analytic and both semantically and $\mathbf{DL}$-equivalent to \eqref{eq:rule gen churchrosser alba}. Running the two procedures in parallel shows that they have the same essentials, namely adjunction and Ackermann lemma. Indeed, the cut rules applied on the premises can be assimilated to instances of the Ackermann lemma. Moreover, introduction rules for any given connective are invertible exactly on the side in which the connective is an adjoint/residual. Notice that disjunction (resp.\ conjunction) is no exception since in the distributive environment it is both a left (resp.\ right) adjoint and a right (resp.\ left) residual.

\section{Power and limits of display calculi: Conclusion}
\label{sec:conclusions}
The present work addresses the question of which axiomatic extensions of a basic DLE-logic admit a proper display calculus obtained by modularly adding structural rules to the proper display calculus of the basic logic. Such axiomatic extensions are referred to as {\em properly displayable} (cf.\ Definition \ref{def: properly displayable logic}).
Our starting point was Kracht's paper \cite{Kracht}, which characterizes properly displayable axiomatic extensions of the basic modal/tense logic as those associated with the {\em primitive axioms} of the language of classical tense logic. In the present paper, we extend Kracht's notion of primitive axiom to primitive inequalities, uniformly defined in any DLE-languages, and prove that Kracht's characterization holds up to semantic equivalence. Specifically, we introduce the class of analytic inductive inequalities as a syntactic extension of primitive inequalities. We show that each analytic inductive inequality can be effectively translated via ALBA into (a set of) analytic rules. In fact, in Section \ref{sec:analytic}, we show that each analytic inductive inequality can be transformed into an analytic rule which is {\em quasi-special} (cf.\ Definition \ref{def:quasispecial}). Moreover, in Section \ref{sec:special rules as expressive as analytic}, we characterize the subclass of analytic inductive inequalities which exactly corresponds to quasi-special rules (cf.\ Definition \ref{def:quasianalyticinductive}), and show that each such inequality is in fact frame-equivalent to a primitive inequality. These results, taken together, characterize up to semantic equivalence the properly displayable axiomatic extensions of any basic DLE-logic as as those associated with the {\em primitive inequalities} of its associated DLE$^\ast$-language.

\begin{center}
\begin{tikzpicture}

\draw (0, 4.5)node {\large{\underline{Rules}}};
\draw (5, 4.5)node {\large{\underline{Inequalities}}};

\draw (0, 3)node {\large{Analytic}};
\draw (5.5, 3)node {\large{Analytic Inductive}};
\draw[very thick, black!60!white, ->] (1, 3) -- node [above, sloped] {Prop \ref{prop:rulestoineq}} (3.5, 3); % (3.5, 1.8);
\draw[very thick, black!60!white, <-] (1, 1.8) -- node [above, sloped] {Prop \ref{prop:type5}} (3.5, 2.8);

\draw (-0.5, 1.5)node {\large{Quasi-Special}};
\draw (6, 1.5)node {\large{Quasi-Special Inductive}};
\draw[very thick, black!60!white, ->] (1, 1.6) -- node [above, sloped] {Section \ref{ssec:quasi-primitive}} (3.5, 1.6);
\draw[very thick, black!60!white, <-] (1, 1.4) -- node [below, sloped] {Section \ref{ssec:quasi-primitive}}(3.5, 1.4);
\draw[very thick, black!60!white, ->] (4.8, 1.2) -- node [above, sloped] {Prop \ref{lemma:analytictospecial}} (4.8, 0.2);

\draw (0, 0)node {\large{Special}};
\draw (4.8, 0)node {\large{Primitive}};
\draw[very thick, black!60!white, ->] (1, 0.1) -- node [above, sloped] {Lemma \ref{lemma:spetopri}} (3.5, 0.1);
\draw[very thick, black!60!white, <-] (1, -0.1) -- node [below, sloped] {Lemma \ref{lemma:pri}}(3.5, -0.1);

\end{tikzpicture}
\end{center}

%%%%%%%%%%%%%%%%%%%%%%%%

\bibliographystyle{abbrv}
\bibliography{reference}

%%%%%%%%%%%%%%%%%%%%%%%%
\newpage

\appendix
\section{Cut elimination for the display calculi $\mathbf{DL}$ and $\mathbf{DL}^\ast$}
\label{appendix:cut elim}
The present appendix focuses on the proof that the calculi $\mathbf{DL}$ and $\mathbf{DL}^\ast$ defined in Section \ref{sec:display calculi DL DL ast}.
\begin{fact}\label{app:c8}
The display calculi $\mathbf{DL}$ and $\mathbf{DL}^\ast$ verify condition C$_8$ (cf.\ Section \ref{PS:para:CanonicalCutElimination}).
\end{fact}

The reduction step for axioms goes as usual:
\begin{center}
{\footnotesize{
\bottomAlignProof
\begin{tabular}{lcr}
\AX$p \fCenter p$
\AX$p \fCenter p$
\BI$p \fCenter p$
\DisplayProof

 & $\rightsquigarrow$ &

\bottomAlignProof
\AX$p \fCenter p$
\DisplayProof
 \\
\end{tabular}
}}
\end{center}

\noindent Now we treat the introductions of the connectives of the propositional base (we also treat here the cases relative to the two additional arrows $\leftarrow$ and $\pdra$ added to our presentation):

\begin{center}
{\scriptsize{
\bottomAlignProof
\begin{tabular}{lcr}
\AX$\textrm{I}\ \fCenter\ \top$
\AXC{\ \ $\vdots$ \raisebox{1mm}{$\pi$}}
\noLine
\UI$\textrm{I}\ \fCenter\ X$
\UI$\top\ \fCenter\ X$
\BI$\textrm{I}\ \fCenter\ X$
\DisplayProof
 & $\rightsquigarrow$ &
\bottomAlignProof
\AXC{\ \ $\vdots$ \raisebox{1mm}{$\pi$}}
\noLine
\UI$\textrm{I}\ \fCenter\ X$
\DisplayProof
 \\
\end{tabular}
}}
\end{center}

\begin{center}
{\scriptsize{
\bottomAlignProof
\begin{tabular}{lcr}
\AXC{\ \ $\vdots$ \raisebox{1mm}{$\pi$}}
\noLine
\UI$ X \ \fCenter\ \textrm{I}$
\UI$ X \fCenter\ \bot$
\AX$\bot \fCenter\ \textrm{I}$
\BI$ X\ \fCenter\ \textrm{I}$
\DisplayProof
 & $\rightsquigarrow$ &
\bottomAlignProof
\AXC{\ \ $\vdots$ \raisebox{1mm}{$\pi$}}
\noLine
\UI$X \fCenter\ \textrm{I}$
\DisplayProof
 \\
\end{tabular}
}}
\end{center}

\begin{center}
{\scriptsize{
\bottomAlignProof
\begin{tabular}{lcr}
\AXC{\ \ \ $\vdots$ \raisebox{1mm}{$\pi_1$}}
\noLine
\UI$X\ \fCenter\ A$
\AXC{\ \ \ $\vdots$ \raisebox{1mm}{$\pi_2$}}
\noLine
\UI$Y\ \fCenter\ B$
\BI$X\,; Y\ \fCenter\ A\wedge B$
\AXC{\ \ \ $\vdots$ \raisebox{1mm}{$\pi_3$}}
\noLine
\UI$A\,; B\ \fCenter\ Z$
\UI$A\wedge B\ \fCenter\ Z$
\BI$X\,; Y\ \fCenter\ Z$
\DisplayProof
 & $\rightsquigarrow$ &
\bottomAlignProof
\AXC{\ \ \ $\vdots$ \raisebox{1mm}{$\pi_2$}}
\noLine
\UI$Y\ \fCenter\ B$
\AXC{\ \ \ $\vdots$ \raisebox{1mm}{$\pi_1$}}
\noLine
\UI$X\ \fCenter\ A$
\AXC{\ \ \ $\vdots$ \raisebox{1mm}{$\pi_3$}}
\noLine
\UI$A\,; B\ \fCenter\ Z$
\UI$B\,; A\ \fCenter\ Z$
\UI$A\ \fCenter\ B > Z$
\BI$X\ \fCenter\ B > Z$
\UI$B\,; X\ \fCenter\ Z$
\UI$X\,; B\ \fCenter\ Z$
\UI$B\ \fCenter\ X > Z$
\BI$Y\ \fCenter\ X > Z$
\UI$X\,; Y\ \fCenter\ Z$
\DisplayProof
 \\
\end{tabular}
}}
\end{center}

%%%%%%%%%%%%%%%%%%%%%%%%%%%%%%%%%%%%%%%%%%%%%%%%%%%%%%%

\begin{center}
{\scriptsize{
\bottomAlignProof
\begin{tabular}{@{}ccc@{}}
\AXC{\ \ \ $\vdots$ \raisebox{1mm}{$\pi_3$}}
\noLine
\UI$Z\ \fCenter\ B\,; A$
\UI$Z\ \fCenter\ B\vee A $
\AXC{\ \ \ $\vdots$ \raisebox{1mm}{$\pi_1$}}
\noLine
\UI$B\ \fCenter\ Y$
\AXC{\ \ \ $\vdots$ \raisebox{1mm}{$\pi_2$}}
\noLine
\UI$A\ \fCenter\ X$
\BI$B\vee A\ \fCenter\ Y\,; X$
\BI$Z\ \fCenter\ Y\,; X$
\DisplayProof
& $\rightsquigarrow$ &
\bottomAlignProof
\AXC{\ \ \ $\vdots$ \raisebox{1mm}{$\pi_3$}}
\noLine
\UI$Z\ \fCenter\ B\,; A$
\UI$Z\ \fCenter\ A\,; B$
\UI$A > Z\ \fCenter\ B$
\AXC{\ \ \ $\vdots$ \raisebox{1mm}{$\pi_1$}}
\noLine
\UI$B\ \fCenter\ Y$
\BI$A > Z\ \fCenter\ Y$
\UI$Z\ \fCenter\ A\,; Y$
\UI$Z\ \fCenter\ Y\,; A$
\UI$Y > Z\ \fCenter\ A$
\AXC{\ \ \ $\vdots$ \raisebox{1mm}{$\pi_2$}}
\noLine
\UI$A\ \fCenter\ X$
\BI$Y > Z\ \fCenter\ X$
\UI$Z\ \fCenter\ Y\,; X$
\DisplayProof
 \\
\end{tabular}
}}
\end{center}
%%%%%%%%%%%%%%%%%%%%%%%%%%%%%%%%%%%%%%%%%%%%%%%%%%%%%%%%%%%%%%%%%%%%%%%%%%%%%%%%%%%%%%%%%%%
%%%%%%%%%%%%%%%%%%%%%%%%%%%%%%%%%%%%%%%%%%%%%%%%%%%%%%%%%%%%%%%%%%%%%%%%%%%%%%%%%%%%%%%%%%%%

\begin{center}
{\scriptsize{
\bottomAlignProof
\begin{tabular}{@{}ccc@{}}
\AXC{\ \ \ $\vdots$ \raisebox{1mm}{$\pi_1$}}
\noLine
\UI$Y\ \fCenter\ A > B$
\UI$Y\ \fCenter\ A\rightarrow B$
\AXC{\ \ \ $\vdots$ \raisebox{1mm}{$\pi_2$}}
\noLine
\UI$X\ \fCenter\ A$
\AXC{\ \ \ $\vdots$ \raisebox{1mm}{$\pi_3$}}
\noLine
\UI$B\ \fCenter\ Z$
\BI$A\rightarrow B\ \fCenter\ X > Z$
\BI$Y\ \fCenter\ X > Z$
\DisplayProof
& $\rightsquigarrow$&
\bottomAlignProof
\AXC{\ \ \ $\vdots$ \raisebox{1mm}{$\pi_2$}}
\noLine
\UI$X\ \fCenter\ A$
\AXC{\ \ \ $\vdots$ \raisebox{1mm}{$\pi_1$}}
\noLine
\UI$Y\ \fCenter\ A > B$
\UI$A\,; Y\ \fCenter\ B$
\AXC{\ \ \ $\vdots$ \raisebox{1mm}{$\pi_3$}}
\noLine
\UI$B\ \fCenter\ Z$
\BI$A\,; Y\ \fCenter\ Z$
\UI$Y\,; A\ \fCenter\ Z$
\UI$A\ \fCenter\ Y > Z$
\BI$X\ \fCenter\ Y > Z$
\UI$Y\,; X\ \fCenter\ Z$
\UI$X\,; Y\ \fCenter\ Z$
\UI$Y\ \fCenter\ X > Z$
\DisplayProof
 \\
\end{tabular}
}}
\end{center}

%%%%%%%%%%%%%%%%%%%%%%%%%%%%%%%%%%%%%%%%%%%%%%%%%%%%%%%

\begin{center}
{\scriptsize{
\bottomAlignProof
\begin{tabular}{@{}ccc@{}}
\AXC{\ \ \ $\vdots$ \raisebox{1mm}{$\pi_1$}}
\noLine
\UI$Y\ \fCenter\ B < A$
\UI$Y\ \fCenter\ B\leftarrow A$
\AXC{\ \ \ $\vdots$ \raisebox{1mm}{$\pi_2$}}
\noLine
\UI$B\ \fCenter\ Z$
\AXC{\ \ \ $\vdots$ \raisebox{1mm}{$\pi_3$}}
\noLine
\UI$X\ \fCenter\ A$
\BI$B\leftarrow A\ \fCenter\ Z < X$
\BI$Y\ \fCenter\ Z< X$
\DisplayProof
&$\rightsquigarrow$ &
\bottomAlignProof
\AXC{\ \ \ $\vdots$ \raisebox{1mm}{$\pi_2$}}
\noLine
\UI$X\ \fCenter\ A$
\AXC{\ \ \ $\vdots$ \raisebox{1mm}{$\pi_1$}}
\noLine
\UI$Y\ \fCenter\ B< A$
\UI$ Y; A\ \fCenter\ B$
\AXC{\ \ \ $\vdots$ \raisebox{1mm}{$\pi_3$}}
\noLine
\UI$B\ \fCenter\ Z$
\BI$ Y; A\ \fCenter\ Z$
\UI$ A; Y\ \fCenter\ Z$
\UI$A\ \fCenter\ Z < Y$
\BI$X\ \fCenter\ Z <Y$
\UI$ X; Y\ \fCenter\ Z$
\UI$ Y; X\ \fCenter\ Z$
\UI$Y\ \fCenter\ Z < X$
\DisplayProof
 \\
\end{tabular}
}}
\end{center}

%%%%%%%%%%%%%%%%%%%%%%%%%%

\begin{center}
{\scriptsize{
\bottomAlignProof
\begin{tabular}{@{}ccc@{}}
\AXC{\ \ \ $\vdots$ \raisebox{1mm}{$\pi_2$}}
\noLine
\UI$A\ \fCenter\ Y$
\AXC{\ \ \ $\vdots$ \raisebox{1mm}{$\pi_3$}}
\noLine
\UI$Z\ \fCenter\ B$
\BI$Y > Z\ \fCenter\ A \pdra B$
\AXC{\ \ \ $\vdots$ \raisebox{1mm}{$\pi_1$}}
\noLine
\UI$A > B\ \fCenter\ X$
\UI$A \pdra B\ \fCenter\ X$
\BI$Y > Z\ \fCenter\ X$
\DisplayProof
&$\rightsquigarrow$&
\bottomAlignProof
\AXC{\ \ \ $\vdots$ \raisebox{1mm}{$\pi_3$}}
\noLine
\UI$Z\ \fCenter\ B$
\AXC{\ \ \ $\vdots$ \raisebox{1mm}{$\pi_1$}}
\noLine
\UI$A > B\ \fCenter\ X$
\UI$B\ \fCenter\ A\,; X$
\BI$Z\ \fCenter\ A\,; X$
\UI$Z\ \fCenter\ X\,; A$
\UI$X > Z\ \fCenter\ A$
\AXC{\ \ \ $\vdots$ \raisebox{1mm}{$\pi_2$}}
\noLine
\UI$A\ \fCenter\ Y$
\BI$X > Z\ \fCenter\ Y$
\UI$Z\ \fCenter\ X\,; Y$
\UI$Z\ \fCenter\ Y\,; X$
\UI$Y > Z\ \fCenter\ X$
\DisplayProof
 \\
\end{tabular}
}}
\end{center}

\begin{center}
{\scriptsize{
\bottomAlignProof
\begin{tabular}{@{}ccc@{}}
\AXC{\ \ \ $\vdots$ \raisebox{1mm}{$\pi_2$}}
\noLine
\UI$Y\ \fCenter\ B$
\AXC{\ \ \ $\vdots$ \raisebox{1mm}{$\pi_3$}}
\noLine
\UI$A\ \fCenter\ Z$
\BI$Y < Z\ \fCenter\ B \pdla A$
\AXC{\ \ \ $\vdots$ \raisebox{1mm}{$\pi_1$}}
\noLine
\UI$B < A \ \fCenter\ X$
\UI$B \pdla A\ \fCenter\ X$
\BI$Y < Z\ \fCenter\ X$
\DisplayProof
 &$\rightsquigarrow$&
\bottomAlignProof
\AXC{\ \ \ $\vdots$ \raisebox{1mm}{$\pi_3$}}
\noLine
\UI$Y\ \fCenter\ B$
\AXC{\ \ \ $\vdots$ \raisebox{1mm}{$\pi_1$}}
\noLine
\UI$B < A \ \fCenter\ X$
\UI$B\ \fCenter\ X; A$
\BI$Y\ \fCenter\ X; A$
\UI$Y\ \fCenter\ A; X$
\UI$ Y < X\ \fCenter\ A$
\AXC{\ \ \ $\vdots$ \raisebox{1mm}{$\pi_2$}}
\noLine
\UI$A\ \fCenter\ Z$
\BI$Y < X\ \fCenter\ Z$
\UI$Y\ \fCenter\ Z\,; X$
\UI$Y\ \fCenter\ X\,; Z$
\UI$Y < Z\ \fCenter\ Y$
\DisplayProof
 \\
\end{tabular}
}}
\end{center}
%\marginnote{From here down it needs to be edited with $f$ and $g$ of arbitrary arity $n_f$ and $n_g$}

\begin{center}
{\scriptsize{
\bottomAlignProof
\begin{tabular}{@{}ccc@{}}
\AXC{\ \ \ $\vdots$ \raisebox{1mm}{$\pi$}}
\noLine
\UnaryInfC{$Y \fCenter K (\vec A_I, \vec A_J)$}
\UnaryInfC{$Y \fCenter g(\vec A_I, \vec A_J)$}

\AXC{\ \ \ $\vdots$ \raisebox{1mm}{$\pi_i$}}
\noLine
\UnaryInfC{$A_i \fCenter X_i$}

\AXC{$$}
\noLine
\UnaryInfC{$\cdots$}

\AXC{\ \ \ $\vdots$ \raisebox{1mm}{$\pi_j$}}
\noLine
\UnaryInfC{$X_j \fCenter A_j$}

\TI$g(\vec A_I, \vec A_J) \fCenter K (\vec X_I, \vec X_J)$
\BI$Y \fCenter K (\vec X_I, \vec X_J)$
\DisplayProof

& $\rightsquigarrow$&

\bottomAlignProof
\AXC{\ \ \ $\vdots$ \raisebox{1mm}{$\pi_j$}}
\noLine
\UI$X_j \fCenter A_j$

\AXC{\ \ \ $\vdots$ \raisebox{1mm}{$\pi$}}
\noLine
\UI$Y \fCenter K(\vec A_I, \vec A_J)$
\UI$K_i (\vec A_I[Y/A_i], \vec A_J) \fCenter A_{i \in I}$
\AXC{\ \ \ $\vdots$ \raisebox{1mm}{$\pi_i$}}
\noLine

\UI$A_i \fCenter X_i$
\BI$K_i (\vec A_I[Y/A_i], \vec A_J) \fCenter X_i$
\UI$Y \fCenter K (\vec A_I[X_i/A_i], \vec A_J)$

\noLine
\UnaryInfC{\ \ \ $\vdots$ \raisebox{1mm}{\phantom{$\pi_1$}}}
\noLine

\UI$Y \fCenter K (\vec X_I, \vec A_J)$
\UI$A_{j \in J} \fCenter K_j (\vec X_I, \vec A_J[Y/A_j])$

\BI$X_j \fCenter K_j (\vec X_I, \vec A_J[Y/A_j])$
\UI$Y \fCenter K (\vec X_I, \vec A_J[X_j/A_j])$

\noLine
\UnaryInfC{\ \ \ $\vdots$ \raisebox{1mm}{\phantom{$\pi_1$}}}
\noLine

\UI$Y \fCenter K (\vec X_I, \vec X_J)$
\DisplayProof
 \\
\end{tabular}
}}
\end{center}

\bigskip

\begin{center}
{\scriptsize{
\bottomAlignProof
\begin{tabular}{@{}ccc@{}}
\AXC{\ \ \ $\vdots$ \raisebox{1mm}{$\pi_i$}}
\noLine
\UnaryInfC{$X_i \fCenter A_i$}

\AXC{$$}
\noLine
\UnaryInfC{$\cdots$}

\AXC{\ \ \ $\vdots$ \raisebox{1mm}{$\pi_j$}}
\noLine
\UnaryInfC{$A_j \fCenter X_j$}

\TI$H (\vec X_I, \vec X_J) \fCenter f(\vec A_I, \vec A_J)$

\AXC{\ \ \ $\vdots$ \raisebox{1mm}{$\pi$}}
\noLine
\UnaryInfC{$H (\vec A_I, \vec A_J) \fCenter Y$}
\UnaryInfC{$f(\vec A_I, \vec A_J) \fCenter Y$}

\BI$H (\vec X_I, \vec X_J) \fCenter Y$
\DisplayProof

& $\rightsquigarrow$&

\bottomAlignProof
\AXC{\ \ \ $\vdots$ \raisebox{1mm}{$\pi_i$}}
\noLine
\UI$X_i \fCenter A_i$

\AXC{\ \ \ $\vdots$ \raisebox{1mm}{$\pi$}}
\noLine
\UI$H (\vec A_I, \vec A_J) \fCenter Y$
\UI$A_{i \in I} \fCenter H_i (\vec A_I[Y/A_i], \vec A_J)$
\BI$X_i \fCenter H_i (\vec A_I[Y/A_i], \vec A_J)$
\UI$H (\vec A_I[X_i/A_i], \vec A_J) \fCenter Y$

\noLine
\UnaryInfC{\ \ \ $\vdots$ \raisebox{1mm}{\phantom{$\pi_1$}}}
\noLine

\UI$H (\vec X_I, \vec A_J) \fCenter Y$
\UI$H_j (\vec X_I, \vec A_J[Y/A_j]) \fCenter A_{j \in J}$

\AXC{\ \ \ $\vdots$ \raisebox{1mm}{$\pi_j$}}
\noLine
\UI$A_j \fCenter X_j$

\BI$H_j (\vec X_I, \vec A_J[Y/A_j]) \fCenter X_j$
\UI$H (\vec X_I, \vec A_J[X_j/A_j]) \fCenter Y$

\noLine
\UnaryInfC{\ \ \ $\vdots$ \raisebox{1mm}{\phantom{$\pi_1$}}}
\noLine

\UI$H (\vec X_I, \vec X_J) \fCenter Y$
\DisplayProof
 \\
\end{tabular}
}}
\end{center}

\section{Invertible rules of $\mathbf{DL}$}
\label{appendix:invertible}
The present appendix characterizes the invertible rules of the calculi $\mathbf{DL}$ defined in Section \ref{sec:display calculi DL DL ast}.
Throughout the present section, fix a language $\mathcal{L}_\mathrm{DLE} = \mathcal{L}_\mathrm{DLE}(\mathcal{F}, \mathcal{G})$,
and let $f\in \mathcal{F}$ and $g\in \mathcal{G}$.

Notice that the following rules are derivable in $\mathbf{DL}$:

\begin{center}
\begin{tabular}{c c c}
\bottomAlignProof
	\AX$A \fCenter X$
	
	\AX$B\fCenter X$
\LeftLabel{$\lor_{L'}$}
	\BI$A\vee B\fCenter X$
	\DisplayProof

& $\qquad$ &
\bottomAlignProof
	\AX$X \fCenter A$
	\AX$X\fCenter B$
\RightLabel{$\land_{R'}$}
	\BI$X\fCenter A\land B$
	\DisplayProof
\\
\end{tabular}
\end{center}
Hence, for the sake of the comparison of the two settings, we can add them to $\mathbf{DL}$ as primitive rules.
\begin{lemma}
The rules $\land_L$, $\land_{R'}$, $\lor_R$, $\lor_{L'}$, $f_L$, $g_R$ are invertible.
%\marginnote{I put here the R' and L' rules, for the rules that show that $\land$ and $\lor$ are $\Delta$-adjoints. We should state in the comparison section that these rules are added to the system to facilitate the comparison or sth}	
\end{lemma}
\begin{proof}
We only show the cases of $f_L$ and $\lor_{L'}$, the remaining cases being similar.
Assume that $f(A_1,\ldots,A_{n_f})\vdash X$. Then we can derive the premise of $f_L$ via the following derivation:

\begin{center}
\begin{tabular}{ccc}
\AXC{$A_1\vdash A_1\quad \ldots\quad\ A_{n_f}\vdash A_{n_f}$}
\UI$H(A_1,\ldots,A_{n_f})\fCenter f(A_1,\ldots,A_{n_f})$
\AX$f(A_1,\ldots,A_{n_f})\fCenter X$
\BI$H(A_1,\ldots,A_{n_f})\fCenter X.$
\DisplayProof
\end{tabular}
\end{center}
Assume $A\lor B\vdash X$. Then we can derive the premises of $\lor_{L'}$ via the following derivation:

\begin{center}
\begin{tabular}{ccc}
	\bottomAlignProof
	\AX$A \fCenter A$
	\UI$A\fCenter A;B$
	\UI$A\fCenter A\lor B$
	\AX$A\lor B\fCenter X$
	\BI$A\fCenter X.$
	\DisplayProof

& $\qquad$ &
\bottomAlignProof
	\AX$B \fCenter B$
	\UI$B\fCenter B;A$
 \UI$B\fCenter A;B$
	\UI$A\fCenter A\lor B$
	\AX$A\lor B\fCenter X$
	\BI$B\fCenter X.$
	\DisplayProof
\\

\end{tabular}
\end{center}
\end{proof}

\begin{lemma}
The rules $\land_R$, $\lor_L$, $f_R$, $g_L$ are not invertible.	
\end{lemma}
\begin{proof}
Notice that for a rule to be invertible, for each instance of the rule it must be the case that the logical interpretation of each premise is valid in the class of models for $\mathbf{DL}$ in which the corresponding conclusion is valid. Hence to disprove the invertibility of a rule it is enough to find a instance of the rule for which there exists a model satisfying the conclusion but not the premises.
We only show this for $f_R$, and $\land_R$, the remaining cases being similar.
To show that $f_R$ is not invertible, consider the conclusion $H(p_1,\ldots,p_{n_f})\vdash f(q_1,\ldots,q_{n_f})$. Let $\mathbb{A}$ be any Heyting algebra with two incomparable elements $b$ $c$, and let $f^\mathbb{A}$ be the $n$-ary operation such that $f(\vec{a})=\bot$ for all $\vec{a}\in\mathbb{A}^n$ (notice that this operation is both join-preserving and meet-reversing in each coordinate). Then by letting $v(p_1)=b$ and $v(q_1)=c$, we have that $\bot\leq\bot$ but $b\nleq c$.

As for $\land_R$, notice preliminarily that the following instance of the conclusion is derivable:

\begin{center}
	\bottomAlignProof
	\AX$A \fCenter A$
	\AX$B \fCenter B$
	\BI$A;B\fCenter A\land B$
	\UI$B;A\fCenter A\land B$
	\DisplayProof
\end{center}

 Suppose for contradiction that $\land_R$ was invertible. Then, from $B;A\vdash A\land B$ we would be able to derive both $A\vdash B$ and $B\vdash A$. But since $B;A\vdash A\land B$ is derivable in $\mathbf{DL}$, this would imply that we can also derive $A\vdash B$ for any $A$ and $B$, which contradicts the soundness of the calculus.
\end{proof}

\section{The display calculi $\mathbf{DL}$ are amenable}
\label{appedix:amenable}
The present appendix sketches the proof that the calculi $\mathbf{DL}$ defined in Section \ref{sec:display calculi DL DL ast} are amenable.
\begin{definition}[Amenable calculus, cf.\ \cite{CiRa14}, Definition 3.1]
\label{def:amenable calculus}
Let $\mathcal{C}$ be a display calculus containing an a-structure constant $\mathbf{I}$ and an s-structure constant $\mathbf{I}'$ and satisfying C1-C8. Let $\mathfrak{S}(a)$ and $\mathfrak{S}(s)$ denote the class of a- and s-structures of $\mathcal{C}$, and let $\mathcal{L}$ be the language of $L_{\mathbf{I}}(\mathcal{C})$. A display calculus satisfying the following conditions is said to be \emph{amenable}.

\begin{enumerate}

\item (interpretation functions) There are functions $l:\mathfrak{S}\mapsto For\mathcal{L}$ and $r:\mathfrak{S}\mapsto For\mathcal{L}$ such that $l(A)=A=r(A)$ for $A\in For\mathcal{L}$, and for arbitrary $X\in\mathfrak{S}(a)$ and $Y\in\mathfrak{S}(s)$:
\begin{enumerate}
\item $X\vdash l(X)$ and $Y\vdash l(Y)$ are derivable in $\mathcal{C}$.
\item if $X\vdash Y$ is derivable in $\mathcal{C}$ then so is $l(X)\vdash r(Y)$.
\end{enumerate}

\item (logical constants) There are logical constants $c_a, c_s\in For(\mathcal{L})$ such that the following sequents are derivable for arbitrary $X\in\mathfrak{S}(a)$ and $Y\in\mathfrak{S}(s)$:
\begin{center}
$c_a\vdash Y\ \ \ \ \ \ \ \ \ \ \ \ \ \ \ \ \ X\vdash c_s$
\end{center}

\item (logical connectives) There are binary connectives $\land, \lor\in\mathcal{L}$ such that the following sequents are derivable for $\star\in\{\lor, \land\}$:
\begin{enumerate}
\item commutativity: $A\star B\vdash B\star A$
\item associativity: $A\star (B\star C)\vdash (A\star B)\star C$ and $(A\star B)\star C\vdash A\star (B\star C)$
\end{enumerate}
Also, for $A,B\in For\mathcal{L}$, $X\in\mathfrak{S}(a)$ and $Y\in\mathfrak{S}(s)$:
\begin{itemize}
\item[(a)$_{\lor}$] $A\vdash Y$ and $B\vdash Y$ implies $\lor(A, B)\vdash Y$
\item[(b)$_{\lor}$] $X\vdash A$ implies $X\vdash \lor(A,B)$ for any formula $B$.
\item[(a)$_{\land}$] $X\vdash A$ and $X\vdash B$ implies $X\vdash\land(A, B)$
\item[(b)$_{\land}$] $A\vdash Y$ implies $\land(A,B)\vdash Y$ for any formula $B$.
\end{itemize}
\end{enumerate}
\end{definition}

\begin{fact}
For any $\mathcal{L}_\mathrm{DLE}$-language, the corresponding calculus $\mathbf{DL}$ is amenable.
\end{fact}
\begin{proof}
The interpretation functions $l$ and $r$ are those defined in Definition \ref{def:rsls}. The constants are $c_a: = \top$ and $c_s: = \bot$. Finally, the derivations
requested by item 3 are straightforward and omitted.
\end{proof}

\section{Analytic inductive inequalities and acyclic $\mathcal{I}_2(\mathbf{DL})$-inequalities}
\label{appedix:analytic and I2}

The following definitions are slight modifications of Definitions 3.7--3.9 in \cite{CiRa14}. The modifications essentially amount to specializing the original inequalities from an arbitrary display calculus $\mathcal{C}$ to $\mathbf{DL}$.

\begin{definition}
For any sequent $X\vdash Y$ in the language of $\mathbf{DL}$, let $inv(X\vdash Y)$ denote the collection of sets of sequents obtained by applying sequences of display postulates and invertible logical rules in $\mathbf{DL}$ (cf.\ Appendix \ref{appendix:invertible}) to it.
\end{definition}
\begin{definition}
An $\mathcal{L}_{\mathrm{DLE}}$-formula is {\em a-soluble} (resp.\ {\em s-soluble}) if there is some $\{U_i\vdash V_i\mid i\in I\}\in inv(s\vdash \mathbf{I})$ (resp.\ $\in inv(\mathbf{I})\vdash s$) containing no logical connective.
\end{definition}
\begin{lemma}
\label{lem:soluble is primitive}
Any $\mathcal{L}_{\mathrm{DLE}}$-formula $s$ is a-soluble (resp.\ s-soluble) iff $s$ is left-primitive (resp.\ right-primitive).
\end{lemma}
\begin{proof}
If $s$ is left-primitive, then every non-leaf node in $+s$ is labelled in one of the following ways: $+f$, $-g$, $\pm \wedge$, or $\pm \vee$. Since the left-introduction (resp.\ right-introduction) rule for any $f\in \mathcal{F}$ (resp.\ $g\in \mathcal{G}$) is invertible and both introduction rules for $\wedge$ and $\vee$ are invertible, a routine induction on the shape of $s$ shows that $s$ is a-soluble. Conversely, if $s$ is not left-primitive, then there exists at least one node in $+s$ which is labelled either $-f$ or $+g$ for some $f\in \mathcal{F}$ or some $g\in \mathcal{G}$. Consider one such node $n$, and let $s'$ be the subterm of $s$ rooted at $n$. We can assume w.l.o.g.\ that all the ancestors of $n$ do not violate the left-primitive requirement. Reasoning like we did before, we can apply suitable invertible rules to all the subformulas of $s$ rooted at the nodes in the path from the root of $+s$ to the direct ancestor of $n$. Then, in the set of sequents obtained as premises of the last rule application, there will be either one sequent of the form $U_i\vdash f(s_1,\ldots, s_{n_f})$ (if $n$ is labelled $-f$) or of the form $g(s_1,\ldots, s_{n_g})\vdash V_i$ (if $n$ is labelled $+g$). In either case, since the right-introduction (resp.\ left-introduction) rule for any $f\in \mathcal{F}$ (resp.\ $g\in \mathcal{G}$) is not invertible, there is no invertible rule which can be applied to transform the main connective into a structural connective, which proves that $s$ is not a-soluble, as required.
\end{proof}
The following definition slightly generalizes the original Definition 3.9 in \cite{CiRa14} from formulas to inequalities.
\begin{definition}
Any $\mathcal{L}_{\mathrm{DLE}}$-inequality $s\leq t$ belongs to the class $\mathcal{I}_2(\mathbf{DL})$ iff there is some $\{U_i\vdash V_i\mid i\in I\}\in inv(s\vdash t)$ such that, for each $i\in I$, each antecedent-part (resp.\ succedent-part) formula in $U_i\vdash V_i$ is s-soluble (resp.\ a-soluble).
\end{definition}
\begin{prop}\label{prop:goodbranchisi2c}
The following are equivalent for any $\mathcal{L}_{\mathrm{DLE}}$-inequality $s\leq t$:
\begin{enumerate}
\item $s\leq t$ belongs to $\mathcal{I}_2(\mathbf{DL})$;
 \item every branch in $+s$ and $-t$ is good.
\end{enumerate}
\end{prop}
\begin{proof}
By Lemma \ref{lem:soluble is primitive}, a term $s$ is left-primitive (resp.\ right-primitive) if and only if $s$ is a-soluble (s-soluble). Moreover, left-primitive (resp.\ right-primitive) terms coincide with positive (resp.\ negative) Skeleton and negative (resp.\ positive) PIA terms (cf.\ discussion at the beginning of Section \ref{ssec:type3}). If $s\leq t$ is such that every branch is good, then $s\leq t$ is of the form illustrated in the picture below:

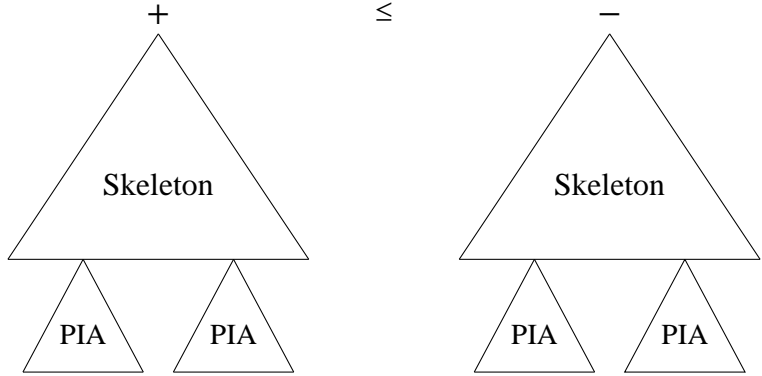
\begin{center}
	\begin{tikzpicture}
	\draw (-5,-1.5) -- (-3,1.5) node[above]{\Large$+$} ;
	\draw (-5,-1.5) -- (-1,-1.5) ;
	\draw (-3,1.5) -- (-1,-1.5);
	\draw (-3,-0.5) node{\large{Skeleton}} ;
	%\draw[dashed] (-3,1.5) -- (-4,-1.5);
	%\draw[dashed] (-3,1.5) -- (-2,-1.5);
	\draw (-4,-1.5) --(-4.8,-3);
	\draw (-4.8,-3) --(-3.2,-3);
	\draw (-3.2,-3) --(-4,-1.5);
	\draw (-2,-1.5) --(-2.8,-3);
	\draw (-2.8,-3) --(-1.2,-3);
	\draw (-1.2,-3) --(-2,-1.5);
	%\draw[dashed] (-4,-1.5) -- (-4,-3);
	%\draw[dashed] (-2,-1.5) -- (-2,-3);
	%\draw[fill] (-4,-3) circle[radius=.1] node[below]{$+p$};
	%\draw[fill] (-2,-3) circle[radius=.1] node[below]{$+p$};
	\draw (-4,-2.5) node{PIA} ;
	\draw (-2,-2.5) node{PIA} ;
	\draw (0,1.8) node{$\leq$};
	\draw (5,-1.5) -- (3,1.5) node[above]{\Large$-$} ;
	\draw (5,-1.5) -- (1,-1.5) ;
	\draw (3,1.5) -- (1,-1.5);
	\draw (3,-0.5) node{\large{Skeleton}} ;
	%\draw[dashed] (3,1.5) -- (4,-1.5);
	%\draw[dashed] (3,1.5) -- (2,-1.5);
	\draw (2,-1.5) --(2.8,-3);
	\draw (2.8,-3) --(1.2,-3);
	\draw (1.2,-3) --(2,-1.5);
	%\draw[dashed] (2,-1.5) -- (2,-3);
	%\draw[fill] (2,-3) circle[radius=.1] node[below]{$+p$};
	\draw
	(4,-1.5) -- (4.8,-3) -- (3.2,-3) -- (4, -1.5);
	%\fill[pattern=north east lines]
	%(4,-1.5) -- (4.8,-3) -- (3.2,-3) -- (4, -1.5);
	%\draw (4,-3.25)node{$\gamma$};
	\draw (4,-2.5) node{PIA} ;
	\draw (2,-2.5) node{PIA} ;
	\end{tikzpicture}
\end{center}

Then it is clear that $s\leq t$ belongs to $\mathcal{I}_2(\mathbf{DL})$. Indeed, after applying exhaustively all possible invertible rules to the Skeleton nodes, the PIA parts are ``moved to the premises'' via an application of the Ackermann rule, as discussed in Section \ref{sec:comparison}. It is straightforward but tedious to show that, when occurring in the premises, each PIA part is guaranteed to occur on the side on which it is soluble. By definition, this implies that $s\leq t$ is in $\mathcal{I}_2(\mathbf{DL})$.

As to the converse direction, notice that each step in the reasoning above can be reversed.
\end{proof}

To finish the comparison, we need to report on some definitions from \cite{CiRa14}. The following one is a slight modification of \cite[Definition 3.18]{CiRa14}, motivated by the purpose of highlighting its similarity with sets of inequalities in Ackermann shape:

\begin{definition}
A nonempty set $\mathcal{S}$ of sequents {\em respects multiplicities}
w.r.t.\ a propositional variable $p$ occurring in any of its sequents if $\mathcal{S}$ can be written in one of the following forms via application of display rules:
\[\{p \vdash U \mid p \mbox{ does not occur in } U\}\cup\{ S \vdash T\mid p \mbox{ only occurs as s-part in } S \vdash T \} \]
\[\{U \vdash p \mid p \mbox{ does not occur in } U\}\cup\{ S \vdash T\mid p \mbox{ only occurs as a-part in } S \vdash T\}.\]
\end{definition}
If $\mathcal{S}$ is a set of sequents respecting multiplicities wrt $p$, then $\mathcal{S}$ can be equivalently transformed into the set $\mathcal{S}_p$ not
containing $p$, and the transformation consists essentially in an application of Ackermann lemma.
\begin{definition}(cf.\ \cite[Definition 3.20]{CiRa14}) (the set $\mathcal{S}_p$). Let $\mathcal{S}$ be a set of sequents respecting multiplicities w.r.t.\ $p$. If $\mathcal{S}$ is uniform in $p$, in the sense that $p$ occurs always as an s-part or an a-part in each sequent of $\mathcal{S}$,
 then let $\mathcal{S}_p: = \{S\vdash T\mid S\vdash T\in \mathcal{S}$ and $p$ does not occur in $S\vdash T \}$. Otherwise, define $\mathcal{S}_p$ as the union of $\{S\vdash T\mid S\vdash T\in \mathcal{S}$ and $p$ does not occur in $S\vdash T \}$ and the set of sequents $S'\vdash T'$ obtained by substituting $p$ for any $U$ such that $p\vdash U$ is in $ \mathcal{S}$ (resp.\ $U\vdash p$ is in $\mathcal{S}$) in each sequent $S\vdash T$ in $\mathcal{S}$.
\end{definition}
The first case of the definition above corresponds to the situation in which a given variable occurring only positively or negatively is eliminated via Ackermann by suitably replacing it by $\bot$ or $\top$. Clearly, if $\mathcal{S}$ respects multiplicities w.r.t.\ $p$, then $p$ does not occur in $\mathcal{S}_p$ (cf.\ \cite[Lemma 3.21]{CiRa14}).
\begin{definition}
(cf.\ \cite[Definition 3.22]{CiRa14})(acyclic set). Let $\mathcal{C}$ display calculus. A finite set $\mathcal{S}$ of sequents built from
structure variables, structure constants and propositional variables using structural connectives is
acyclic if (i) the sequents in $\mathcal{S}$ do not contain any variables; or (ii) there exists a variable $p$ such that $\mathcal{S}$ respects multiplicities w.r.t.\ $p$ and $\mathcal{S}_p$ is acyclic.
\end{definition}
\begin{lemma}\label{lem:intermediatething}
Let $\mathcal{S}$ be an acyclic set of sequents in the variables $p_1,\ldots,p_n$ containing no logical connectives, such that for each variable $p_i$ there exist $s_1, s_2\in \mathcal{S}$ such that $p_i$ occurs in antecedent (resp.\ succedent) position in $s_1$ (resp.\ in $s_2$). Then there exists a $p$ such that $\mathcal{S}$ can be written in one of the following forms via application of display rules:
\[\{p \vdash U \mid \mbox{ no logical variable occurs in } U\}\cup\{ S \vdash T\mid p \mbox{ only occurs as s-part in } S \vdash T \} \]
\[\{U \vdash p \mid \mbox{ no logical variable occurs in } U\}\cup\{ S \vdash T\mid p \mbox{ only occurs as a-part in } S \vdash T\}.\]
\end{lemma}
\begin{proof}
By induction on the number of variables appearing in $\mathcal{S}$. If it contains only one variable, $p$, then the statement immediately follows from the fact that $\mathcal{S}$ respects multiplicities w.r.t.\ $p$.

Assume that the statement holds for sets of sequents $\mathcal{S}$ on $n$ variables, and let $\mathcal{S}$ contain $n+1$ variables. Assume for contradiction that the statement is false for each variable $p$ such that $\mathcal{S}$ respects multiplicities w.r.t.\ $p$. This means that for every such $p$, there is a sequent $p\vdash U$ (or $U\vdash p$) as above such that $U$ contains a propositional variable $q$.

Then for every such $p$, the set $\mathcal{S}_p$ inherits the same issue: Indeed substituting $U$ for $p$ cannot possibly create terms free from propositional variables, given that $U$ contains $q$. The induction hypothesis implies that each $\mathcal{S}_p$ is not acyclic. Then $\mathcal{S}$ is not acyclic, a contradiction.
\end{proof}

The following definition is aimed at adapting \cite[Definition 3.23]{CiRa14} to the setting of DLE-logics.
\begin{definition}\label{def:acyclicineq}
	(acyclic inequality). An inequality $s\leq t$ in $\mathcal{I}_2(\mathbf{DL})$ is \emph{acyclic} if there is a set $\{\rho_i\}_{i\in I}$ of semi-structural rules\footnote{A \emph{semi-structural} is a rule whose conclusion is constructed from structure variables and structure constants using structural connectives, and whose premises might additionally contain propositional variables.} which is obtained by applying the procedure described in Section \ref{sec:comparison} to $s\leq t$ such that the set of premises of
	each $\rho_i$ is acyclic.
\end{definition}

\begin{prop}
The following are equivalent for any inequality $s\leq t$:
\begin{enumerate}
\item $s\leq t$ is acyclic and belongs to $\mathcal{I}_2(\mathbf{DL})$;
\item $s\leq t$ is analytic inductive.
\end{enumerate}
\end{prop}
\begin{proof}
Let $s\leq t$ be analytic $(\Omega,\epsilon)$-inductive. Then by Proposition \ref{prop:goodbranchisi2c}, $s\leq t$ is in $\mathcal{I}_2(\mathbf{DL})$. To finish the proof we need to show that it is acyclic. This amounts to proving that the set of premises obtained by applying the Ackermann rule in the procedure described in Section \ref{sec:comparison} is acyclic. By assumption, $s\leq t$ has the following shape:

\[\xi_1(\vec{\phi}_1/\vec{x}_1,\vec{\psi}_1/\vec{y}_1, \vec{\gamma}_1/\vec{z}_1, \vec{\theta}_1/\vec{w}_1)\leq \xi_2(\vec{\psi}_2/\vec{x}_2,\vec{\phi}_2/\vec{y}_2, \vec{\theta}_2/\vec{z}_2,\vec{\gamma}_2/\vec{w}_2),\] where
$\xi_1(\vec{!x}_1, \vec{!y}_1, \vec{!z}_1, \vec{!w}_1)$ and $\xi_2(\vec{!x}_2, \vec{!y}_2, \vec{!z}_2, \vec{!w}_2)$ respectively are a positive and a negative Skeleton-formula ---cf.\ page \pageref{page: positive negative PIA}---(hence $\xi_1$ is left-primitive and $\xi_2$ is right-primitive) which are scattered, monotone in $\vec{x}$ and $\vec{z}$ and antitone in $\vec{y}$ and $\vec{w}$. Moreover, the formulas in $\vec{\phi}$ and $\vec{\gamma}$ are positive PIA (and hence right-primitive), and the formulas in $\vec{\psi}$ and $\vec{\theta}$ are negative PIA (and hence left-primitive). Finally, every $\phi$ and $\psi$ contains at least one $\varepsilon$-critical variable, whereas all $+\gamma$ and $-\theta$ are $\varepsilon^\partial$-uniform. Without loss of generality we may assume that all formulas in $\vec{\phi_i}$, $\vec{\psi_i}$, $\vec{\gamma_i}$ and $\vec{\theta_i}$ for $i\in\{1,2\}$ are definite PIA (cf.\ Footnote \ref{footnote: def definite PIA}). %\marginnote{Define definitite PIA}

\noindent Let us apply the procedure described in \cite{CiRa13, CiRa14} to the calculus $\mathbf{DL}$ and the inequality above, seen as a sequent. By exhaustively applying in reverse all invertible rules of $\mathbf{DL}$ which are applicable to the sequent we get the following: \[\Xi_1(\vec{\phi}_1/\vec{x}_1,\vec{\psi}_1/\vec{y}_1, \vec{\gamma}_1/\vec{z}_1, \vec{\theta}_1/\vec{w}_1)\vdash \Xi_2(\vec{\psi}_2/\vec{x}_2,\vec{\phi}_2/\vec{y}_2, \vec{\theta}_2/\vec{z}_2,\vec{\gamma}_2/\vec{w}_2),\] where $\Xi_1$ and $\Xi_2$ denote the structures associated with $\xi_1$ and $\xi_2$ respectively.
At this point, the procedure in \cite{CiRa13, CiRa14} calls for the display of the subformulas on which it is not possible to apply invertible rules as a-parts or s-parts of the premises of the rule-to be. The equivalence of the rule below to the sequent above is guaranteed by the Ackermann lemma:
\begin{center}
	\AxiomC{$\vec X_1\vdash\vec\phi_1\quad \quad \vec Y_2\vdash\vec\phi_2 \quad\quad \vec\psi_1\vdash\vec Y_1\quad\quad \vec\psi_2\vdash\vec X_2\quad\quad\vec Z_1\vdash\vec\gamma_1\quad \quad \vec W_2\vdash\vec\gamma_2 \quad\quad \vec\theta_1\vdash\vec W_1\quad\quad \vec\theta_2\vdash\vec Z_2$}
	\RightLabel{.}
	\UnaryInfC{$\Xi_1(\vec X_1,\vec Y_1,\vec Z_1,\vec W_1)\vdash\Xi_2(\vec X_2,\vec Y_2,\vec Z_2,\vec W_2)$}
	\DisplayProof
\end{center}
\noindent On each of the premises of the rule above, more invertible rules of $\mathbf{DL}$ can be applied in reverse. Applying them exhaustively yields
\begin{center}
	\AxiomC{$\vec X_1\vdash\vec\Phi_1\quad \quad \vec Y_2\vdash\vec\Phi_2 \quad\quad \vec\Psi_1\vdash\vec Y_1\quad\quad \vec\Psi_2\vdash\vec X_2\quad\quad\vec Z_1\vdash\vec\Gamma_1\quad \quad \vec W_2\vdash\vec\Gamma_2 \quad\quad \vec{\Theta}_1\vdash\vec W_1\quad\quad \vec{\Theta}_2\vdash\vec Z_2$}
	\RightLabel{.}
	\UnaryInfC{$\Xi_1(\vec X_1,\vec Y_1,\vec Z_1,\vec W_1)\vdash\Xi_2(\vec X_2,\vec Y_2,\vec Z_2,\vec W_2)$}
	\DisplayProof
\end{center}

By the definition of inductive inequality, if some $\Omega$-minimal variable occurs in any $\Phi$ or $\Psi$ subterm, then no other variable can occur in that subterm. Hence, the premises of the rule respect multiplicities w.r.t.\ these variables, which can then be eliminated. Likewise, one can show, by induction on $\Omega$, that all variables can be eliminated, that is, $s\leq t$ is acyclic, as required.

For the converse direction, assume that $s\leq t$ is acyclic and belongs to $\mathcal{I}_2(\mathbf{DL})$. We may assume without loss of generality that all variables in $s\leq t$ occur both positively and negatively, %in both antecedent and succedent position,
since otherwise they can be eliminated by replacing them with $\top$ and $\bot$. By Proposition \ref{prop:goodbranchisi2c}, every branch of the signed generation trees $+s$ and $-t$ is good, and by Definition \ref{def:acyclicineq} the following set is acyclic: $$\vec X_1\vdash\vec\Phi_1\quad \quad \vec Y_2\vdash\vec\Phi_2 \quad\quad \vec\Psi_1\vdash\vec Y_1\quad\quad \vec\Psi_2\vdash\vec X_2\quad\quad\vec Z_1\vdash\vec\Gamma_1\quad \quad \vec W_2\vdash\vec\Gamma_2 \quad\quad \vec{\Theta}_1\vdash\vec W_1\quad\quad \vec{\Theta}_2\vdash\vec Z_2.$$

The assumption that each variable occurs both positively and negatively implies that each variable occurs both in antecedent and in consequent position in the sequents above. Hence, by Lemma \ref{lem:intermediatething}, there exists a propositional variable $p$ such that the above set can be written in one of the following forms via application of display rules:
\[\{p \vdash U \mid \mbox{ no logical variable occurs in } U\}\cup\{ S \vdash T\mid p \mbox{ only occurs as s-part in } S \vdash T \} \]
\[\{U \vdash p \mid \mbox{ no logical variable occurs in } U\}\cup\{ S \vdash T\mid p \mbox{ only occurs as a-part in } S \vdash T\}.\]
 Let us define a strict partial order $\Omega$ and an order-type $\epsilon$ on the variables occurring in the set of premises as follows:
 We declare these $p$ as $\Omega$-minimal elements and we let $\epsilon(p):=1$ the set of premises is of the second form and $\epsilon(p):=\partial$ otherwise. Clearly, the set of premises respects multiplicities w.r.t.\ $p$ which can then be eliminated. In the new set of sequents produced the same reasoning applies. The new variable will be placed above all the $\Omega$-minimal elements. Since the set is acyclic, this process is guaranteed to end after a finite number of rounds, defining an $\epsilon$ and $\Omega$ for all the variables present. It is routine to check that $s\leq t$ is analytic $(\Omega,\epsilon)$-inductive.
\end{proof}
\end{document}